\documentclass[a4paper,11pt]{article}
\usepackage[utf8]{inputenc}
\usepackage{tikz}
\usepackage{amsthm}
\usepackage{amsmath,stmaryrd, mathrsfs}
\usepackage{amsfonts}
\usepackage{amssymb}
\usepackage{dsfont}
\usepackage{tikz-cd}
\usepackage{subcaption}
\usepackage{float}
\usepackage{url}
\usepackage[hidelinks]{hyperref}
\usepackage{booktabs}
\usepackage{import}
\usepackage{enumitem}
\usepackage[normalem]{ulem}
\usepackage{soul}
\usepackage{upgreek}
\usepackage{soul}

\usepackage{geometry}
\numberwithin{equation}{section}

\newtheorem{theorem}{Theorem}[section]
\newtheorem{lemma}[theorem]{Lemma}
\newtheorem{proposition}[theorem]{Proposition}
\newtheorem{corollary}[theorem]{Corollary}
\newtheorem{nota}[theorem]{Convention}
\newtheorem{remark}[theorem]{Remark}

\newtheorem{definition}[theorem]{Definition}

\def\trim{  \overline{T} }

\def\deg{\mathrm{deg}\,}

\renewcommand\leq{\leqslant}
\renewcommand\geq{\geqslant}
\renewcommand\rho{\varrho}

\def\bg{{\overline{\Gamma}}}
\def\hp{\hat{\Pi}_{\ve}}
\def\sT{\mathsf{D}}
\def\hC{\hat{\mC}}
\def\limI{\mathrm{I}}
\def\n{\mathbf{n}}

\newcommand{\RR}{\mathbf{R}}
\newcommand{\NN}{\mathbf{N}}
\newcommand{\ZZ}{\mathbf{Z}}

\newcommand{\TT}{\mathbf{T}}

\newcommand{\EE}{\mathbb{E}}

\newcommand{\PP}{\mathbb{P}}

\newcommand{\mG}{\mathcal{G}}
\newcommand{\mL}{\mathcal{L}}

\newcommand{\mO}{\mathcal{O}}
\newcommand{\mP}{\mathcal{P}}

\newcommand{\mC}{\mathcal{C}}
\newcommand{\mW}{\mathcal{W}}
\newcommand{\mM}{\mathcal{M}}
\newcommand{\mT}{\mathcal{T}}
\newcommand{\mI}{\mathcal{I}}

\newcommand{\mF}{\mathcal{F}}
\newcommand{\mD}{\mathcal{D}}
\newcommand{\mX}{\mathcal{X}}
\newcommand{\mR}{\mathcal{R}}
\newcommand{\mE}{\mathcal{E}}
\newcommand{\mS}{\mathcal{S}}

\newcommand{\mA}{\mathcal{A}}
\newcommand{\mH}{\mathcal{H}}

\newcommand{\mK}{\mathcal{K}}

\newcommand{\mf}[1]{\mathfrak{#1}}

\newcommand{\ve}{\varepsilon}
\newcommand{\vr}{\varrho}

\newcommand{\bigslant}[2]{{\raisebox{.17em}{$#1$}\hspace{-0.7ex}\left/\hspace{-0.2ex}\raisebox{-.17em}{$\, #2$}\right.}}
\newcommand{\lqm}{``}
\newcommand{\rqm}{'' }

\newcommand*{\ud}{\mathrm{\,d}}
 
\usetikzlibrary{shapes.misc}
\usetikzlibrary{shapes.symbols}
\usetikzlibrary{snakes}
\usetikzlibrary{decorations}
\usetikzlibrary{decorations.markings}
\usetikzlibrary{arrows.meta}
\usetikzlibrary{shapes.geometric}
\usetikzlibrary{decorations.pathmorphing,calc}

\def\drawx{\draw[-,solid, line width = 1.0mm] (-3pt,-3pt) -- (3pt,3pt);\draw[-,solid, line width = 1.0mm] (-3pt,3pt) -- (3pt,-3pt);}
\tikzset{
	dot/.style={circle,fill=black,draw=black, solid,inner sep=0pt,minimum size=1.20mm},
	blue/.style={circle,fill=blue,draw=blue, solid,inner sep=0pt,minimum size=1.20mm},
	green/.style={circle,fill=green,draw=green, solid,inner sep=0pt,minimum	size=1.20mm},
	odot/.style={circle,thick, fill=white,draw=black, solid,inner
sep=0pt,minimum size=1.05mm},
	root/.style={circle, thick, fill=white,draw=purple, solid,inner
sep=0pt,minimum size=1.05mm},
	bigroot/.style={circle,very thick, fill=white,draw=purple, solid,inner
sep=0pt,minimum size=1.5mm},
	rootb/.style={circle, thick, fill=white,draw=blue, solid,inner sep=0pt,minimum size=0.95mm},
	idot/.style={circle, thick, fill=white,draw=purple, solid,inner sep=0pt,minimum size=0.95mm},
	wdot/.style={circle, thick, fill=white,draw=black, solid,inner sep=0pt,minimum size=0.95mm},
	square/.style={rectangle,fill=black,draw=black, solid,inner sep=0pt,minimum size=1mm},
	empty/.style={circle,fill=white,draw=white, solid,inner sep=0pt,minimum size=0.5mm},
	var/.style={circle,fill=black!10,draw=black,inner sep=0pt, minimum size=
	2mm},
	symb/.style={circle,fill=symbols,draw=symbols, solid,inner sep=0pt,minimum size=0.5mm},
	yy/.style={circle,fill=gray!20,draw=black,inner sep=0pt,minimum size=0.8mm},
	>=stealth,
	dotred/.style={circle,fill=black!50,inner sep=0pt, minimum size=2mm},
	generic/.style={semithick,shorten >=1pt,shorten <=1pt},
	dist/.style={ultra thick,draw=testcolor,shorten >=1pt,shorten <=1pt},
	testfcn/.style={ultra thick,testcolor,shorten >=1pt,shorten <=1pt,<-},
	testfcnx/.style={ultra thick,testcolor,shorten >=1pt,shorten <=1pt,<-,
		postaction={decorate,decoration={markings,mark=at position 0.6 with {\drawx}}}},
	kprime/.style={semithick,shorten >=1pt,shorten <=1pt,densely dashed,->},
	kprimex/.style={semithick,shorten >=1pt,shorten <=1pt,densely dashed,->,
		postaction={decorate,decoration={markings,mark=at position 0.4 with {\drawx}}}},
	kernel/.style={semithick,shorten >=1pt,shorten <=1pt,->},
	multx/.style={shorten >=1pt,shorten <=1pt,
		postaction={decorate,decoration={markings,mark=at position 0.5 with {\drawx}}}},
	kernelx/.style={semithick,shorten >=1pt,shorten <=1pt,->,
		postaction={decorate,decoration={markings,mark=at position 0.4 with {\drawx}}}},
	kernel1/.style={->,semithick,shorten >=1pt,shorten <=1pt,postaction={decorate,decoration={markings,mark=at position 0.45 with {\draw[-] (0,-0.1) -- (0,0.1);}}}},
	kernel2/.style={->,semithick,shorten >=1pt,shorten <=1pt,postaction={decorate,decoration={markings,mark=at position 0.45 with {\draw[-] (0.05,-0.1) -- (0.05,0.1);\draw[-] (-0.05,-0.1) -- (-0.05,0.1);}}}},
	kernelBig/.style={semithick,shorten >=1pt,shorten <=1pt,decorate, decoration={zigzag,amplitude=1.5pt,segment length = 3pt,pre length=2pt,post length=2pt}},
	rho/.style={dotted,semithick,shorten >=1pt,shorten <=1pt},
	renorm/.style={shape=circle,fill=white,inner sep=1pt},
	labl/.style={shape=rectangle,fill=white,inner sep=1pt},
	xi/.style={circle,fill=symbols!10,draw=symbols,inner sep=0pt,minimum size=1.2mm},
	xix/.style={crosscircle,fill=symbols!10,draw=symbols,inner sep=0pt,minimum size=1.2mm},
	xib/.style={circle,fill=symbols!10,draw=symbols,inner sep=0pt,minimum size=1.6mm},
	xibx/.style={crosscircle,fill=symbols!10,draw=symbols,inner sep=0pt,minimum size=1.6mm},
	not/.style={circle,fill=symbols,draw=symbols,inner sep=0pt,minimum size=0.5mm},
	>=stealth,
        midarrow/.style={ decoration={ markings, mark=at position 0.7 with
		{\arrow{>}} }, postaction={decorate} }
	}

\makeatletter
\def\DeclareSymbol#1#2#3{\expandafter\gdef\csname
MH@symb@#1\endcsname{\tikz[baseline=#2,scale=0.4]{#3}}
\expandafter\gdef\csname
MH@symb@#1s\endcsname{\scalebox{0.6}{\tikz[baseline=#2,scale=0.3]{#3}}}}
\def\<#1>{\csname MH@symb@#1\endcsname}
\makeatother

\tikzset{every loop/.style={}}

\DeclareSymbol{rl}{0}{
  \node[dot] (x) at (0,0.4) {};
  \node[dot] (y) at (2,0.4) {};
  \draw[purple, thick] (x) -- (y);
  \node at (x.south) [below=2pt] {$x$};
  \node at (y.south) [below=2pt] {$y$};
}

\DeclareSymbol{single}{0}{ \node[root] (r) at (1,0) { }}

\DeclareSymbol{0}{0}{
  \node[root] (r) at (0,-.3) { };
  \node[root] (t) at (0,.7) { };
  \draw[midarrow] (r)  -- (t);
}

\DeclareSymbol{0_b}{0}{
  \node[dot] (r) at (0,-.3) { };
  \node[dot] (t) at (0,.7) { };
  \draw[midarrow] (r)  -- (t);
}

\DeclareSymbol{0_bl}{0}{
  \node[dot] (r) at (0,-.3) {};
  \node[dot] (t) at (0,.7) {};
  \draw[midarrow] (r)  -- (t);
  \node[font=\scriptsize, anchor=west] at (t) {$v_{\Gamma_{12}}$};
  \node[font=\scriptsize, anchor=west] at (r) {$v_{\Gamma_{34}}$};
}

\DeclareSymbol{1}{0}{
  \node[root] (r) at (0,-0.8) {};
  \node[odot] (m) at (0,0.2) {};
  \node[root] (t) at (0,1.2) {};
  \draw[midarrow] (r.north) -- (m.south) {};
  \draw[midarrow] (m.north) -- (t.south) {};
}

\DeclareSymbol{1_b}{0}{
  \node[root] (r) at (0,-0.8) {};
  \node[dot] (m) at (0,0.2) {};
  \node[root] (t) at (0,1.2) {};
  \draw[midarrow] (r.north) -- (m.south) {};
  \draw[midarrow] (m.north) -- (t.south) {};
}

\DeclareSymbol{1_bl}{0}{
  \node[root] (r) at (0,-0.8) {};
  \node[dot] (m) at (0,0.2) {};
  \node[root] (t) at (0,1.2) {};
  \draw[midarrow] (r.north) -- (m.south) {};
  \draw[midarrow] (m.north) -- (t.south) {};
  \node[font=\scriptsize, anchor=west] at (m) {$v_{\Gamma_{1234}}$};
}

\DeclareSymbol{2}{0}{
  \node[root] (r) at (0,-1.3) {};
  \node[odot] (m1) at (0,-0.3) {}; 
  \node[odot] (m2) at (0,0.7) {}; 
  \node[root] (t) at (0,1.7) {}; 
  \draw[midarrow] (r.north) -- (m1.south);
  \draw[midarrow] (m1.north) -- (m2.south);
  \draw[midarrow] (m2.north) -- (t.south);
}

\DeclareSymbol{2_b}{0}{
  \node[root] (r) at (0,-1.3) {};
  \node[odot] (m1) at (0,-0.3) {}; 
  \node[dot] (m2) at (0,0.7) {}; 
  \node[root] (t) at (0,1.7) {}; 
  \draw[midarrow] (r.north) -- (m1.south);
  \draw[midarrow] (m1.north) -- (m2.south);
  \draw[midarrow] (m2.north) -- (t.south);
}

\DeclareSymbol{3}{0}{
  \node[root] (r) at (0,-1.8) {};
  \node[odot] (m1) at (0,-0.8) {};
  \node[odot] (m2) at (0,0.2) {};
  \node[odot] (m3) at (0,1.2) {};
  \node[root] (t) at (0,2.2) {};
  \draw[midarrow] (r.north) -- (m1.south);
  \draw[midarrow] (m1.north) -- (m2.south);
  \draw[midarrow] (m2.north) -- (m3.south);
  \draw[midarrow] (m3.north) -- (t.south);
}

\DeclareSymbol{c3_1}{0}{
  \node[root] (r) at (0,-1.3) {};
  \node[odot] (m1) at (0,-0.3) {};
  \node[dot] (m2) at (0,0.7) {};
  \node[root] (t) at (0,1.7) {};
  \draw[midarrow] (r.north) -- (m1.south);
  \draw[midarrow] (m1.north) -- (m2.south);
  \draw[midarrow] (m2.north) -- (t.south);
  \draw[midarrow] (m2) to[out=-45,in=45,looseness=11] (m2);
}

\DeclareSymbol{c3_2}{0}{
  \node[root] (r) at (0,-1.3) {};
  \node[dot] (m1) at (0,-0.3) {};
  \node[odot] (m2) at (0,0.7) {};
  \node[root] (t) at (0,1.7) {};
  \draw[midarrow] (r.north) -- (m1.south);
  \draw[midarrow] (m1.north) -- (m2.south);
  \draw[midarrow] (m2.north) -- (t.south);
  \draw[midarrow] (m1) to[out=-45,in=45,looseness=11] (m1);
}

\DeclareSymbol{c3_3}{0}{
  \node[odot] (m1) at (0,-0.8) {};
  \node[dot] (m2) at (0,0.2) {};
  \node[root] (r) at (-0.5,1.2) {};
  \node[root] (t) at (0.5,1.2) {};
  \draw[midarrow] (m1.east) to[out=0,in=0] (m2.east);
  \draw[midarrow] (m2.west) to[out=180, in=180] (m1.west);
  \draw[midarrow] (m2.north) -- (t);
  \draw[midarrow] (m2.north) -- (r);
}

\DeclareSymbol{c4}{0}{
  \node[root] (r) at (0,-2.2) {};
  \node[odot] (m1) at (0,-1.2) {};
  \node[odot] (m2) at (0,-.2) {};
  \node[odot] (m3) at (0,.8) {};
  \node[odot] (m4) at (0,1.8) {};
  \node[root] (t) at (0,2.8) {};
  \draw[midarrow] (r.north) -- (m1.south);
  \draw[midarrow] (m1.north) -- (m2.south);
  \draw[midarrow] (m2.north) -- (m3.south);
  \draw[midarrow] (m3.north) -- (m4.south);
  \draw[midarrow] (m4.north) -- (t.south);
  \draw[red, very thick, dash pattern=on 0.5pt off 0.5pt] (m1) to[out=0,in=0] (m3);
  \draw[red, very thick, dash pattern=on 0.5pt off 0.5pt] (m2) to[out=180,in=180] (m4);
}

\DeclareSymbol{cc4}{0}{
  \node[root] (r) at (0,-1.2) {};
  \node[dot] (m1) at (0,-.2) {};
  \node[dot] (m2) at (0,.8) {};
  \node[root] (t) at (0,1.8) {};
  \draw[midarrow] (r.north) -- (m1.south);
  \draw[midarrow] (m1.north) -- (m2.south);
  \draw[midarrow] (m1) to[out=0,in=0] (m2);
  \draw[midarrow] (m2) to[out=180,in=180] (m1);
  \draw[midarrow] (m2.north) -- (t.south);
}

\DeclareSymbol{cc4div}{0}{
  \node[dot] (m1) at (0,-.2) {};
  \node[dot] (m2) at (0,.8) {};
  \draw[midarrow] (m1.north) -- (m2.south);
  \draw[midarrow] (m1) to[out=0,in=0] (m2);
  \draw[midarrow] (m2) to[out=180,in=180] (m1);
}

\DeclareSymbol{c21_half}{0}{
  \node[root] (r) at (0,-1.3) {};
  \node[odot] (m1) at (0,-.3) {};
  \node[odot] (m2) at (0,.7) {};
  \node[root] (t) at (0,1.7) {};
  \draw[midarrow] (r.north) -- (m1.south);
  \draw[midarrow] (m1.north) -- (m2.south);
  \draw[midarrow] (m2.north) -- (t.south);
  \draw[red, very thick, dash pattern=on 0.5pt off 0.5pt] (m1) to[out=0,in=0] (m2);
}

\DeclareSymbol{c21_halfMin}{0}{
\node[root] (r) at (0,0) { };
\node[rootb] (t) at (0,1.2) {};
\draw (r)  -- (t);
\draw (0,0.6) to[out=45,in=90] (0.6,0.6);
\draw (0,0.6) to[out=-45,in=-90] (0.6,0.6);
\node[wdot] at (0,0.6){};
}

\DeclareSymbol{c31_half}{0}{
\node[root] (r) at (0,0) { };
\node[dot] at (0,0.6){};
\node[dot] at (0,1.2) {};
\node[dot] at (0,1.8) {};
\node[rootb] (t) at (0,2.4) {};
\draw (r)  -- (t);
\draw [red, very thick, dash pattern=on 0.5pt off 0.5pt] (0,0.6) to[out=0,in=0] (0.0,1.2);
}

\DeclareSymbol{c31_halfMin}{0}{
\node[root] (r) at (0,0) { };
\node[dot] at (0,1.2) {};
\node[rootb] (t) at (0,1.8) {};
\draw (r)  -- (t);
\draw (0,0.6) to[out=45,in=90] (0.6,0.6);
\draw (0,0.6) to[out=-45,in=-90] (0.6,0.6);
\node[wdot] at (0,0.6){};
}

\DeclareSymbol{c32_half}{0}{
\node[root] (r) at (0,0) { };
\node[dot] at (0,0.6){};
\node[dot] at (0,1.2) {};
\node[dot] at (0,1.8) {};
\node[rootb] (t) at (0,2.4) {};
\draw (r)  -- (t);
\draw [red, very thick, dash pattern=on 0.5pt off 0.5pt] (0,1.2) to[out=0,in=0] (0.0,1.8);
}

\DeclareSymbol{c32_halfMin}{0}{
\node[root] (r) at (0,0) { };
\node[dot] at (0,0.6){};
\node[rootb] (t) at (0,1.8) {};
\draw (r)  -- (t);
\draw (0,1.2) to[out=45,in=90] (0.6,1.2);
\draw (0,1.2) to[out=-45,in=-90] (0.6,1.2);
\node[wdot] at (0,1.2) {};
}

\DeclareSymbol{c33_half}{0}{
\node[root] (r) at (0,0) { };
\node[dot] at (0,0.6){};
\node[dot] at (0,1.2) {};
\node[dot] at (0,1.8) {};
\node[rootb] (t) at (0,2.4) {};
\draw (r)  -- (t);
\draw [red, very thick, dash pattern=on 0.5pt off 0.5pt] (0,0.6) to[out=0,in=0] (0.0,1.8);
}

\DeclareSymbol{c33_halfMin}{0}{
\node[root] (r) at (0,0) { };
\node[dot] at (0,1.2) {};
\node[rootb] (t) at (.6,1) {};
\draw (0,0.6) to[out=50,in=-50] (0,1.2);
\draw (0,0.6) to[out=130,in=-130] (0,1.2);
\node[wdot] (m1) at (0,0.6){};
\draw (r)  -- (m1);
\draw (m1)  -- (t);
}

\DeclareSymbol{c1}{0}{
  \node[root] (r) at (0,-0.7) {};
  \node[odot] (m) at (0,0.3) {};
  \node[root] (t) at (0,1.3) {};
  \draw[midarrow] (r.north) -- (m.south);
  \draw[midarrow] (m.north) -- (t.south);
  \node[root] (r1) at (0.8,-0.7) {};
  \node[odot] (m1) at (0.8,0.3) {};
  \node[root] (t1) at (0.8,1.3) {};
  \draw[midarrow] (r1.north) -- (m1.south);
  \draw[midarrow] (m1.north) -- (t1.south);
  \draw[red, very thick, dash pattern=on 0.5pt off 0.5pt] (m) -- (m1);
}

\DeclareSymbol{cc1}{0}{
  \node[root] (r) at (0,-.55) {};
  \node[root] (r1) at (1.5,-.55) {};
  \node[dot] (m) at (.75,.2) {};
  \node[root] (t) at (0,.95) {};
  \node[root] (t1) at (1.5,.95) {};
  \draw[midarrow] (r) -- (m);
  \draw[midarrow] (r1) -- (m);
  \draw[midarrow] (m) -- (t);
  \draw[midarrow] (m) -- (t1);
}

\DeclareSymbol{cc1_narrow}{0}{
  \node[root] (r) at (0,-.4) {};
  \node[root] (r1) at (1.2,-.4) {};
  \node[dot] (m) at (.6,.2) {};
  \node[root] (t) at (0,.8) {};
  \node[root] (t1) at (1.2,.8) {};
  \draw (r) -- (m);
  \draw (r1) -- (m);
  \draw (m) -- (t);
  \draw (m) -- (t1);
}

\DeclareSymbol{c21}{0}{
  \node[root] (r) at (0,-1.2) {};
  \node[odot] (m1) at (0,-.2) {};
  \node[odot] (m2) at (0,.8) {};
  \node[root] (t) at (0,1.8) {};
  \draw[midarrow] (r.north) -- (m1.south);
  \draw[midarrow] (m1.north) -- (m2.south);
  \draw[midarrow] (m2.north) -- (t.south);
  \node[root] (r1) at (0.8,-1.2) {};
  \node[odot] (n1) at (0.8,-0.2) {};
  \node[odot] (n2) at (0.8,0.8) {};
  \node[root] (t1) at (0.8,1.8) {};
  \draw[midarrow] (r1.north) -- (n1.south);
  \draw[midarrow] (n1.north) -- (n2.south);
  \draw[midarrow] (n2.north) -- (t1.south);
  \draw[red, very thick, dash pattern=on 0.5pt off 0.5pt] (m1) to[out=180,in=180] (m2);
  \draw[red, very thick, dash pattern=on 0.5pt off 0.5pt] (n1) to[out=0,in=0] (n2);
}

\DeclareSymbol{cc21}{0}{
  \node[root] (r) at (0,-0.8) {};
  \node[dot] (m) at (0,0.2) {};
  \node[root] (t) at (0,1.2) {};
  \draw[midarrow] (r) -- (m);
  \draw[midarrow] (m) -- (t);
\draw[midarrow] (m) to[out=-135,in=135,looseness=10] (m);
  \node[root] (r1) at (0.8,-0.8) {};
  \node[dot] (n) at (0.8,0.2) {};
  \node[root] (t1) at (0.8,1.2) {};
  \draw[midarrow] (r1) -- (n);
  \draw[midarrow] (n) -- (t1);
\draw[midarrow] (n) to[out=-45,in=45,looseness=10] (n);
}

\DeclareSymbol{c21_bl}{0}{
  \node[root] (r) at (0,-0.7) {};
  \node[dot] (m) at (0,0.3) {};
  \node[root] (t) at (0,1.3) {};
  \draw[midarrow] (r) -- (m);
  \draw[midarrow] (m) -- (t);
  \draw[midarrow] (m) to[out=-45,in=45,looseness=11] (m);
}

\DeclareSymbol{loop}{0}{
  \node[dot] (m) at (0,0.3) {};
  \draw[midarrow] (m) to[out=-45,in=45,looseness=11] (m);
}

\DeclareSymbol{c22}{0}{
  \node[root] (r) at (0,-1.2) {};
  \node[odot] (m1) at (0,-.2) {};
  \node[odot] (m2) at (0,.7) {};
  \node[root] (t) at (0,1.7) {};
  \draw[midarrow] (r.north) -- (m1.south);
  \draw[midarrow] (m1.north) -- (m2.south);
  \draw[midarrow] (m2.north) -- (t.south);
  \node[root] (r1) at (0.8,-1.2) {};
  \node[odot] (n1) at (0.8,-.2) {};
  \node[odot] (n2) at (0.8,.7) {};
  \node[root] (t1) at (0.8,1.7) {};
  \draw[midarrow] (r1.north) -- (n1.south);
  \draw[midarrow] (n1.north) -- (n2.south);
  \draw[midarrow] (n2.north) -- (t1.south);
  \draw[red, very thick, dash pattern=on 0.5pt off 0.5pt] (m1) -- (n1);
  \draw[red, very thick, dash pattern=on 0.5pt off 0.5pt] (m2) -- (n2);
}

\DeclareSymbol{cc22}{0}{
  \node[root] (r) at (0,-1.2) {};
  \node[root] (r1) at (0.8,-1.2) {};
  \node[dot] (d1) at (0.4,-.2) {};
  \node[dot] (d2) at (0.4,.7) {};
  \node[root] (t) at (0,1.7) {};
  \node[root] (t1) at (0.8,1.7) {};
  \draw[midarrow] (r) -- (d1);
  \draw[midarrow] (r1) -- (d1);
  \draw[midarrow] (d1) to[out=0,in=0] (d2);
  \draw[midarrow] (d1) to[out=180,in=180] (d2);
  \draw[midarrow] (d2) -- (t);
  \draw[midarrow] (d2) -- (t1);
}

\DeclareSymbol{ccc22}{0}{
  \node[root] (r) at (0,-0.5) {};
  \node[root] (r1) at (0.8,-0.5) {};
  \node[dot] (d1) at (0.4,0.25) {};
  \node[dot] (d2) at (0.4,0.85) {};
  \node[root] (t) at (0,1.5) {};
  \node[root] (t1) at (0.8,1.5) {};
  \draw (r) -- (d1);
  \draw (r1) -- (d1);
  \draw[purple, thick] (d1) -- (d2);
  \draw (d2) -- (t);
  \draw (d2) -- (t1);
}

\DeclareSymbol{bubbleA}{0}{
  \node[dot] (d1) at (0.4,-.2) {};
  \node[dot] (d2) at (0.4,.7) {};
  \draw[midarrow] (d1) to[out=0,in=0] (d2);
  \draw[midarrow] (d1) to[out=180,in=180] (d2);
}

\DeclareSymbol{bubbleB}{0}{
  \node[dot] (d1) at (0.4,-.2) {};
  \node[dot] (d2) at (0.4,.7) {};
  \draw[midarrow] (d2) to[out=0,in=0] (d1);
  \draw[midarrow] (d1) to[out=180,in=180] (d2);
}

\DeclareSymbol{cccc22}{0}{
  \node[root] (r) at (0,-0.5) {};
  \node[root] (r1) at (0.8,-0.5) {};
  \node[dot] (d) at (0.4,0.55) {};
  \node[root] (t) at (0,1.5) {};
  \node[root] (t1) at (0.8,1.5) {};
  \draw (r) -- (d);
  \draw (r1) -- (d);
  \draw (d) -- (t);
  \draw (d) -- (t1);
}

\DeclareSymbol{c23}{0}{
  \node[root] (r) at (0,-1.2) {};
  \node[odot] (m1) at (0,-.2) {};
  \node[odot] (m2) at (0,0.8) {};
  \node[root] (t) at (0,1.8) {};
  \draw[midarrow] (r.north) -- (m1.south);
  \draw[midarrow] (m1.north) -- (m2.south);
  \draw[midarrow] (m2.north) -- (t.south);
  \node[root] (r1) at (0.8,-1.2) {};
  \node[odot] (n1) at (0.8,-.2) {};
  \node[odot] (n2) at (0.8,0.8) {};
  \node[root] (t1) at (0.8,1.8) {};
  \draw[midarrow] (r1.north) -- (n1.south);
  \draw[midarrow] (n1.north) -- (n2.south);
  \draw[midarrow] (n2.north) -- (t1.south);
  \draw[red, very thick, dash pattern=on 0.5pt off 0.5pt] (m1) -- (n2);
  \draw[red, very thick, dash pattern=on 0.5pt off 0.5pt] (m2) -- (n1);
}

\DeclareSymbol{dc23}{0}{
  \node[root] (r) at (0,-1.0) {};
  \node[odot] (m1) at (0,0.0) {};
  \node[odot] (m2) at (0,1.0) {};
  \node[root] (t) at (0,2.0) {};
  \draw[midarrow] (r.north) -- (m1.south);
  \draw[midarrow] (m1.north) -- (m2.south);
  \draw[midarrow] (m2.north) -- (t.south);
  \node[root] (r1) at (1.2,-1.0) {};
  \node[odot] (n1) at (1.2,0.0) {};
  \node[odot] (n2) at (1.2,1.0) {};
  \node[root] (t1) at (1.2,2.0) {};
  \draw[midarrow] (r1.north) -- (n1.south);
  \draw[midarrow] (n1.north) -- (n2.south);
  \draw[midarrow] (n2.north) -- (t1.south);
  \draw[red, very thick, dash pattern=on 0.5pt off 0.5pt] (m1) -- (n2);
  \draw[red, very thick, dash pattern=on 0.5pt off 0.5pt] (m2) -- (n1);
}

\DeclareSymbol{cc23}{0}{
  \node[root] (r) at (0,-1.3) {};
  \node[root] (r1) at (0.8,-1.3) {};
  \node[dot] (d1) at (0.4,-.3) {};
  \node[dot] (d2) at (0.4,.7) {};
  \node[root] (t) at (0,1.7) {};
  \node[root] (t1) at (0.8,1.7) {};
  \draw[midarrow] (r) -- (d1);
  \draw[midarrow] (d1) -- (r1);
  \draw[midarrow] (d2) to[out=0,in=0] (d1);
  \draw[midarrow] (d1) to[out=180,in=180] (d2);
  \draw[midarrow] (d2) -- (t);
  \draw[midarrow] (t1) -- (d2);
}

\DeclareSymbol{dcc23}{0}{
  \node[root] (r) at (0,-1.0) {};
  \node[root] (r1) at (1.2,-1.0) {};
  \node[dot] (d1) at (0.6,0.00) {};
  \node[dot] (d2) at (0.6,1.00) {};
  \node[root] (t) at (0,2.0) {};
  \node[root] (t1) at (1.2,2.0) {};
  \draw[midarrow] (r) -- (d1);
  \draw[midarrow] (d1) to (r1);
  \draw[midarrow] (d1) to[out=0,in=0] (d2);
  \draw[midarrow] (d2) to[out=180,in=180] (d1);
  \draw[midarrow] (d2) -- (t);
  \draw[midarrow] (t1) -- (d2);
}

\DeclareSymbol{c31}{0}{
  \node[root] (r) at (0.8,-0.6) {};
  \node[odot] (m1) at (0.8,0.0) {};
  \node[odot] (m2) at (0.8,0.6) {};
  \node[odot] (m3) at (0.8,1.2) {};
  \node[root] (t) at (0.8,1.8) {};
  \draw (r.north) -- (m1.south);
  \draw (m1.north) -- (m2.south);
  \draw (m2.north) -- (m3.south);
  \draw (m3.north) -- (t.south);
  \draw[red, very thick, dash pattern=on 0.5pt off 0.5pt] (m1.east) to[out=0, in=0] (m3.east);
  \node[root] (r1) at (0,-0.6) {};
  \node[odot] (m11) at (0,0.0) {};
  \node[odot] (m21) at (0,0.6) {};
  \node[odot] (m31) at (0,1.2) {};
  \node[root] (t1) at (0,1.8) {};
  \draw (r1.north) -- (m11.south);
  \draw (m11.north) -- (m21.south);
  \draw (m21.north) -- (m31.south);
  \draw (m31.north) -- (t1.south);
  \draw[red, very thick, dash pattern=on 0.5pt off 0.5pt] (m11.west) to[out=180, in=180] (m31.west);
  \draw[red, very thick, dash pattern=on 0.5pt off 0.5pt] (m21.east) to[out=0, in=180] (m2.west);
}

\DeclareSymbol{cc31}{0}{
  \node[root] (r) at (0,-0.6) {};
  \node[root] (r1) at (0.8,-0.6) {};
  \node[root] (r2) at (1.6,-0.6) {};
  \node[root] (r3) at (2.4,-0.6) {};
  \node[dot] (m) at (0.4,0.6) {};
  \node[dot] (m2) at (2.0,0.6) {};
  \node[dot] (t) at (1.2,1.8) {};
  \draw (r.north) -- (m.south);
  \draw (r1.north) -- (m.south);
  \draw (r2.north) -- (m2.south);
  \draw (r3.north) -- (m2.south);
  \draw[purple, thick] (m) to (t);
  \draw[purple, thick] (m2) to (t);
}

\DeclareSymbol{c41}{0}{
  \node[root] (r) at (1.2,-2.2) {};
  \node[odot] (m1) at (1.2,-1.2) {};
  \node[odot] (m2) at (1.2,-.2) {};
  \node[odot] (m3) at (1.2,.8) {};
  \node[odot] (m4) at (1.2,1.8) {};
  \node[root] (t) at (1.2,2.8) {};
  \draw[midarrow] (r.north) -- (m1.south);
  \draw[midarrow] (m1.north) -- (m2.south);
  \draw[midarrow] (m2.north) -- (m3.south);
  \draw[midarrow] (m3.north) -- (m4.south);
  \draw[midarrow] (m4.north) -- (t.south);
  \node[root] (r1) at (0,-2.2) {};
  \node[odot] (m11) at (0,-1.2) {};
  \node[odot] (m21) at (0,-.2) {};
  \node[odot] (m31) at (0,.8) {};
  \node[odot] (m41) at (0,1.8) {};
  \node[root] (t1) at (0,2.8) {};
  \draw[midarrow] (r1.north) -- (m11.south);
  \draw[midarrow] (m11.north) -- (m21.south);
  \draw[midarrow] (m21.north) -- (m31.south);
  \draw[midarrow] (m31.north) -- (m41.south);
  \draw[midarrow] (m41.north) -- (t1.south);
  \draw[red, very thick, dash pattern=on 0.5pt off 0.5pt] (m41.east) to[out=0, in=180] (m2.west);
  \draw[red, very thick, dash pattern=on 0.5pt off 0.5pt] (m31) to[out=45, in=240] (m4);
  \draw[red, very thick, dash pattern=on 0.5pt off 0.5pt] (m21.east) to[out=0, in=180] (m1.west);
  \draw[red, very thick, dash pattern=on 0.5pt off 0.5pt] (m11) to[out=45, in=240] (m3);
}

\DeclareSymbol{ccc41}{0}{
  \node[root] (r1) at (-1.5,-1.5) {};
  \node[dot] (m1) at (-.75,-.75) {};
  \node[dot] (m2) at (-.75,.75) {};
  \node[dot] (m3) at (.75,-.75) {};
  \node[dot] (m4) at (.75,.75) {};
  \node[root] (r2) at (-1.5,1.5) {};
  \node[root] (r3) at (1.5,-1.5) {};
  \node[root] (r4) at (1.5,1.5) {};
  \draw[midarrow] (r1) -- (m1);
  \draw[midarrow] (r2) -- (m2);
  \draw[midarrow] (m3) -- (r3);
  \draw[midarrow] (m4) -- (r4);
  \draw[midarrow] (m1) -- (m2);
  \draw[midarrow] (m2) -- (m3);
  \draw[midarrow] (m4) -- (m3);
  \draw[midarrow] (m2) --(m4);
  \draw[midarrow] (m3) -- (m1);
  \draw[midarrow] (m1) -- (m4);
}

\DeclareSymbol{cc41}{0}{
  \node[root] (r) at (0.8,-2.2) {};
  \node[dot] (m1) at (0.8,-1.2) {};
  \node[dot] (m2) at (0.8,-.2) {};
  \node[dot] (m3) at (0.8,0.8) {};
  \node[dot] (m4) at (0.8,1.8) {};
  \node[root] (t) at (0.8,2.8) {};
  \node[root] (t1) at (0.0,.8) {};
  \node[root] (r1) at (0,-.2) {};
  \draw[midarrow] (r.north) -- (m1.south);
  \draw[midarrow] (m1.north) -- (m2.south);
  \draw[midarrow] (m2.north) -- (m3.south);
  \draw[midarrow] (m3.north) -- (m4.south);
  \draw[midarrow] (m4.north) -- (t.south);
  \draw[midarrow] (m1.west) to[out=140, in =220] (m3.west);
  \draw[midarrow] (m4.east) to[out=0, in =0] (m1.east);
  \draw[midarrow] (m2) to[out=30, in =-30] (m4);
  \draw[midarrow] (m3) to (t1);
  \draw[midarrow] (r1) to (m2);
}

\DeclareSymbol{c42}{0}{
  \node[root] (r) at (1.2,-2.2) {};
  \node[odot] (m1) at (1.2,-1.2) {};
  \node[odot] (m2) at (1.2,-.2) {};
  \node[odot] (m3) at (1.2,.8) {};
  \node[odot] (m4) at (1.2,1.8) {};
  \node[root] (t) at (1.2,2.8) {};
  \draw[midarrow] (r.north) -- (m1.south);
  \draw[midarrow] (m1.north) -- (m2.south);
  \draw[midarrow] (m2.north) -- (m3.south);
  \draw[midarrow] (m3.north) -- (m4.south);
  \draw[midarrow] (m4.north) -- (t.south);
  \node[root] (r1) at (0,-2.2) {};
  \node[odot] (m11) at (0,-1.2) {};
  \node[odot] (m21) at (0,-.2) {};
  \node[odot] (m31) at (0,.8) {};
  \node[odot] (m41) at (0,1.8) {};
  \node[root] (t1) at (0,2.8) {};
  \draw[midarrow] (r1.north) -- (m11.south);
  \draw[midarrow] (m11.north) -- (m21.south);
  \draw[midarrow] (m21.north) -- (m31.south);
  \draw[midarrow] (m31.north) -- (m41.south);
  \draw[midarrow] (m41.north) -- (t1.south);
  \draw[red, very thick, dash pattern=on 0.5pt off 0.5pt] (m41.east) to[out=0, in=180] (m4.west);
  \draw[red, very thick, dash pattern=on 0.5pt off 0.5pt] (m31) to[out=-20, in=160] (m1);
  \draw[red, very thick, dash pattern=on 0.5pt off 0.5pt] (m21.east) to[out=0, in=180] (m3.west);
  \draw[red, very thick, dash pattern=on 0.5pt off 0.5pt] (m11) to[out=45, in=240] (m2);
}

\DeclareSymbol{c42plain}{0}{
  \node[root] (r) at (1.2,-2.2) {};
  \node[odot] (m1) at (1.2,-1.2) {};
  \node[odot] (m2) at (1.2,-.2) {};
  \node[odot] (m3) at (1.2,.8) {};
  \node[odot] (m4) at (1.2,1.8) {};
  \node[root] (t) at (1.2,2.8) {};
  \draw[midarrow] (r.north) -- (m1.south);
  \draw[midarrow] (m1.north) -- (m2.south);
  \draw[midarrow] (m2.north) -- (m3.south);
  \draw[midarrow] (m3.north) -- (m4.south);
  \draw[midarrow] (m4.north) -- (t.south);
  \node[root] (r1) at (0,-2.2) {};
  \node[odot] (m11) at (0,-1.2) {};
  \node[odot] (m21) at (0,-.2) {};
  \node[odot] (m31) at (0,.8) {};
  \node[odot] (m41) at (0,1.8) {};
  \node[root] (t1) at (0,2.8) {};
  \draw[midarrow] (r1.north) -- (m11.south);
  \draw[midarrow] (m11.north) -- (m21.south);
  \draw[midarrow] (m21.north) -- (m31.south);
  \draw[midarrow] (m31.north) -- (m41.south);
  \draw[midarrow] (m41.north) -- (t1.south);
  \draw[red, very thick, dash pattern=on 0.5pt off 0.5pt] (m41.east) to[out=0, in=180] (m4.west);
  \draw[red, very thick, dash pattern=on 0.5pt off 0.5pt] (m31.east) to[out=0, in=180] (m3.west);
  \draw[red, very thick, dash pattern=on 0.5pt off 0.5pt] (m21.east) to[out=0, in=180] (m2.west);
  \draw[red, very thick, dash pattern=on 0.5pt off 0.5pt] (m11.east) to[out=0, in=180] (m1.west);
}

\DeclareSymbol{cc42}{0}{
  \node[root] (r) at (0.8,-2.2) {};
  \node[dot] (m1) at (0.8,-1.2) {};
  \node[dot] (m2) at (0.8,-.2) {};
  \node[dot] (m3) at (0.8,0.8) {};
  \node[dot] (m4) at (0.8,1.8) {};
  \node[root] (t) at (0.8,2.8) {};
  \node[root] (t1) at (0.0,1.8) {};
  \node[root] (r1) at (0,.8) {};
  \draw[midarrow] (r.north) -- (m1.south);
  \draw[midarrow] (m1.north) -- (m2.south);
  \draw[midarrow] (m2.north) -- (m3.south);
  \draw[midarrow] (m3.north) -- (m4.south);
  \draw[midarrow] (m4.north) -- (t.south);
  \draw[midarrow] (r1) to (m3);
  \draw[midarrow] (m3.east) to[out=0, in = 0] (m1.east);
  \draw[midarrow] (m1.west) to[out=180, in =180] (m2.west);
  \draw[midarrow] (m2.east) to[out=0, in =0] (m4.east);
  \draw[midarrow] (m4) to (t1);
}

\DeclareSymbol{cc42plain}{0}{
  \node[root] (r) at (0.0,-2.2) {};
  \node[root] (r1) at (1.6,-2.2) {};
  \node[dot] (m1) at (0.8,-1.2) {};
  \node[dot] (m2) at (0.8,-.2) {};
  \node[dot] (m3) at (0.8,0.8) {};
  \node[dot] (m4) at (0.8,1.8) {};
  \node[root] (t) at (0.0,2.8) {};
  \node[root] (t1) at (1.6,2.8) {};
  \draw[midarrow] (r.north) -- (m1.south);
  \draw[midarrow] (r1.north) -- (m1.south);
  \draw[midarrow] (m1.east) to[out=0, in = 0] (m2.east);
  \draw[midarrow] (m1.west) to[out=180, in = 180] (m2.west);
  \draw[midarrow] (m2.east) to[out=0, in = 0] (m3.east);
  \draw[midarrow] (m2.west) to[out=180, in = 180] (m3.west);
  \draw[midarrow] (m3.east) to[out=0, in = 0] (m4.east);
  \draw[midarrow] (m3.west) to[out=180, in = 180] (m4.west);
  \draw[midarrow] (m4.north) -- (t.south);
  \draw[midarrow] (m4.north) -- (t1.south);
}

\DeclareSymbol{cc42C}{0}{
  \node[root] (r) at (0.8,-1.7) {};
  \node[dot] (m2) at (0.8,-.7) {};
  \node[dot] (m3) at (0.8,0.3) {};
  \node[dot] (m4) at (0.8,1.3) {};
  \node[root] (t) at (0.8,2.3) {};
  \node[root] (t1) at (0.0,1.3) {};
  \node[root] (r1) at (0,0.3) {};
  \draw[midarrow] (r.north)  -- (m2.south);
  \draw[midarrow] (m2.north) -- (m3.south);
  \draw[midarrow] (m3.north) -- (m4.south);
  \draw[midarrow] (m4.north) -- (t.south);
  \draw[midarrow] (r1) to (m3);
  \draw[midarrow] (m3.east) to[out=0, in = 0] (m2.east);
  \draw[midarrow] (m2.east) to[out=0, in =0] (m4.east);
  \draw[midarrow] (m4) to (t1);
}

\DeclareSymbol{cc42D}{0}{
  \node[root] (r) at (0.8,-1.3) {};
  \node[dot] (m3) at (0.8,-.3) {};
  \node[dot] (m4) at (0.8,.7) {};
  \node[root] (t) at (0.8,1.7) {};
  \node[root] (t1) at (0.0,.7) {};
  \node[root] (r1) at (0,-.3) {};
  \draw[midarrow] (r.north)  -- (m3.south);
  \draw[midarrow] (m3.north) -- (m4.south);
  \draw[midarrow] (m4.north) -- (t.south);
  \draw[midarrow] (r1) to (m3);
  \draw[midarrow] (m3.east) to[out=0, in = 0] (m4.east);
  \draw[midarrow] (m4) to (t1);
}

\DeclareSymbol{f21}{0}{
\node[root] (r1) at (0,0.5) { };
\node[root] (r2) at (1.6,0.5) { };
\node[dot] at (0.8,0.5) {};
\draw (r1) -- (0.8,0.5) -- (r2);
\path (0.8,0.5)  edge [loop above, min distance = 15mm, in = 150, out = 30]
(0.8,0.5) ;
}

\DeclareSymbol{f22}{0}{
\node[root] (r1) at (0,0.5) { };
\node[root] (r2) at (2.4,0.5) { };
\node[root] (r3) at (0,-0.5) { };
\node[root] (r4) at (2.4,-0.5) { };
\node[dot] (i1) at (0.8,0) {};
\node[dot] (i2) at (1.6,0) {};
\draw (r1) to (i1) to (r3);
\draw (r4) to (i2) to (r2);
\draw (i1)  to[out=40,in=140] (i2) ;
\draw (i1) to[out=-40,in=-140] (i2) ;
}

\DeclareSymbol{bubble}{0}{
\node[root] (r1) at (0,0) { };
\node[root] (r2) at (2.4,0) { };
\node[dot] (i1) at (0.8,0) {};
\node[dot] (i2) at (1.6,0) {};
\draw (r1) to (i1);
\draw (r2) to (i2);
\draw (i1)  to[out=50,in=130] (i2) ;
\draw (i1) to[out=-50,in=-130] (i2) ;
\draw (i1) to[out=0,in=-180] (i2) ;
}

\DeclareSymbol{f33}{0}{
\node[root] (r) at (0,0) { };
\node[wdot] (m12) at (0.4,1.2) {};
\node[rootb] (t) at (-.6,1) {};
\draw (0,0.6) to[out=20,in=-100] (m12);
\draw (0,0.6) to[out=90,in=-180] (m12);
\node[wdot] (m1) at (0,0.6){};
\draw (r)  -- (m1);
\draw (m1)  -- (t);
\node[root] (r1) at (0.8,0) { };
\node[rootb] (t1) at (1.4,1) {};
\draw (0.8,0.6) to[out=150,in=-80] (m12);
\draw (0.8,0.6) to[out=90,in=0] (m12);
\node[wdot] (m11) at (0.8,0.6){};
\draw (r1)  -- (m11);
\draw (m11)  -- (t1);
}

\DeclareSymbol{ex1}{0}{%
\node[root] (r1) at (-1,.5) {};
\node[root] (r2) at (-1,-.5) {};
\node[root] (r3) at (3,.5) {};
\node[root] (r4) at (3,-.5) {};
  \node[dot] (a) at (0,0) {};
  \node[dot] (b) at (1,0) {};
  \node[dot] (c) at (2,0) {};
  \draw[midarrow] (r1) to (a);
  \draw[midarrow] (r2) to (a);
  \draw[midarrow] (c) to (r3);
  \draw[midarrow] (c) to (r4);
  \draw[midarrow] (a) to[out=90,in=90] (b);
  \draw[midarrow] (a) to[out=-90,in=-90] (b);
  \draw[midarrow] (b) to[out=90,in=90] (c);
  \draw[midarrow] (b) to[out=-90,in=-90] (c);
}

\DeclareSymbol{ex1L}{0}{%
\node[root] (r1) at (-1.8,1.0) {};
\node[root] (r2) at (-1.8,-1.0) {};
\node[root] (r3) at (5.4,1.0) {};
\node[root] (r4) at (5.4,-1.0) {};
  \node[dot, label={[label distance = 1.3mm] below:$x_1$}] (a) at (0,0) {};
  \node[dot, label={[label distance = 1.3mm] below:$x_2$}] (b) at (1.8,0) {};
  \node[dot, label={[label distance = 1.3mm] below:$x_3$}] (c) at (3.6,0) {};
  \draw[midarrow] (r1) to (a);
  \draw[midarrow] (r2) to (a);
  \draw[midarrow] (c) to (r3);
  \draw[midarrow] (c) to (r4);
  \draw[midarrow] (a) to[out=90,in=90] (b);
  \draw[midarrow] (a) to[out=-90,in=-90] (b);
  \draw[midarrow] (b) to[out=90,in=90] (c);
  \draw[midarrow] (b) to[out=-90,in=-90] (c);
}

\DeclareSymbol{ex2}{0}{%
  \node[dot] (a) at (0,0) {};
  \node[dot] (b) at (1.8,0) {};
  \node[dot] (c) at (0.9,1.4) {};
  \draw[midarrow] (a) to[out=20,in=160] (b);
  \draw[midarrow] (a) to[out=-20,in=-160] (b);
  \draw[midarrow] (a) -- (c);
  \draw[midarrow] (b) -- (c);
}

\DeclareSymbol{ex3}{0}{%
  \node[dot] (v) at (0,0) {};
  \node[dot] (c) at (0,1.4) {};
  \draw[midarrow] (v) to[out=70,in=-70] (c);
  \draw[midarrow] (v) to[out=110,in=-110] (c);
}

\DeclareSymbol{bphze1}{0}{%
\node[root]  (r) at (-4.5,0.2) {};
\node[dot]  (m1) at (-3.0,0.2) {};
\node[dot]  (m2) at (-1.5,0.2) {};
\node[dot]  (m3) at ( 0.0,0.2) {};
\node[dot]  (m4) at ( 1.5,0.2) {};
\node[root]  (t) at (3.0,0.2) {};
\node[font=\scriptsize, anchor=north, yshift=-10pt] at (m1) {$1$};
\node[font=\scriptsize, anchor=north, yshift=-10pt] at (m2) {$2$};
\node[font=\scriptsize, anchor=north, yshift=-10pt] at (m3) {$3$};
\node[font=\scriptsize, anchor=north, yshift=-10pt] at (m4) {$4$};
\draw[midarrow] (m1.east) -- (m2.west);
\draw[midarrow] (m2.east) -- (m3.west);
\draw[midarrow] (m3.east) -- (m4.west);
\draw[midarrow] (m1.south) to[out=-90,in=-90] (m2.south);
\draw[midarrow] (m2.north) to[out=90,in=90]   (m1.north);
\draw[midarrow] (m3.north) to[out=90,in=90]   (m4.north);
\draw[midarrow] (m3.south) to[out=-90,in=-90]   (m4.south);
\draw[midarrow] (r) to (m1);
\draw[midarrow] (m4) to (t);
}%

\DeclareSymbol{bphze2}{0}{%
\node[dot]  (m1) at (-3.0,0.2) {};
\node[dot]  (m2) at (-1.5,0.2) {};
\node[font=\scriptsize, anchor=north, yshift=-10pt] at (m1) {$1$};
\node[font=\scriptsize, anchor=north, yshift=-10pt] at (m2) {$2$};
\draw[midarrow] (m1.east) -- (m2.west);
\draw[midarrow] (m1.south) to[out=-90,in=-90] (m2.south);
\draw[midarrow] (m2.north) to[out=90,in=90]   (m1.north);
}%

\DeclareSymbol{bphze3}{0}{%
\node[dot]  (m1) at (-3.0,0.2) {};
\node[dot]  (m2) at (-1.5,0.2) {};
\node[font=\scriptsize, anchor=north, yshift=-10pt] at (m1) {$3$};
\node[font=\scriptsize, anchor=north, yshift=-10pt] at (m2) {$4$};
\draw[midarrow] (m1.east) -- (m2.west);
\draw[midarrow] (m1.south) to[out=-90,in=-90] (m2.south);
\draw[midarrow] (m2.north) to[out=90,in=90]   (m1.north);
}%

\DeclareSymbol{bphze4}{0}{%
\node[dot]  (m1) at (-3.0,0.2) {};
\node[dot]  (m2) at (-1.5,0.2) {};
\node[dot]  (m3) at ( 0.0,0.2) {};
\node[dot]  (m4) at ( 1.5,0.2) {};
\node[font=\scriptsize, anchor=north, yshift=-10pt] at (m1) {$1$};
\node[font=\scriptsize, anchor=north, yshift=-10pt] at (m2) {$2$};
\node[font=\scriptsize, anchor=north, yshift=-10pt] at (m3) {$3$};
\node[font=\scriptsize, anchor=north, yshift=-10pt] at (m4) {$4$};
\draw[midarrow] (m1.east) -- (m2.west);
\draw[midarrow] (m2.east) -- (m3.west);
\draw[midarrow] (m3.east) -- (m4.west);
\draw[midarrow] (m1.south) to[out=-90,in=-90] (m2.south);
\draw[midarrow] (m2.north) to[out=90,in=90]   (m1.north);
\draw[midarrow] (m3.north) to[out=90,in=90]   (m4.north);
\draw[midarrow] (m3.south) to[out=-90,in=-90]   (m4.south);
}%

\DeclareSymbol{var1}{0}{
  \node[root] (r) at (0,-0.7) {};
  \node[odot] (m) at (0,0.3) {};
  \node[root] (t) at (0,1.3) {};
  \draw[midarrow] (r.north) -- (m.south);
  \draw[midarrow] (m.north) -- (t.south);
  \node[root] (r1) at (0.8,-1.7) {};
  \node[odot] (m2) at (0.8,-0.7) {};
  \node[odot] (m1) at (0.8,0.3) {};
  \node[odot] (m3) at (0.8,1.3) {};
  \node[root] (t1) at (0.8,2.3) {};
  \draw[midarrow] (r1.north) -- (m2.south);
  \draw[midarrow] (m2.north) -- (m1.south);
  \draw[midarrow] (m1.north) -- (m3.south);
  \draw[midarrow] (m3.north) -- (t1.south);
  \draw[red, very thick, dash pattern=on 0.5pt off 0.5pt] (m) -- (m1);
  \draw[red, very thick, dash pattern=on 0.5pt off 0.5pt] (m2) to[out=0,in=0] (m3);
}
\DeclareSymbol{var2}{0}{
  \node[root] (r) at (0.8,-0.7) {};
  \node[odot] (m) at (0.8,0.3) {};
  \node[root] (t) at (0.8,1.3) {};
  \draw[midarrow] (r.north) -- (m.south);
  \draw[midarrow] (m.north) -- (t.south);
  \node[root] (r1) at (0,-1.7) {};
  \node[odot] (m2) at (0,-0.7) {};
  \node[odot] (m1) at (0,0.3) {};
  \node[odot] (m3) at (0,1.3) {};
  \node[root] (t1) at (0,2.3) {};
  \draw[midarrow] (r1.north) -- (m2.south);
  \draw[midarrow] (m2.north) -- (m1.south);
  \draw[midarrow] (m1.north) -- (m3.south);
  \draw[midarrow] (m3.north) -- (t1.south);
  \draw[red, very thick, dash pattern=on 0.5pt off 0.5pt] (m) -- (m1);
  \draw[red, very thick, dash pattern=on 0.5pt off 0.5pt] (m2) to[out=180,in=180] (m3);
}

\DeclareSymbol{varc42}{0}{
  \node[root] (r) at (1.2,-3.2) {};
  \node[odot] (m1) at (1.2,-1.7) {};
  \node[odot] (m2) at (1.2,-.2) {};
  \node[odot] (m3) at (1.2,1.3) {};
  \node[odot] (m4) at (1.2,2.8) {};
  \node[root] (t) at (1.2,4.3) {};
  \draw (r.north) -- node[midway, xshift = 4.5pt] {\tiny $2_0$} (m1.south);
  \draw (m1.north) -- node[midway,xshift = 4.5pt] {\tiny $2_1$} (m2.south);
  \draw (m2.north) -- node[midway, xshift = 4.5pt] {\tiny $2_2$} (m3.south);
  \draw (m3.north) -- node[midway, xshift = 4.5pt] {\tiny $2_3$} (m4.south);
  \draw (m4.north) -- node[midway, xshift = 4.5pt] {\tiny $2_4$} (t.south);
  \node[root] (r1) at (0,-3.2) {};
  \node[odot] (m11) at (0,-1.7) {};
  \node[odot] (m21) at (0,-.2) {};
  \node[odot] (m31) at (0,1.3) {};
  \node[odot] (m41) at (0,2.8) {};
  \node[root] (t1) at (0,4.3) {};
  \draw (r1.north) --node[midway, xshift = -4.0pt] {\tiny $1_0$} (m11.south);
  \draw (m11.north) --node[midway, xshift = -4.0pt] {\tiny $1_1$} (m21.south);
  \draw (m21.north) --node[midway, xshift = -4.0pt] {\tiny $1_2$} (m31.south);
  \draw (m31.north) --node[midway, xshift = -4.0pt] {\tiny $1_3$} (m41.south);
  \draw (m41.north) --node[midway, xshift = -4.0pt] {\tiny $1_4$} (t1.south);
  \draw[red, very thick, dash pattern=on 0.5pt off 0.5pt] (m41.east) to[out=0, in=180] (m4.west);
  \draw[red, very thick, dash pattern=on 0.5pt off 0.5pt] (m31) to[out=-20, in=160] (m1);
  \draw[red, very thick, dash pattern=on 0.5pt off 0.5pt] (m21.east) to[out=0, in=180] (m3.west);
  \draw[red, very thick, dash pattern=on 0.5pt off 0.5pt] (m11) to[out=45, in=240] (m2);
}

\DeclareSymbol{varcc42}{0}{
  \node[root] (r) at (1.5,-3.2) {};
  \node[dot] (m1) at (1.5,-1.7) {};
  \node[dot] (m2) at (1.5,-.2) {};
  \node[dot] (m3) at (1.5,1.3) {};
  \node[dot] (m4) at (1.5,2.8) {};
  \node[root] (t) at (1.5,4.3) {};
  \node[root] (t1) at (0.0,2.8) {};
  \node[root] (r1) at (0,1.3) {};
  \draw (r.north) -- node[midway, xshift = 4.5pt] {\tiny $1_0$}(m1.south);
  \draw (m1.north) --node[midway, xshift = 4.5pt] {\tiny $1_1$} (m2.south);
  \draw (m2.north) --node[midway, xshift = 4.5pt] {\tiny $1_2$} (m3.south);
  \draw (m3.north) --node[midway, xshift = 4.5pt] {\tiny $1_3$} (m4.south);
  \draw (m4.north) --node[midway, xshift = 4.5pt] {\tiny $1_4$} (t.south);
  \draw (r1) to node[midway, yshift = 4.5pt] {\tiny $2_0$}(m3);
  \draw (m3.east) to[out=0, in = 0] node[midway, xshift = 4.5pt] {\tiny $2_1$} (m1.east);
  \draw (m1.west) to[out=180, in =180] node[midway, xshift = -4.0pt] {\tiny $2_2$} (m2.west);
  \draw (m2.east) to[out=0, in =0] node[midway, xshift = 4.5pt] {\tiny $2_3$} (m4.east);
  \draw (m4) to node[midway, yshift = 4.5pt] {\tiny $2_4$} (t1);
}

\DeclareSymbol{varB1}{0}{
  \node (r) at (0,-1.7) {};
  \node (r1) at (1.5,-1.7) {};
  \node[dot] (d1) at (0.75,-.2) {};
  \node[dot] (d2) at (0.75,1.3) {};
  \node (t) at (0,2.8) {};
  \node (t1) at (1.5,2.8) {};
  \draw (r) --node[midway, xshift = -5.5pt] {\tiny $\alpha_i$} (d1);
  \draw (r1) --node[midway, xshift = 9.5pt] {\tiny $\alpha_{i+1}$} (d1);
  \draw (d1) to[out=0,in=0] node[midway, xshift = 9.5pt] {\tiny $\beta_{j+2}$} (d2);
  \draw (d1) to[out=180,in=180]node[midway, xshift = -8.5pt] {\tiny $\beta_{j+1}$} (d2);
  \draw (d2) -- node[midway, xshift = -5.5pt] {\tiny $\beta_{j}$}(t);
  \draw (d2) -- node[midway, xshift = 9.5pt] {\tiny $\beta_{j+3}$}(t1);
}

\DeclareSymbol{varB2}{0}{
  \node (r) at (0,-1.7) {};
  \node (r1) at (1.5,-1.7) {};
  \node[dot] (d1) at (0.75,-.2) {};
  \node[dot] (d2) at (0.75,1.3) {};
  \node (t) at (0,2.8) {};
  \node (t1) at (1.5,2.8) {};
  \draw (r) --node[midway, xshift = -5.5pt] {\tiny $\beta_j$} (d1);
  \draw (r1) --node[midway, xshift = 9.5pt] {\tiny $\beta_{j+1}$} (d1);
  \draw (d1) to[out=0,in=0] node[midway, xshift = 9.5pt] {\tiny $\alpha_{i+2}$} (d2);
  \draw (d1) to[out=180,in=180]node[midway, xshift = -8.5pt] {\tiny $\alpha_{i+1}$} (d2);
  \draw (d2) -- node[midway, xshift = -5.5pt] {\tiny $\alpha_{i}$}(t);
  \draw (d2) -- node[midway, xshift = 9.5pt] {\tiny $\alpha_{i+3}$}(t1);
}

\DeclareSymbol{varB3}{0}{
  \node (r) at (0,-1.7) {};
  \node (r1) at (1.5,-1.7) {};
  \node[dot] (d1) at (0.75,-.2) {};
  \node[dot] (d2) at (0.75,1.3) {};
  \node (t) at (0,2.8) {};
  \node (t1) at (1.5,2.8) {};
  \draw (r) --node[midway, xshift = -5.5pt] {\tiny $\alpha_i$} (d1);
  \draw (r1) --node[midway, xshift = 9.5pt] {\tiny $\beta_{j}$} (d1);
  \draw (d1) to[out=0,in=0] node[midway, xshift = 9.5pt] {\tiny $\beta_{j+1}$} (d2);
  \draw (d1) to[out=180,in=180]node[midway, xshift = -8.5pt] {\tiny $\alpha_{i+1}$} (d2);
  \draw (d2) -- node[midway, xshift = -9.5pt] {\tiny $\alpha_{i+2}$}(t);
  \draw (d2) -- node[midway, xshift = 9.5pt] {\tiny $\beta_{j+2}$}(t1);
}

\DeclareSymbol{varB4}{0}{
  \node (r) at (0,-1.7) {};
  \node (r1) at (1.5,-1.7) {};
  \node[dot] (d1) at (0.75,-.2) {};
  \node[dot] (d2) at (0.75,1.3) {};
  \node (t) at (0,2.8) {};
  \node (t1) at (1.5,2.8) {};
  \draw (r) --node[midway, xshift = -5.5pt] {\tiny $\alpha_i$} (d1);
  \draw (r1) --node[midway, xshift = 9.5pt] {\tiny $\beta_{j+2}$} (d1);
  \draw (d1) to[out=0,in=0] node[midway, xshift = 9.5pt] {\tiny $\beta_{j+1}$} (d2);
  \draw (d1) to[out=180,in=180]node[midway, xshift = -8.5pt] {\tiny $\alpha_{i+1}$} (d2);
  \draw (d2) -- node[midway, xshift = -9.5pt] {\tiny $\alpha_{i+2}$}(t);
  \draw (d2) -- node[midway, xshift = 9.5pt] {\tiny $\beta_{j}$}(t1);
}

\DeclareSymbol{varB5}{0}{
  \node (r) at (0,-1.7) {};
  \node (r1) at (1.5,-1.7) {};
  \node[dot] (d1) at (0.75,-.2) {};
  \node[dot] (d2) at (0.75,1.3) {};
  \node (t) at (0,2.8) {};
  \node (t1) at (1.5,2.8) {};
  \draw (r) -- (d1);
  \draw (r1) -- (d1);
  \draw (d1) to[out=0,in=0] (d2);
  \draw (d1) to[out=180,in=180] (d2);
  \draw (d2) -- (t);
  \draw (d2) -- (t1);
}

\DeclareSymbol{varC}{0}{
  \node (r) at (-.3,-.7) {};
  \node (r1) at (1.8,-.7) {};
  \node[dot] (d1) at (0.75,.55) {};
  \node (t) at (-.3,1.8) {};
  \node (t1) at (1.8,1.8) {};
  \draw (r) --  node[midway, xshift = -9pt] {\tiny $\pi \alpha_{i}$}(d1);
  \draw (r1) --  node[midway, xshift = +12.5pt] {\tiny $\pi \alpha_{i+1}$}(d1);
  \draw (d1) --  node[midway, xshift = -9pt] {\tiny $\pi \beta_{j}$}(t);
  \draw (d1) --  node[midway, xshift = +14.5pt] {\tiny $\pi \beta_{j+3}$}(t1);
}

\DeclareSymbol{varC2}{0}{
  \node (r) at (-.3,-.7) {};
  \node (r1) at (1.8,-.7) {};
  \node[dot] (d1) at (0.75,.55) {};
  \node (t) at (-.3,1.8) {};
  \node (t1) at (1.8,1.8) {};
  \draw (r) --  node[midway, xshift = -9pt] {\tiny $\pi \beta_{j}$}(d1);
  \draw (r1) --  node[midway, xshift = +12.5pt] {\tiny $\pi \beta_{j+1}$}(d1);
  \draw (d1) --  node[midway, xshift = -9pt] {\tiny $\pi \alpha_{i}$}(t);
  \draw (d1) --  node[midway, xshift = +14.5pt] {\tiny $\pi \alpha_{i+3}$}(t1);
}

\DeclareSymbol{varC3}{0}{
  \node (r) at (-.3,-.7) {};
  \node (r1) at (1.8,-.7) {};
  \node[dot] (d1) at (0.75,.55) {};
  \node (t) at (-.3,1.8) {};
  \node (t1) at (1.8,1.8) {};
  \draw (r) --  node[midway, xshift = -9pt] {\tiny $\pi \alpha_{i}$}(d1);
  \draw (r1) --  node[midway, xshift = +12.5pt] {\tiny $\pi \beta_{j}$}(d1);
  \draw (d1) --  node[midway, xshift = -14pt] {\tiny $\pi \alpha_{i+2}$}(t);
  \draw (d1) --  node[midway, xshift = +14.5pt] {\tiny $\pi \beta_{j+2}$}(t1);
}

\DeclareSymbol{varC4}{0}{
  \node (r) at (-.3,-.7) {};
  \node (r1) at (1.8,-.7) {};
  \node[dot] (d1) at (0.75,.55) {};
  \node (t) at (-.3,1.8) {};
  \node (t1) at (1.8,1.8) {};
  \draw (r) --  node[midway, xshift = -9pt] {\tiny $\pi \alpha_{i}$}(d1);
  \draw (r1) --  node[midway, xshift = +12.5pt] {\tiny $\pi \beta_{j+2}$}(d1);
  \draw (d1) --  node[midway, xshift = -15pt] {\tiny $\pi \alpha_{i+2}$}(t);
  \draw (d1) --  node[midway, xshift = +14.5pt] {\tiny $\pi \beta_{j}$}(t1);
}

\DeclareSymbol{varC5}{0}{
  \node (r) at (-.3,-.7) {};
  \node (r1) at (1.8,-.7) {};
  \node[dot, label=above:{\tiny $v$}] (d1) at (0.75,.55) {};
  \node (t) at (-.3,1.8) {};
  \node (t1) at (1.8,1.8) {};
  \draw[thick, blue] (r) --  node[midway, xshift = -9pt] {\tiny $\alpha_{i}$}(d1);
  \draw[thick, green!60!black] (r1) --  node[midway, xshift = +12.5pt] {\tiny $\beta_{j}$}(d1);
  \draw[thick, red!70!black] (d1) --  node[midway, xshift = -15pt] {\tiny $\alpha_{i+1}$}(t);
  \draw[thick, orange] (d1) --  node[midway, xshift = +14.5pt] {\tiny $\beta_{j+1}$}(t1);
}

\DeclareSymbol{varD1}{0}{
  \node (r) at (0,-1.7) {};
  \node (r1) at (1.5,-1.7) {};
  \node[dot] (d1) at (0.75,-.2) {};
  \node[dot] (d2) at (0.75,1.3) {};
  \node (t) at (0,2.8) {};
  \node (t1) at (1.5,2.8) {};
  \draw[thick, blue] (r) --node[midway, xshift = -5.5pt] {\tiny $\alpha_i$} (d1);
  \draw[thick, red!70!black] (r1) --node[midway, xshift = 9.5pt] {\tiny $\alpha_{i+1}$} (d1);
  \draw (d1) to[out=0,in=0] node[midway, xshift = 9.5pt] {\tiny $\beta_{j+2}$} (d2);
  \draw (d1) to[out=180,in=180]node[midway, xshift = -8.5pt] {\tiny $\beta_{j+1}$} (d2);
  \draw[thick, green!60!black] (d2) -- node[midway, xshift = -5.5pt] {\tiny $\beta_{j}$}(t);
  \draw[thick, orange] (d2) -- node[midway, xshift = 9.5pt] {\tiny $\beta_{j+3}$}(t1);
}

\DeclareSymbol{varD2}{0}{
  \node (r) at (0,-1.7) {};
  \node (r1) at (1.5,-1.7) {};
  \node[dot] (d1) at (0.75,-.2) {};
  \node[dot] (d2) at (0.75,1.3) {};
  \node (t) at (0,2.8) {};
  \node (t1) at (1.5,2.8) {};
  \draw[thick, green!60!black] (r) --node[midway, xshift = -6.5pt] {\tiny $\beta_{j'}$} (d1);
  \draw[thick, orange] (r1) --node[midway, xshift = 11.5pt] {\tiny $\beta_{j'+1}$} (d1);
  \draw (d1) to[out=0,in=0] node[midway, xshift = 9.5pt] {\tiny $\alpha_{i+2}$} (d2);
  \draw (d1) to[out=180,in=180]node[midway, xshift = -8.5pt] {\tiny $\alpha_{i+1}$} (d2);
  \draw[thick, blue] (d2) -- node[midway, xshift = -5.5pt] {\tiny $\alpha_{i}$}(t);
  \draw[thick, red!70!black] (d2) -- node[midway, xshift = 9.5pt] {\tiny $\alpha_{i+3}$}(t1);
}

\DeclareSymbol{varD3}{0}{
  \node (r) at (0,-1.7) {};
  \node (r1) at (1.5,-1.7) {};
  \node[dot] (d1) at (0.75,-.2) {};
  \node[dot] (d2) at (0.75,1.3) {};
  \node (t) at (0,2.8) {};
  \node (t1) at (1.5,2.8) {};
  \draw[thick, blue] (r) --node[midway, xshift = -5.5pt] {\tiny $\alpha_i$} (d1);
  \draw[thick, green!60!black] (r1) --node[midway, xshift = 9.5pt] {\tiny $\beta_{j'}$} (d1);
  \draw (d1) to[out=0,in=0] node[midway, xshift = 10.5pt] {\tiny $\beta_{j'+1}$} (d2);
  \draw (d1) to[out=180,in=180]node[midway, xshift = -8.5pt] {\tiny $\alpha_{i+1}$} (d2);
  \draw[thick, red!70!black] (d2) -- node[midway, xshift = -9.5pt] {\tiny $\alpha_{i+2}$}(t);
  \draw[thick, orange] (d2) -- node[midway, xshift = 11.5pt] {\tiny $\beta_{j'+2}$}(t1);
}

\DeclareSymbol{varD4}{0}{
  \node (r) at (0,-1.7) {};
  \node (r1) at (1.5,-1.7) {};
  \node[dot] (d1) at (0.75,-.2) {};
  \node[dot] (d2) at (0.75,1.3) {};
  \node (t) at (0,2.8) {};
  \node (t1) at (1.5,2.8) {};
  \draw[thick, blue] (r) --node[midway, xshift = -5.5pt] {\tiny $\alpha_i$} (d1);
  \draw[thick, orange] (r1) --node[midway, xshift = 11.5pt] {\tiny $\beta_{j'+2}$} (d1);
  \draw (d1) to[out=0,in=0] node[midway, xshift = 10.5pt] {\tiny $\beta_{j'+1}$} (d2);
  \draw (d1) to[out=180,in=180]node[midway, xshift = -8.5pt] {\tiny $\alpha_{i+1}$} (d2);
  \draw[thick, red!70!black] (d2) -- node[midway, xshift = -9.5pt] {\tiny $\alpha_{i+2}$}(t);
  \draw[thick, green!60!black] (d2) -- node[midway, xshift = 9.5pt] {\tiny $\beta_{j'}$}(t1);
}

\usepackage{authblk}

\usepackage{orcidlink}

\title{Fluctuations in the weakly coupled 4D Anderson Hamiltonian}
\author[1]{Simon Gabriel\orcidlink{0009-0005-5048-970X}}
\author[2]{Tommaso Rosati\orcidlink{0000-0001-5255-6519}}

\affil[1]{UC Berkeley, USA}
\affil[2]{University of Warwick, UK}

\begin{document}

\maketitle

\begin{abstract}
We study the weak coupling limit of the Anderson Hamiltonian in the
critical dimension $d=4$. In a perturbative sense, we prove Gaussian
fluctuations about the Green's function of the Laplacian. The fluctuations are
described by an explicit
effective variance, up to a critical value of the coupling constant at which
we expect a phase transition in the structure of the fluctuations.
The proof is based on a combinatorial analysis of Feynman diagrams,
and on a detailed study of the BPHZ renormalisation of the model.
We
 characterise the limiting distribution in terms of primitive blow-ups,
 and prove that no Laplacian renormalisation is present.
Our approach seems applicable to a broad class of equations.

\end{abstract}

\noindent\hspace{1cm}{\small\textit{{Keywords.}} Scaling critical singular SPDEs; Anderson model; BPHZ
renormalisation.}

\noindent\hspace{1cm}{\small\textit{{MSC classification.}} Primary: 60H15, Secondary: 35R60.}

\newpage

\setcounter{tocdepth}{2}
\tableofcontents

\newpage

\section{Introduction and main results}

This article concerns the elliptic Anderson model in resolvent form:
\begin{equation}\label{e:rslvnt-naive}
  \begin{aligned}
  (1 - \Delta) v = \lambda \xi v + \delta_y \;, \qquad \forall x \in \TT^{4} \;,
  \end{aligned}
  \end{equation}
with $\TT^4$ the $4$-dimensional torus, $\lambda \in \RR$ a coupling constant,
and $\xi$ space white noise, a
centred generalised Gaussian field on $ \TT^{4}$ with correlation function $\EE[\xi(x) \xi(y)] =
\delta(x-y)$.

The Anderson model arises in the study of particles in
random media, such as 
electrons in a random potential \cite{anderson1958absence} or branching
diffusions in a random environment \cite{MR1069840}.
The noise is a source of disorder
that models inhomogeneities, such as microscopic impurities in a material. A
fundamental question in the study of any disordered system is how the disorder
affects its large scale behaviour. This leads to
the problems of intermittency and localisation of the spectrum of
the Anderson Hamiltonian $\mH= \Delta - \lambda \xi$, both in discrete and in 
continuous settings. Namely, at large scales (on infinite volume)
solutions to the parabolic problem $\partial_t u = \mH u $  concentrate on high
peaks of the potential: see for example \cite{ZMRS,
KLMS, fm_anderson} in the discrete setting and \cite{CvZ, DL, HL, KPvZ} in the
continuous case. 

In the continuum, the Anderson model has so far been defined only in dimensions
$d \leq 3$. In these dimensions, the associated parabolic or elliptic equations are subcritical in the jargon of singular
stochastic PDEs, or superrenormalisable in that of the physics literature, and therefore \eqref{e:rslvnt-naive} is well
posed \cite{allez2015continuous,labbe19} after suitable renormalisation, and for $\lambda$
sufficiently small. On finite
volume both in the discrete setting and in the continuous one in dimension
$d\leq 3$, the Anderson model is always delocalised, since the
Hamiltonian is a perturbation of the Laplacian, and the solution $v$ to
\eqref{e:rslvnt-naive}, after suitable renormalisation, remains a strictly
positive function.

Instead, in dimension $d=4$ the continuous model becomes critical. Its leading order terms are
formally scale invariant, since by considering the equation $\Delta v = \lambda\xi
v$, and setting $v^{(\zeta)}(x)=v(x/\zeta)$ for any $\zeta >0$,
the invariance of white noise
$\xi(x/\zeta) \overset{d}{=} \zeta^{-d/2} \xi(x)$ leads to expect
$v^{(\zeta)} \overset{d}{=}v^{(1)}$. 
As a consequence, in dimensions $4$ and higher, we
do not expect the Anderson model on finite volume to be a perturbation of the Laplacian, since at small scales the noise becomes as relevant (in $d=4$) or more (in $d
>4$) than the Laplacian. This breaks the heuristic that guarantees delocalisation of the continuum
model on finite volume, since a solution to \eqref{e:rslvnt-naive} in $d \geq
4$, in whatever sense, may not be strictly positive.

Indeed, when the equation
is scale invariant or even supercritical, one might expect that localisation
occurs already on finite volume, and in some appropriate continuum limit, since
it appears in the large scale limit of discrete models independently of the
dimension \cite{fm_anderson}. This picture seems beyond our current mathematical
understanding, and to the best of our
knowledge the present work is the first attempt in studying it. Our
main results, which describe the fluctuations of \eqref{e:rslvnt-naive} about the
Green's function of the Laplacian in a weak coupling scaling, suggest that
unless the coupling constant is suitably (logarithmically) tamed, the $4d$
Anderson Hamiltonian will not be a perturbation of the Laplacian. Moreover, they
suggest the existence of a
\emph{critical} choice
of the weak coupling constant
at which to expect 
a non-trivial scaling limit (similar to the $2d$ stochastic heat flow \cite{csz_shf})
that might capture the behaviour of the model at the crossover between localisation and delocalisation.

Overall, critical
and supercritical models have so far been tackled only on a case by case basis,
with notable examples including 
the multiplicative stochastic heat equation (SHE) 
\cite{csz_she, dun_gu_fb_she, csz_shf,tao_nonlin}
and the KPZ equation \cite{csz_kpz, gu_kpz, cd_kpz} in {$ d = 2$}, 
the anisotropic KPZ and Burgers-type equations in $d=2$ \cite{cannizzaro2023weak,
cannizzaro2024gaussian},  diffusions in random environments
\cite{ToninelliGFF, OttoGFF, armstrong2024superdiffusive},
the $\Phi^4$ measure in $d=4$ \cite{Aizenman2024},
superprocesses \cite{EA_sup}, and Dean--Kawasaki dynamics \cite{DK_triv}.
We note that there is no unified theory to treat these systems, and a
variety of tools have been necessary. For SHE and KPZ results build on explicit chaos
decompositions or on tools from Malliavin calculus.
The study of Burgers-type
equations relies on the knowledge of an explicit Gaussian invariant measure,
superprocesses and Dean-Kawasaki build on martingale formulations, and for the
$\Phi^4_4$ model triviality is understood only at the level of the measure
and not of the dynamics.
 
One of the main difficulties when dealing with such systems is
that tools from singular SPDEs, such as regularity structures \cite{Regularity}, paracontrolled
distributions \cite{GIP}, or renormalisation group techniques \cite{kup, duch}, do not apply. These theories hinge on
the idea that the small scale structure of a solution to a nonlinear stochastic PDE is governed by
its Gaussian linearised dynamics and a finite number collection of functionals thereof. In
critical systems, because of the self-similarity of the equation, there is no
distinction between small and large scales and this assumption breaks down.
 
In this work, we consider \eqref{e:rslvnt-naive} for small
values of the coupling constant, $|\lambda| \ll 1$, in a sense that will be made
precise below. If the coupling constant is small, we expect the system to fluctuate around the Green's
function of the Laplacian (the case $\lambda = 0$). Our results suggest that such fluctuations should
be Gaussian up to a critical value of the coupling constant in a suitable
scaling. We predict an exact form of the covariance structure before that
critical value, as well as determine the exact value
at which the phase transition is expected to occur.

Our approach is perturbative, and in the terminology of regularity structures our main result can be thought of as akin to the construction
of the full infinite model associated to the critical equation (whereas in a
subcritical equation a model would consist only of finitely many terms). As in that theory, the power of our
approach lies in the fact that is seems applicable to a much broader class of
equations, beyond the Anderson model, most of which have previously evaded a rigorous
mathematical treatment.

We consider a weak coupling
limit, in which we couple a mollification of the
noise at a vanishing spatial scale $\ve \in (0, 1)$
to a coupling constant that vanishes at a logarithmic rate $$\lambda_\ve =
\hat{\lambda} \big(\log \tfrac{1}{\ve}\big)^{-1/2},$$ for some $\hat{\lambda} \in \RR$.
Namely, for a mollifying sequence
$\rho_\ve(x) = \ve^{-d}\vr(x/\ve)$ with $\rho \geq 0$ radial, smooth,
compactly supported, and with $\smallint \vr =1$, we consider the equation
\begin{equation*}
  \begin{aligned}
  (1 - \Delta) v_{\ve}  =  \rho_{\ve} \star \big( \lambda_{\ve}  \xi v_{\ve} + \delta_y\big) \;, \qquad \forall x \in \TT^{4} \;,
  \end{aligned}
\end{equation*}
and formally decompose its solution perturbatively {into a power series of the}  coupling constant:
\begin{equation*}
  v_{\varepsilon} (x, y) \; \lqm =" \; G_{\ve} (x-y) +
\sum_{n =1}^{\infty} \lambda_{\ve}^{n} \, \mI_{n, \ve} (x,y)  \;,
\end{equation*}
for some sequence of random fields $ (\mI_{n, \ve} )_{n \geq 1}$ independent of
$\lambda$ and with $G_\ve
= \vr_\ve \star G$ where $G(x) \simeq |x|^{-2}$ is the Green's function of $(1 -
\Delta)$.

This construction does not account for the necessity of renormalising the
solution by removing
diverging quantities that appear already in subcritical dimensions. Here we use
BPHZ renormalisation, which is one of the standard ways to systematically
renormalise solutions to singular SPDEs (as well as Feynman diagrams). We define the BPHZ-renormalised terms $\mf{R}_\ve \mI_{n,
\ve}$ as in Section~\ref{sec_bphz} and through them the renormalised function
\begin{equation}\label{e:expansion}
  \mf{R}_\ve v_\ve (x,y)\; \lqm =" \; G_{\ve} (x-y) +
  \sum_{n =1}^{\infty} \lambda_{\ve}^{n} \, \mf{R}_\ve \mI_{n, \ve} (x,y) \;,
\end{equation}
which formally solves the equation
\begin{equation}\label{e:resolvent-renormalised}
  (1 - \Delta) \mf{R}_\ve v_{\ve}  =  \rho_{\ve} \star \big( \lambda_{\ve}  (\xi-c_\ve) \mf{R}_\ve v_{\ve} + \delta_y\big) \;,
\end{equation}
for a diverging sequence of constants $c_{\ve} \sim
\lambda_{\ve} \ve^{-2}$.
Note that it is unclear whether the sum on the right of \eqref{e:expansion}
converges, or even whether the resolvent
equation for $\mf{R}_\ve v_\ve$ is solvable. Therefore the identity in
\eqref{e:expansion} is only formal and instead of working with
\eqref{e:resolvent-renormalised}, our results concern only the perturbative terms on
the right of \eqref{e:expansion}, {see also Remark~\ref{rem_spde}}. 

To state our main theorem, {we introduce the} centred Gaussian random field  $\limI$
with covariance
\begin{equation}\label{e:Jcorr}
  \EE [ \limI(x, y) \limI(x', y')] = \int_{\TT^4} G(x - z) G(y -z) G(x'-z )
G(y'-z) \ud z \;.
\end{equation}
Then we obtain the following scaling limit {for the expansion terms in
\eqref{e:expansion}.}
\begin{theorem}\label{thm:main}
  For every $n\geq 1$ there exists $ \sigma_{n} \in (0, \infty) $  such that in distribution
  in $ \mS^{\prime}
  ( (\TT^{4})^2 ; \RR)^{\NN_*} $:
  \begin{equation*}
  \begin{aligned}
  \lim_{\ve \to 0} \left( \lambda_{\ve}^{n-1} \,
  \mf{R}_{\ve} \, \mI_{n, \ve}
  \right)_{n \geqslant 1} = \Big( \sigma_{n} \, \hat{\lambda}^{n-1} \, \limI^{(n)} \Big)_{n \geqslant
  1}\;, \qquad \forall \hat{\lambda} \in \RR  \;,
  \end{aligned}
  \end{equation*}
  where $\big( \limI^{(n)} \big)_{n \geq 1}$ is an infinite-dimensional
  multivariate Gaussian, with each $ \limI^{(n)} $ equal in
  distribution to the field $ \limI $ described by~\eqref{e:Jcorr}, and the coefficients
  $\sigma_n$ defined in~\eqref{e:sigma-n-def}.
  \end{theorem}
  It is implied  that we can characterise the covariance
structure between the terms $\sigma_{n} \limI^{(n)}$ in Theorem~\ref{thm:main}, which can be expressed in a similar fashion to the
variances, see again~\eqref{e:sigma-n-def}. We stress that the terms
$\limI^{(n)}$ are not independent.
Moreover, the limiting Gaussian fields are summable up to a critical value of
$|\hat{\lambda}|$.
  \begin{theorem}\label{thm:main2}
    For all $|\hat{\lambda} |< \sqrt{2} \pi$ the series $\mathcal{U} (x,y) := \sum_{n =1}^{ \infty} \sigma_{n} \hat{\lambda}^{n-1} \limI^{(n)}
    (x,y)$ is convergent almost surely, as a random distribution in  $\mS'(
    (\TT^4)^{2} ; \RR)$, and satisfies that in law:
    \begin{equation}\label{eq_sigma_eff}
    \begin{aligned}
	    \mathcal{U}  \overset{d}{=} \sigma_{\mathrm{eff}} ( \hat{\lambda}) \, \limI \;, \qquad \sigma_{\mathrm{eff}}^2(
      \hat{\lambda}) = \frac{2 \pi^{2}}{2 \pi^{2} - \hat{\lambda}^2 } \;.
    \end{aligned}
    \end{equation}
    \end{theorem}
    This theorem formally implies that for $ |\hat{\lambda}| <
    \sqrt{2} \pi$ 
    \begin{equation*}
    \begin{aligned}
    \mf{R}_\ve v_{\ve}(x,y) = G_{\ve}(x-y) + \lambda_{\ve} \; \mathcal{U}(x,y) +  o(\lambda_\ve)\,,
    \end{aligned}
    \end{equation*}
    which suggests that for $|\hat{\lambda}|< \sqrt{2} \pi$ fluctuations of $ \mf{R}_\ve v_{\ve}$ around $ G_{\ve}$ vanish as $ \ve \to 0 $.
    On the other hand, at the critical threshold $|\hat{\lambda}|=\sqrt{2} \pi$,
    we expect the fluctuations to be of the same order as $G_\ve$.  Our result should be understood as a central limit theorem. The terms
$\mI_{n, \ve}$ are centred through the renormalisation procedure, after which,
at the critical dimension one must account for logarithmic variance blow-ups
    through the weak coupling. At the critical value of $\hat{\lambda}$ we
    do not expect the fluctuations to be Gaussian.
    {Variance blow ups of this kind appear naturally when studying scaling critical SPDEs,
    but have also been considered in subcritical systems \cite{hai_var, gt_var, gh_var} where only finitely many
perturbation terms exhibit this behaviour.}
We
    observe that the effective variance $ \sigma_{\mathrm{eff}}^2(
      \hat{\lambda})$ is independent both of the size of
    the torus and of the mass term in the equation as it only relies on the
    leading order term of the Green's function $G(x)= (4 \pi^2
    |x|^2)^{-1}(1+o(1))$ near $x=0$.

The proofs of our results are based on a combinatorial analysis of the Feynman diagrams that
describe the moments of the renormalised terms $\mf{R}_\ve \mI_{n, \ve}$. We
provide a graphical characterisation of the diagrams that contribute to the limiting
fluctuations in terms of nested primitive blow-ups, see Section~\ref{sec_exact}. To
obtain this characterisation, we must prove among other that the renormalisation
has no contribution in the weak coupling limit.
This is quite delicate,
since in other models the renormalisation is expected to impact the weak coupling
limit (essentially through higher order Taylor remainders). This is the mechanism
that should lead for
instance to the appearance of a renormalised Laplacian with an effective
diffusivity in the anisotropic KPZ equation \cite{cannizzaro2023weak}, although for that
model a perturbative argument has not yet been implemented. See the discussion
in Section~\ref{sec_negative}. 

Eventually, our proof provides an
exact combinatorial characterisation of the contribution of each diagram that
does not vanish in the weak coupling limit, leading to the effective
variance $\sigma_{\rm{eff}}^{2}$. The latter is a geometric series
$\sigma_{\rm{eff}}^2(\hat{\lambda}) = \sum_{n \geq 0} (\hat{\lambda} / \sqrt{2}
\pi)^{2n}$, which is a consequence of the iterative structure of the nested
primitive blow-ups that contribute to the variance, see
Section~\ref{sec:effective-variance}. 
This matches the structure of the Gaussian fluctuations of the {multiplicative} Stochastic Heat Equation (SHE)
before the phase transition. The SHE is the
evolution $\partial_t u = \Delta u + \xi u$ with $\xi$ space-time white noise in
$d=2$. In the SHE model, fluctuations are proven to be Gaussian in a weak coupling scaling,
up to a phase transition and with effective variance given by
a geometric series \cite{csz_she} (unlike ours, the overall result in \cite{csz_she} is
not perturbative, although the proof rests on a perturbative expansion). 

Despite significant progress in the study of SHE, even at the critical
value of the coupling constant \cite{csz_shf, tsai2024stochastic}, it has remained so far unclear
how those results would extend to other models, including Anderson, since they rely on an exact decomposition into
It\^o chaoses and do not require to deal with infinite renormalisations.
In fact, to the best of our knowledge, our work is the first that deals with a
scaling critical SPDE that requires nontrivial renormalisation: in anisotropic
KPZ and Burgers no renormalisation is required because of the antisymmetry of
the nonlinearity (which leads to an explicit Gaussian invariant measure)
\cite{cannizzaro2023weak}, and in SHE no renormalisation is required if
one considers It\^o solutions to the equation. Incidentally, It\^o solutions to SHE admit an explicit It\^o
chaos decomposition that lies at the heart of the analysis of \cite{csz_she,csz_kpz}. See also \cite{hu}, which studies a
parabolic Anderson model (in arbitrary dimension) in which the classical 
product is replaced by a generalised Wick product leading to a similar explicit
chaos decomposition. By contrast, in our model such a structure is not available and we
replace it with the study of general Feynman diagrams via primitive blowups.
In fact, in the particular case of Anderson, it is not even clear whether to
expect Gaussian fluctuations, since the solution to the
parabolic model does not possess finite second moments (for all positive times) in any subcritical
dimension $d\geq 2$ \cite{allez2015continuous}. 

Our results give a first indication that fluctuations might
indeed be Gaussian. Moreover, we believe that {non--trivial}  Gaussian fluctuations are a universal feature of critical
SPDEs in weak coupling regimes for small enough coupling constant. The perturbative
approach presented in this work, through the
study of primitive blowups, may help unveil
similar structures in many other models beyond the reach of current
techniques, such as the dynamic $\Phi^4_4$ equation (note that our work
has a direct link to the $\Phi^4_4$ model, see Remark~\ref{rem_phi4}), or KPZ-type
equations.

One challenging problem that is left open is the convergence of the full
solution, and not just that of the formal power series.
In a work concerning random critical initial data \cite{grz_ac},
we could prove that a truncated Wild expansion is
both a good approximation for the solution to the equation and converges to the desired
Gaussian limit.
See also \cite{cd_ac} for an alternative approach, and \cite{HLR} for a treatment of the
subcritical problem.
While some of the arguments of that approach may hold also in the
present case, other aspects are highly model dependent. In
particular, \cite{grz_ac} uses the nonlinear damping of the Allen--Cahn equation to prove that
the truncated expansion remains close to the original solution (in the Anderson model
there is no nonlinear damping). In this regard, we note recent progress in wave
turbulence and kinetic theory, with the use of similar perturbative
arguments and where related challenges are faced, see for
example \cite{DH}.

\subsection*{Outline of the paper}
In Section~\ref{sec:trees} we introduce the terms in the expansion
\eqref{e:expansion}, including the BPHZ renormalisation of the trees, and the diagrammatic notation that will be
used throughout the work. In Section~\ref{sec:proof-main} we collect all the
elements that lead to the proof of the main results, Theorems~\ref{thm:main}
and~\ref{thm:main2}. In particular, we split our analysis into non-negative and
negative diagrams. Non-negative diagrams are then analysed in
Section~\ref{sec:nonneg} where we characterise their contribution in terms of
nested bubbles. We then compute the exact contribution of such diagrams in
Sections~\ref{sec_exact} and~\ref{sec:effective-variance}. Finally we treat
negative diagrams, which require renormalisation, in Section~\ref{sec_negative}.

\subsection*{Notation}
We write $\NN = \{0, 1, 2, \ldots\}$ and $\NN_{*} = \NN \setminus \{ 0 \}$.
For a finite set $A$ we indicate with $|A|$ its cardinality.
For every $d\in \NN_*$ we write $\TT^d$ for the $d$-dimensional torus $\TT^d =
\bigslant{[-\pi, \pi]^d}{\sim} $ with $\sim$ being the equivalence relation that
glues opposite boundaries, for $x, y \in \TT^d $ we write $|x-y|$ for the
distance on the torus.
We also write $\mS(\TT^d; \RR)  = C^{\infty}( \TT^{d}; \RR) $ for the set of Schwartz (smooth) functions over
$\TT^d$. Whenever $d$ is clear from context we write $\mS$. We write $x_{1:n}$ shorthand for the collection of variables $(x_1,
\ldots, x_n)$, or more generally
 $ x_{A} $
indicates the set of variables $ \{ x_{v}  \; \colon \; v \in A \} $ for an
arbitrary index set $ A $.
For two functions $f, g \colon \mX \to \RR$,
for some set $\mX$, we write $f \lesssim g$ if there exists a constant $C>0$ such that $f(x)\leq C g(x)$ for all $x \in \mX$.
For two finite rooted trees $ T_{1}, T_{2}$, we write $[ T_{1} , T_{2}] $ for the grafted
tree, which is obtained by connected both roots to a new common root:
\begin{equation*}
  [T_1, T_2] =
  \begin{tikzpicture}[baseline=-7]
          \node[dot] (0) at (0,0) {};
          \draw (0) to (-.4,-0.4) node at (-.5, -0.6) {\scalebox{0.8}{$T_1$}} ;
          \draw (0) to (0.4,-0.4) node at (0.5,-.6) {\scalebox{0.8}{$T_2$}} ;
        \end{tikzpicture}.
\end{equation*}


\subsection*{Acknowledgments}
We are particularly thankful to Giuseppe Cannizzaro and Nikos Zygouras for their
support in the development of this project, and to Rongchan Zhu and Xiangchan Zhu for a conversation that
sparked our interest in the Anderson model. We are also grateful to  Nikolay
Barashkov, Yvain Bruned, Ajay Chandra, Martin Hairer, Tom Klose, Colin Piernot, and Hao Shen for helpful discussions.

SG acknowledges financial support through the ERC Grant 101045082, held by Hendrik Weber, the NSF Grant
DMS-2424139 while in residence at SLMath in Berkeley, and the DFG
Excellence Cluster in Münster EXC 2044–390685587.
TR acknowledges financial support through the Leverhulme Trust ECF 2024-543.


\section{Trees, Feynman diagrams, and their renormalisation}\label{sec:trees}

In this section, we will introduce the expansion terms $ \mI_{n, \ve}$ from~\eqref{e:expansion-v}.
To this end,
 it is convenient to introduce some suitable notations concerning
trees, Feynman diagrams, and their connection to moments of the
terms $ \mI_{n, \ve} $.
Throughout the paper, we will express different analytic quantities, such as integrals
that define the elements $ \mI_{n, \ve} $ and their moments, through
\emph{directed} graphs.

\subsection{Trees, Feynman diagrams and their valuation}

A directed graph is a couple $ G =
(V, E) $ of vertices $ V $ and (directed) edges $ E \subseteq V \times V $.
The direction of the edges (from the bottom to the top) does not play an important role at first, but will
be useful in keeping track of combinatorial properties
later on. 
{We will denote by
$\mI_{n} $ the directed graph (tree)} that consist of $ n $ internal vertices and two leaves. For example
\begin{equation*}
\begin{aligned}
\mI_{2} =  \<2> \;.
\end{aligned}
\end{equation*}
We colour in red the leaves of the tree, which correspond to nodes $v$ of degree
$\deg (v) =1 $, i.e. the total number of incoming and outgoing edges.
Unlike in the introduction, we indicate with $ \mI_{n} $ only the graph
(the dependence of $ \ve $ appears only when considering the associated
random variable). The set $ V $ of nodes of $ \mI_{n} $ consists of a set
\begin{equation*}
\begin{aligned}
V_{\star} =  \{ v \in V \, : \, \deg \, v = 2 \} \;,
\end{aligned}
\end{equation*}
of \textbf{internal vertices} that we identify with empty black nodes, and a set
\begin{equation*}
\begin{aligned}
L = \{ v \in V \, : \, \deg \, v = 1 \}  = V \setminus V_{\star}  \;,
\end{aligned}
\end{equation*}
of \textbf{leaves} that are identified by empty red nodes. Whenever we are in this setting, we write $
E_{\star} $ for the set of \textbf{internal edges}, namely which connect internal
vertices:
\begin{equation*}
\begin{aligned}
E_{\star}
= \{ e \in E  \; \colon \; e \in V_{\star} \times V_{\star} \} \;,
\end{aligned}
\end{equation*}
and $ E_L $ for the set of \textbf{legs} which are edges that connect to leaves:
\begin{equation*}
\begin{aligned}
E_L =
 \{ e \in E  \; \colon \; e \in V_{\star} \times L \cup L \times
V_{\star} \} = E \setminus E_{\star} \;.
\end{aligned}
\end{equation*}

\subsubsection{Trees as iterated stochastic integrals}

We start by introducing the terms $ \mI_{n, \ve} $ that describe the
perturbative expansion of $v_\ve$ in~\eqref{e:expansion-v} to arbitrary order in $\lambda_\ve$.
We represent graphically the first order term $G_\ve$ in~\eqref{e:expansion-v} as a tree consisting of one line (the kernel $ G_\ve $) connecting two
vertices, corresponding to the variables $ x, y \in \TT^{4} $:
\begin{equation*}
\begin{aligned}
\mI_{0, \varepsilon}  (x,y) = \<0>^{y}_{x, \ve} \;.
\end{aligned}
\end{equation*}
We have added a direction to the edge, so that we can distinguish the top and
the bottom of the tree (we will often omit the variables $x,y$ in the
graphical representation).
Now, we iteratively define
\begin{equation}\label{e_def_In}
\begin{aligned}
 \mI_{n+1, \varepsilon}(x,y) :=
G_{\varepsilon} \star
(\xi\cdot \mI_{n, \varepsilon}(\cdot , y)) (x) \,,
\end{aligned}
\end{equation}
so that formally
\begin{equation}\label{e:expansion-v}
	v_\ve = \sum_{n =0}^{\infty} \lambda_{\ve}^n \, \mI_{n, \ve} \;.
\end{equation}
Each term in the expansion is therefore associated to a
``linear'' tree in which every inner node corresponds to a noise $ \xi
$, and each edge corresponds to a kernel $ G_{\ve} $,
and we integrate over all variables associated to internal
nodes. The fact that black nodes are empty (a white interior) indicates the presence of a noise at that
node. For example
\begin{equation}\label{eq_laddertrees}
\begin{aligned}
\mI_{1, \varepsilon} = \<1>_{\ve} \;, \qquad \mI_{2, \varepsilon} = \<2>_{\ve}\;, \qquad
\mI_{3,\varepsilon} = \<3>_{\ve} \;, \qquad
\ldots \qquad \;.
\end{aligned}
\end{equation}
Each empty red node is associated to a variable $ x $ (at the bottom)
or $ y $ (at the top).

\subsubsection{Feynman diagram representation of moments}

Next, we introduce pairings of the trees $ \mI_{n} $ to describe moments
of $ \mI_{n , \ve} $.
Each tree $ \mI_{n, \ve} $ from~\eqref{e_def_In} corresponds to a random variable in the $ n
$-th inhomogeneous It\^o chaos, and it can therefore be decomposed into
its homogeneous chaos components. This can be achieved by summing over all internal
pairings of noises within the tree
\begin{equation}\label{e:int-cont}
\begin{aligned}
\mI_{n, \ve} = \sum_{\kappa \in \mK_{n}^{\mathrm{int}}} \mI_{n, \kappa, \ve} \;.
\end{aligned}
\end{equation}
Here $ \mK_{n}^{\mathrm{int}} $ is the set of all possible internal pairings. A pairing is a bijection $ \kappa
\colon V_{\kappa}^{(1)} \to V_{\kappa}^{(2)} $, where $ [n] = \{ 1, \dots, n \}
$ enumerates the set of internal vertices of $
\mI_{n} $ and $ V_{\kappa}^{(i)} \subseteq [n] $ are two disjoint sets
of the same size. We identify $\kappa$ with its inverse $\kappa^{-1} \colon V^{(2)}_{\kappa} \rightarrow V^{(1)}_{\kappa}$, so that a pairing $ \kappa $ indicates that vertices in $ V_{\kappa}^{(1)},
V_{\kappa}^{(2)} $ are paired according to the rule $ \kappa $, namely if $v=\kappa(w)$ or equivalently $w=\kappa^{-1}(v)$.
We write $ \kappa = \emptyset$ for the empty pairing, when $ |
V_{\kappa}^{(1)}| = |V_{\kappa}^{(2)}|=0$.
We say that a pairing $ \kappa $ is \emph{complete}, if $
V_{\kappa}^{(1)} \sqcup V_{\kappa}^{(2)} = [n] $, which is only possible if $ n$ is even.

To any pairing $ \kappa $ we associate a \emph{multi-graph} $ \mI_{n, \kappa} = (V(\mI_{n,\kappa}), E(\mI_{n,\kappa}))
$, according to the rule that a pairing $ \kappa $  identifies two nodes $ v $ and $ w $
provided $ \kappa
(v) = w $.
This corresponds to
quotienting the original graph by the equivalence relation induced by $\kappa$.
Moreover, in the new graph we denote paired nodes with a full
black node, to distinguish these nodes from the
unpaired ones (with a white interior), which will still be associated to instances of the noise.
For example, if $ n =2 $ then we can pair the two vertices
\begin{equation}\label{e:loop}
  \<c21_half> = \<c21_bl> \;,
\end{equation}
where we represent pairings through red dotted
lines connecting vertices $v,w$ whenever $v=\kappa(w)$.

Before we proceed, let us comment on the structure of paired graphs. A
directed multigraph is a couple $(V,E)$, where the edge set $E\subseteq_m
V\times V$ is a finite multi-subset of $V\times V$, meaning that edges may
appear more than once.
For example $\{e_1, e_1, e_2, e_2, e_2, e_3\}\subseteq_m V\times V$ if every $e_i\in V\times V$.
From here onwards, we will omit the subindex $m$ since no confusion can arise.
The multigraph definition is necessary,
for instance due to the following pairing
\begin{equation*}
  \begin{aligned}
    \<c4> = \<cc4> \;.
  \end{aligned}
\end{equation*}
Now, to every paired graph $ \mI_{n , \kappa} $
we associate a homogeneous
Gaussian integral $ \mI_{n , \kappa, \ve} $ as follows. For every $ \ve \in
(0, 1) $, we first associate to $ \mI_{n, \kappa} $ an $ \ve $-dependent
kernel through the mapping $ \mW_{\ve} $:
\begin{equation*}
\begin{aligned}
\mW_{\ve} \mI_{n, \kappa} (x_{V_\star}) = \prod_{e \in E_{\star}}
G_{\ve}  (x_{e_{+}} - x_{e_{-}}) \;,
\end{aligned}
\end{equation*}
where $ E $ is the edge set of the multi-graph $ \mI_{n, \kappa} = (E, V) $, with every edge being
of the form $ e = (e_{-}, e_{+}) \in V \times V $.
We then set
\begin{equation}\label{eq_contractedHomIntegral}
\begin{aligned}
\mI_{n, \kappa, \ve} (x_L) = \int_{( \mathbf{T}^{4})^{ V_{\star} }} \mW_{\ve} \mI_{n, \kappa}
(x_{V_\star}) \prod_{e \in E_L} G_{\varepsilon}( x_{e_+}- x_{e_-}) \prod_{v \in V_{o}} \xi (x_{v}) \ud x_{V_{\star}} \;,
\end{aligned}
\end{equation}
where $ V_{o} = V_{\star}  \setminus (  V^{(1)}_\kappa \sqcup V^{(2)}_\kappa )$ is the set of unpaired inner variables. The integral
\eqref{eq_contractedHomIntegral} should be interpreted as a homogeneous Gaussian integral in the sense of
\cite{Nualart}.
For example,  $ \mI_{1, \ve}  $ is Gaussian and~\eqref{e:int-cont} yields the
 representation  by the empty
pairing
\begin{equation*}
\begin{aligned}
 \mI_{1, \ve} = \mI_{1, \varnothing, \ve}
=  \<1>_{ \varnothing, \varepsilon}  \;,
\end{aligned}
\end{equation*}
since no other pairings are possible.
On the other hand, $ \mI_{2 , \ve} $ can be decomposed in terms of
\begin{equation} \label{e:ex-2}
\begin{aligned}
\<2>_{\,\varepsilon}= \<2>_{ \varnothing, \varepsilon} + \<c21_half>_{\!\varepsilon} \;.
\end{aligned}
\end{equation}
Moreover, we don't expect the limit of  $
\lambda_{\varepsilon}^{n -1} \mI_{n , \kappa, \varepsilon}$ to be function valued, but
instead to lie in the Besov--H\"older space $
\mC^{0 -} ( ( \mathbf{T}^{4})^{2}; \RR) $, which is a space of distributions.
Therefore, we evaluate $\mI_{n,\kappa ,\ve}$
against smooth test functions $\varphi \in \mathcal{S}( (\TT^4)^{L} ;
\RR)$, and write
\begin{equation}\label{eq_def_Itested}
\begin{aligned}
\mI_{n, \kappa, \varepsilon} ( \varphi )
:=
\int_{(\mathbf{T}^{4})^2} \mI_{n, \kappa, \varepsilon}
(x_L) \varphi (x_L) \ud x_L  \,,
\end{aligned}
\end{equation}
and similarly $ \mI_{n, \varepsilon}(\varphi)$ according to~\eqref{e:int-cont}.

On the other hand, two-point correlations of a given tree are given by the sum
over all possible pairings between two copies of that tree. For
example,
\begin{equation*}
\begin{aligned}
\EE   \left[ \mI_{1, \ve }(x_{1}, y_{1}) \mI_{1, \ve }(x_{2}, y_{2})\right] =
\<c1>_{\ve} \;, \qquad \EE  \left[ \mI_{2, \ve }(x_{1}, y_{1}) \mI_{2, \ve }(x_{2},
y_{2})\right] = \<c21>_{\ve} + \<c22>_{\ve} +\<c23>_{\ve} \;.
\end{aligned}
\end{equation*}
and more generally,
\begin{equation}\label{eq_2pt_corr}
\begin{aligned}
\EE \left[ \mI_{n, \ve} (x_{1}, y_{1}) \mI_{k, \ve} (x_{2}, y_{2}) \right] =
\sum_{\kappa \in \mK(\mI_{n}, \mI_{k})} (\mI_{n}, \mI_{k})_{\kappa, \ve}
(x_{12}, y_{12}) \;,
\end{aligned}
\end{equation}
where $ \mK (\mI_{n}, \mI_{k}) $ is the set of all \emph{complete} pairings between the
two trees $ \mI_{n}$ and $ \mI_{k} $. That is, bijections $ \kappa \colon
V^{(1)}_{\kappa} \to V^{(2)}_{\kappa} $, where $ V^{(1)}_{\kappa} \sqcup
V^{(2)}_{\kappa} = V_{\star} (\mI_{n}) \sqcup V_{\star} (\mI_{k}) $. Then, as
before the pairing $ \kappa $ produces a new graph $ (\mI_{n},
\mI_{k})_{\kappa} $ in which paired nodes are coloured full black.

\begin{remark}\label{rem:labels}
 All graphs and diagrams are considered labelled, unless explicitly
  stated, even if the labelling is not referred
  to explicitly.
In particular, the vertices of a single tree $\mI_n$ are considered
  labelled from $0$ to $n+1$, from the bottom to the top, with $0$ being the
  bottom, and $n+1$ the top. For a pairing $\kappa \in \mK(\mI_n, \mI_k)$,
  it will be useful to fix some rule for the labelling of $(\mI_n,
  \mI_k)_\kappa$. The particular choice of the rule does not matter, however for
  completeness we use the following:
  \begin{itemize}
    \item Start from the root of the tree $\mI_n$ and follow it upwards in
    its embedding in $(\mI_n,
    \mI_k)_\kappa$. Label the root $0$ in $(\mI_n,
    \mI_k)_\kappa$, and every time a new vertex is visited, label it
    increasingly ($1, 2, \ldots$)  starting from zero.
    \item Upon reaching the upper leaf of $ \mI_{n}$, start from the root of $\mI_k$ and repeat the same procedure.
  \end{itemize}
\end{remark}

As before, we can associate an integral kernel to each graph on the right--hand side of
\eqref{eq_2pt_corr}, through the map $ \mW_{\ve} $:
\begin{equation*}
\begin{aligned}
\mW_{\ve } (\mI_{n}, \mI_{k})_{\kappa} (x_{V_\star})  = \prod_{e \in
E_{\star} }
G_{\ve} (x_{e_{+}} -
x_{e_{-}}) \;,
\end{aligned}
\end{equation*}
such that
\begin{equation*}
\begin{aligned}
(\mI_{n}, \mI_{k})_{\kappa, \ve } (x_{L})= \int_{(\TT^{4})^{
V_{\star}}} \mW_{\ve} (\mI_{n}, \mI_{k})_{\kappa} (x_{V_{\star} })
\prod_{e \in E_L} G_{\varepsilon}( x_{e_+}- x_{e_-})  \ud x_{V_{\star}}
\;,
\end{aligned}
\end{equation*}
where $ E_{\star} $ is the set of inner edges in $ (\mI_{n}, \mI_{k})_{\kappa} = (V,E)$.

Therefore, the pairing between integration variables leads to Feynman diagrams, which are (deterministic)
integrals associated to the graphs. They can be
distinguished pictorially from the stochastic integrals~\eqref{eq_laddertrees},
 since all their inner vertices are
coloured with full black dots. For example,
\begin{equation*}
\begin{aligned}
\<c1> = \<cc1> \;, \qquad \<c21> = \<cc21> \;, \qquad \<c22> = \<cc22>\;,
\qquad \<c23> = \<cc23> \;,
\end{aligned}
\end{equation*}
where we represented the last two paired graphs as multi-graphs. Moreover, we note that different pairings
can give rise to the same diagram shapes, however representing different analytic
expressions due to the distribution of variables among the
leaves. For example,~\eqref{eq_2pt_corr} yields
\begin{equation} \label{e:exmpl}
\begin{aligned}
\EE  \left[ \, \mI_{2, \ve }(x_{1}, y_{1}) \, \mI_{2, \ve }(x_{2}, y_{2}) \, \right] =
\<cc21>_\ve +  \<cc22>_\ve +  \<cc23>_\ve\,,
\end{aligned}
\end{equation}
where in the last term on the right side the variables $\{x_{1},
y_{2}\}$ and $\{y_{1}, x_{2}\}$
are attached to the top and bottom nodes, respectively.
This is indicated by the direction of the edges.

Instead of considering only two--point correlations, the above can
be generalised
to an arbitrary number of trees and multi-point
correlations.
To this end, we
introduce the following sets of complete pairings
among a set of trees $ \{ \tau_{i}
\}_{i =1}^{k} $,  $ \tau_{i} \in  \{ \mI_{n} \}_{n \geqslant 1} $:
\begin{equation*}
\begin{aligned}
\mK(\tau_{1}, \dots, \tau_{k}) & = \{ \text{complete pairings of }
\tau_{1}, \dots, \tau_{k}  \} \;.
\end{aligned}
\end{equation*}
For any $ \kappa \in \mK (\tau_{1}, \dots, \tau_{k}) $ we call the graph
$ (\tau_{1}, \dots, \tau_{k})_{\kappa} $ a Feynman diagram.
\begin{definition}\label{def:FDgeneral}
  A directed multi-graph $ \Gamma = (V, E) $ is a Feynman diagram, if it is
endowed with a set $L \subseteq V$ of leaves.
We write
  \begin{equation*}
    V_\star = V\setminus L, \qquad E_\star = \{ e \in E \text{ s.t. } e \in V_\star \times V_\star \}\;, \qquad E_L = E \setminus E_\star
  \end{equation*}
  for the set of internal vertices, internal edges, and legs respectively.
Since $E$ is a multiset, each edge $e$ can appear multiple times. We denote
with $\mf{n}(e) \in \NN$ such multiplicity.
\end{definition}
We say that a Feynman diagram $ \Gamma = (V, E) $ is an \textbf{Anderson-Feynman diagram}, if there exists a
$ k \in \NN $ and trees $ \tau_{i}  \in  \{
\mI_{n} \}_{n \in \NN} $, $i =1 ,\ldots, k $, and a pairing $ \kappa \in
\mK (\tau_{1}, \dots , \tau_{k}) $ such that
\begin{equation}\label{eq_AndersonFeyn}
\begin{aligned}
\Gamma = (\tau_{1}, \dots, \tau_{k})_{\kappa} \;,
\end{aligned}
\end{equation}
with the set $L$ of leaves being the union of the leaves of the trees $\tau_i$.
Note that $ k $ is uniquely determined by $ \Gamma $,  since $ 2 k $ is the
number of leaves (or legs) of $ \Gamma $.
To highlight the dependence on $ k $, we will
refer to $ \Gamma $ as a $ 2k $-Anderson--Feynman diagram.
{Moreover, we notice that every inner node of an Anderson--Feynman diagrams is
connected to precisely four half--edges, i.e. it is graph theoretically of degree
four.} 

{
\begin{remark}[Connection to $\Phi^{4}$--Feynman diagrams]\label{rem_phi4}
Anderson–Feynman diagrams are those Feynman diagrams in which each internal vertex 
has degree four and edges represent Green’s functions. 
They therefore coincide with the diagrams generated by the $\Phi^{4}$--measure in
Euclidean quantum field theory, see for example \cite[Chapter~7.4]{zinn2021quantum}. 
Formally, the Anderson Hamiltonian is related to the (wrong-sign) $\Phi^{4}$ measure via a Hubbard–Stratonovich transformation \cite[Chapter~4]{Efetov_1996}.
To obtain the $\Phi^{4}$--measure with the conventional negative sign,
the coupling constant $\lambda$ must be taken purely imaginary. 
This does not affect our analysis: the series defining $\sigma^{2}_{\mathrm{eff}}$ remains summable. 
Most notable, this leads to a sign change in the denominator of \eqref{eq_sigma_eff}.
It therefore recovers, 
at the perturbative level, the triviality of the $\Phi^{4}_{4}$--measure \cite{Aizenman2024}. 
Further details will be addressed elsewhere.
\end{remark}
}

Feynman diagram are interpreted as deterministic integral expressions, which
are functions of the leaf variables. We generalise the valuation of a Feynman diagram
from the case $ ( \mI_{n}, \mI_{k})_{\kappa, \varepsilon}$ as follows:
\begin{definition}\label{def:valuation}
  To a connected Feynman diagram $ \Gamma$, we associate the integral
\begin{equation}\label{eq_def_val}
\begin{aligned}
\Pi_{\varepsilon}\Gamma
( x_{L }) := \begin{cases}
  \int_{(\TT^{4})^{ V_{\star} }} \mW_{\ve} \Gamma
  (x_{V_\star})
  \prod_{e \in E_L}
  G_{\ve}(x_{e_+} - x_{e_-})
   \ud
  x_{V_{\star}} \;, & \text{ if } L \neq \emptyset \;, \\
  \int_{(\TT^{4})^{ V_{\star}\setminus \{v \} }} \mW_{\ve} \Gamma
  (x_{V_\star})
   \ud
  x_{V_{\star} \setminus \{v \}} \;, & \text{ if } L = \emptyset \;,
\end{cases}
\end{aligned}
\end{equation}
where $v \in V_\star$ in the second case is arbitrary, and
\begin{equation}\label{e_def_mW}
\begin{aligned}
\mW_{\ve} \Gamma (x_{V_\star}) := \prod_{e \in E_{\star}}
G_{\ve}(x_{e_{+}} - x_{e_{-}}) \;.
\end{aligned}
\end{equation}
If $\Gamma$ consists of multiple connected components
$\{\Gamma_i\}_{i=1}^n$, then $\Pi_\ve \Gamma = \prod_{i=1}^n \Pi_\ve \Gamma_{i}$.
\end{definition}

In the case $L= \emptyset$ we integrate over all but one internal variable (the
one associated to the vertex $v$). In this case, the
integral does neither depend on the choice of $v$ nor on the variable
$x_{v}$ by translation invariance.
The case $L  = \emptyset$ is not of interest in the discussion above, but
it will be used in the renormalisation procedure drescribed in
Section~\ref{sec_bphz} below, and it is crucial to obtain the correct definition
of renormalisation constant. This matches the notion of vacuum diagrams
in~\cite[Definition~2.9]{BPHZ}.

Moreover, because the limiting random fields that we are going to treat are distribution-valued, the kernel $ \Pi_{\varepsilon} \Gamma (
x_{L}) $ will be tested against smooth functions.
Thus, for a smooth test function $\varphi \in \mS((\TT^4)^L ; \RR) $, we define
\begin{equation}\label{eq_def_valTested}
  \Pi_{\varepsilon}\Gamma ( \varphi) := \int_{(\TT^4)^L} \Pi_{\varepsilon}\Gamma
  ( x_{L })  \varphi(x_L) \ud x_L \;.
\end{equation}
Finally, with this notation at hand, the following formula holds for $ k $-point
correlation functions of any $\{\tau_i \}_{i =1}^{k}$, with $
\tau_{i}\in \{\mI_{n}\}_{n \in \NN}$:
\begin{equation} \label{e:moments}
\begin{aligned}
 \EE[ \tau_{1, \ve } (x_{1}, y_{1}) \cdots \tau_{k, \ve} (x_{k}, y_{k}) ] = \sum_{\kappa \in \mK  (\tau_{1}, \ldots,
\tau_{k})}
\Pi_{\varepsilon}(\tau_{1}, \dots, \tau_{k} )_{\kappa}
(x_{1:k}, y_{1 : k})
\;.
\end{aligned}
\end{equation}

\subsection{BPHZ renormalisation}\label{sec_bphz}
Even in lower dimensions, and in fact in the entire singular subcritical
regime, the limit of single trees $ \mI_{n, \ve} $ as $ \ve \to 0 $ does not necessarily 
exist. This is due to the presence of divergences, which need to be removed through a renormalisation procedure.
Let us start with an example. The second tree $ \mI_{2, \ve} $ diverges as $ \ve \to 0 $, because of the presence of the loop~\eqref{e:loop}:
\begin{equation}\label{eq_tadpole}
  \EE \, \, \<2> = \Pi_\ve \, \,  \<c21_bl> = \Pi_\ve \, \,  \<loop> \cdot \Pi_\ve \, \,  \<1_b> \, \,,
\end{equation}
where the loop $ \Pi_\ve \, \, \<loop> = G_\ve(0) \simeq \ve^{-2}$ leads to a polynomial divergence. These divergences are different from the ones that are cured by the weak coupling, since the former appear when computing expectations of trees, while the latter appear from blow-ups in their variances.
In particular, the first kind of divergence can be cured through the
subtraction of suitable counterterms, while variance blow-ups can only be cured
through weak coupling.

In some cases the renormalisation procedure is simple. For example, the loop in~\eqref{eq_tadpole} must only be subtracted: We set
\begin{equation}\label{e:ren-example}
  \mf{R}_\ve \, \, \, \<2> = \<2>_\ve - \Pi_\ve \<loop> \; \<1_b>_\ve =  \<2>_{\varnothing, \ve}\;,
\end{equation}
where on the right side we have used the notation introduced in~\eqref{e:ex-2} to indicate the component of the original tree in the second homogeneous chaos.
Note that by virtue of the resulting renormalisation, we find that the analogue of~\eqref{e:exmpl} for the renormalised trees reads
\begin{equation*}
\begin{aligned}
\EE  \left[ \, \mf{R}_{\ve} \, \mI_{2 }(x_{1}, y_{1}) \, \mf{R}_{\ve} \, \mI_{2 }(x_{2}, y_{2}) \, \right] =
 \<cc22>_\ve + \<cc23>_\ve  \;,
\end{aligned}
\end{equation*}
where the divergent diagram $ \<cc21>_\ve  $ appearing in~\eqref{e:exmpl} has
been removed by the renormalisation procedure. The resulting tree~\eqref{e:ren-example} has mean zero and a variance that is finite in
subcritical dimensions, but which diverges logarithmically in dimension four.
Precisely such variance blow ups will be compensated by the weak coupling.
This example leads to the two main points that we must address in this section:
\begin{itemize}
  \item A systematic way of defining the map $\mf{R}_{\ve}$: This is given by the well-known BPHZ renormalisation procedure.
  \item A link between renormalisation of trees and the computation of moments, which amounts to renormalisation of Feynman diagrams. This is necessary in our arguments because the relevance of the weak coupling appears only when computing second moments.
\end{itemize}
In fact, we will proceed somewhat backwards, starting from the definition of BPHZ renormalisation $\mR$ for Feynman diagrams and building from that a renormalisation map at the level of the trees.

\begin{remark}\label{rem_spde}
  The result of our next discussion is the construction of a renormalisation
  procedure $\mf{R}_\ve$ for trees. Such renormalisation is widely used in many
  different settings in the study of singular SPDEs~\cite{BrunedHairerZabotti19,Ajay, BCCH21}. In
  particular, we expect this map to coincide, mutatis mutandis, with the
  renormalisation procedure in~\cite{Ajay}. In fact our setting is simpler
  because we do not require the positive renormalisation that appears in
  regularity structures and we are only considering Gaussian noise. Checking
  this equivalence, and in addition establishing a link between the
  renormalisation of the trees and the renormalisation of the original equation
  (cf.~\cite{BCCH21}) requires some technical work that is unnecessary for the heart of our
  arguments, and would not make our work self-contained. It is therefore left
to a future project.
\end{remark}

\subsubsection{Degree, contractions, extractions, and BPHZ for Feynam diagrams}

In order to understand whether the associated integral $ \Pi_{\ve} \Gamma $ of
a Feynman diagram $ \Gamma$ diverges, and to what extent, we require the notion
of degree.

\begin{definition} \label{def:subdiagram}
  Given a Feynman diagram $\Gamma = (V, E)$ we say that $\overline{\Gamma} = (
\overline{V}, \overline{E})$ is a subdiagram of $\Gamma$ if $\overline{V}
\subseteq V_\star$ and $\overline{E} \subseteq E_\star $, where the last
inclusion is in the sense of multi-sets, i.e. the multiplicity
$\overline{\mf{n}} (e)$ of $e$ in $\overline{\Gamma}$ is smaller than the
multiplicity $\mf{n}(e)$ of $e$ in $E$: $\overline{\mf{n}}(e) \leqslant
\mathfrak{n}(e)$.
We say $ \overline{\Gamma}$ is a full subdiagram, if $\overline{\mf{n}}(e) =
\mathfrak{n}(e)$ for every $ e \in \overline{E}$.
\end{definition}

Given a connected subdiagram $ \overline{\Gamma} =
( \overline{V}, \overline{E}) \subseteq \Gamma $
 we associate to it the degree
\begin{equation}\label{e_def_deg}
\begin{aligned}
\deg \overline{\Gamma}
:= 4 (| \overline{V}|-1) - 2 |\overline{E}|  \,,
\end{aligned}
\end{equation}
where $4$ is the spatial dimension, and $-2$ is the homogeneity of the Green's function in four dimensions (see Section~\ref{sec:nonneg} for a more general definition).
If a diagram consists of several connected components, we define its degree as the sum of the degrees of its connected components.
In the following discussion it will be furthermore important to extract and
contract subdiagrams of Feynman diagrams. We define
contractions as follows.

\begin{definition}\label{def:contract}
  Given a Feynman diagram $ \Gamma = (V, E) $ and a connected, full subdiagram $
  \overline{\Gamma} = ( \overline{V} , \overline{E}) $,
  we define the \emph{contracted diagram} $ \mK_{\overline{\Gamma}} \Gamma = (
  V ( \mK_{ \overline{\Gamma}} \Gamma), E
  (\mK_{ \overline{\Gamma}} \Gamma)) $ as follows. We replace all vertices of
  $ \overline{\Gamma} $ with a single new vertex $ v_{\bg} $, meaning that
  \begin{equation*}
  \begin{aligned}
  V (\mK_{ \overline{\Gamma}} \Gamma ) = V \setminus \overline{V} \sqcup \{ v_{\bg} \} \;,
  \end{aligned}
  \end{equation*}
  and for every edge $ e = (e_{-}, e_{+}) \in E
  \setminus \overline{E} $ we replace $ e $ with $ \mK_{\overline{\Gamma}} (e) =
  (e_{-},v_{\bg}) $ (resp. $ \mK_{\overline{\Gamma}} (e) = ( v_{\bg},
e_{+}) $) if $
  e_{+} \in \overline{V}$ (resp. $ e_{-} \in \overline{V} $) and leave it unchanged otherwise. All edges in $\overline{\Gamma}$ are deleted:
  \begin{equation*}
    E(\mK_{\overline{\Gamma}} \Gamma ) = \mK_{\overline{\Gamma}} (E \setminus \overline{E}) \;.
  \end{equation*}
  We say that $\overline{\Gamma}$ has been contracted to $v_\bg$, and whenever
  convenient we use the notation
  \begin{equation*}
    \mK_{\overline{\Gamma}} \Gamma = \bigslant{\Gamma}{\overline{\Gamma}} \;.
  \end{equation*}
  Finally, if the subdiagram $\overline{\Gamma} = ( \overline{V} , \overline{E})
  $ is full, but not necessarily connected (with connected components
  $\overline{\Gamma}_i$ for $i = 1, \ldots, n$), then we define the contracted
  diagram
  \begin{equation*}
    \mK_{ \overline{\Gamma}} \Gamma = \mK_{\overline{\Gamma}_n} \cdots \mK_{\overline{\Gamma}_1} \Gamma \;,
  \end{equation*}
  where the order of the contractions does not matter and we view
  $\overline{\Gamma}_{i+1}$ as a subdiagram of $ \mK_{\overline{\Gamma}_i}
  \cdots \mK_{\overline{\Gamma}_1} \Gamma$, see Remark~\ref{rem:subdiam}.
  \end{definition}
  The following is an example of a contraction (we omit the labelling of the vertices for simplicity):
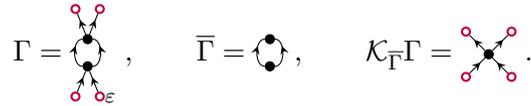
\begin{equation*}
  \Gamma =  \<cc22>_\ve \;, \qquad \bg = \<bubbleA> \;, \qquad \mK_{\bg} \Gamma = \<cc1>  \;.
\end{equation*}
\begin{remark}\label{rem:subdiam}
  The contraction $\mK_{\bg}$ induces a map from subdiagrams of $\Gamma$ to
subdiagrams of $\mK_{\bg} \Gamma$. Namely if $\widetilde{\Gamma}$ is another
subdiagram of $\Gamma$, then $\mK_{\bg} \widetilde{\Gamma}$ is a subdiagram of
$\mK_{\bg} \Gamma$. The former is defined through the same procedure described in
Definition~\ref{def:contract},
by contracting all vertices in $V(\widetilde{\Gamma}) \cap V(\Gamma)$ to
$v_\bg$. We call $\mK_{\bg} \widetilde{\Gamma}$ the image of
$\widetilde{\Gamma}$ in $\mK_{\bg} \Gamma$.
\end{remark}
Next, we introduce \emph{strictly} negative divergences, which are the ones that we will cure through renormalisation.
  Let $\Gamma$ be a Feynman diagram and set
\begin{equation}\label{e_def_Gminus}
\begin{aligned}
 \mG^{-}_{\Gamma}
  :=\Big\{ \overline{\Gamma} \subset \Gamma: \; \overline{\Gamma} \; \text{is
connected and full with } \deg \overline{\Gamma}< 0\;\Big\}.
\end{aligned}
\end{equation}  \begin{remark}
    It is important to note that we will only renormalise divergences that are
\emph{strictly} negative. Also diagrams with degree equal to zero lead to
logarithmic divergences, however, they are accounted for by the weak coupling.
  \end{remark}
  Now, for a given divergent diagram $\overline{\Gamma} \in \mG_{\Gamma}^-$, we define the extraction map
\begin{equation}\label{e_def_contrextr}
\begin{aligned}
\mC_{\overline{\Gamma}} \Gamma = \overline{\Gamma} \cdot
\mK_{\overline{\Gamma}} \Gamma\,.
\end{aligned}
\end{equation} Here the product of Feynman diagrams corresponds simply to the disjoint union of Feynman diagrams. For example
  \begin{equation*}
    \Gamma = \<cc4> \quad \text{and} \quad \qquad \bg = \<cc4div> \qquad \text{lead to} \quad \mC_{\bg} \Gamma = \<cc4div> \, \cdot \, \<1_b>\;,
  \end{equation*}
  where we note that $\deg (\bg) = -2$, so that $\bg \in
\mG^-_{\Gamma}$.
  \begin{remark}
    Note that with our definitions both a subdiagram $\bg$ and the contracted
diagram $\mK_{\bg}\Gamma$ are Feynman diagrams.
In a common terminology, that we
do not require, $\bg$ is referred to as a vacuum diagram, because it does not
have any legs. In particular, we view $\mC_{\bg} \Gamma$ as an element of the
polynomial ring generated by all Feynman diagrams. Moreover, the map $\Pi_\ve$
extends uniquely to a ring homomorphism from this polynomial ring to $\RR$, by
setting $\Pi_\ve(\Gamma_1 \cdot \Gamma_2) = \Pi_\ve(\Gamma_1)
\Pi_\ve(\Gamma_2)$.
  \end{remark}
  In order to cure multiple divergences, which possibly overlap,
 we must extend our definition to subsets of divergences.
As it turns out, it is sufficient to consider forests of divergences.
\begin{definition}\label{def:forest-div}
A forest is a subset
$\mf{F} \subset \mathcal G^{-}_{\Gamma}$ such that for any $\overline{\Gamma}_1,\overline{\Gamma}_2\in \mf{F}$ either
$\overline{\Gamma}_1\subset \overline{\Gamma}_2$, or $ \overline{\Gamma}_2\subset \overline{\Gamma}_1$, or
$\overline{\Gamma}_1$ and $ \overline{\Gamma}_2$ are vertex-disjoint. We denote by $\mathcal F^{-}_{\Gamma}$
the set of all forests in $\mG^{-}_{\Gamma}$.
\end{definition}
We extend the definition
of the extraction map to forests of divergent subdiagrams: For any $\mf{F} \in \mF^{-}_{\Gamma}$ we set
  \begin{equation}\label{e:cont-for}
    \mC_{\mf{F}} \Gamma = \prod_{\overline{\Gamma} \in \mf{F}} \mC_{\overline{\Gamma}} \Gamma = 
    \mC_{\mf{F}\setminus \vr(\mf{F})}\prod_{ \bg \in \vr(\mf{F})}  \mC_{\bg} \Gamma \;,
  \end{equation}
  where the order in the first product does not matter, and the second is an
equivalent inductive definition, in which we have denoted with $\vr(\mf{F})$
the set of roots (maximal diagrams in terms of inclusion) of $\mf{F}$. It is
implied in the
definition, that once we
extract a larger diagram, the smaller one will be extracted from the larger
one.
For this reason the forest condition is important.
 More precisely, recall that for $\Gamma_1, \Gamma_2 \in \mf{F}$, either $\Gamma_1
\subseteq \Gamma_2$, or $\Gamma_2 \subseteq \Gamma_1$, or the two are vertex
disjoint. In the first case we set  $\mC_{\Gamma_2}  \mC_{\Gamma_1} \Gamma =
\mC_{\Gamma_1} \Gamma_2 \cdot \mK_{\Gamma_2} \Gamma$. In the second case we set
$\mC_{\Gamma_2}  \mC_{\Gamma_1} \Gamma =  \mC_{\Gamma_2} \Gamma_1 \cdot
\mK_{\Gamma_1} \Gamma$. In the last case, we set $\mC_{\Gamma_2}
\mC_{\Gamma_1} \Gamma =  \Gamma_1 \cdot \Gamma_{2} \cdot \mK_{\Gamma_2}
\mK_{\Gamma_1} \Gamma$, where we view $\Gamma_2$ as a subdiagram of
$\mK_{\Gamma_1} \Gamma$, cf. Remark~\ref{rem:subdiam}.
For example, consider
  \begin{equation}\label{exmp_bphz}
    \Gamma = \<bphze1> \;, \qquad \Gamma_{12} = \<bphze2>\;, \qquad \Gamma_{34} = \<bphze3>\;, \qquad \Gamma_{1234} = \<bphze4>\;,
  \end{equation}
  where $\deg(\Gamma_{12}) = \deg(\Gamma_{34}) = \deg(\Gamma_{1234}) =-2$.
The diagram $ \Gamma $ arises for example when calculating the
mean of $ \mI_{8, \varepsilon}$.
 Then for the forest
    $\mf{F} = \{ \Gamma_{12}, \Gamma_{34}, \Gamma_{1234} \}$, we have
  \begin{equation*}
    \mC_{\mf{F}} \Gamma =   \<bphze2> \cdot \<bphze3> \cdot \<0_bl> \cdot \<1_bl> \;.
  \end{equation*}
Because the evaluation of a Feynman diagram does not depend on the labelling,
the choice of labels is irrelevant.
However we kept track of it precisely in this example, to show that there is no ambiguity in the definition of~\eqref{e:cont-for}.
Finally, we are ready to state the BPHZ renormalisation of a Feynman diagram
$\Gamma$ through Zimmerman's forest formula, {see for example
\cite[Proposition~3.3]{BPHZ}:}
\begin{equation}\label{eq_zimmermann}
  \mR \Gamma = \sum_{\mf{F} \in \mathcal{F}_\Gamma^{-}} (-1)^{|\mf{F}|}  \mC_{\mf{F}} \Gamma \;,
  \qquad \hat{\Pi}_\ve \Gamma = \Pi_\ve \mR \Gamma \,.
\end{equation}
Note that in the valuation on the right side we used
Definition~\ref{def:valuation}, so that all the diagrams that are contracted and
extracted are integrated over all variables but one (since we are in the case $L = \emptyset$ of
that definition), while only the remainder of the original diagram is integrated
over all internal variables.
Following up on the example in~\eqref{exmp_bphz}, we have
\begin{equation*}
  \begin{aligned}
  \hp \Gamma = & \Pi_\ve \Gamma - 2 \Pi_\ve \Gamma_{12}  \cdot \Pi_{\ve}  \mK_{\Gamma_{12}} \Gamma - \Pi_\ve \Gamma_{1234}   \cdot \Pi_{\ve} \mK_{\Gamma_{1234}} \Gamma + (\Pi_\ve \Gamma_{12})^2 \cdot \Pi_{\ve} \mK_{\Gamma_{12}\sqcup \Gamma_{34}} \Gamma \\
  &+ 2 \Pi_\ve \Gamma_{12} \cdot \Pi_{\ve} \mK_{\Gamma_{12}} \Gamma_{1234}
\cdot \Pi_{\ve} \mK_{\Gamma_{1234}} \Gamma -  (\Pi_\ve \Gamma_{12})^2 (\Pi_\ve
\mK_{\Gamma_{12} \sqcup \Gamma_{34}} \Gamma_{1234}) \Pi_{\ve}
\mK_{\Gamma_{1234}} \Gamma \;,
  \end{aligned}
\end{equation*}
where we used the symmetry of the diagram to group elements together.
This should be interpreted as a form of counterbalancing of inner divergences, which is deeply connected to Taylor expansions. For clarity, in a simpler case, we have
\begin{equation*}
  \hp \<cc4>  = \Pi_\ve\<cc4> - \Pi_\ve \<cc4div> \cdot \Pi_\ve \<1_b>\;.
\end{equation*}
Our discussion so far has led to the definition of BPHZ renormalisation of
Feynman diagrams, and their associated deterministic integrals. The next step is to use this definition to construct a renormalisation of the trees, which are random variables.

\subsubsection{The BPHZ renormalisation of trees}
At the level of trees, the BPHZ renormalisation $\mf{R}_\ve$ is defined as follows. We recall from the identity~\eqref{e:int-cont} that for any $n \in \NN_*$ we have
$\mI_{n, \ve} = \sum_{\kappa \in \mK_{n}^{\mathrm{int}}} \mI_{n, \kappa, \ve}$,
where $\kappa$ is the set of internal pairings. Therefore, it suffices to define the map $\mf{R}_\ve$ on each homogeneous component $\mI_{n, \kappa}$ and then set
\begin{equation}\label{e:ren-1}
  \mf{R}_\ve \mI_{n} := \sum_{\kappa \in \mK_n^{\mathrm{int}}} \mf{R}_\ve \mI_{n, \kappa} \,.
\end{equation}
First, we note that $\mI_{n,\kappa,\ve}$ can be interpreted as a Feynman diagram in
which a subset of the leaves is integrated against noise variables.
To this end, let us define
$\mI_{n, \kappa}^{\rm{int}} = \mI_{n, \kappa}$ as a graph (so it has the
same set of edges and vertices), with the difference that we alter its
set of leaves.
More precisely,
the set of leaves $L^{\rm{int}}$
associated to $\mI^{\rm{int}}_{n, \kappa}$ is given by
\begin{equation*}
  L^{\rm{int}} = L \cup V_o\;, \qquad L = \{\text{leaves of
} \mI_{n,\kappa}\} \;,
\end{equation*}
i.e. nodes that are associated to the noise are also considered
leaves\footnote{While we call the nodes $ V_{o}$ leaves, they are not in the
sense of graph theory. Every node in $ V_{o}$ has degree $2$, i.e. it is connected to two edges.}. Hence, while $\mI_{n, \kappa}^{\rm{int}}$ and $ \mI_{n, \kappa}$
are identical as graphs, they are not identical as Feynman diagrams, following
Definition~\ref{def:FDgeneral}.
Consequently,~\eqref{eq_contractedHomIntegral} can be rewritten through the map
$\Pi_\ve$ from~\eqref{eq_def_valTested} as
\begin{equation*}
  \mI_{n,\kappa,\ve} (x_L) = \int_{( \mathbf{T}^{4})^{ V_{o} }} \Pi_\ve \mI_{n, \kappa}^{\rm{int}}
  (x_L, x_{V_{o}})  \prod_{v \in V_{o}} \xi (x_{v}) \ud x_{V_{o}} \;,
\end{equation*}
with
\begin{equation*}
  \Pi_\ve \mI_{n, \kappa}^{\rm{int}} (x_L, x_{V_{o}}) := \int_{(\TT^4)^{V_\star^{\rm{int}}}} \prod_{e \in E}
  G_{\ve}(x_{e_{+}} - x_{e_{-}}) \ud x_{V_{\star}^{\rm{int}}} \;,
\end{equation*}
where $V_{\star}^{\rm{int}} = V \setminus L^{\rm{int}}$.

Next, we apply $\hat{\Pi}_\ve$ from~\eqref{eq_zimmermann} to the
Feynman diagram
$\mI^{\rm{int}}_{n, \kappa}$ in which noise nodes are now considered leaves.
In this way, we define the renormalisation map $\mf{R}_{\ve}$ as
\begin{equation}\label{e:ren-2}
  \mf{R}_{\ve} \mI_{n, \kappa} (x_L) := \int_{( \mathbf{T}^{4})^{ V_{o} }} \hat{\Pi}_\ve \mI_{n, \kappa}^{\rm{int}}
  (x_L, x_{V_{o}})  \prod_{v \in V_{o}} \xi (x_{v}) \ud x_{V_{o}} \;.
\end{equation}
For example,
\begin{equation*}
  \mf{R}_\ve  \, \<3> = \mf{R}_\ve  \, \<c3_1> + \mf{R}_\ve  \, \<c3_2>+ \<c3_3>_{\varnothing, \ve}  + \<3>_{\varnothing, \ve} = \<c3_3>_{\varnothing, \ve} + \<3>_{\varnothing, \ve}\;,
\end{equation*}
where as usual the subscript $\varnothing$ indicates the homogeneous iterated Gaussian integral, and where we used that:
\begin{equation*}
  \mf{R}_\ve  \, \<c3_1> = \<c3_1>_{\ve} - \Pi_\ve \<loop> \cdot \<2_b>_{ \ \ \
\ve} = 0 \,,
\quad \text{and}
\quad
\mf{R}_\ve  \, \<c3_2> = 0
\;.
\end{equation*}
In this setting, the identity~\eqref{e:moments} which characterises moments of trees in terms of pairings, extends to BPHZ renormalised trees and BPHZ
renormalised Feynman diagrams as follows.
\begin{proposition}\label{prop:bphz}
	For every $ k \in \NN $ and $ \{ \tau_{i} \}_{i =1}^{k}$, with $ \tau_{i} \in \{ \mI_{n} \}_{n \geqslant 1}
$, we have
\begin{equation*}
\begin{aligned}
 \EE[ \mf{R}_{\ve} \tau_{1 } (x_{1}, y_{1}) \cdots \mf{R}_{\ve} \tau_{k} (x_{k}, y_{k}) ] = \sum_{\kappa \in \mK  (\tau_{1}, \ldots,
\tau_{k})} \hat{\Pi}_{\ve} (\tau_{1}, \dots, \tau_{k} )_{\kappa} (x_{1:k}, y_{1:k})\;.
\end{aligned}
\end{equation*}
\end{proposition}
The proof of this proposition can be found in Appendix~\ref{app:bphz}. In the next
section, we present the main aspects of the proofs of the main results.

\begin{remark}
By virtue of the above proposition, all the technical analysis of our work will take
place at the level of Feynman diagrams, and we all but avoid working with renormalised~trees.
\end{remark}

\section{Proof of the main results}\label{sec:proof-main}

In this section we collect all the main results which lead to the proofs of
Theorems~\ref{thm:main} and~\ref{thm:main2}.

The first step towards the proof of Theorem~\ref{thm:main} is to show convergence of finite-dimensional
distributions of $ (\mI_{n, \ve})_{n \geqslant  1} $ in weak coupling.
We must show that for every $N\in \NN_*$
the sequence $ \big(\big( \lambda_{\ve}^{n-1} \, \mf{R}_{\ve} \, \mI_{n}
\big)_{n = 1}^{N}\big)_{\varepsilon \in (0,1)} $  is tight, that any limit point is Gaussian, and finally
we must identify the limiting mean vector (zero) and covariance matrix. To
prove tightness of the process, it suffices to obtain a uniform (in $\ve$)
control of the second
moments of this term. Here we use Proposition~\ref{prop:bphz}.
We will present the argument for arbitrary moments,
 since it will be useful in proving Gaussianity of the limit.
Our objective is to study the following limit, if it exists, for smooth test
functions $\varphi_i \in \mS((\TT^4)^2;\RR)$:
\begin{equation*}
\begin{aligned}
\lim_{\ve \to 0} &  \lambda_{\ve}^{\sum_{l
=1}^{k} (n_{l}-1)}  \EE[ \mf{R}_{\ve} \mI_{n_{1}, \ve} ( \varphi_1) \cdots
\mf{R}_{\ve} \mI_{n_{k}, \ve} (\varphi_k) ] \\
& = \sum_{\kappa \in \mK  (\mI_{n_{1} }, \ldots,
\mI_{n_{k}})}  \lim_{\ve \to 0}  \lambda_{\ve}^{\sum_{l
=1}^{k} (n_{l}-1)}  \hat{\Pi}_{\ve} (\mI_{n_{1}}, \dots, \mI_{n_{k}} )_{\kappa} (\varphi_1 \otimes \cdots \otimes \varphi_k)\;.
\end{aligned}
\end{equation*}
We will study this limit by identifying exactly those pairings that
\emph{contribute} to the weak coupling limit, and distinguish them from those that lead
to a zero contribution.
\begin{definition}\label{def:contrib}
For any $k \in \NN_*$ and any set of trees $ \{ \mI_{n_{l}}  \}_{l =1}^{k} $, we say that a
pairing $ \kappa \in \mK (\mI_{n_{1}}, \dots, \mI_{n_{k}}) $ is
\emph{contributing}, if there exists a $\varphi \in \mS( (\TT^4)^{2k}; \RR)$ such that
\begin{equation} \label{e:contrib}
\begin{aligned}
\limsup_{ \ve \to 0}  \lambda_{\ve}^{\sum_{l
=1}^{k} (n_{l}-1)}
\big\vert \hat{\Pi}_{\ve} (\mI_{n_{1}}, \dots, \mI_{n_{k}} )_{\kappa}  (\varphi) \big\vert > 0  \;.
\end{aligned}
\end{equation}
We say the pairing is \emph{not contributing} if~\eqref{e:contrib} vanishes for any test function $\varphi$. We write
$\mK^{\rm{cont}}(\mI_{n_1}, \ldots \mI_{n_k})$ for the set of all contributing
pairings for the given set of trees.
\end{definition}
Our approach is to characterise contributing contractions in terms of graphical properties of
the Feynman diagram $ (\mI_{i_{1}}, \dots, \mI_{i_{k}} )_{\kappa} $, and then use
this characterisation to control the moments.
Our first observation is that the contribution of a Feynman diagram in the weak coupling limit depends
strongly on its degree. We say that a Feynman diagram is \emph{positive, negative, or zero
degree} according to the following definition.

\begin{definition}\label{def:degree}
We say that a Feynman diagram $ \Gamma $ is
\begin{enumerate}
\item \textbf{Positive} if $\deg \overline{\Gamma} > 0 $ for all
subdiagrams $ \overline{\Gamma} \subseteq \Gamma $.
\item \textbf{Negative} if there exists a subdiagram $ \overline{\Gamma}
\subseteq \Gamma  $ with $ \deg \overline{\Gamma} < 0 $.
\item \textbf{Zero degree} if $ \deg \overline{\Gamma} \geqslant 0 $ for all
subdiagrams $ \overline{\Gamma} \subseteq
\Gamma $ and there exists a subdiagram $ \overline{\Gamma}
$ with $ \deg \overline{\Gamma} = 0 $.
\end{enumerate}
\end{definition}

We will prove that positive and negative diagrams
(after renormalisation)  do not
contribute, implying that \emph{only zero degree} diagrams have a chance of
contributing. This is the content of the following two results.

\begin{lemma}[No contribution of positive diagrams]\label{lem:pos}
Fix any $ n_{1}, \ldots, n_{k} \in \NN_* $, for some $ k \in \NN_* $. Let $ \kappa \in \mK
(\mI_{n_{1}}, \dots, \mI_{n_{k}}) $ be a pairing such that $ \Gamma =  (
\mI_{n_{1}}, \dots, \mI_{n_{k}})_{\kappa} $ is positive. Then $
\hat{\Pi}_{\ve} \Gamma (x_{1:k}, y_{1: k})= \Pi_{\ve} \Gamma (x_{1:k}, y_{1: k}) $ and
\begin{equation*}
 \sup_{\ve \in (0,1)}
 | {\Pi}_{\ve} \Gamma  (\varphi)  | < \infty \;, \qquad \forall \varphi \in
\mS( (\TT^4)^{2k}; \RR) \;.
\end{equation*}
In particular, $ \kappa$ does not contribute.
\end{lemma}
The proof of this lemma follows from Lemma~\ref{lem_nonneg_diagrams} and is a
consequence of the well-known fact that positive Feynman
diagrams correspond to converging integrals. In particular, no
weak coupling is necessary to bound these integrals. \\

The next result is more subtle, because we chose the ``subcritical'' BPHZ
renormalisation, which does not renormalise logarithmic divergences in
the critical dimension.
Thus, there is fine balance between logarithmic divergences and weak couplings
$ \lambda_{\ve}$, as explained at length in Section~\ref{sec_negative}. The
result holds for Feynman diagrams of 2 and 4 legs only, which are the only ones
that we will require to prove convergence to zero in $L^2 ( \PP)$. We expect the result
to extend to Feynman diagrams of higher order as well.
\begin{lemma}[No contribution of negative diagrams]\label{lem:negvanish}
Let $ k \in \{1,2\}$,
and fix any $ n_{1},  \ldots, n_{k} \in \NN_* $. Let $ \kappa \in \mK
(\mI_{n_{1}}, \dots, \mI_{n_{k}}) $  be a pairing such that $ \Gamma
=(
\mI_{n_{1}}, \dots, \mI_{n_{k}})_{\kappa} $ is negative. Then  $ \kappa$ does
not contribute.
\end{lemma}
Overall, proving this result requires revisiting the BPHZ theorem, and the
proof can be found at the very end of Section~\ref{sec_negative}, in
Section~\ref{sec:the-end-proof}.\\

Finally, we turn our attention to zero degree diagrams. Note that by our definition of the BPHZ renormalisation map we have the identity $\Pi_{\ve} \Gamma = \hat{\Pi}_{\ve} \Gamma$. The right concept to
characterise the contribution of a zero degree diagram is that of
\emph{primitive blow-ups}.

\begin{definition}[Primitive blow-up]\label{def:primitive} Given a Feynman diagram $ \Gamma =
(V, E) $ and a subdiagram $ \overline{\Gamma} = ( \overline{V}, \overline{E}) $,
we say that $\overline{\Gamma}$ is a \textbf{primitive blow-up} in $\Gamma$, if
$ \overline{\Gamma} $ is full, $ \deg ( \overline{\Gamma}) = 0 $, and for every subdiagram $
\widetilde{\Gamma} \subsetneq \overline{\Gamma}$, we have $ \deg (
\widetilde{\Gamma})> 0$.

\end{definition}
In other words, primitive blow-ups contain no subdivergence.
We will show that each primitive subdiagram will produce exactly one
logarithmic blow-up. Therefore, the total power of logarithmic blow-ups, that bound
the valuation of a Feynman diagram $ \Gamma $, corresponds to the total number of
primitive blow-ups that one can extract.
Here are some examples of primitive blow-ups that correspond to
Anderson-Feynman diagrams (meaning that the subdiagram we obtain after removing the legs is a primitive blow-up):
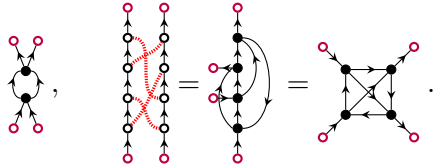
\begin{equation*}
\begin{aligned}
  \<cc22> \;,
  \qquad
  \<c41> = \<cc41> = \<ccc41>  \;.
\end{aligned}
\end{equation*}
Incidentally, the second primitive blow-up is the smallest one that one can
find which is not the \textbf{bubble} (the first example above).
In our weak coupling limit, we want to capture the maximal number of blow-ups possible.

\subsection{Nested bubble diagrams and their extraction sequences}

Since the smallest primitive blow-up is given by the bubble,
it is intuitively clear that the maximal number of blow-ups is attained if one can iteratively extract such bubbles.
Note that such blow-ups may appear in parallel, as for example in
\[
   \<ex1>\,,
\]
or they may be nested one in the other, meaning that we see the next bubble only after we have contracted the previous one. For example:
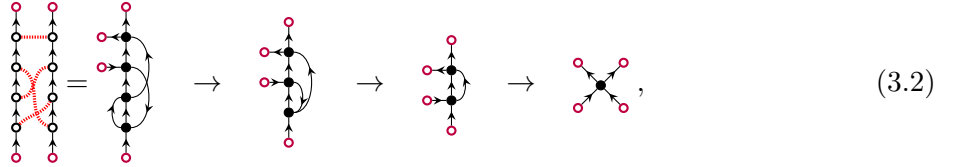
\begin{equation}\label{e:nbex}
  \<c42> = \<cc42> \quad \rightarrow \quad \<cc42C> \quad \rightarrow \quad \<cc42D> \quad \rightarrow \quad \<cc1> \;,
\end{equation}
where as mentioned the succession of diagrams above has been obtained by repeatedly contracting a bubble to a single point:
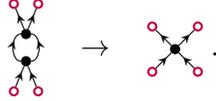
\begin{equation*}
  \<cc22>  \quad \rightarrow \quad \<cc1> \;.
\end{equation*}
We define nested bubble diagrams, such as the one appearing in~\eqref{e:nbex}, as follows.
\begin{definition}[Nested bubble diagram]\label{def:bubbl}
We say that an Anderson-Feynman diagram $ \Gamma $ with $|V_\star|=n \in \NN_*$ is a
\textbf{nested bubble diagram}, if there exists a
sequence of Feynman diagrams $ \{ \Gamma_{i} \}_{i =1}^{n} $, and a sequence $ \{ \overline{\Gamma}_{i} \}_{i =1}^{n-1} $ of
subdiagrams $\overline{\Gamma}_i \subseteq \Gamma_i$ such that:
\begin{enumerate}
\item $ \overline{\Gamma}_{i} \subseteq \Gamma_{i} $ is a bubble diagram (it is isomorphic to $\<bubbleA>$ or $\<bubbleB>$).
\item $ \Gamma_{1} = \Gamma $ and $ \Gamma_{i+1} = \mK_{
\overline{\Gamma}_{i}} \Gamma_{i} $, for all $ i \in \{ 1, \dots, n -1\} $.
\item In the diagram $ \Gamma_{n} $ all internal edges have been contracted: $\Gamma_n = \scalebox{0.7}{\<cc1>}$.
\end{enumerate}
\end{definition}

While Definition~\ref{def:bubbl} is the closest to our intuition, it is
often more convenient to think about successive extractions of bubbles in terms
of increasing sequences of subdiagrams
$\{\tilde{\Gamma}_i\}_{i=1}^{n-1}$  of the original diagram $\Gamma$, in which
each $\tilde{\Gamma}_i$ consists of the first $i$ bubbles extracted. The only
difficulty in doing so, is that as soon as they are nested, the bubbles $\overline{\Gamma}_i$ do not live in the original
diagram, but in a contracted one.

  To overcome this issue, it is convenient to \lqm invert\rqm a contraction as
  follows. Consider some $\Gamma_{c} \subseteq \Gamma$ which is the disjoint
  union of connected, full subdiagrams. Denote with $V_{\Gamma_c}^\Gamma
  \subseteq V( \mK_{\Gamma_c} \Gamma)$ the set of vertices in $\bigslant{\Gamma}{\Gamma_c}$ to which
  we contract $\Gamma_c$ (their cardinality corresponds to the
  number of connected components of $\Gamma_c$), following
  Definition~\ref{def:contract}. Then for any $\overline{\Gamma} \subseteq
  \bigslant{\Gamma}{\Gamma_c}$ define
  \begin{equation*}
    \mK_{\Gamma_c}^{-1} \overline{\Gamma} = \max\{ \Gamma_c \subseteq \tilde{\Gamma}
    \subseteq \Gamma \text{ such that } \bigslant{\tilde{\Gamma}}{\Gamma_c} =  \overline{\Gamma} \cup V_{\Gamma_c}^\Gamma \} \;,
  \end{equation*}
  where we have used Remark~\ref{rem:subdiam}, the maximum is taken with respect to inclusion of subdiagrams, and we
  view $V_{\Gamma_c}^\Gamma $ as a set of singletons (vertices without incident edges) in
  $\bigslant{\Gamma}{\Gamma_c}$, if they are not contained in $\overline{\Gamma}$.

In this setting, the correct definition of the set $\tilde{\Gamma}_i$, that
consists of the first $i$ bubbles to be contracted, is given inductively
by:
\begin{equation}\label{e:gamma-tild-def}
  \tilde{\Gamma}_1 = \overline{\Gamma}_1 \;, \qquad \tilde{\Gamma}_{i+1} = \mK^{-1}_{\tilde{\Gamma}_i} \overline{\Gamma}_{i+1} \;, \qquad \forall i \in \{1, \ldots, n-2\} \;.
\end{equation}
We note that the $\tilde{\Gamma}_i$ satisfy the following properties. Since the
verification is immediate from the definitions we omit a proof.

\begin{lemma}\label{lem:extraction-sequence}
  For any two collections $\{\Gamma_i\}_{i=1}^n$ and
  $\{\overline{\Gamma}_i\}_{i=1}^{n-1}$ as in Definition~\ref{def:bubbl}, the sequence
  $\{\tilde{\Gamma}_i\}_{i=1}^{n-1}$ defined in~\eqref{e:gamma-tild-def}
  satisfies the following properties:
  \begin{enumerate}
    \item The sequence is increasing: $\tilde{\Gamma}_{i-1} \subseteq
    \tilde{\Gamma}_{i}$ for all $i \in \{2,\ldots, n-1\}$ and $\tilde{\Gamma}_1$
    is isomorphic to $\<bubbleA>$ or $\<bubbleB>$.
    \item The quotient $\bigslant{\tilde{\Gamma}_{i}}{\tilde{\Gamma}_{i-1}}$ is
    of the form $\bigslant{\tilde{\Gamma}_{i}}{\tilde{\Gamma}_{i-1}}
    =\overline{\Gamma}_i \cup V_{\tilde{\Gamma}_{i-1}}^{\Gamma} \subseteq
    \bigslant{\Gamma}{\tilde{\Gamma}_{i-1}}$, where $\overline{\Gamma}_{i}$ is
    isomorphic to $\<bubbleA>$ or $\<bubbleB>$, for $i \in \{2, \ldots, n-1\}$.
  \end{enumerate}
  Moreover, any sequence $\{\tilde{\Gamma}_i\}_{i=1}^{n-1}$ satisfying the two
  points above gives rise to sequences $\Gamma_i$ and $\overline{\Gamma}_i$ as
  in Definition~\ref{def:bubbl} by defining $\overline{\Gamma}_i =
  \bigslant{\tilde{\Gamma}_i}{\tilde{\Gamma}_{i-1}} \setminus \{
  \text{\emph{singletons}}\}$.
\end{lemma}

This remark motivates the following definition.

\begin{definition}[Bubble extraction sequences]\label{def:extraction-sequences}
  Given an Anderson-Feynman diagram $\Gamma = (V, E)$ with $|V_\star|=n \in \NN_*$,
  we say that any sequence of subdiagrams $ \{ \tilde{\Gamma}_{i} \}_{i =1}^{n-1}
  $ satisfying the two points of Lemma~\ref{lem:extraction-sequence} is a bubble extraction sequence for $\Gamma$.
  We denote by $ \mathbb{S} ( \Gamma)$ the set of all bubble extraction
  sequences. If $n=1$ we define the only extraction sequence to be the empty
one.
\end{definition}

In this setting, our fundamental result concerning the convergence of zero-degree Feynman diagrams is the
following.

\begin{proposition}\label{prop:bubble} The following holds for any $k \in \NN_*$.
  \begin{enumerate}
    \item A zero degree, connected $2k$-Anderson-Feynman diagram $ \Gamma $ is contributing if and only if it is
    a connected nested bubble diagram. If $\Gamma$ is contributing, then $k=2$.
    \item Moreover, if $\Gamma$ is contributing, then
    \begin{equation}\label{eq_prop_bubble_supp1}
      \lim_{\ve \to 0} \lambda_{\ve}^{2(|V_\star| -1)} \, \Pi_{\ve} \Gamma \,
    (\varphi) = \sigma_\Gamma^2  \, \hat{\lambda}^{2(|V_\star| -1)} \,
    \<cc1>  \, (\varphi) \,,
    \end{equation}
    where for any smooth test function $\varphi \in \mS( (\mathbf{T}^{4})^4;\RR)$,
\begin{equation*}
\begin{aligned}
\<cc1>  \, (\varphi)  :=
\int_{ ( \mathbf{T}^{d} )^{L}} \varphi (x_{L} )
\int_{\mathbf{T}^{d}} \prod_{v \in L} G
(x_{\star} - x_{v}) \ud x_{\star}  \ud x_{L}\,,
\end{aligned}
\end{equation*}
 and
    \begin{equation*}
      \sigma_\Gamma^2 = \Big(\frac{ 1}{ \sqrt{8} \pi}\Big)^{2(|V_\star| -1)}
    \frac{|\mathbb{S}( \Gamma)|}{ (|V_\star| -1 )! } \;,
    \end{equation*}
    with $\mathbb{S}$ as in Definition~\ref{def:extraction-sequences}.
\end{enumerate}
\end{proposition}
\begin{proof}
  The proof of the first statement is a consequence of
  Lemma~\ref{lem_nonneg_diagrams}. The second statement follows from
  Corollary~\ref{cor_contr_diagram} and is overall the aim of
  Section~\ref{sec_exact}.
\end{proof}

\begin{remark}
  note that whenever $\Gamma$ is a $4$-Anderson-Feynman diagram, then it is
  of the form $\Gamma = (\mI_n, \mI_m)_{\kappa}$, for some $n, m$ and some
  pairing $\kappa$. Therefore, $|V_\star| = (n+m)/2$ and the weak coupling
  scaling in~\eqref{eq_prop_bubble_supp1} corresponds to the one in
  Theorem~\ref{thm:main}.
\end{remark}

As a consequence of Proposition~\ref{prop:bubble}, we can now provide a formula
for the covariance of the limiting Gaussian process appearing in
Theorem~\ref{thm:main}. We set
\begin{equation}\label{e:sigma-n-def}
  \sigma_{n}^{2} := \sum_{\kappa \in \mK^{\rm{cont}}(\mI_n, \mI_n)}
\sigma^{2}_{(\mI_n, \mI_n)_\kappa}\,.
\end{equation}
Moreover, we are able to define the overall effective variance coefficients
as follows.
\begin{definition}\label{def:sigma-n}
  For every $n \in \NN_*$, we set
  \begin{equation*}
    \sigma_{\rm{eff}}^{2,(n)} := \sum_{n_1 + n_2 = n} \sum_{\kappa \in
    \mK^{\mathrm{cont}}(\mI_{n_1}, \mI_{n_2})} \sigma_{(\mI_{n_1},
\mI_{n_2})_\kappa}^{2} \;,
  \end{equation*}
  with $\sigma_{(\mI_{n_1}, \mI_{n_2})_\kappa}^{2}$ as in
Proposition~\ref{prop:bubble},
and $  \mK^{\mathrm{cont}}$ introduced in Definition~\ref{def:contrib}.
\end{definition}
As a consequence of Proposition~\ref{prop:bubble} and
Proposition~\ref{prop:bphz}, we expect that the effective variance $\sigma_{\rm{eff}}^2$ appearing in
Theorem~\ref{thm:main} is given by
\begin{equation}\label{e:guess-var}
  \sigma_{\rm{eff}}^2 (\hat{\lambda}) = \sum_{n=1}^\infty \sigma_{\rm{eff}}^{2,(2n)} \; \hat{\lambda}^{2n-2} \;,
\end{equation}
where we have used that $\sigma^{2,(2n+1)}_{\rm{eff}} = 0$ by antisymmetry (reflecting $\xi \mapsto - \xi$). In fact, this series simplifies further,
since we can show that $\sigma^{2,(2n)}_{\rm{eff}}$ is geometric.
\begin{proposition}\label{prop:effective-variance}
  For every $n \in \NN_\star$, and with $\sigma_n$ from Definition~\ref{def:sigma-n} we have:
  \begin{equation*}
    \begin{aligned}
      \sigma_{\rm{eff}}^{2,(2n)} = \bigg(\frac{1}{2 \pi^{2}}\bigg)^{n-1}
\;, \qquad  \sigma_{\rm{eff}}^{2,(2n+1)}  = 0 \;.
    \end{aligned}
  \end{equation*}
\end{proposition}
The proof of this proposition can be found in Section~\ref{sec:effective-variance}.

\subsection{Proofs of Theorems~\ref{thm:main} and~\ref{thm:main2}}

We have now collected all the main results that lead to the proof of both
Theorem~\ref{thm:main} and Theorem~\ref{thm:main2}. We start with the first
result.

\begin{proof}[Proof of Theorem~\ref{thm:main}]
  Tightness of the sequence $\left( \lambda_\ve^{n-1} \, \mf{R}_\ve \,
\mI_{n,\ve} \right)_{n=1}^N $ is a consequence of Mitoma's criterion.
Combining
 Proposition~\ref{prop:bphz}, with Lemma~\ref{lem:pos},
Lemma~\ref{lem:negvanish}, and Proposition~\ref{prop:bubble}, implies
that for any pairing $\kappa\in\mK(\mI_{n},
\mI_{n})$
\begin{equation*}
\begin{aligned}
\sup_{ \ve \in (0,1)} \lambda_\ve^{2n-2}|\hat{\Pi}_{\ve} (\mI_n, \mI_n )_{\kappa} (\varphi \otimes
\psi ) | < \infty\,,
\end{aligned}
\end{equation*}
for any $ \varphi, \psi \in \mS ( (\TT^{4})^{2}; \RR) $.

  Next we must identify the law of the limit. First, we observe that the
limiting random variables are centred. This is a consequence of
Proposition~\ref{prop:bubble}, because there is no internal pairing $\kappa
\in \mK(\mI_n)$ that is contributing, since any contributing pairing must
belong to a $4$-Anderson-Feynman diagram.
Moreover, we observe that negative
Feynman diagrams arise only from internal contractions, see Lemma~\ref{lem:neg}. Therefore,
Lemma~\ref{lem:negvanish} implies that all homogeneous chaos terms that give
rise to negative diagrams vanish in the weak coupling limit in $L^2$, and therefore also in
$L^p$ for any $p>2$ by hypercontractivity. Thus, we must only
concentrate on positive and non-negative contractions in the following.

Given these preliminaries, Gaussianity of the limit follows through the fourth moment
theorem \cite{nualart2005central}. We note that {mixed} cumulants are given through the Link-Cluster Theorem and
  Proposition~\ref{prop:bphz} by
  \begin{equation*}
    \rm{Cumulant} (\mf{R}_{\ve} \mI_{n_1, \ve} (\varphi_1), \ldots, \mf{R}_{\ve} \mI_{n_k, \ve} (\varphi_k)) = \sum_{\kappa \in \mK_0(\mI_{n_1}, \ldots, \mI_{n_k})} \hat{\Pi}_{\ve} (\mI_{n_1}, \ldots, \mI_{n_k})_\kappa (\varphi_1 \otimes \cdots \otimes \varphi_k) \;,
  \end{equation*}
  where $\mK_0(\mI_{n_1}, \ldots, \mI_{n_k})$ is the subset of all pairings $\kappa
  \in \mK (\mI_{n_1}, \ldots, \mI_{n_k})$ that result in a connected
  Anderson-Feynman diagram. Since for $k>2$ none of these pairings
  contribute by the discussion above concerning negative diagrams, by
  Lemma~\ref{lem:pos} concerning positive diagrams, and by
  Proposition~\ref{prop:bubble} concerning zero-degree diagrams, we see that all cumulants
  of
order $k>2$ vanish for the limiting variable, and therefore it is Gaussian.

Finally, the covariance of the limit is determined by Proposition~\ref{prop:bubble},
which implies that the only contributing pairings are those $\kappa \in \mK^{\rm{cont}}(\mI_n, \mI_m)$,
and that each such pairing contributes with a constant $\sigma^{2}_{(\mI_n,
\mI_m)_\kappa}$. Hence,
for every pair $\varphi, \psi \in \mS( (\TT^4)^2 ; \RR )$:
  \begin{equation} \label{e:corr}
    \EE \left[  \limI^{(n)} (\varphi) \limI^{(m)} (\psi) \right] = \frac{ \sigma_{n, m}^{2}}{\sigma_{n}  \sigma_{m} }  \<cc1> (\varphi \otimes \psi)\;, \qquad
    \sigma_{n,m}^{2} := \sum_{\kappa \in \mK^{\rm{cont}}(\mI_n, \mI_m)}
\sigma^{2}_{(\mI_n, \mI_m)_\kappa} \;,
  \end{equation}
  with the last constant as in Proposition~\ref{prop:bubble},
and $ \sigma_{n}^{2}= \sigma_{n, n}^{2}$ (cf.~\eqref{e:sigma-n-def}). This completes the proof of the result.
\end{proof}

We conclude the section with the proof of Theorem~\ref{thm:main2}.

\begin{proof}[Proof of Theorem~\ref{thm:main2}]
  This result is a consequence of Proposition~\ref{prop:bubble} and
  Proposition~\ref{prop:effective-variance}. In fact the first result implies
  that convergence of $\mathcal{U}$ is equivalent to checking the convergence of the
  series in~\eqref{e:guess-var}. In turn, by Proposition~\ref{prop:effective-variance}, this amounts to controlling
  \begin{equation*}
    \sum_{n=1}^\infty \bigg(\frac{1}{2 \pi^{2}}\bigg)^{n-1} \; \hat{\lambda}^{2n-2} = \sum_{n=0}^\infty \left( \frac{\hat{\lambda}^2}{2\pi^2} \right)^n =
    \frac{2 \pi^{2}}{2 \pi^{2} - \hat{\lambda}^2} \;,
  \end{equation*}
  where the identities hold under the assumption $|\hat{\lambda}| < \sqrt{2} \pi $.
\end{proof}

\section{Non--negative Feynman diagrams}\label{sec:nonneg}

This section is devoted to the proof of Lemma~\ref{lem:pos} as well as the first point of Proposition~\ref{prop:bubble}.
However, for later use and because the results of this section apply also to other models than the 4D Anderson model, we introduce a more general setting and work with typed Feynman diagrams in arbitrary dimension $d\in \NN_*$.
For this purpose, we consider a fixed set $\mf{T} \subseteq (-d, \infty)$ of
types, and assume that any $\mf{t} \in \mf{T}$ is associated to a kernel
$K_{\mf{t}, \varepsilon} \colon \TT^d \to \RR$, satisfying the bound
\begin{equation}\label{e_orderbound}
  \begin{aligned}
	  C_{\mf{t}}:= { \sup_{ \varepsilon \in (0,1)}}  \sup_{x \in
  \mathbf{T}^{d}} (|x|+ \varepsilon)^{-\mathfrak{t} } | K_{ \mf{t}, \varepsilon}(
  x)| < \infty\,.
  \end{aligned}
  \end{equation}
  Then we can add types to Feynman diagrams as follows.

\begin{definition}\label{def:FDtyped}
  A directed, typed, multi-graph $ \Gamma = (V, E, \mf{t}) $ is a \emph{typed} Feynman
  diagram, if it is a Feynman diagram according to Definition~\ref{def:FDgeneral}, and if
  in addition it is endowed with a map $\mf{t} \colon E \rightarrow \mf{T}$.
\end{definition}
Since all Feynman diagrams in this section are typed, we will omit the word \lqm
typed\rqm from now onwards.
Again we associate to any Feynman diagram an integral as in the previous section, but now with general kernels $ K_{\mf{t}, \varepsilon}$ instead of the Green's function $ G_{\varepsilon}$:
\begin{equation*}
\begin{aligned}
\Pi_{\varepsilon}\Gamma
( x_{L }) := \int_{(\TT^{d})^{ V_{\star} }} \mW_{\ve} \Gamma
(x_{V_\star})
\prod_{e \in E_L}
 K_{\mf{t}(e), \varepsilon} (x_{e_+} - x_{e_-})
 \ud
x_{V_{\star}} \;,
\end{aligned}
\end{equation*}
where
\begin{equation*}
\begin{aligned}
\mW_{\ve} \Gamma (x_{V_\star}) = \prod_{e \in E_{\star}}
K_{ \mf{t}(e),\ve}(x_{e_{+}} - x_{e_{-}}) \;.
\end{aligned}
\end{equation*}
Note that by assumption $ \mathfrak{t} ( e) >-d  $, so that each single kernel is integrable.
For any connected subdiagram $\overline{\Gamma}$ (which inherits the types
from the original diagram) of a Feynman diagram $\Gamma$,
 we can define the degree
\begin{equation}\label{e:deg-extended}
\begin{aligned}
\deg ( \overline{\Gamma}) = d ( | V_{\star} ( \overline{\Gamma})| -1) + \sum_{e \in E(\overline{\Gamma})} \mathfrak{t}(e) \,.
\end{aligned}
\end{equation}
If a diagram is composed of several connected components, then we define its degree as the sum of the degrees of its connected components.
This is equivalent to the definition~\eqref{e_def_deg} we have used so far in the case in which
$\mf{t} \equiv -2$, which corresponds to the case in which the kernels involved are
Green's functions of the 4D Laplacian. According to this definition, we can also
define non--negative diagrams analogously to Definition~\ref{def:degree}, simply
by using the extended notion of degree from~\eqref{e:deg-extended}.

The main result of this section is a bound on non--negative Feynman
diagrams, stated in Proposition~\ref{prop_countlogs_nonneg}.
However, in order to state this result, we must introduce Hepp sectors, which
will allow us
to break up integrals $\Pi_{\varepsilon} \Gamma ( \varphi) $ into smaller
pieces, and estimate them individually.

\subsection{Hepp sectors and trees, and their contribution}

Let $ \Gamma $ be a
Feynman diagram and denote by $
\mT_{V_{\star} } $  the set of finite rooted binary trees with
leaf set given by the inner vertices $ V_{\star}$ of $\Gamma$.
We call $\mT_{V_{\star}}$ the set of Hepp trees over $V_\star$.

Every $ T \in \mT_{V_{\star}}$ induces a partial order $\prec$ on its vertices, with the root
being the smallest element and the leaves the largest ones. For $u \neq u ' $,
we write $ u \perp
u' $ if neither $u \prec u ' $ nor $ u \succ u ' $.
Moreover, let $ \trim$ denote the trimming of $T$,
where we removed all leaves and edges connected to them:
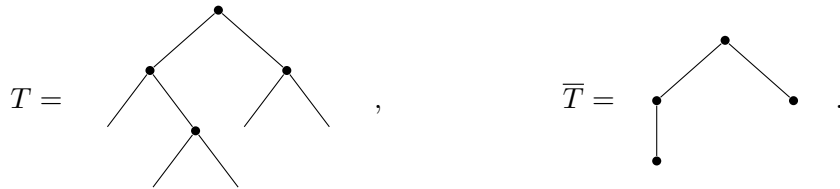
\begin{equation*}
  T   = \quad
  \begin{tikzpicture}[
    baseline={([yshift=-.6ex]current bounding box.center)},
    grow'=down,
    level distance=8mm,
    level 1/.style={sibling distance=18mm},
    level 2/.style={sibling distance=12mm},
    int/.style={circle, fill=black, inner sep=1.2pt},
    leaf/.style={circle, fill=white, inner sep=1.2pt},
    lab/.style={font=\small}
  ]
    \begin{scope}[shift={(-3.0,0)}]
      \node[int] (r) {}
        child { node[int] (a) {}
          child { node[leaf] {} }
          child { node[leaf] {} }
        }
        child { node[int] (b) {}
          child { node[int] (c) {}
            child { node[leaf] {} }
            child { node[leaf] {} }
          }
          child { node[leaf] {} }
        };
    \end{scope}
  \end{tikzpicture} \quad   \;,
\qquad \qquad \qquad  \trim  = \quad
\begin{tikzpicture}[
  baseline={([yshift=-.6ex]current bounding box.center)},
  grow'=down,
  level distance=8mm,
  level 1/.style={sibling distance=18mm},
  level 2/.style={sibling distance=12mm},
  int/.style={circle, fill=black, inner sep=1.2pt},
  leaf/.style={circle, fill=white, inner sep=1.2pt},
  lab/.style={font=\small}
]
  \begin{scope}[shift={(3.3,0)}]
    \node[int] (r2) {}
      child { node[int] (a2) {} }
      child { node[int] (b2) {}
        child { node[int] (c2) {} }
      };
  \end{scope}
\end{tikzpicture} \quad \;.
\end{equation*}
We say that a map $ \mathbf{n} \colon \trim \to \NN $ is
\emph{compatible} with a tree $T$, if $ \mathbf{n}(u')  \geqslant  \mathbf{n}(u)$
whenever $ u' \succ u $,
where we have identified $ \trim $ with the set of inner
nodes (not leaves) of $ T$. We write $\mA(T)$ for the set of all maps compatible with $T$. Finally, we denote the root of $ T $ (and $\trim$)
by $ \bullet$.

Now, for any pair $ \mathbf{T} = (T, \mathbf{n}) $ where $ \mathbf{n} $ is compatible with
$ T $, we define the Hepp sector
\begin{equation}\label{eq_def_Hepp_dom}
\begin{aligned}
D_{\mathbf{T}} = \left\{
x \in (\TT^{d})^{V_{\star}}\, :\, (\sqrt{d} \pi)  2^{- \mathbf{n} (v \wedge w)-1} < | x_{v} - x_{w} |
\leqslant (\sqrt{d} \pi)   2^{- \mathbf{n} (v \wedge w)} \ \  \forall v, w \in
V_{\star} \right\}
\end{aligned}
\end{equation}
where $ v \wedge w $ is the least common ancestor of $ v $ and $ w $ in the
tree $ T $, and $| \cdot |$ is the distance on the torus.
Here we included the factor $ \sqrt{d} \pi$ so that the 
 sets $D_{\mathbf{T}}$ cover
the full space $ (\mathbf{T}^{d})^{V_{\star}}$, up to diagonals, see
Lemma~\ref{lem_hepp_almost_covers}.

Before we proceed, let us introduce a further notion that will be useful later
on.

\begin{definition}\label{def:heap-ordering}
  For any partially ordered set $(\mX, \prec)$, we say
that a bijection $ \ell \colon \mX
  \to \{1, \ldots, |\mX|\}$ is a \emph{heap ordering} of $\mX$, if for any $ a, b \in \mX$ such
  that $ a \prec b $, we have $ \ell(a) \geq \ell(b)$.
\end{definition}

The reason why inequalities are inverted (imposing $\ell(a) \leq \ell(b)$ would
appear more natural) in the definition above is that we consider 
{leaves} of trees as maximal
elements, while the usual convention is to consider them as minimal ones.
However, in the present setting this is natural, because roots of Hepp sectors are connected to
the largest scales {(smallest $ \n$)} in a Feynman diagram.\\

Now, let us fix a Feynman diagram $\Gamma$ and define for every $u \in
\overline{T}$ the subdiagram $\Gamma_u$ as the
(possibly disconnected) full subdiagram spanned by vertices in
\begin{equation}\label{e:Gu}
\begin{aligned}
V_{\star} ( \Gamma_{u}) := \{ v \in V_{\star} \, : \, v \succ u \}\,, \qquad E(\Gamma_{u}) = \{e \in E \, \colon \, \{e_-, e_+\} \subseteq V(\Gamma_u)\}\;.
\end{aligned}
\end{equation}
Moreover, we set:
\begin{equation}\label{e:null-defn}
  \begin{aligned}
    \overline{\deg} (u) & = d(|V_\star(\Gamma_u)|-1) + \sum_{e \in E(\Gamma_u)}
\mathfrak{t}(e) \;, \qquad
\text{and} \qquad
    \mathrm{Null}(T) := | \{ u \in \trim \, :\, \overline{\deg}(u) = 0 \} | \;.
  \end{aligned}
\end{equation}
Note that the only difference between $\overline{\deg}(u)$ and $\deg(\Gamma_u)$
is that $\overline{\deg}$ is not given by the sum of the degrees over the
connected components of $\Gamma_u$, if it consists of disconnected components.
In particular, we have $\overline{\deg}(u) > \deg(\Gamma_u)$ whenever $
\Gamma_{u}$ is not connected.
The quantity of zero degree subdiagrams spanned by inner vertices, $\mathrm{Null} (T)$, determines the rate of logarithmic blow-up of the integral associated to a Feynman diagram over a fixed Hepp sector. This is the content of the following proposition, which is the main result of this section.

\begin{proposition}\label{prop_countlogs_nonneg}
  There exists a constant $C>0$ such that the following holds uniformly over any diagram
  $\Gamma$, test function $\varphi \in \mS$, and $\ve \in (0, 1)$. Let $ \Gamma $ be a non--negative $2k$-Feynman diagram with $
  \mathfrak{t} ( e) \in
  \mathbf{Z}\cap (-d, \infty) $ for all $ e \in E_{\star}$, then
  \begin{equation*}
  \begin{aligned}
  |\Pi_{\varepsilon} \Gamma( \varphi)|
  \leqslant
\| \varphi\|_{\infty}
C^{|V_\star|+ | E_{\star}| + 2k}
  \sum_{T \in \mT_{V_{\star}}}
  \big(
  \log { \tfrac{1}{ \varepsilon}}
  \big)^{\mathrm{Null} (T)}\,.
  \end{aligned}
  \end{equation*}
for every $ \varphi \in
\mS( (\TT^4)^{2k}; \RR) $.
  \end{proposition}
  \begin{remark}
    The result should be understood as a sector-by-sector estimate. Namely, the proof actually shows that there exists a $C_1>0$ such that for every $T \in \mT_{V_\star}$
    \begin{equation*}
      \begin{aligned}
        \int_{D(T)} |\mW_{\ve} \Gamma
(x_{V_\star})| \ud x_{V_{\star}} \leq C_1^{|V_\star|} \big( \log { \tfrac{1}{ \varepsilon}}
\big)^{\mathrm{Null} (T)} \;, \qquad \text{where} \qquad D(T) := \bigcup_{\mathbf{n} \in \mA(T)} D_{(T, \mathbf{n})} \;.
      \end{aligned}
    \end{equation*}
    Finally, the restriction to $ \mathfrak{t}$ being integer valued can be easily lifted, but adds an additional technical burden that we avoid.
  \end{remark}
\begin{remark}
  A similar result to Proposition~\ref{prop_countlogs_nonneg} was already
  obtained in~\cite[Proposition~6.1]{BerglundBruned}, and in any case the proof
  follows through a careful
  analysis of the classical proof for convergent Feynman diagrams, see for
  example the exposition in~\cite{BPHZ}. However, the
  result~\cite[Proposition~6.1]{BerglundBruned} is not presented in full
  generality (it allows at most for blow-ups of order $(\log(1/\ve))^2$ due to the
  nature of the problem).
  Therefore, we provide here a complete proof: the same techniques
  are also used in later proofs.
\end{remark}

To explain the result, we continue our analysis of the example in~\eqref{e:nbex}
\begin{equation*}
	\begin{tikzpicture}[scale =0.6]
	\node[dot]  (m1) at (-3.0,0.0) {};
	\node[dot]  (m2) at (-1.5,0.0) {};
	\node[dot]  (m3) at ( 0.0,0.0) {};
	\node[dot]  (m4) at ( 1.5,0.0) {};
	\node[bigroot]  (l1) at ( -4.2,0.0) {};
	\node[bigroot]  (l2) at ( 2.7,0.5) {};
	\node[bigroot]  (l3) at ( 2.7,-.5) {};
	\node[bigroot]  (l4) at ( 0,-1) {};
\node[font=\scriptsize, anchor=north, xshift = - 3pt] at (m1) {$1$};
\node[font=\scriptsize, anchor=north, xshift = 3 pt] at (m2) {$2$};
\node[font=\scriptsize, anchor=north, xshift = - 5pt] at (m3) {$3$};
\node[font=\scriptsize, anchor=north] at (m4) {$4$};
\draw[midarrow, thick] (m1.east) -- (m2.west);
\draw[midarrow, thick] (m2.east) -- (m3.west);
\draw[midarrow, thick] (m3.east) -- (m4.west);
\draw[midarrow, thick] (m3.north) to[out=90,in=90]   (m1.north);
\draw[midarrow, thick] (m1.south) to[out=-90,in=-90] (m2.south);
\draw[midarrow, thick] (m2.north) to[out=90,in=90]   (m4.north);
	\draw[midarrow, thick] (l1) to (m1);
	\draw[midarrow, thick] (l4) to (m3);
	\draw[midarrow, thick] (m4) to (l3);
	\draw[midarrow, thick] (m4) to (l2);
\end{tikzpicture}\,.
\end{equation*}
Among many others, we then consider the following Hepp sectors:
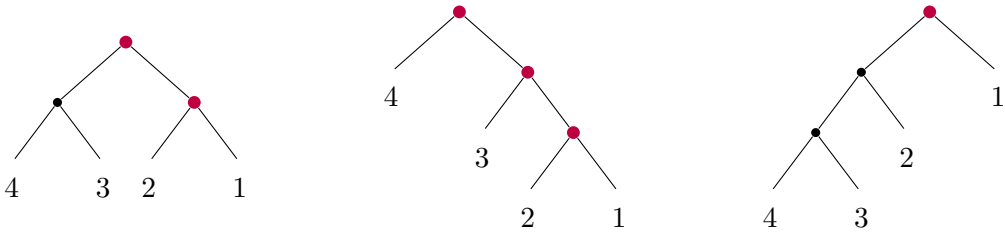
\begin{equation*}
\begin{tikzpicture}[
  baseline={([yshift=-.6ex]current bounding box.center)},
  grow'=down,
  level distance=8mm,
  level 1/.style={sibling distance=18mm},
  level 2/.style={sibling distance=12mm},
  int/.style={circle, fill=black, inner sep=1.2pt},
  red/.style={circle, fill=purple, inner sep=1.7pt},
  leaf/.style={circle, fill=white, inner sep=1.2pt},
  lab/.style={font=\small}
]
  \begin{scope}[shift={(-6.0,0)}]
    \node[red] (rA) {}
      child { node[red] (aA) {}
        child { node[leaf,label=below:$1$] {} }
        child { node[leaf,label=below:$2$] {} }
      }
      child { node[int] (bA) {}
        child { node[leaf,label=below:$3$] {} }
        child { node[leaf,label=below:$4$] {} }
      };
  \end{scope}
\end{tikzpicture}
\qquad \qquad
\begin{tikzpicture}[
  baseline={([yshift=-.6ex]current bounding box.center)},
  grow'=down,
  level distance=8mm,
  level 1/.style={sibling distance=18mm},
  level 2/.style={sibling distance=12mm},
  int/.style={circle, fill=black, inner sep=1.2pt},
  red/.style={circle, fill=purple, inner sep=1.7pt},
  leaf/.style={circle, fill=white, inner sep=1.2pt},
  lab/.style={font=\small}
]
  \begin{scope}[shift={(0,0)}]
    \node[red] (rB) {}
      child { node[red] (uB) {}
        child { node[red] (vB) {}
          child { node[leaf,label=below:$1$] {} }
          child { node[leaf,label=below:$2$] {} }
        }
        child { node[leaf,label=below:$3$] {} }
      }
      child { node[leaf,label=below:$4$] {} };
  \end{scope}
\end{tikzpicture}
\qquad \qquad
\begin{tikzpicture}[
  baseline={([yshift=-.6ex]current bounding box.center)},
  grow'=down,
  level distance=8mm,
  level 1/.style={sibling distance=18mm},
  level 2/.style={sibling distance=12mm},
  int/.style={circle, fill=black, inner sep=1.2pt},
  red/.style={circle, fill=purple, inner sep=1.7pt},
  leaf/.style={circle, fill=white, inner sep=1.2pt},
  lab/.style={font=\small}
]
  \begin{scope}[shift={(6.0,0)}]
    \node[red] (rC) {}
      child { node[leaf,label=below:$1$] {} }
      child { node[int] (uC) {}
        child { node[leaf,label=below:$2$] {} }
        child { node[int] (vC) {}
          child { node[leaf,label=below:$3$] {} }
          child { node[leaf,label=below:$4$] {} }
        }
      };
  \end{scope}
\end{tikzpicture}
\end{equation*}
where we have drawn in thick red the nodes for which $\overline{\deg}(u)=0$. 
Notice that legs of the Feynman diagrams do not play a role in this evaluation.
We
see that our bound guarantees that the first term blows up at order
$(\log(1/\ve))^2$, the second one at order $(\log(1/\ve))^3$, and the last one at
order $\log(1/\ve)$. Notably, the second Hepp tree is connected to the idea of
$\Gamma$ being an iterated bubble diagram: every inner node corresponds to
a bubble that appears iteratively, as we remove one after the other, from the
leaves upwards. For example, after pairing the points $1$ and $2$ (which
form a bubble), we can pair $1$ (or $2$) and $3$ which has now become a
bubble, and so forth, see Definition~\ref{def:extraction-sequences}.

\subsection{Proof of Proposition~\ref{prop_countlogs_nonneg}}

Let us fix a Feynman diagram $ \Gamma$ and an associated Hepp tree $ T \in
\mT_{V_{\star}}$.
We will perform an inductive proof over the size of the Hepp tree.
To this end, it will be convenient to define the map $ \eta: \trim \to
\mathbf{R}$
\begin{equation}\label{eq_def_eta}
\begin{aligned}
\eta (u):= d+ \sum_{ \substack{e \in E_{\star} \\  u = e_{+}\wedge e_{-}}} \mathfrak{t} (e)\,,
\end{aligned}
\end{equation}
which satisfies
\begin{equation}\label{eq_def_degbar}
  \overline{\deg}(u) = \sum_{v \in \trim, \, v \succ u} \eta(v) \geqslant 0 \;,
\end{equation}
where the last inequality holds for positive or non-negative Feynman diagrams, since $\overline{\deg} (u )\geqslant \deg(\Gamma_u)$.
In this setting, we obtain the following estimate.

\begin{lemma}\label{lem_estimate_nonneg_Hepp}
Fix a Feynman diagram $\Gamma$ as in Proposition~\ref{prop_countlogs_nonneg}.
Let $T \in \mT_{V_{\star}}$ be a Hepp tree and set $\eta$ as in~\eqref{eq_def_eta}. Then there exists a $C>0$ (independent in $ \ve$, $ \Gamma$,
and $ T$) such that
\begin{equation}\label{eq_estnonneg_Hepp}
\begin{aligned}
\sum_{\mathbf{n} \in \mA(T)}
\prod_{u \in \trim}
(2^{- \mathbf{n}(u)} + \varepsilon)^{ \eta(u) -d}
2^{-d \, \mathbf{n}(u)}
\leqslant
C^{|V_\star|}
\big(
\log { \tfrac{1}{ \varepsilon}}
\big)^{ \mathrm{Null}(T)}\,.
\end{aligned}
\end{equation}
\end{lemma}

Now, we are ready to prove the main result of this section. Afterwards, we
provide the proof of Lemma~\ref{lem_estimate_nonneg_Hepp}.

\begin{proof}[Proof of Proposition~\ref{prop_countlogs_nonneg}]
  The value of the constant $C>0$ that appears below may change from line to line, in
  order to keep the presentation clean.
  We start with the following upper bound, which is derived by taking the
supremum over the test function and integrating out the legs (using the assumption $\mf{t}>-d$):
  \begin{equation}\label{eq_pos_est}
  \begin{aligned}
  |\Pi_{\ve} \Gamma ( \varphi)| \leqslant \|
\varphi\|_{\infty} C^{2k} \sum_{T \in \mT_{ V_{\star}}}  \sum_{\mathbf{n} \in \mA(T)}  \int_{D_{(T, \mathbf{n})}}
  | \mW_{\ve} \Gamma (x_{V_{\star}}) | \ud x_{V_{\star}} \;.
  \end{aligned}
  \end{equation}
  Furthermore, for all $ x\in D_{(T, \mathbf{n})}$ we estimate via~\eqref{e_orderbound}:

  \begin{equation}\label{eq_supp1_estLinf}
  \begin{aligned}
  |\mW_{\ve} \Gamma (x_{V_{\star}})|
  &=
  \prod_{e \in E_{ \star}}
 | K_{ e, \ve}(x_{e_{+}} - x_{e_{-}}) |
  \leqslant
  C^{| E_{\star}|}
  \prod_{e \in E_{\star}}
  ( | x_{ e_{+}} - x_{e_{-}}| + \varepsilon )^{\mathfrak{t}(e)}\\
  & \leqslant
  C^{| E_{\star}|}
  \prod_{e \in E_{\star}}
  (2^{- \mathbf{n}(e_{+} \wedge e_{-})} + \varepsilon)^{ \mathfrak{t} (e)}
  =
 C^{| E_{\star}|}
  \prod_{ u \in \trim}
  (2^{- \mathbf{n}(u)} + \varepsilon)^{\sum_{e \,: \, u = e_{+}\wedge e_{-}} \mathfrak{t}
  (e)}\,.
  \end{aligned}
  \end{equation}
  Hence, we find the upper bound
  \begin{equation}\label{eq_pos_Hepp_est_supp1}
  \begin{aligned}
  |\Pi_{\ve} \Gamma ( \varphi)|
  \leqslant
\| \varphi\|_{\infty}
  C^{|E_\star|+2k }
  \sum_{T \in \mT_{V_{\star}}}  \sum_{\mathbf{n} \in \mA ( T)}
  \prod_{u \in \trim}
  (2^{- \mathbf{n}(u)} + \varepsilon)^{ \eta(u) -d}
  2^{-d \, \mathbf{n}(u)}\,,
  \end{aligned}
  \end{equation}
  with $\eta$ defined as in~\eqref{eq_def_eta}.
  Applying Lemma~\ref{lem_estimate_nonneg_Hepp} for every $ T \in
  \mT_{V_{\star}}$, we obtain
  \begin{equation*}
  \begin{aligned}
  |\Pi_{\ve} \Gamma ( \varphi) |
  &\leqslant
  \| \varphi\|_{\infty}
  C^{|V_\star|+ | E_{\star}| + 2k}
  \sum_{T \in \mT_{V_{\star}}}
  \big(
  \log { \tfrac{1}{ \varepsilon}}
  \big)^{ \mathrm{Null}(T)}  \,,
  \end{aligned}
  \end{equation*}
  which concludes the proof of the result.
  \end{proof}

\begin{proof}[Proof of Lemma~\ref{lem_estimate_nonneg_Hepp}]
We perform a proof by induction over the size of $|
\overline{T}| =m$, i.e. $ |
V_{\star}| = m+1$.
First, for $m=0$ the sum in~\eqref{eq_estnonneg_Hepp} is empty, hence the
estimate is trivially true. Now fix
\begin{equation*}
  N_{\varepsilon} = \lfloor \log_{2} \tfrac{1}{\varepsilon} \rfloor \;.
\end{equation*}
For $m=1$, we have for all $K \leqslant
N_{\varepsilon}$
\begin{equation*}
\begin{aligned}
\sum_{n=K}^{\infty}
(2^{- n} + \varepsilon)^{ \eta -d}
2^{-d \, n}
& \leqslant
 4 ( N_{\varepsilon} - K +1 )^{ \mathds{1}_{\eta =0}}
( 2^{-  K}
+ \varepsilon )^{ \eta} \,,
\end{aligned}
\end{equation*}
using Lemma~\ref{lem_simple_sumargument},
where we wrote for convenience $ n = \mathbf{n} (u)$ and $ \eta = \eta(u)$ for the
only node $u \in \trim$ and used that $\eta (u) \in \NN$ by our assumption $\mf{t}(e) \in \ZZ$.
For clarity, let us write $\mA_K(T)$ for the set of all compatible scalings such
that $\mathbf{n}(u) \geqslant K$ for all $u \in \trim$:
\begin{equation}\label{e:def-AK}
  \mA_K(T) = \{ \mathbf{n} \in \mA(T) \text{ s.t. } \mathbf{n}(u) \geq K \;, \ \forall u \in \trim \}\;.
\end{equation}
Next, let us fix a tree $T$ with $| \overline{T}|=m+1$ assume that the following bound
\begin{equation}\label{eq_istep}
\begin{aligned}
\sum_{\mathbf{n} \in \mA_{K}(T')}
\prod_{u \in \trim'}
(2^{- \mathbf{n}(u)} + \varepsilon)^{ \eta(u) -d}
2^{-d \, \mathbf{n}(u)}
 \leqslant  C^{| \overline{T} '|}
\big(
2 N_{\varepsilon}- K +1 \big)^{ \mathrm{Null}(T')}
(2^{- K}+ \varepsilon)^{ \overline{\deg} (\vr')}
\,,
\end{aligned}
\end{equation}
holds true for all sub-Hepp trees $ T' \subseteq T$ with $ | \overline{T}'| \leqslant m$. Here a subtree is defined as the set of all vertices and leaves of $T$ that are the children of some $\vr' \in \trim$, which is the root of $T'$.
note that
$ 2N_{\varepsilon} -K \to \infty $ for all $ K \leqslant N_{\varepsilon}$,
because of the factor $2$ in front of $ N_{\varepsilon}$.

Now, by construction there exist trees  $ T_{1}, T_{2}$ (with roots $ \vr_{1},
\vr_{2} \in \trim$) such that $T= [T_1, T_2]$ is obtained by grafting $T_1$ and
$T_2$ onto a new root. Note that $ T_{i}$ may be
empty, if $ \vr_{i}$ is a leaf in $ V_{\star} $.
We then write, with $n = \mathbf{n}(\bullet)$ and $\eta_0 = \eta(\bullet)$ (where we recall that $\bullet$ is the root of $T$):
\begin{equation}\label{eq_supp1_heppsummation}
\begin{aligned}
\sum_{\mathbf{n} \in \mA_K(T)}
\prod_{u \in \overline{T}} &
(2^{- \mathbf{n}(u)} + \varepsilon)^{ \eta(u) -d}
2^{-d \, \mathbf{n}(u)}\\
&=
\sum_{ n = K}^{ N_{\varepsilon}}
(2^{- n} + \varepsilon)^{\eta_0 -d}
2^{-d \, n}
\prod_{i =1}^{2}
\sum_{\mathbf{n}|_{ T_{i}} \in \mA_{n}(T_{i})}
\prod_{u \in \overline{T}_{i}}
(2^{- \mathbf{n}(u)} + \varepsilon)^{ \eta(u) -d}
2^{-d \, \mathbf{n}(u)}\\
& \qquad \qquad +\sum_{\mathbf{n} \in \mA_{N_{\varepsilon}}(T)}
\prod_{u \in \overline{T}}
(2^{- \mathbf{n}(u)} + \varepsilon)^{ \eta(u) -d}
2^{-d \, \mathbf{n}(u)}\,,
\end{aligned}
\end{equation}
where $ \mathbf{n} |_{T_{i}}$ denotes the scales restricted to the nodes in $
T_{i}$.
First, we note that the second term on the right side is bounded by
\begin{equation}\label{eq_supp2_aboveN}
\begin{aligned}
\sum_{\mathbf{n} \in \mA_{N_{\varepsilon}}(T)}
\prod_{u \in \overline{T}}
(2^{- \mathbf{n}(u)} + \varepsilon)^{ \eta(u) -d}
2^{-d \, \mathbf{n}(u)}
& \leqslant
\prod_{u \in \overline{T}}
(2 \varepsilon)^{ \eta(u) -d}
\sum_{n = N_{\varepsilon}}^{\infty}
2^{-d \,n}\\
&\leqslant
\prod_{u \in \overline{T}}
(2\varepsilon)^{ \eta(u) -d}
2^{-d\,  N_{\varepsilon}}
=
2^{-d | \overline{T}|}
(2\varepsilon)^{\sum_{u \succ \bullet} \eta ( u)}
\leqslant C^{| \overline{T} |} \,,
\end{aligned}
\end{equation}
up to choosing an appropriate $C>0$, since $\sum_{u \succ \bullet} \eta ( u)
\geqslant 0 $ by assumption.
On the other hand, for the first term on the right side of~\eqref{eq_supp1_heppsummation} we apply the induction hypothesis, which
yields the upper bound
\begin{equation*}
\begin{aligned}
&\sum_{ n = K}^{ N_{\varepsilon}}
(2^{- n} + \varepsilon)^{\eta_0 -d}
2^{-d \, n}
\prod_{i =1}^{2}
\sum_{\mathbf{n}|_{T_{i}} \in \mA_{n}(T_{i})}
\prod_{u \in \overline{T}_{i}}
(2^{- \mathbf{n}(u)} + \varepsilon)^{ \eta(u) -d}
2^{-d \, \mathbf{n}(u)}\\
& \leqslant
C^{| \overline{T}_{1}| + | \overline{T}_{2}|}
\sum_{n= K}^{ N_{\varepsilon}}
(2^{- n} + \varepsilon)^{ \eta_0 -d}
2^{-d \, n}
\prod_{i =1}^{2}
\big(
2N_{\varepsilon}- n+1 \big)^{ \mathrm{Null}(T_i)}
(2^{- n}+ \varepsilon)^{ \sum_{u \succ \vr_{i}} \eta(u)}\\
& =
C^{| \overline{T}| -1}
\sum_{ n = K}^{ N_{\varepsilon}}
\big( 2N_{\varepsilon}- n +1 \big)^{ \mathrm{Null}(T_1) + \mathrm{Null}(T_2)}
(2^{- n} + \varepsilon)^{  \sum_{u \succ \bullet} \eta(u) -d}
2^{-d \, n}\,.
\end{aligned}
\end{equation*}
We then consider the following two cases:
\begin{enumerate}
\item If $  \sum_{u \succ \bullet} \eta(u) =0$, then $ \mathrm{Null}(T) =
\mathrm{Null} (T_1) + \mathrm{Null} (T_1) +1 $. In this case, the previous sum is bounded by:
\begin{equation*}
\begin{aligned}
C^{| \overline{T}| -1}
\sum_{ n = K}^{ N_{\varepsilon}} &
\big( 2 N_{\varepsilon}-  n +1 \big)^{\mathrm{Null}(T)-1}
(2^{- n} + \varepsilon)^{  -d}
2^{-d \, n} \\
& \leqslant C^{| \overline{T}| -1}
\sum_{n = K}^{ N_{\varepsilon}}
\big( 2 N_{\varepsilon}- n +1 \big)^{ \mathrm{Null} (T) -1} \\
& \leqslant C^{| \overline{T}| -1} \big( 2 N_{\varepsilon}- K +1 \big)^{ \mathrm{Null} (T)-1} (N_\ve - K +1)
\leqslant C^{| \overline{T}| -1} \big( 2 N_{\varepsilon}- K +1 \big)^{ \mathrm{Null} (T)} \,.
\end{aligned}
\end{equation*}
This leads to an estimate of the form~\eqref{eq_istep} for $ T' = T$.

\item If $  \sum_{u \succ \bullet} \eta(u) >0$, then  $ \mathrm{Null}(T) = \mathrm{Null} (T_1) + \mathrm{Null} (T_1) $.
Hence, by estimating
$ N_{\varepsilon}- n \leqslant N_{\varepsilon}- K$, we
obtain
\begin{equation*}
\begin{aligned}
&
C^{| \overline{T}| -1}
\sum_{n = K}^{ N_{\varepsilon}}
\big( 2 N_{\varepsilon}- n +1 \big)^{ \mathrm{Null} (T_1) + \mathrm{Null} (T_1)}
(2^{- n} + \varepsilon)^{  \sum_{u \succ \bullet} \eta(u) -d}
2^{-d \, n}
\\
& \leqslant
C^{| \overline{T}| -1}
\big(2 N_{\varepsilon}- K +1 \big)^{ \mathrm{Null}(T)}
\sum_{n = K}^{ \infty}
(2^{- n} + \varepsilon)^{  \sum_{u \succ \bullet} \eta(u) -d}
2^{-d \, n}\\
& \leqslant C^{| \overline{T}|-1}
4 \big(2  N_{\varepsilon}- K +1 \big)^{ \mathrm{Null}(T)}
(2^{- K} + \varepsilon)^{  \sum_{u \succ \bullet} \eta(u)}
\end{aligned}
\end{equation*}
where we used Lemma~\ref{lem_simple_sumargument} in the last step. This is
again an estimate of the correct order, provided that $C \geqslant 4$.
\end{enumerate}
This completes the proof of the induction argument.
Using~\eqref{eq_istep} for $ T'= T$ and setting $K=0$ delivers the overall estimate claimed in the lemma. The proof is complete.
\end{proof}
The next result is a supporting statement for the proof of the preceding lemma.

\begin{lemma}\label{lem_simple_sumargument}
For every $ \alpha \geqslant 0 $ and $K \leqslant N_{\varepsilon}  = \log_{2} \tfrac{1}{ \varepsilon}$
\begin{equation*}
\begin{aligned}
\sum_{n=K}^{\infty}
(2^{- n} + \varepsilon)^{ \alpha -d}
2^{-d \, n}
& \leqslant
\begin{cases}
 \frac{2}{ 1- 2^{- \alpha\wedge d}} ( 2^{-  K}
+ \varepsilon )^{ \alpha}  & \quad \text{if } \ \alpha>0 \,,\\
\frac{1}{1- 2^{-d}} ( N_{\varepsilon} - K + 1)\,,  & \quad \text{if } \ \alpha=0
\,.
\end{cases}
\end{aligned}
\end{equation*}
In particular, if $ \alpha \in \mathbf{N}$ then
\begin{equation*}
\begin{aligned}
\sum_{n=K}^{\infty}
(2^{- n} + \varepsilon)^{ \alpha -d}
2^{-d \, n}
\leqslant 4 ( N_{\varepsilon} - K +1 )^{ \mathds{1}_{\alpha =0}}
( 2^{-  K}
+ \varepsilon )^{ \alpha} \,.
\end{aligned}
\end{equation*}
\end{lemma}

\begin{proof}
First consider the case $ \alpha>0 $ and $ \alpha-d \leqslant   0 $. Then
\begin{equation*}
\begin{aligned}
\sum_{n=K}^{\infty}
(2^{- n} + \varepsilon)^{ \alpha -d}
2^{-d \, n}
&\leqslant
\sum_{n=K}^{N_{\varepsilon}-1}
2^{- n(  \alpha -d)} 2^{-d \, n}
+
 \varepsilon^{ \alpha -d}
\sum_{n=N_{\varepsilon}}^{ \infty}
2^{-d \, n}\\
& =
\frac{2^{- \alpha\,  K} - 2^{- \alpha\,  N_{\varepsilon}}}{ 1- 2^{- \alpha}}
+
\frac{
 \varepsilon^{ \alpha -d}
2^{-d \, N_{\varepsilon}}}{1- 2^{-d}} \\
& \leqslant
\frac{1}{ 1- 2^{- \alpha }}
\big( 2^{- \alpha\, K}
+ \varepsilon^{\alpha} \big)
\leqslant
\frac{2}{ 1- 2^{- \alpha }}
( 2^{-  K}
+ \varepsilon )^{ \alpha}\,.
\end{aligned}
\end{equation*}
On the other hand, if $ \alpha=0$, we instead
get the bound
\begin{equation*}
\begin{aligned}
\sum_{n=K}^{\infty}
(2^{- n} + \varepsilon)^{ \alpha -d}
2^{-d \, n}
&\leqslant ( N_{\varepsilon}- K )
+
\frac{
 1}{1- 2^{-d}} \,.
\end{aligned}
\end{equation*}
Lastly, if $ \alpha- d >0$, then
\begin{equation*}
\begin{aligned}
\sum_{n=K}^{\infty}
(2^{- n} + \varepsilon)^{ \alpha -d}
2^{-d \, n}
&\leqslant
(2^{- K} + \varepsilon)^{ \alpha -d}
\sum_{n=K}^{\infty}
2^{-d \, n}
=
(2^{- K} + \varepsilon)^{ \alpha -d}
\frac{2^{-d\, K}}{1 - 2^{-d}} \\
&\leqslant
(2^{- K} + \varepsilon)^{ \alpha -d}
\frac{(2^{- K}+ \varepsilon)^{d}}{1 - 2^{-d}}
=
\frac{1}{1 - 2^{-d}}
(2^{- K} + \varepsilon)^{ \alpha }\,.
\end{aligned}
\end{equation*}
Combining the bounds respectively yields the statement of the lemma.
\end{proof}

\subsection{Application to the Anderson model}

From Proposition~\ref{prop_countlogs_nonneg} we deduce the following lemma, which proves that only nested bubble diagrams contribute in the weak coupling limit of the Anderson model.

\begin{lemma}\label{lem_nonneg_diagrams}
Fix any $k \in \NN_*$ and let $ \Gamma$ be a connected  $2k$-Anderson-Feynman diagram that
is positive or has zero degree, cf. Definition~\ref{def:degree}. Then for every
$ \varphi \in \mS ( ( \TT^{4})^{L}; \RR) $
\begin{equation}\label{eq_estimate_nonneg}
\begin{aligned}
\sup_{\varepsilon \in (0,1)}
|\lambda_{\varepsilon}|^{| E_{\star}|}
| \Pi_{\varepsilon} \Gamma ( \varphi)| \lesssim_{|
E_{\star}|, k} |\hat{\lambda}|^{|E_\star|} \,.
\end{aligned}
\end{equation}
Moreover, if $k \neq 2$, or if $k=2$ but $ \Gamma$ is \emph{not} a nested bubble diagram, then
\begin{equation}\label{eq_nonbubble_vanish}
\begin{aligned}
\lim_{ \varepsilon \to 0}
|\lambda_{\varepsilon}|^{| E_{\star}|}
| \Pi_{\varepsilon} \Gamma ( \varphi)|
= 0\,.
\end{aligned}
\end{equation}
\end{lemma}

\begin{proof}
Let $ \Gamma $ be a non--negative Anderson--Feynman diagram with $2k$ legs, i.e. it is
of the form $ \Gamma =  (\mI_{n_{1} }, \ldots, \mI_{n_{k}})_{\kappa} $, for some $k
\geqslant 1 $ and some complete pairing $\kappa$.
Then for every $T \in \mT_{V_{\star}}$
\begin{equation}\label{eq_Estar_card}
\begin{aligned}
| E_{\star}|
= \sum_{\ell =1}^{k}
(n_{\ell}-1)
=
2
\sum_{\ell =1}^{k} \frac{n_{\ell}}{2} -k
= 2 | V_{\star}|-k
= 2 |\trim| +(2-k)\,.
\end{aligned}
\end{equation}
Thus, by applying Proposition~\ref{prop_countlogs_nonneg}, we have for some $C>0$:
\begin{equation}\label{eq_supp1_nonneg}
\begin{aligned}
|\lambda_{\varepsilon}|^{| E_{\star}|}
|\Pi_{\varepsilon} \Gamma ( \varphi) |
&\leq
\| \varphi\|_{\infty} C^{|E_\star| +2k } |\hat{\lambda}|^{|E_\star|}
\sum_{T \in \mT_{V_{\star}}}
\big(
\log{ \tfrac{1}{ \varepsilon}}
\big)^{\mathrm{Null}(T)- |\trim|+\frac{k-2}{2}}\,.
\end{aligned}
\end{equation}
We analyse the following three cases to conclude~\eqref{eq_estimate_nonneg}:

\begin{itemize}

  \item \underline{$k =1$}.
In this case $ \Gamma = ( \mI_{n})_{\kappa}$, for some $n$
that is even. Then $ \deg \, \Gamma = d (\tfrac{n}{2}-1) - 2
(n-1)=2-d<0$, since the pairing $ \kappa$ halves the number of vertices. Hence, a non--negative diagram can never emerge from only
internal pairings.

\item \underline{$k=2$}.
Then the exponent on the right side of~\eqref{eq_supp1_nonneg}
is non--positive, because $ \mathrm{Null}(T) \leqslant |T|$ by definition.


\item \underline{$k \geqslant 3$}.
In this case, we will prove that $\mathrm{Null}(T) \leqslant |T| -
\tfrac{k-1}{2}$ for every $ T \in \mT_{V_{\star}}$,
which implies that the right side in~\eqref{eq_supp1_nonneg} always vanishes in the limit $\ve \to 0$.
To see this, we first note that because $ k \geqslant 3$
\begin{equation}\label{eq_k3_posdeg}
\begin{aligned}
\sum_{v \in \trim} \eta (v)
=
\deg \, \Gamma
= d ( | V_{\star}| -1 ) -2 (2 |V_{\star}| -k)
= 2k-4 >0
\,,
\end{aligned}
\end{equation}
since $ d= 4$.
Thus, the last integration at the root of the Hepp tree $T$ can not yield a
logarithmic divergence.
Therefore, $\mathrm{Null} (T) = \mathrm{Null} (T_{1}) + \mathrm{Null}( T_{2})$ where $ T_{1}, T_{2}$ are the two
subtrees obtained by splitting off the root of $T = [T_1, T_2]$. We will write $ V_{i}$ for the set of
leaves in $ T_{i}$. Note that $ V_{1}$ and $ V_{2}$ form a partition of $
V_{\star} $.

Next, note that there is at least one edge that connects
$ V_{1}$ to $ V_{ 2}$ in $ \Gamma $, otherwise $ \Gamma $ would be disconnected.
Hence,
\begin{equation*}
\begin{aligned}
\mathrm{Null} (T) =
 \mathrm{Null} (T_{1}) + \mathrm{Null}( T_{2})
\leqslant \frac{1}{2} \big(| E_{\star}| -1\big)
=
\frac{1}{ 2}
\big(2| \trim| +2-k\big)  - \frac{1}{2} \,,
\end{aligned}
\end{equation*}
 where we used Lemma~\ref{lem_eta_nonneg} below, which states that
$\mathrm{Null} (T_{i}) $ is bounded by half the number of edges present in the
subgraph that is spanned by the leaves of $ T_{i}$.
In combination with~\eqref{eq_supp1_nonneg}, this finally yields
\begin{equation*}
\begin{aligned}
|\lambda_{\varepsilon}|^{| E_{\star}|}
| \Pi_{\varepsilon} \Gamma ( \varphi) |
\lesssim_{|E_\star|, \varphi, \hat{\lambda}}
\frac{1}{ \sqrt{\log \tfrac{1}{ \varepsilon}}} \to 0  \,, \quad \text{as }
\varepsilon \to 0 \,.
\end{aligned}
\end{equation*}
\end{itemize}
To conclude~\eqref{eq_nonbubble_vanish} we first note that if $ k \geqslant
3$ then $ \Gamma$ is never a nested bubble diagram because $ \deg\,
\Gamma >0$, cf.~\eqref{eq_k3_posdeg}.
On the other hand, in the case $k=2$, it suffices to show that if there exists a Hepp tree $ T
\in \mT_{V_{\star}}$ such that $ \mathrm{Null}(T) = |\trim|$, then $ \Gamma$ must be a
nested bubble diagram.

Indeed fix any Heap ordering $\{u_i\}_{i=1}^{|\trim|} = \trim $ of the inner
vertices (see Definition~\ref{def:heap-ordering}),
and set $\widetilde{\Gamma}_i = \Gamma_{u_i}$. Now, we iteratively contract the
subdiagrams $\widetilde{\Gamma}_{i}$. Let us define $\Gamma_1 =
\Gamma$ and then for $i\geqslant 1$ set $\Gamma_{i+1} = \mK_{\bg_{i}} \Gamma_i$, where $\bg_i$
is the image of $\widetilde{\Gamma}_{i}$ in the contracted diagram $\Gamma_{i}$,
following Remark~\ref{rem:subdiam}.
Then we claim that $\{\Gamma_i\}_{i=1}^{|\trim|+1}$ and
$\{\bg_i\}_{i=1}^{|\trim|}$
form a sequence of nested bubble diagrams as in Definition~\ref{def:bubbl}. The
only thing we must show is that each $\bg_i$ is isomorphic to a bubble diagram.
This is the case for $\bg_1$ because it necessarily consists of the
subdiagram spanned by two leaves and in order to have degree zero, these must
be connected by two edges, forming a bubble (since self--loops are not possible
as they lead to strictly negative divergences). For all $i>1$ we argue similarly. Let $v_{i,1}, v_{i,2}$ be the two
children of $u_i$ in $ T $. If any of the $v_{i,j}$ is not a leaf, then it must
have already appeared as some $u_k$ for $k< i$, because we chose a Heap
ordering. This means that all the leaf variables that have $u_k$ as their ancestor have been
contracted to a single node in $\Gamma_k$, and therefore also in $\Gamma_{i}$.
Hence, $\bg_i$ is again spanned by two vertices in $\Gamma_{i-1}$, and the same argument as before shows that it must be a bubble diagram. This concludes the proof.
\end{proof}

\begin{lemma}\label{lem_eta_nonneg}
  Let $ \Gamma $ be a non--negative Anderson--Feynman diagram and $ T \in \mT_{
  V_{\star}}$ an associated  Hepp tree. We have the upper bound $ \mathrm{Null} ( T ) \leqslant \tfrac{ |E_{\star} |}{2}$.
  \end{lemma}


  \begin{proof}

  Let $ F \subset \trim $ be the subset of inner nodes $u$ such that
  \begin{equation}\label{eq_supp1_boundK}
  \begin{aligned}
\overline{\deg} ( u) =
\sum_{v \succ u} \eta (v) =0 \quad \text{and} \quad
\overline{\deg} ( w) >0 \quad \text{for all } w \prec u \,,
  \end{aligned}
  \end{equation}
  i.e.  all nodes $w$ on the path from $u \in F $ to the root $\bullet$ satisfy $\sum_{v \succ w} \eta
  (v)>0$, and
  \begin{equation*}
  \begin{aligned}
  0=\sum_{v \succ u} \eta(v)
  = d (|V(\Gamma_u)| -1) - 2 | E (\Gamma_u)|\,.
  \end{aligned}
  \end{equation*}
  In particular, if $T_u \subseteq \trim$ is the set of inner
nodes that are children of $u \in F$, namely $T_u = \{v \in \trim \text{ such that }
v\succ u\}$, then since $d=4$
  \begin{equation*}
    |T_u|= |V(\Gamma_u)| -1 = \frac{1}{2} |E (\Gamma_u) | \;, \qquad \forall u \in F\;.
  \end{equation*}
  Because all inner nodes $v \in \trim$ satisfying $ \overline{\deg}(v) =0$ must be descendants of nodes in $ F $,
  we have the upper bound
  \begin{equation*}
  \begin{aligned}
  \mathrm{Null} (T) \leqslant \sum_{u \in F} | T_{u}|
  = \sum_{u \in F} \frac{| E (\Gamma_u)|}{ 2}
  \leqslant \frac{| E_{\star}|}{2} \,,
  \end{aligned}
  \end{equation*}
  where in the last step we used that the subdiagrams $ \Gamma_{u}$ are
  disjoint.
  \end{proof}

\section{Exact contributions of nested bubble diagrams}\label{sec_exact}

In the previous section, we proved that any non--negative contributing diagram must be a
nested bubble diagram, see Lemma~\ref{lem_nonneg_diagrams}.
In particular, this lemma proves an upper bound to the
contribution of a single contributing diagram:
\begin{equation}\label{eq_bubble_upperbound}
\begin{aligned}
\lim_{\varepsilon \to 0}
|\lambda_{\varepsilon}|^{| E_{\star}|}
|\Pi_{\varepsilon} \Gamma ( \varphi) |
\leqslant  (C |\hat{\lambda}|)^{| E_{\star}|}
\big| \{
T \in \mT_{V_{ \star}} \,: \,  \overline{\deg}(u) \equiv 0
\}\big|\,,
\end{aligned}
\end{equation}
where $ C>0$ is some constant. However,~\eqref{eq_bubble_upperbound} is not
sharp, and the objective of this section is to obtain the precise limiting
behaviour of a nested bubble diagram. The fundamental challenge we face is
captured by the following example. We must be able to distinguish the
contributions of the following two diagrams:
\begin{equation}\label{eq_exactcontr_ex1}
  \<c42plain> = \<cc42plain> \;, \qquad \qquad \<c42> = \<cc42> \;.
\end{equation}
Here we will link the contribution of a diagram to its symmetries. The
distinguishing feature between the two diagrams above is that the first one has
three bubbles that we can extract, while the second one only has one: in this
case the next bubble will only appear once we contract the first one. Diagrams that
contain many nested bubbles turn out to contribute less, because if one bubble
is nested in another, then in order to contribute there is a natural order
forced on the integration variables. More precisely, the scales associated to the innermost
variables must be smaller than the outermost ones. As we will see, the correct
way to count the \lqm symmetries\rqm of nested bubble diagrams, in order to link
symmetries to the limiting contribution, is by counting
the number of nested bubble extraction sequences from
Definition~\ref{def:extraction-sequences}. Eventually, since there are $6$ ways
to extract three bubbles from the first diagram in~\eqref{eq_exactcontr_ex1}, but only one way to extract
three bubbles from the second diagram, the contribution of the first will be $6$ times
the one of the second diagram.


The main technical result of this section is the following
proposition, which is in the spirit of Proposition~\ref{prop_countlogs_nonneg},
and allows us to estimate the exact contribution of a nested bubble diagram over
a given Hepp ordering. For the statement of this result we require tree
factorials. For a finite rooted tree $\tau$ obtained by grafting $\tau = [\tau_1, \ldots,
\tau_n]$ for some $n \in \NN_*$, we define
\begin{equation*}
  \tau ! = |\tau | \tau_1 ! \cdots \tau_n ! \;, \qquad \bullet ! = 1 \;,
\end{equation*}
where with $\bullet$ we indicate the tree that consists only of the root node,
and $|\tau| = 1 + |\tau_1| + \cdots + |\tau_n|$ is the number of vertices
(including leaves) of $\tau$.
Finally, let us define the set of \emph{contributing Hepp trees} over a given diagram
$\Gamma$:
\begin{equation}\label{e:def_contr_hepp}
  \mT_{V_\star}^{0}(\Gamma) = \{ T \in \mT_{V_\star} \text{ such that } \rm{Null}(T) = |\trim| \} \;.
\end{equation}
Here $\rm{Null}(T)$ is defined as in~\eqref{e:null-defn} with $\mf{t} \equiv -2$.
Whenever clear from context we write $\mT_{V_\star}^{0}$ instead of $\mT_{V_\star}^{0}(\Gamma)$.

\begin{proposition}\label{prop_bubbleDiag_exact}
Let $ \Gamma $ be a connected nested bubble diagram as in Definition~\ref{def:bubbl}, then for all $\varphi \in
\mS( (\TT^d)^L; \RR)$
\begin{equation*}
\begin{aligned}
\lim_{ \varepsilon \to 0} \lambda_{\varepsilon}^{| E_{\star}|}
\Pi_{\varepsilon} \Gamma ( \varphi)
=
\bigg(\frac{\hat{\lambda}}{ \sqrt{8} \pi} \bigg)^{| E_{\star} |} \sum_{T \in \mT_{V_{\star}}^{0}(\Gamma)}
\frac{1}{ \overline{T} !}
\<cc1_narrow> ( \varphi ) \;,
\end{aligned}
\end{equation*}
with
\begin{equation*}
\begin{aligned}
\<cc1_narrow> ( \varphi )
 :=
\int_{ ( \mathbf{T}^{d} )^{L}} \varphi (x_{L} )
\int_{\mathbf{T}^{d}} \prod_{v \in L} G
(x_{\star} - x_{v}) \ud x_{\star}  \ud x_{L}\,.
\end{aligned}
\end{equation*}
\end{proposition}
The proof of this result is presented in Section~\ref{sec:proof_prop_bubbleDiag_exact} below.
Now, a combinatorial argument yields a representation of the limit in terms of the
number of possible bubble extraction sequences, which links the proposition to
the discussion at the beginning of this section.

\begin{corollary}\label{cor_contr_diagram}
Let $ \Gamma $ be a connected nested bubble diagram as in Definition~\ref{def:bubbl} and write $m =
\tfrac{1}{2} | E_{\star}| = | V_{\star}| -1 $, then for all $\varphi \in \mS(
(\TT^d)^L ; \RR^d)$
\begin{equation*}
\begin{aligned}
\lim_{ \varepsilon \to 0}
\lambda_{\varepsilon}^{2m}
\Pi_{\varepsilon} \Gamma ( \varphi)
=
\Big(\frac{ \hat{\lambda}}{ \sqrt{8} \pi}\Big)^{2m}
\frac{|\mathbb{S}( \Gamma)|}{m! }
\<cc1_narrow> ( \varphi )  \,,
\end{aligned}
\end{equation*}
where $ \mathbb{S}( \Gamma )$ is the set of bubble extraction sequences of $
\Gamma$ from Definition~\ref{def:extraction-sequences}, and we used the same
notation as in Proposition~\ref{prop_bubbleDiag_exact}.
\end{corollary}
\begin{proof}
For the proof of this result recall from Definition~\ref{def:heap-ordering} that
a map $\ell : \overline{T} \to \{1, \ldots, |
\overline{T}|\}$ is a heap ordering with respect to the partial order of
the tree $\trim$ if $\ell(a) \geq \ell(b)$ whenever $a \prec b$.

Then the hooklength formula~\cite{Sagan1989}\footnote{In~\cite{Sagan1989}, heap orderings are called ``natural orderings''
and assign to the root the minimial label $1$, instead of the maximal label.
These two formulations are equivalent, since $ u \mapsto | \overline{T}| +1 - \ell (u)$
defines such a natural ordering.} states that
\begin{equation*}
  \begin{aligned}
  \frac{|\overline{T}|! }{ \overline{T} !}
  =
  \big|
  \{
  \ell \,: \, \ell \text{ is a heap ordering of } { \trim }
  \}\big|\,,
  \end{aligned}
\end{equation*}
which combined with Proposition~\ref{prop_bubbleDiag_exact} leads to
\begin{equation*}
  \begin{aligned}
  \lim_{ \varepsilon \to 0} \lambda_{\varepsilon}^{| E_{\star}|}
  \Pi_{\varepsilon} \Gamma ( \varphi)
  =
  \bigg(\frac{\hat{\lambda}}{ \sqrt{8} \pi} \bigg)^{| E_{\star} |} \bigg( \sum_{T \in \mT_{V_{\star}}^{0}(\Gamma)}
  \frac{1}{ m !} \big|
  \{
  \ell \,: \, \ell \text{ is a heap ordering of } {\trim}
  \}\big| \bigg)
  \<cc1_narrow> ( \varphi ) \;.
  \end{aligned}
  \end{equation*}
  Then the result is a consequence of Lemma~\ref{lem_bijection}, which proves
  that the sets
  \begin{equation}\label{e:def-Tord}
    \mT^{\rm{ord}}_{V_\star} (\Gamma) = \{
    (T,\ell) \,: \, T \in \mT_{V_\star}^0(\Gamma) \;, \text{ and }\ell \text{ is a heap ordering of } { \trim }
    \}
  \end{equation}
  and $\mathbb{S}(\Gamma)$ are in bijection. This completes the proof of the result.
\end{proof}

The missing ingredient of the proof of Corollary~\ref{cor_contr_diagram} is
the following lemma, which states that every $ T \in \mT_{V_{\star}}^{0}(\Gamma)$ and
heap ordering $ \ell$ of $\trim$, is associated to exactly one bubble extraction sequence in
$\mathbb{S} ( \Gamma) $.

\begin{lemma}\label{lem_bijection}
  For every Anderson--Feynman diagram $ \Gamma$, there exists a bijection $ F  : \mathbb{S} (\Gamma) \to
  \mT_{V_{\star}}^{\mathrm{ord}}(\Gamma)$, with the notation of
  Definition~\ref{def:extraction-sequences} and~\eqref{e:def-Tord}.
  \end{lemma}

\begin{proof}
We construct a map $ F : \mathbb{S} (\Gamma) \to
\mT_{V_{\star}}^{\mathrm{ord}}(\Gamma)$ below, and check
that it is indeed a bijection.

Consider a bubble extraction sequence $\mE = \{ \tilde{\Gamma}_i \}_{i=1}^{n-1} \in \mathbb{S}(\Gamma)$,
where $n = |V_\star(\Gamma)|$. To construct a tree and associated heap ordering
$(T, \ell)$ we proceed as follows:
Our starting point is a forest $T^{(0)} =\{T_v, v \in V_\star\}$ of singletons,
meaning that each $T_v$ is a finite rooted tree consisting only of the root node.

\begin{enumerate}
\item Let $\overline{V}_1=\{v, w\}$ be the vertex set of the bubble
	$\overline{\Gamma}_1 \subseteq \tilde{\Gamma}_1$, see
	{Lemma~\ref{lem:extraction-sequence}.} 
	Note that $\tilde{\Gamma}_1$ consists
	only of $\overline{\Gamma}_1$ and possibly a number of
	singletons. Since this is the first bubble, we could also fix $\overline{\Gamma}_1 =
\tilde{\Gamma}_1$, but we refrain from doing so to make the next steps clearer.
Then define $T_{u_{\overline{\Gamma}_1}} := [T_{v}, T_{w}]_{u_{\overline{\Gamma}_1}}$ by which we mean the tree
that has root labelled $v_{\overline{\Gamma}_1}$ and leaves $v, w$. The choice
of the label $u_{\overline{\Gamma}_1}$ corresponds to the fact that
$u_{\overline{\Gamma}_1}$ is the vertex created through the contraction
$\bigslant{\Gamma}{\tilde{\Gamma}_1}$ following Definition~\ref{def:contract}.
We then set $T^{(1)} = (T^{(0)} \setminus \{T_{v_1}, T_{v_2}\} )\cup \{
T_{\overline{\Gamma}_1}\}$. For example, 
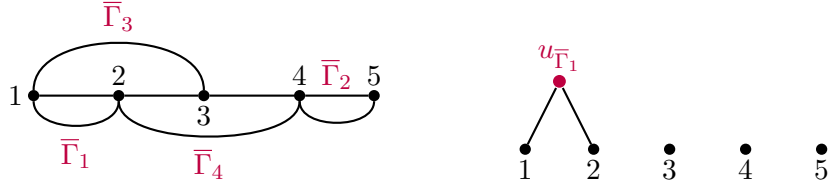
\begin{equation*}
  \begin{tikzpicture}[thick, scale = 0.7]
    \tikzset{
      vtx/.style={circle,fill=black,inner sep=2pt}
    }
    \foreach \i/\x in {1/0,2/1.6,3/3.2,4/5.0,5/6.4}{
      \node[dot] (v\i) at (\x,0) {};
    }
    \node[left] at (v1) {$1$};
    \node[above] at (v2) {$2$};
    \node[below] at (v3) {$3$};
    \node[above] at (v4) {$4$};
    \node[above] at (v5) {$5$};
    \draw (v1) -- (v2) -- (v3) -- (v4) -- node[midway,above=-1.5pt,purple] {$\overline{\Gamma}_2$} (v5);
    \draw (v1) to[out=-90, in =-90] node[midway,below=0.5pt,purple] {$\overline{\Gamma}_1$} (v2);
    \draw (v4) to[out=-90, in =-90] (v5);
    \draw (v1) .. controls +(0,1.3) and +(0,1.3) .. node[midway,above=0.5pt,purple] {$\overline{\Gamma}_3$} (v3);
    \draw (v2) .. controls +(0,-1.0) and +(0,-1.0) .. node[midway,below=0.5pt,purple] {$\overline{\Gamma}_4$} (v4);
    \end{tikzpicture}
    \qquad \qquad
    \begin{tikzpicture}[thick, scale =0.5]
      \tikzset{
        vtx/.style={circle,fill=black,inner sep=1.7pt}
      }
    \begin{scope}[xshift=11cm]
      \node[vtx,fill=purple] (u1) at (0,1.8) {};
      \node[dot] (w1) at (-0.9,0) {};
      \node[dot] (w2) at (0.9,0) {};
      \draw (u1) -- (w1);
      \draw (u1) -- (w2);
      \node[above,purple] at (u1) {$u_{\overline{\Gamma}_1}$};
      \node[below] at (w1) {$1$};
      \node[below] at (w2) {$2$};
      \node[dot] (w3) at (2.9,0) {};
      \node[dot] (w4) at (4.9,0) {};
      \node[dot] (w5) at (6.9,0) {};
      \node[below] at (w3) {$3$};
      \node[below] at (w4) {$4$};
      \node[below] at (w5) {$5$};
    \end{scope}
    \end{tikzpicture}
\end{equation*}
{where the $\overline{\Gamma}_i$ are suggestive of the bubbles that appear
  iteratively as we remove one after the other (the actual bubble
$\overline{\Gamma}_i$ appear only after some contractions).}

\item We iterate this procedure. Indeed, we are in the same position we
started from. Assume that for $i\in \{1, \ldots, n-2\}$ we are given a forest $T^{(i)}$ of the form $T^{(i)} = \{T_v \,
\colon \, v \in V(\bigslant{\Gamma}{\tilde{\Gamma}_{i}})\}$. Then we define the
forest $T^{(i+1)}$ as follows. Let $\{v, w\}$ be the vertices of the bubble
$\overline{\Gamma}_{i+1} \subseteq \bigslant{\Gamma}{\tilde{\Gamma}_{i}}$, see 
{Lemma~\ref{lem:extraction-sequence}.} Then
define the grafted tree $T_{u_{\overline{\Gamma}_{i+1}}}= [T_{v}, T_{w}]_{u_{\overline{\Gamma}_{i+1}}}$, and set $T^{(i+1)} = ( T^{(i)} \setminus \{T_v, T_w\}) \cup
\{T_{u_{\overline{\Gamma}_{i+1}}}\}$.


\end{enumerate}

With this construction, the forest $T^{(n-1)}$ is a single rooted trees, and we set $T =
T^{(n-1)}$. In our example we would obtain:

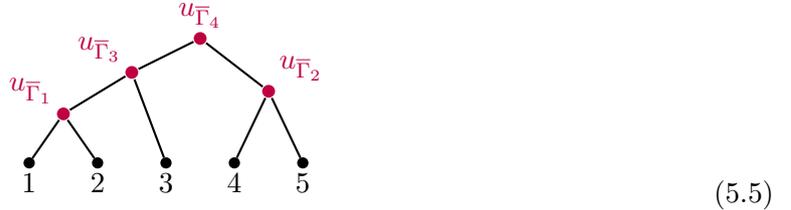
\begin{equation}\label{e:hepp-tree-example}
  \begin{tikzpicture}[thick,scale=0.5]
    \tikzset{
      leaf/.style={circle,fill=black,inner sep=2pt},
      internal/.style={circle,fill=purple,inner sep=1.7pt}
    }
    \node[dot] (x1) at (0,0)   {};
    \node[dot] (x2) at (1.8,0) {};
    \node[dot] (x3) at (3.6,0) {};
    \node[dot] (x4) at (5.4,0) {};
    \node[dot] (x5) at (7.2,0) {};
    \node[below] at (x1) {$1$};
    \node[below] at (x2) {$2$};
    \node[below] at (x3) {$3$};
    \node[below] at (x4) {$4$};
    \node[below] at (x5) {$5$};
    \node[internal] (u1) at (0.9,1.3) {};
    \node[internal] (u3) at (2.7,2.4) {};
    \node[internal] (u2) at (6.3,1.9) {};
    \node[internal] (u4) at (4.5,3.3) {};
    \node[above left,purple]  at (u1) {$u_{\overline{\Gamma}_1}$};
    \node[above left,purple]  at (u3) {$u_{\overline{\Gamma}_3}$};
    \node[above right,purple] at (u2) {$u_{\overline{\Gamma}_2}$};
    \node[above,purple]       at (u4) {$u_{\overline{\Gamma}_4}$};
    \draw (x1) -- (u1);
    \draw (x2) -- (u1);
    \draw (u1) -- (u3);
    \draw (x3) -- (u3);
    \draw (u3) -- (u4);
    \draw (u4) -- (u2);
    \draw (x4) -- (u2);
    \draw (x5) -- (u2);
  \end{tikzpicture}
\end{equation}
Furthermore, by construction we have that $T \in
\mT_{V_\star}^{0}(\Gamma)$, and we have also obtained a heap ordering via the
labelling of $\trim$, by
assigning $\ell(u_{\overline{\Gamma}_i}) = i$.

Finally, this construction can be inverted. Fix any $ (T, \ell) \in
\mT_{V_{\star}}^{\mathrm{ord}}(\Gamma)$. We write $u_i$ for the inner node $u
\in \trim$ such that $\ell(u)=i$. Then define
\begin{equation*}
  \tilde{\Gamma}_{i}= \bigcup_{j \leq i} \Gamma_{u_j}  \;, \qquad \forall i \in \{1, \ldots, n-1\}\;,
\end{equation*}
where we have followed the notation of~\eqref{e:Gu}. It can be checked following
similar arguments as above, that $\{\tilde{\Gamma}_i\}_{i=1}^{n-1}$ is a bubble
extraction sequence in the sense of Definition~\ref{def:extraction-sequences}.
This complete the proof of the result.
\end{proof}

With the conclusion of Corollary~\ref{cor_contr_diagram}, we are
ready to move towards the proof of Proposition~\ref{prop_bubbleDiag_exact}. 
This requires some further concepts.

\subsection{Zero degree forests}\label{sec:zero-forests}

In this subsection, we introduce forests that consist of degree zero subdiagrams.
These are related to many of the concepts we already introduced: in particular
to contributing Hepp trees as in~\eqref{e:def_contr_hepp} and to heap orderings
as seen in the previous subsection. The following concept is the analogue of
Definition~\ref{def:forest-div}.

\begin{definition}[Zero degree forests]\label{def_zeroforest}
  Let $ \Gamma$ be a nested bubble diagram.
  A collection $ \mf{F} $ of connected, full subdiagrams of $ \Gamma$ is called a
  \emph{degree zero forest}, if for any $\overline{\Gamma}_1,\overline{\Gamma}_2\in \mf{F}$ either
  $\overline{\Gamma}_1\subset \overline{\Gamma}_2$, or $ \overline{\Gamma}_2\subset \overline{\Gamma}_1$, or
  $\overline{\Gamma}_1$ and $ \overline{\Gamma}_2$ are vertex-disjoint,
  and if $ \deg\, \overline{\Gamma}=0$ for every $ \overline{\Gamma} \in \mf{F} $.
  \end{definition}
Every contributing Hepp tree $ T \in \mT_{V_{\star}}^{0}(\Gamma)$, as
in~\eqref{e:def_contr_hepp}, induces a zero degree forest. 

\begin{lemma}\label{lem_forestT}
For a nested  bubble diagram diagram $ \Gamma$ and an associated Hepp tree $ T
\in \mT_{V_{\star}}^{(0)}(\Gamma)$, the collection
\begin{equation*}
\begin{aligned}
\mf{F}_{T}: = \{ \Gamma_{u} \,: \, u \in \trim \}\,,
\end{aligned}
\end{equation*}
defines a zero degree forest, with $ \Gamma_{u}$ as in~\eqref{e:Gu}.
\end{lemma}
We do not provide a proof, because this observation is immediate, up to the fact that $
\Gamma_{u}$ as in~\eqref{e:Gu} with $u \in \trim$ is always connected, provided $
T \in \mT_{V_{\star}}^{(0)}(\Gamma)$. 
To see this, we first notice that $T \in \mT_{V_{\star}}^{(0)}(\Gamma)$ implies $
\eta\equiv 0 $. In particular, for $ \eta $ to vanish at every node, exactly two
(previously unobserved) edges
have to be
taken into account for in \eqref{eq_def_eta}, connecting leaves in the left branch of $
\trim$ rooted at $
u $ to leaves in the right branch.

Given a nested bubble diagram $ \Gamma = (V,E)$, it will be useful to iteratively extract
subdiagrams from a forest $\mf{F} $ of zero degree subdiagrams by ``rewiring''
the edges of the original diagram $ \Gamma$. This operation is performed through the iterative
application of extractions maps $\hat{\mC}_{\gamma}$ for $\gamma \in \mf{F}$.
The figure to keep in mind is the following:
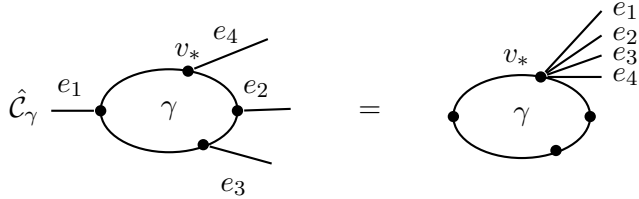
\begin{equation*}
  \hat{\mC}_{\gamma} \;
  \begin{tikzpicture}[thick,>=stealth, scale=0.5, baseline=-0.2em]
    \tikzset{
      dot/.style={circle,fill=black,inner sep=1.5pt}
    }
    \node[above] at (-3.0,0.1) {$e_1$};
    \draw (-0.4,0) ellipse (1.8 and 1.1);
    \node at (-0.4,0) {$\gamma$};
    \node[dot] (bL)  at (-2.2,0) {};          
    \node[dot] (cL)  at (1.4,0)   {};         
    \node[dot] (dL)  at (0.5,-0.9) {};        
    \node[dot] (vL)  at (0.1,1.05) {};        
    \draw (-3.5,0) -- (bL);
    \node[above] at (vL) {$v_*$};
    \draw (vL) -- (2.2,1.9);
    \node[above] at (1.1,1.5) {$e_4$};
    \draw (cL) -- (2.8,0.05);
    \node[above] at (1.9,0.08) {$e_2$};
    \draw (dL) -- (2.3,-1.4);
    \node[below] at (1.3,-1.5) {$e_3$};
  \end{tikzpicture}
   \qquad = \qquad
  \begin{tikzpicture}[thick,>=stealth, scale=0.5, baseline = 0.0em]
    \tikzset{
      dot/.style={circle,fill=black,inner sep=1.5pt}
    }
     \draw (-0.4,0) ellipse (1.8 and 1.1);
     \node at (-0.4,0) {$\gamma$};
     \node[dot] (bL)  at (-2.2,0) {};          
     \node[dot] (cL)  at (1.4,0)   {};         
     \node[dot] (dL)  at (0.5,-0.9) {};        
     \node[dot] (vL)  at (0.1,1.05) {};        
     \node[above left] at (vL) {$v_*$};
     \coordinate (p1) at (1.7,2.15);
     \coordinate (p2) at (1.7,2.80);
     \coordinate (p3) at (1.7,1.62);
     \coordinate (p4) at (1.7,1.05);
     \draw (vL) -- (p1);
     \draw (vL) -- (p2);
     \draw (vL) -- (p3);
     \draw (vL) -- (p4);
     \node[right] at (p1) {$\scriptsize e_2$};
     \node[right] at (p2) {$\scriptsize e_1$};
     \node[right] at (p3) {$\scriptsize e_3$};
     \node[right] at (p4) {$\scriptsize e_4$};
    \end{tikzpicture}
\end{equation*}
However, a rigorous definition of these extraction maps requires some care, because we must systematically fix   the vertex $v_{\star} \in
V(\gamma)$ to which we rewire the edges.

For a fixed, connected, Feynman diagram $\Gamma = (V, E)$, we introduce the set $
\mD_{ \Gamma} $ of connected Feynman
diagrams $ \tilde{\Gamma}= ( V, \tilde{E})$ over the same vertex set of
$\Gamma$ such that there exists a bijection
\begin{equation}\label{e:tau-def}
  \tau_{\tilde{\Gamma}, \Gamma}: \tilde{E} \to E \;,
\end{equation}
such that $ \tau_{\tilde{\Gamma}, \Gamma} ( \tilde{E}_{L}) = E_{L} $, i.e. legs are
mapped to legs.
For brevity we may write $\tau$ instead of $\tau_{\tilde{\Gamma}, \Gamma}$. In
words $\tilde{\Gamma}$ is a Feynman diagram over the same set of vertices as
$\Gamma$, with the same number of edges.
We equip the inner vertices $ V_{\star} $ of both $ \Gamma$ and $
\tilde{\Gamma}$ with a
total ordering, by numbering them from $1$ to $| V_{\star}| $. The choice of the
ordering is irrelevant, and we can fix for instance a labelling similar to the
one from Remark~\ref{rem:labels} (in that case we are labelling also leaf
nodes). Given a subdiagram $\gamma \subseteq \Gamma$ we will then write
$v_\star(\gamma) \in V(\gamma)$ for the maximal vertex in $\gamma$ with respect
to this ordering.

Given a connected, full subdiagram $ \gamma$ of $ \Gamma$,
we then define the extraction map $ \hat{\mC}_{\gamma}: \mD_{\Gamma} \to
\mD_{\Gamma}$ as follows. For any $e= (e_-, e_+) \in \tilde{E}$ (the edge set of some
$\tilde{\Gamma} \in \mD_\Gamma$) we set
\begin{equation} \label{e:def-hat-C}
  \hat{\mC}_{\gamma} e = \begin{cases}
    (v_{\star}(\gamma), e_+)\,, & \text{if } e_- \in V_{\star}(\gamma) \text{ and } e_+ \notin V (\gamma) \,, \\
    (e_-, v_{\star}(\gamma))\,, & \text{if } e_+ \in V_{\star}(\gamma) \text{ and } e_- \notin V (\gamma) \,, \\
    e\,, & \text{else}\,.
  \end{cases}
\end{equation}
And in this way we define
\begin{equation*}
  \hat{\mC}_{\gamma} \tilde{\Gamma} = (V, \{ \hat{\mC}_{\gamma} e \,: \, e \in \tilde{E} \})\,.
\end{equation*}
Note that since $\gamma$ was connected, then also $\hat{\mC}_{\gamma}
\tilde{\Gamma}$ is connected, so that indeed $\hat{\mC}_{\gamma} \tilde{\Gamma}
\in \mD_{\Gamma}$.
Given a zero degree forest, we can then iteratively extract one divergence
after the other using the extraction map $ \hat{\mC}$, leading to the operator
\begin{equation}\label{eq_forest_extr}
\begin{aligned}
\mathfrak{K}_{\mf{F}} \Gamma := \prod_{\gamma \in \mf{F}} \hat{\mC}_{\gamma} \Gamma
\,.
\end{aligned}
\end{equation}
The order in which we perform the contractions does not matter:
\begin{lemma}
  For any
  Feynman diagram $\Gamma$, any two zero degree subdiagrams $\gamma_1,
  \gamma_2 \subseteq \Gamma$ such that $\gamma_1 \subseteq \gamma_2$ or
  $\gamma_1, \gamma_2$ are vertex disjoint, and any $\tilde{\Gamma} \in \mD_{\Gamma}$ it holds that $\hat{\mC}_{\gamma_1}
  \hat{\mC}_{\gamma_2} \tilde{\Gamma} = \hat{\mC}_{\gamma_2}
  \hat{\mC}_{\gamma_1} \tilde{\Gamma} $.
\end{lemma}
\begin{proof}
The case
$\gamma_1$ and $\gamma_2$ vertex disjoint is immediate.
{On the other hand, if $ \gamma_{1}\subset \gamma_{2}$ then} the extraction map acts only on the
following three sets of edges $E_{i} \subseteq \tilde{E}$: 
\begin{equation*}
  \begin{aligned}
    E_1 & = \{ (e_-, e_+) \text{ s.t. } e_\pm \in V(\gamma_1), e_{\mp} \in V(\gamma_2) \setminus V(\gamma_1) \} \;, \\
    E_2 & = \{ (e_-, e_+) \text{ s.t. } e_\pm \in V(\gamma_2), e_{\mp} \in V  \setminus
    V(\gamma_2) \} \;.
  \end{aligned}
\end{equation*}
Because $v_\star(\gamma_1) { \leqslant } v_\star(\gamma_2)$ in the ordering of the vertices, we find
\begin{equation*}
  \begin{aligned}
    \hat{\mC}_{\gamma_1} \hat{\mC}_{\gamma_2} e & = \hat{\mC}_{\gamma_1} e = \hat{\mC}_{\gamma_2} \hat{\mC}_{\gamma_1}  e \;, \qquad \forall e \in E_1 \;, \\
    \hat{\mC}_{\gamma_1} \hat{\mC}_{\gamma_2} e & = \hat{\mC}_{\gamma_2} e = \hat{\mC}_{\gamma_2} \hat{\mC}_{\gamma_1}  e \;, \qquad \forall e \in E_2 \;.
  \end{aligned}
\end{equation*}
This concludes the proof.
\end{proof}

Before we proceed, let us make some comments to clarify this definition.
\begin{remark}
  The construction of $\hat{\mC}_\gamma$ is different from the
  contraction-extraction appearing in th BPHZ renormalisation through the map
  $\mC_\gamma$ in~\eqref{e_def_contrextr}. In the latter, we contract the
  subdiagram completely and pull it out of the original diagram, while here we
  rewire the vertices to one selected vertex. Once we integrate over all
  variables, there is no difference between the two operations, because we are
  in a spatially homogeneous setting, see also~\cite{BPHZ}. However, the map $\hat{\mC}_\gamma$ is more
  convenient in the technical analysis that we are going to perform.
\end{remark}

  To further clarify the setting, let us consider the diagram that we used as an
  example in the proof of Lemma~\ref{lem_bijection}, which for simplicity we
  have taken without legs and edge orientation to keep the picture simple:
  
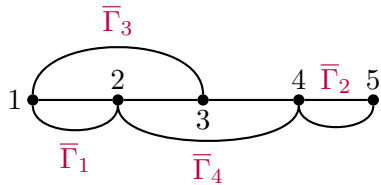
\begin{equation*}
    \begin{tikzpicture}[thick, scale = 0.7]
      \tikzset{
        vtx/.style={circle,fill=black,inner sep=2pt}
      }
      \foreach \i/\x in {1/0,2/1.6,3/3.2,4/5.0,5/6.4}{
        \node[dot] (v\i) at (\x,0) {};
      }
      \node[left] at (v1) {$1$};
      \node[above] at (v2) {$2$};
      \node[below] at (v3) {$3$};
      \node[above] at (v4) {$4$};
      \node[above] at (v5) {$5$};
      \draw (v1) -- (v2) -- (v3) -- (v4) -- node[midway,above=-1.5pt,purple] {$\overline{\Gamma}_2$} (v5);
      \draw (v1) to[out=-90, in =-90] node[midway,below=0.5pt,purple] {$\overline{\Gamma}_1$} (v2);
      \draw (v4) to[out=-90, in =-90] (v5);
      \draw (v1) .. controls +(0,1.3) and +(0,1.3) .. node[midway,above=0.5pt,purple] {$\overline{\Gamma}_3$} (v3);
      \draw (v2) .. controls +(0,-1.0) and +(0,-1.0) .. node[midway,below=0.5pt,purple] {$\overline{\Gamma}_4$} (v4);
      \end{tikzpicture}
  \end{equation*}
  Again, the $\overline{\Gamma}_i$ are suggestive of the bubbles that appear
  iteratively as we remove one after the other (the actual bubble
  $\overline{\Gamma}_i$ appear only after some contractions).
  In this case, the forest induced by the Hepp tree~\eqref{e:hepp-tree-example}
  is given by
  \begin{equation*}
    \mf{F}_{T} = \{ \Gamma_{u_i}, \; \colon \; i \in \{1,2,3,4\} \} \;.
  \end{equation*}
  We then find that
  \begin{equation*}
    \hat{\mC}_{\Gamma_{u_1}} \Gamma =
    \begin{tikzpicture}[thick, scale = 0.7, baseline = -0.2em]
      \tikzset{
        vtx/.style={circle,fill=black,inner sep=2pt}
      }
      \foreach \i/\x in {1/0,2/1.6,3/3.2,4/5.0,5/6.4}{
        \node[dot] (v\i) at (\x,0) {};
      }
      \node[above] at (v1) {$1$};
      \node[above left] at (v2) {$2$};
      \node[below] at (v3) {$3$};
      \node[above] at (v4) {$4$};
      \node[above] at (v5) {$5$};
      \draw (v1) -- (v2) -- (v3) -- (v4) -- node[midway,above=-1.5pt,red] {} (v5);
      \draw (v1) to[out=-90, in =-90] node[midway,below=0.5pt,red] {} (v2);
      \draw (v4) to[out=-90, in =-90] (v5);
      \draw (v2) .. controls +(0,1.3) and +(0,1.3) .. node[midway,above=0.5pt,red] {} (v3);
      \draw (v2) .. controls +(0,-1.0) and +(0,-1.0) .. node[midway,below=0.5pt,red] {} (v4);
      \end{tikzpicture}\,,
  \end{equation*}
  and
  \begin{equation*}
    \hat{\mC}_{\Gamma_{u_2}}\hat{\mC}_{\Gamma_{u_1}} \Gamma =
    \begin{tikzpicture}[thick, scale = 0.7, baseline = -0.2em]
      \tikzset{
        vtx/.style={circle,fill=black,inner sep=2pt}
      }
      \foreach \i/\x in {1/0,2/1.6,3/3.2,4/5.0,5/6.4}{
        \node[dot] (v\i) at (\x,0) {};
      }
      \node[above] at (v1) {$1$};
      \node[above left] at (v2) {$2$};
      \node[below] at (v3) {$3$};
      \node[above] at (v4) {$5$};
      \node[above] at (v5) {$4$};
      \draw (v1) -- (v2) -- (v3) -- (v4) -- node[midway,above=-1.5pt,red] {} (v5);
      \draw (v1) to[out=-90, in =-90] node[midway,below=0.5pt,red] {} (v2);
      \draw (v4) to[out=-90, in =-90] (v5);
      \draw (v2) .. controls +(0,1.3) and +(0,1.3) .. node[midway,above=0.5pt,red] {} (v3);
      \draw (v2) .. controls +(0,-1.0) and +(0,-1.0) .. node[midway,below=0.5pt,red] {} (v4);
      \end{tikzpicture}\,,
  \end{equation*}
  where we observe that the only difference to the diagram above is the flipping
  of the vertices $4$ and $5$. Next,
  \begin{equation*}
    \hat{\mC}_{\Gamma_{u_3}}\hat{\mC}_{\Gamma_{u_2}}\hat{\mC}_{\Gamma_{u_1}} \Gamma=
    \begin{tikzpicture}[thick, scale = 0.7, baseline = -0.2em]
      \tikzset{
        vtx/.style={circle,fill=black,inner sep=2pt}
      }
      \foreach \i/\x in {1/0,2/1.6,3/3.2,4/5.0,5/6.4}{
        \node[dot] (v\i) at (\x,0) {};
      }
      \node[above] at (v1) {$1$};
      \node[above left] at (v2) {$2$};
      \node[below left] at (v3) {$3$};
      \node[above] at (v4) {$5$};
      \node[above] at (v5) {$4$};
      \draw (v1) -- (v2) -- (v3) -- (v4) -- node[midway,above=-1.5pt,red] { } (v5);
      \draw (v1) to[out=-90, in =-90] node[midway,below=0.5pt,red] { } (v2);
      \draw (v4) to[out=-90, in =-90] (v5);
      \draw (v2) .. controls +(0,1.3) and +(0,1.3) .. node[midway,above=0.5pt,red] { } (v3);
      \draw (v3) .. controls +(0,-1.0) and +(0,-1.0) .. node[midway,below=0.5pt,red] { } (v4);
      \end{tikzpicture}\,.
  \end{equation*}
  Finally, $\hat{\mC}_{\mf{F}_T} \Gamma = \hat{\mC}_{\Gamma_{u_3}}\hat{\mC}_{\Gamma_{u_2}}\hat{\mC}_{\Gamma_{u_1}} \Gamma$.
  With these tools at hand, we are ready to work towards the proof of Proposition~\ref{prop_bubbleDiag_exact}.


\subsection{Proof of Proposition~\ref{prop_bubbleDiag_exact}}\label{sec:proof_prop_bubbleDiag_exact}

The proof of Proposition~\ref{prop_bubbleDiag_exact} follows from a
sequence of extractions and rewirings such as the one in~\eqref{eq_forest_extr}. At the
end of this, we will be able to provide the proof of the proposition.
Recall that our objective is to study the weak coupling limit of the following integral:
\begin{equation}\label{eq_Gamma_test}
\begin{aligned}
\Pi_{\varepsilon} \Gamma ( \varphi)
=
\int_{ (\mathbf{T}^{d})^{V_{\star}} }
\Psi_{\varepsilon}( x_{V_{\star}} )
\mW_{\ve} \Gamma (x_{V_{\star}})
 \ud
x_{V_{\star}} \;,
\end{aligned}
\end{equation}
where we defined
\begin{equation}\label{eq_def_Psi}
\begin{aligned}
\Psi_{\varepsilon} ( x_{V_{\star}}) =
\int_{(\mathbf{T}^{d})^{L}}  \varphi (x_{L})
\prod_{e \in E_L}
 G_{\varepsilon} (x_{e_+} - x_{e_-}) \ud x_{L} \;.
\end{aligned}
\end{equation}
To simplify matters, we will observe that we can replace $\Psi_\ve(x_{V_\star})$ with its limit for
$\ve \to 0$
\begin{equation}\label{eq_def_tf_loc}
\Psi ( x_{v_{\star}}) := \lim_{ \varepsilon \to 0} \Psi_{\varepsilon} ( x_{v_{\star}}, \ldots, x_{v_{\star}}) \;,
\end{equation}
where all the variables are evaluated at the same point $x_{v_\star}$, see
Lemma~\ref{lem_testfunc_loc} below. Here $v_\star \in
V_\star$ could in principle be chosen arbitrarily, but for concreteness we set to be
$v_\star= v_\star(\Gamma)$, the maximal element
of $V_\star$ following the order in the discussion below~\eqref{e:tau-def}.

In addition, we will decompose the
integral \eqref{eq_Gamma_test} into the sum over
  \emph{disjoint} Hepp sectors, see
  Lemma~\ref{lem_disjoin_hepp}, and reduce ourselves to spatial scales larger than
  {$2^{-N_\ve} \simeq \varepsilon$}, where we recall from the
  previous section that 
  \begin{equation}\label{e:Ne}
    N_\ve = \lfloor \log_2  \tfrac{1}{\ve} \rfloor \;.
  \end{equation}
  For any fixed Feynman diagram $\Gamma$ and $T \in \mT_{V_\star}$, we hence define the following subset:
  \begin{equation}\label{e:ANE}
  \mA^{\#, N_\ve}(T) :=
   \{ \mathbf{n} \in \mA (T) \,: \, \forall u \neq u ' \
  \text{ } \ \mathbf{n}(u) \neq \mathbf{n}(u') \ \text{ and }\ \mathbf{n} (u)
  \leqslant N_{\varepsilon} -1 \}\,,
  \end{equation}
  the set $\mA(T)$
  of compatible scalings.
  In this setting, we find the following result.

\begin{lemma}\label{lem_exact_heppsplit}
  Let $ \Gamma$ be a nested bubble diagram, then
  \begin{equation*}
  \begin{aligned}
  \lim_{ \varepsilon \to 0}  |\lambda_{\varepsilon}|^{|E_{\star}|} \bigg\vert
  \Pi_{\varepsilon} \Gamma ( \varphi)
  -
  \sum_{T \in \mT_{V_{\star}}}
  \sum_{\substack{\mathbf{n} \in \mA^{\#, N_\ve}	(T)}}
  \int_{D_{(T, \mathbf{n})}}
  \Psi ( x_{v_\star})
  \mW_{\ve} \Gamma (x_{V_{\star}})
   \ud
  x_{V_{\star}} \bigg\vert = 0 \,,
  \end{aligned}
  \end{equation*}
  {where $ \Psi( x_{v_\star })$ was defined in \eqref{eq_def_tf_loc}.}
  \end{lemma}
  The proof of this result is postponed to
  Section~\ref{sec:test-functions}. Instead, we pass to the heart of our argument first. The following lemma proves that inner edges can be rewired in such a
  way that bubbles can be integrated out individually.

\begin{lemma}\label{lem_bubbleTree_exact}
Let $ \Gamma$ be a nested bubble diagram and $T \in \mT_{V_{\star}}^{0}(\Gamma)$
as in~\eqref{e:def_contr_hepp}. Then
\begin{equation*}
\begin{aligned}
&\lim_{\varepsilon \to 0}
|\lambda_{\varepsilon}|^{|E_{\star} |}
\sum_{\substack{\mathbf{n} \in \mA^{\#, N_\ve} (T)}}
\bigg|
 \int_{D_{ \bf T}}
\Psi ( x_{v_*})
\mW_{\ve} \Gamma (x_{V_{\star}})
 \ud
x_{V_{\star}}
-
\int_{D_{\bf T}}
\Psi ( x_{v_\star})
\mW_{\ve} \mathfrak{K}_{\mf{F}_{T}} \Gamma (x_{V_{\star}})
\ud x_{V_{\star}}
\bigg| = 0 \,.
\end{aligned}
\end{equation*}
\end{lemma}

\begin{proof}
Let us consider a heap ordering $ \gamma_{ 1}, \ldots,
\gamma_{|\trim|}$ of the zero degree forest $\mf{F}_{T}$ from Lemma~\ref{lem_forestT}.
That is, any ordering such that $ \gamma_{i} \subset \gamma_{j} $ implies $ i \geqslant 
j$. Moreover, let us denote with $u_i \in \trim$ the node such that $\gamma_i =
\Gamma_{u_i}$, as defined in~\eqref{e:Gu}.
We will perform a telescopic sum argument, using the identity
\begin{equation*}
\begin{aligned}
\mW_{\ve} \Gamma - \mW_{\ve} \mathfrak{K}_{\mf{F}_{T}} \Gamma
= \sum_{i =1}^{|\trim|} \mW_{\ve}  \mathfrak{K}_{\mf{F}_{i-1}} \Gamma - \mW_{\ve}
\hat{\mC}_{ \gamma_{i}} \mathfrak{K}_{\mf{F}_{i-1}} \Gamma \,,
\end{aligned}
\end{equation*}
where $ \mf{F}_{i} := \{ \gamma_{j}\, : \, j \leqslant i  \}$, i.e. we first
contract the ``larger''  divergences (in the sense of inclusions) before contracting the ``smaller'' ones. Hence, it suffices to show for $ \overline{\Gamma}=
\mathfrak{K}_{\mf{F}_{i -1}} \Gamma =(V, \overline{E})$
\begin{equation}\label{e_bubbleTree_exact_onestep}
\begin{aligned}
\lim_{ \varepsilon \to 0 }
| \lambda_{\varepsilon}|^{|E_{\star} |}
\sum_{\substack{\mathbf{n} \in \mA^{\#, N_\ve}(T)}}
\bigg|
 \int_{D_{ \bf T}}
\Psi ( x_{v_\star})
\Big(
\mW_{\ve}  \overline{\Gamma} (x_{V_{\star}})
-
\mW_{\ve} \hat{\mC}_{ \gamma_{i}} \overline{\Gamma} (x_{V_{\star}})\Big)
\ud x_{V_{\star}}
\bigg|
=0\,.
\end{aligned}
\end{equation}
First, we observe that if $T \in \mT_{V_\star}^0 (\Gamma)$, then also $T \in \mT_{V_\star}^0(\mathfrak{K}_{\mf{F}_{i}}
\Gamma)$ for all $i=1,\ldots, |\trim|$ instead of $ \Gamma$.
In fact, if for any $u \in \trim$ and $i \leq |\trim|$, we define $\Gamma_u^i$
to be the full subdiagram of $\mathfrak{K}_{\mf{F}_{i}}
\Gamma$  (with $\mathfrak{K}_{\mf{F}_{0}}
\Gamma=\Gamma$) spanned by the vertices $\{ v \in V_{\star}\,: \, v \succ u  \} $, then
the number of edges $|E(\Gamma_u^{i})| = |E(\Gamma_u^{0})|$ does not depend on
$i$.\

Next, let $ \partial \overline{E}(\gamma_i) \subset \overline{E}_{\star}$ denote
the set of edges of $\overline{\Gamma}$ that contain
exactly one vertex in $ V ( \gamma_i )$:
\begin{equation*}
  \partial \overline{E}(\gamma_i) = \{ (e_-, e_+) \in \overline{E}_\star \;
\colon \; e_{\pm} \in V(\gamma_i) \;, e_{\mp} \not\in V(\gamma_i) \} \;,
\end{equation*}
{we may assume that $\partial \overline{E}(\gamma_i) \neq
\emptyset$, as otherwise \eqref{e_bubbleTree_exact_onestep} vanishes.}

Then from~\eqref{e:def-hat-C}, we obtain
\begin{equation}\label{eq_supp3_extract}
\begin{aligned}
\big|\mW_{\ve}  \overline{\Gamma} & (x_{V_{\star}})
- \mW_{\ve} \hat{\mC}_{ \gamma_{i}} \overline{\Gamma} (x_{V_{\star}})\big|\\
&= \prod_{e \in \overline{E}_{\star} \setminus \partial \overline{E} ( \gamma_{i}) } G_{ \varepsilon}( x_{e_{+}}- x_{e_{-}})
\bigg| \prod_{\partial \overline{E}( \gamma_{i})} G_{ \varepsilon}( x_{e_{+}}- x_{e_{-}})
-
\prod_{\partial \overline{E}( \gamma_{i})}  G_{ \varepsilon}( x_{v_{\star} (
\gamma_{i})}- x_{\mf{s}(\gamma_i, e)})\bigg|\,,
\end{aligned}
\end{equation}
where $\mf{s}(\gamma_i, e) \in \{e_+, e_-\}$ is the vertex of the edge $e$ such
that $\mf{s}(\gamma_i, e) \not\in V(\gamma_i)$.
Moreover, with another telescopic sum we estimate
\begin{equation}\label{eq_supp2_extract}
\begin{aligned}
\bigg| \prod_{e \in \partial \overline{E}( \gamma_{i})} & G_{ \varepsilon}( x_{e_{+}}- x_{e_{-}})
-
\prod_{e \in \partial \overline{E}( \gamma_{i})}  G_{ \varepsilon}( x_{v_{\star} (
\gamma_{i})}- x_{\mf{s}(\gamma_i, e_j)})\bigg|\\
\leqslant &
\sum_{j =1}^{m}
\Big|
 G_{ \varepsilon}( x_{e_{j,+}}- x_{e_{j,-}})
-
 G_{ \varepsilon}( x_{v_{\star} (\gamma_{i})}- x_{\mf{s}(\gamma_i, e_j)})
\Big| \\
& \qquad \qquad
\prod_{\ell = 1}^{j -1}
 G_{ \varepsilon}( x_{v_{\star} (\gamma_{i})}- x_{\mf{s}(\gamma_i, e_\ell)})
\prod_{\ell = j+1}^{m}
G_{\varepsilon}( x_{e_{\ell,+}}- x_{e_{\ell,-}})
\,,
\end{aligned}
\end{equation}
where $ e_{1}, \ldots, e_{m}$ is an arbitrary enumeration of the edges in $
\partial \overline{E} ( \gamma_{i}) $.

Now, we fix $x \in D(T, \mathbf{n})$, so that via the Taylor expansion in
Lemma~\ref{lem_taylor_new} we obtain
\begin{equation*}
\begin{aligned}
&\Big|
 G_{ \varepsilon}( x_{e_{j,+}}- x_{e_{j,-}})
-
 G_{ \varepsilon}( x_{v_{\star} (\gamma_{i})}- x_{\mf{s}(\gamma_i, e_j)})
\Big|
\lesssim
\frac{2^{- \mathbf{n} ( v_{\star} ( \gamma_{i}) \wedge \mf{r}(\gamma_i, e_j) )}}{
(2^{- \mathbf{n} ( e_{j,+} \wedge e_{j, -})} + \varepsilon)}
( |x_{e_{j, +}}- x_{e_{j , -}}| + \varepsilon)^{-2}
\,,
\end{aligned}
\end{equation*}
where $ \mf{r}(\gamma_i, e_j) \in V( \gamma_{i})$ denotes the vertex of $ e_{j}$
that lies in $ \gamma_{i}$ (or similarly $\{\mf{s}(\gamma_i, e_j),
\mf{r}(\gamma_i, e_j)\} = \{e_+, e_- \}$).
In words, we made a small scale improvement within $ \gamma_{i}$ of order $\mathbf{n} (
v_{\star} ( \gamma_{i}) \wedge {\mf{r}(\gamma_i, e_j)} )$ by replacing the Green's function
at $ e_{j}$ (which is of order $\mathfrak{t}(e_{j})=-2$) with a kernel of order
$\mathfrak{t}(e_{j})=-3$.
Note that in order to apply Lemma~\ref{lem_taylor_new}, we must check the
  condition
  \begin{equation*}
    |x_{v_{\star} (\gamma_{i})}- x_{\mf{r}(\gamma_i, e_j)}| < |x_{\mf{r}(\gamma_i, e_j)}- x_{\mf{s}(\gamma_i, e_j)}|, |x_{v_{\star} (\gamma_{i})}- x_{\mf{s}(\gamma_i, e_j)}| \;,
  \end{equation*}
  which is verified by our assumption that $\mathbf{n} \in \mA^{\#, N_\ve}$, so
  that $\mathbf{n} (\mf{r}(\gamma_i, e_j) \wedge \mf{s}(\gamma_i, e_j))+1
  \leqslant \mathbf{n} (v_{\star} (\gamma_{i}) \wedge \mf{s}(\gamma_i, e_j))$.
Moreover, $ \mathbf{n} ( v_{\star} ( \gamma_{i}) \wedge  \mf{r}(\gamma_i, e_j)  ) \geqslant
\mathbf{n} ( u_i )$ and in addition if $w_i$ is the parent of $u_i {= v_{\star} (
\gamma_{i}) }$ in $T$, then
$\mathbf{n}(u_i \wedge \mf{s}(\gamma_i, e_j)) \leq \mathbf{n}(w_i)$.
Above, we also used that
\begin{equation*}
\begin{aligned}
v_{\star} (\gamma_{i}) \wedge  \mf{s}(\gamma_i, e_\ell)
= e_{\ell , +}\wedge e_{\ell,-}\,, \qquad \forall \ell =1, \ldots, m\,.
\end{aligned}
\end{equation*}

Hence, together with~\eqref{eq_supp2_extract} we have derived the upper bound
\begin{equation*}
\begin{aligned}
\bigg| \prod_{e \in \partial \overline{E}( \gamma_{i})} & G_{ \varepsilon}( x_{e_{+}}- x_{e_{-}})
-
\prod_{e \in \partial \overline{E}( \gamma_{i})}  G_{ \varepsilon}( x_{v_{\star} (
\gamma_{i})}- x_{w_{e}})\bigg| \lesssim
\frac{2^{- \mathbf{n} ( u_i )}}{
(2^{- \mathbf{n} ( w_i)} + \varepsilon)}
\prod_{\ell = 1}^{m}
\big(2^{- \mathbf{n} (e_{\ell , +}\wedge e_{\ell,-})}+ \varepsilon\big)^{-2}\,.
\end{aligned}
\end{equation*}
{Notice that $ u_{i}$ always has a parent $ w_{i} \in \trim$, i.e. $ u_{i}$ is not
the root, provided $ \emptyset \neq \partial \overline{E} ( \gamma_{i}) \subset \overline{E}_{\star} $.}

Finally, in combination with~\eqref{eq_supp3_extract}, we then have the bound
\begin{equation*}
\begin{aligned}
\big|\mW_{\ve}  \overline{\Gamma} (x_{V_{\star}})
-
\mW_{\ve} \hat{\mC}_{ \gamma_{i}} \overline{\Gamma} (x_{V_{\star}})\big|
\lesssim
\frac{2^{- \mathbf{n} ( u_i )}}{
(2^{- \mathbf{n} ( w_i )} + \varepsilon)}
\prod_{e \in \overline{E}_{\star}}
\big(2^{- \mathbf{n} (e_{ +}\wedge e_{-})}+ \varepsilon\big)^{-2}
\,,
\end{aligned}
\end{equation*}
which in turn yields, together with $ \mathrm{Leb} ( D_{\bf T} ) \lesssim
\prod_{u \in \trim} 2^{- d \mathbf{n} (u)}$,
\begin{equation*}
\begin{aligned}
&|\lambda_{\varepsilon}|^{| E_{\star}|}
\sum_{\substack{\mathbf{n} \in \mA^{\#, N_\ve}(T)}}
\bigg|
 \int_{D_{ \bf T}}
\Psi ( x_{v_\star})
\Big(
\mW_{\ve}  \overline{\Gamma} (x_{V_{\star}})
-
\mW_{\ve} \hat{\mC}_{ \gamma_{i}} \overline{\Gamma} (x_{V_{\star}})\Big)
\ud x_{V_{\star}}
\bigg| \\
&\lesssim
|\lambda_{\varepsilon}|^{| E_{\star}|}
\sum_{\substack{\mathbf{n} \in\mA^{\#, N_\ve}(T)}}
\frac{2^{- \mathbf{n} ( u_i )}}{
(2^{- \mathbf{n} ( w_i )} + \varepsilon)}
\prod_{ u \in \trim} \big( 2^{- \mathbf{n} ( u )} + \varepsilon \big)^{ \eta(u)-d}
2^{-d \mathbf{n} ( u)} \\
& \lesssim
|\lambda_{\varepsilon}|^{| E_{\star}|}
\sum_{\substack{\mathbf{n} \in\mA^{\#, N_\ve}(T)}}
\prod_{ u \in \trim} \big( 2^{- \mathbf{n} ( u )} + \varepsilon \big)^{ \eta^{i}(u)-d}
2^{-d \mathbf{n} ( u)}\,,
\end{aligned}
\end{equation*}
where $\eta$ is as in~\eqref{eq_def_eta} and
\begin{equation*}
  \eta^i(u) = \begin{cases}
    \eta(u) + 1\,, & \text{if } u = u_i \;,\\
    \eta(u) - 1\,, & \text{if } u = w_i \;,\\
    \eta(u)\,, & \text{else}\;.
  \end{cases}
\end{equation*}
Because by assumption $ \overline{\deg} \equiv 0$, and thus $ \eta\equiv 0 $  on $ \overline{T}$, the right side
vanishes by Lemma~\ref{lem_estimate_nonneg_Hepp} due to the gain of $+1$ at $ u_i$.
This concludes the proof of~\eqref{e_bubbleTree_exact_onestep} and therefore of
this result.
\end{proof}

\begin{lemma}\label{lem_nbd_singleTree}
Let $ \Gamma$ be a nested bubble diagram and $ T \in \mT_{V_{\star}}^{0}(\Gamma)$, then
\begin{equation*}
\begin{aligned}
\lim_{ \varepsilon \to 0}
\lambda_{\varepsilon}^{|E_{\star} |}
\sum_{\substack{\mathbf{n} \in \mA^{\#, N_\ve} (T)}}
\int_{D_{\bf T}}
\Psi ( x_{v_*})
\mW_{\ve} \mathfrak{K}_{\mf{F}_{T}} \Gamma (x_{V_{\star}})
\ud x_{V_{\star}}
=
\frac{1}{  \overline{T} !}
\bigg(
\frac{ \hat{\lambda}^{2}}{ 8 \pi^{2}} \bigg)^{| \overline{T}|}
\<cc1_narrow> ( \varphi ) \,.
\end{aligned}
\end{equation*}
\end{lemma}

\begin{proof}
We prove the result by induction over the size of $ |
\overline{T}| = | V_{\star}| -1$. More precisely, we will prove by induction that for every
nested bubble diagram $\Gamma$ and contributing Hepp tree $T
 \in \mT^0_{V_\star}(\Gamma)$
\begin{equation}\label{eq_ind_hyp_treefac}
\begin{aligned}
\lim_{ \varepsilon \to 0}
|\lambda_{\varepsilon}|^{|E_{\star} |}
\sup_{\substack{x_{v_\star} \in \TT^4 \\ K \in [ N_{\varepsilon}-1]}}
 \bigg|
\sum_{\substack{\mathbf{n} \in \mA^{\#, N_\ve}_{K} (T)}}
\int_{D_{\bf T}}
\mW_{\ve} \mathfrak{K}_{\mf{F}_{T}} \Gamma (x_{V_{\star}})
\ud x_{V_{\star} \setminus \{v_\star\}}
{-}
\frac{1}{  \overline{T} !}
\bigg(
\frac{  \log{2^{N_\ve-K}}}{ 8 \pi^{2}} \bigg)^{| \overline{T}|}
\bigg| =0 \;,
\end{aligned}
\end{equation}
with $[n]= \{0, \ldots,
n\}$, and where we have defined, similarly to~\eqref{e:def-AK}:
\begin{equation*}
  \mA^{\#, N_\ve}_{K} (T) :=
   \{ \mathbf{n} \in \mA^{\#, N_\ve} (T) \,: \, \mathbf{n} (u) \geqslant K \;, \forall u \in \trim \} \;.
\end{equation*}

For $ | \overline{T}| = 1$, we write $ V_{\star} =\{v, v_\star\} $ for the two inner
vertices. Then, using Lemma~\ref{lem_disjoin_hepp}~1. with $d=4$ we see that
\begin{equation}\label{e:compute-it}
\begin{aligned}
\lambda_{\varepsilon}^{2}
\sum_{\mathbf{n} (\bullet) =K }^{ N_{\varepsilon}-1}
\int_{D_{\bf T}}
G_{\varepsilon}^{2} ( x_{v_\star} - x_{v})
\ud x_{v} & =
\lambda_{\varepsilon}^{2}
\int_{ ( 2 \pi) 2^{- N_{\varepsilon}} <|x_{v_\star} - x_{v}| \leqslant ( 2
\pi) 2^{-K} }
G_{\varepsilon}^{2} ( x_{v_\star} - x_{v})
\ud x_{v} \\
& = \lambda_{\varepsilon}^{2}
\int_{ ( 2 \pi) 2^{- N_{\varepsilon}} <|x| \leqslant ( 2
\pi) 2^{-K} }
G_{\varepsilon}^{2} ( x)
\ud x  \;.
\end{aligned}
\end{equation}
Note that the integration variable $x$ satisfies $|x| > (2 \pi) 2^{-N_\ve} > 2
\ve$. Therefore, we are able to apply
\eqref{e:G-bd2} from Lemma~\ref{lem_greens_function_order_bound} in order to
compute the asymptotic behaviour of~\eqref{e:compute-it}.
We find the following upper bound for some $C >0 $ that is allowed to
change from line to line:
\begin{equation*}
  \begin{aligned}
  \lambda_{\varepsilon}^{2}
  \int_{ ( 2 \pi) 2^{- N_{\varepsilon}} <|x| \leq( 2
  \pi) 2^{-K} } G_{\varepsilon}^{2} ( x)
  \ud x & \leq \lambda_\ve ^2 \int_{( 2 \pi) 2^{- N_{\varepsilon}}}^{(2
  \pi) 2^{-K}} \frac{ 2 \pi^2 \zeta^3}{ (4 \pi^2)^2 ( \zeta - \ve)^4} \ud \zeta + C \lambda_\ve^2\\
  & \leq \lambda_{\ve}^2
  \int_{( 2 \pi) 2^{- N_{\varepsilon}}}^{(2
  \pi) 2^{-K}} \frac{ 2 \pi^2 \zeta^3}{(4 \pi^2)^2 \zeta^4} \ud \zeta + C
\lambda_\ve^2 \\
    &= \frac{1}{8 \pi^2} \lambda_\ve^2 \left( \log( 2^{N_{\varepsilon}} ) -  \log(2^{K} )
  \right)+ C \lambda_\ve^2 \;,
  \end{aligned}
\end{equation*}
In the second inequality, we used that from the definition~\eqref{e:Ne} of $N_\ve$
\begin{equation*}
  (2 \pi -1)\ve \leq ( 2 \pi) 2^{- N_{\varepsilon}}-\ve \leq (4 \pi -1 ) \ve \;,
\end{equation*}
  so that
  \begin{equation*}
    \bigg\vert \frac{ 2 \pi^2 \zeta^3}{ (4 \pi^2)^2 ( \zeta - \ve)^4} - \frac{ 2 \pi^2 \zeta^3}{ (4 \pi^2)^2 \zeta^4}\bigg \vert \lesssim \frac{ \ve (\zeta-\ve)^3}{ (\zeta-\ve)^5} = \frac{ \ve }{ (\zeta-\ve)^2} \;,
  \end{equation*}
  which implies
  \begin{equation*}
    \int_{( 2 \pi) 2^{- N_{\varepsilon}}}^{(2
  \pi) 2^{-K}} \frac{ \ve }{ (\zeta-\ve)^2} \ud \zeta \lesssim 1 \;.
  \end{equation*}
  {This error term was then absorbed into the constant $ C$.} 

We can produce a matching lower bound, again by employing~\eqref{e:G-bd1}:
  \begin{equation*}
    \begin{aligned}
   \lambda_{\varepsilon}^{2}
  \int_{ ( 2 \pi) 2^{- N_{\varepsilon}} <|x| \leq ( 2
  \pi) 2^{-K} }
  G_{\varepsilon}^{2} ( x)
  \ud x & \geq \lambda_{\varepsilon}^{2}
  \int_{ ( 2 \pi) 2^{- N_{\varepsilon}} <|x|\leq ( 2
  \pi) 2^{-K} } \frac{1}{(4 \pi)^2 (|x|+\ve)^4} \ud x - C \lambda_\ve^2 \\
  & = \frac{1}{8 \pi^2} \lambda_\ve^2 \left( \log( 2^{N_{\varepsilon}} ) -  \log(2^{K} ) \right) - C \lambda_\ve^2 \;.
    \end{aligned}
  \end{equation*}
Together, these estimates lead to
\begin{equation}\label{e:in-between}
\begin{aligned}
\lim_{\ve \to 0}
\sup_{ K \in [N_\ve -1 ]}
\, \lambda_{\varepsilon}^{2}
\bigg|\int_{ ( 2 \pi) 2^{- N_{\varepsilon}} <|x| \leqslant ( 2
\pi) 2^{-K} }
G_{\varepsilon}^{2} ( x)
\ud x
-
\frac{1}{8 \pi^2} \log( 2^{N_{\varepsilon} - K } )
\bigg|= 0
 \,,
\end{aligned}
\end{equation}
so that in this case the induction hypothesis~\eqref{eq_ind_hyp_treefac} holds.


Now, let us assume~\eqref{eq_ind_hyp_treefac} holds for all nested bubble diagrams $ \Gamma' = ( V ',
E')$ and $ T'
\in \mT_{V_{\star}'}^0 (\Gamma') $
such that $| \overline{T}'| \leqslant  n$, for some $n \in \mathbf{N}$, and let $ \Gamma $ be a nested bubble diagram and $ T \in \mT_{V_{\star}}^0(\Gamma) $, with $
| \overline{T}| = | V_{\star} | -1 =n+1$.
Let us denote $T_1, T_2$ the two subtrees such that $T=[T_1, T_2]_{\bullet}$,
where $\bullet$ denotes the root node of $T$, and let $u_i$ be the root of $T_i$ and
$\Gamma_i = \Gamma_{u_i}$, and denote the vertex set of $\Gamma_i$ by $V_i$. Note that $ \Gamma_{1}$ and $
\Gamma_{2}$ are connected by exactly two edges.
Now, we find that
\begin{equation*}
  \mf{K}_{\mf{F}_T} \Gamma =  \hat{\mC}_{\Gamma_{\bullet}}\mf{K}_{\mf{F}_{T_1}} \mf{K}_{\mf{F}_{T_2}}  \Gamma  \;,
\end{equation*}
since $\mf{F}_{T_2}$ and $\mf{F}_{T_2}$ are disjoint (hence the order
in which we apply them does not matter). Here $\Gamma_{\bullet}$ denotes
the subdiagram spanned by all the inner vertices (that is, the original diagram
without the leg vertices and edges). In particular, we note that
$\hat{\mC}_{\Gamma_{\bullet}}$ rewires only legs, all the inner
edges remain fixed. Therefore, if we denote with $v_{\star,i}$
the maximal inner vertex in $V_i$ (see the construction in
Section~\ref{sec:zero-forests}), then we can represent
\begin{equation*}
  \mW_\ve \hat{\mC}_{\Gamma_{\bullet}} \mf{K}_{\mf{F}_{T_1}} \mf{K}_{\mf{F}_{T_2}}
  \Gamma (x_{V}) = \mW_\ve \mf{K}_{\mf{F}_{T_1}} \Gamma_1  (x_{V_1})  \mW_\ve
  \mf{K}_{\mf{F}_{T_2}}  \Gamma_2 (x_{V_2}) G_{\ve}^{2}( x_{v_{\star,1}}- x_{v_{\star,2}}) \;.
\end{equation*}
Furthermore, assume without loss of generality that $v_{\star, 2} \leqslant 
v_{\star,1}$ (in
the ordering of the inner vertices), so that in particular $v_{\star,1} = v_\star$.
This allows us to write
\begin{equation}\label{eq_supp1_treefac}
\begin{aligned}
\sum_{\substack{\mathbf{n} \in \mA_{ K}^{\#, N_\ve} (T)}} &
\int_{D_{\bf T}}
\mW_{\ve} \mathfrak{K}_{\mf{F}_{T}} \Gamma (x_{V_{\star}})
\ud x_{V_{\star} \setminus \{v_\star\} }
\\
& = \sum_{ n=K}^{N_{\varepsilon} -1}
\int_{2^{-n-1} < \frac{|x|}{2\pi} \leq 2^{-n} } I_{1,\ve} (n, x_{v_{\star}})
I_{2,\ve}(n, x_{v_{\star}} -x) G_{\ve}^{2}(x) \ud x  \,,
\end{aligned}
\end{equation}
where we defined
\begin{equation*}
  I_{i, \ve} (n, x_{v_{\star,i}}) := \sum_{\substack{\mathbf{n}_{i} \in \mA_{n}^{\#, N_\ve} (T_{i})}} \int_{D_{(T_i, \mathbf{n}_i)}}
  \mW_{\ve} \mathfrak{K}_{\mf{F}_{T_{i}}} \Gamma_{i} (x_{V_i})
  \ud x_{V_{i} \setminus \{v_{\star,i}\} }\;.
\end{equation*}
Moreover, we set
\begin{equation*}
  A_{i}(n, \ve) := \frac{1}{  \overline{T}_i !}
  \bigg(
  \frac{  \log{2^{N_\ve-n}}}{ 8 \pi^{2}} \bigg)^{| \overline{T}_i|} \;, \qquad
E_{i}(n, \ve) := \sup_{x_{v_{\star,i}} \in \TT^4} |I_{i, \ve} (n, x_{v_{\star,i}}) - A_{i}(n, \ve)| \;,
\end{equation*}
such that we find
\begin{equation*}
  \begin{aligned}
    \bigg\vert & \sum_{\substack{\mathbf{n} \in \mA_{ K}^{\#, N_\ve} (T)}}
\int_{D_{\bf T}}
\mW_{\ve} \mathfrak{K}_{\mf{F}_{T}} \Gamma (x_{V_{\star}})
\ud x_{V_{\star} \setminus \{v_\star\} } -  \sum_{ n=K}^{N_{\varepsilon} -1} A_1(n, \ve) A_{2}(n, \ve)
\int_{2^{-n-1} < \frac{|x|}{2\pi} \leq 2^{-n} }  G_{\ve}^{2}(x) \ud x \bigg\vert \\
& \leqslant  \sum_{ n=K}^{N_{\varepsilon} -1}(A_1(n, \ve)E_{2}(n, \ve)+ A_2(n, \ve)E_{1}(n, \ve) +E_1(n, \ve)E_{2}(n, \ve))
\int_{2^{-n-1} < \frac{|x|}{2\pi} \leq 2^{-n} }  G_{\ve}^{2}(x) \ud x \;.
  \end{aligned}
\end{equation*}
Now, by the induction hypothesis we know that
\begin{equation*}
  \lim_{\ve \to 0} \lambda_{\ve}^{|E_i|} \max_{n \leq N_\ve-1} E_{i}(n, \ve) =0 \;.
\end{equation*}
This, in combination with the estimate~\eqref{e:in-between}, implies the result, as
long as we can prove the correct asymptotic behaviour for
\begin{equation*}
   \sum_{ n=K}^{N_{\varepsilon} -1} A_1(n, \ve) A_{2}(n, \ve)
\int_{2^{-n-1} < \frac{|x|}{2\pi} \leq 2^{-n} }  G_{\ve}^{2}(x) \ud x \;.
\end{equation*}
Here, by Lemma~\ref{lem_fixingestimate} we find that
\begin{equation*}
  \begin{aligned}
    \lim_{\ve \to 0} & \max_{K \in [N_\ve-1]}\lambda_{\ve}^{|E_1|+|E_2|+2} \sum_{ n=K}^{N_{\varepsilon} -1} A_1(n, \ve)
    A_{2}(n, \ve) \bigg\vert  \int_{2^{-n-1} < \frac{|x|}{2\pi} \leq 2^{-n} }  G_{\ve}^{2}(x) \ud x -  \frac{1}{ 8 \pi^{2}}
    \log{ (2)} \bigg\vert \\
    & \lesssim \lim_{\ve \to 0} \max_{K \in [N_\ve-1]} \lambda_{\ve}^{|E_1|+|E_2|+2}\sum_{ n=K}^{N_{\varepsilon} -1} A_1(n, \ve)
    A_{2}(n, \ve) ( 2^{-n} + 2^{- ( N_{\ve} - n -1 )}) \\
    & \lesssim  \lim_{\ve \to 0} \lambda_{\ve}^{|E_1|+|E_2|+2} N_{\ve}^{| E_{1}| +| E_{2}|}
    \sum_{ n=0}^{N_{\varepsilon} -1} ( 2^{-n} + 2^{- ( N_{\ve} - n -1)} ) = 0 \;.
  \end{aligned}
\end{equation*}
Therefore, the problem is further reduced to the limiting behaviour of
\begin{equation*}
 \lambda_{\ve}^{| E_{1}| + | E_{2}| +2} \frac{ \log{ (2)} }{ 8 \pi^{2}}
  \sum_{ n=K}^{N_{\varepsilon} -1} A_1(n, \ve) A_{2}(n, \ve)  \;,
\end{equation*}
which is controlled by Lemma~\ref{lem_logsum}, leading to the desired~\eqref{eq_ind_hyp_treefac}. The statement of the lemma now follows by setting $
K =0$ in~\eqref{eq_ind_hyp_treefac}.
\end{proof}

\begin{lemma}\label{lem_fixingestimate}
  There exists a $C>0$ such that for every $\ve \in (0,1)$ and every $ n \in\{ 0,
  \ldots, N_{\ve}-1 \}$:
  \begin{equation*}
  \begin{aligned}
  \bigg|
  \int_{ ( 2 \pi) 2^{- (n +1)} < | x |
  \leqslant ( 2 \pi) 2^{- n}}
  \,
  G_{ \varepsilon}^{2}  (x)
  \ud x
  -
  \frac{1}{ 8 \pi^{2}}
  \log{ (2)
  }
  \bigg|
  \leqslant C ( 2^{-n} +  2^{-( N_{\ve}- n-1)})\,.
  \end{aligned}
  \end{equation*}
  \end{lemma}
  \begin{proof}
  First, we note that
  \begin{equation}\label{e:remove-epsilon}
  \begin{aligned}
  \frac{1}{ 8 \pi^{2}}
  \log{ (2) }
  =
  \frac{1}{16 \pi^{4}}
  \int_{ ( 2 \pi) 2^{- (n +1)} < | x |
  \leqslant ( 2 \pi) 2^{- n}}
  \frac{1}{|x|^{4}}
  \ud x\,.
  \end{aligned}
  \end{equation}
  Therefore, it suffices to estimate for $n \geq N_\ve-1$
  \begin{equation*}
  \begin{aligned}
  &\bigg|
  \int_{ ( 2 \pi) 2^{- (n +1)} < | x |
  \leqslant ( 2 \pi) 2^{- n}}
  \,
  G_{ \varepsilon}^{2}  (x)
  -
  \frac{1}{16 \pi^{4}}
  \frac{1}{|x|^{4}}
  \ud x
  \bigg|\\
  & =
  \bigg|
  \int_{ ( 2 \pi) 2^{- (n +1)} < | x |
  \leqslant ( 2 \pi) 2^{- n}}
  \bigg(
  G_{ \varepsilon}  (x)
  -
  \frac{1}{4 \pi^{2}}
  \frac{1}{|x|^{2}}
  \bigg)
  \bigg(
  G_{ \varepsilon}  (x)
  +
  \frac{1}{4 \pi^{2}}
  \frac{1}{|x|^{2}}
  \bigg)
  \ud x
  \bigg|\\
  & \lesssim
  \int_{ ( 2 \pi) 2^{- (n +1)} < | x |
  \leqslant ( 2 \pi) 2^{- n}}
  \bigg(
  1 + \log{ ( 1+ \tfrac{1}{|x|}) } +
  \frac{ \ve}{(|x| + \ve)^{3}}
  \bigg)
  \frac{1}{(|x| + \ve)^{2}}
  \ud x
  \,,
  \end{aligned}
  \end{equation*}
  where we used Lemmas~\ref{lem:G} and~\ref{lem_Gmol2ratiomol} in the last
inequality, in combination with the fact that $|x|> 2 \ve$.
  In particular, we obtain
  \begin{equation*}
  \begin{aligned}
  \bigg|
  \int_{ ( 2 \pi) 2^{- (n +1)} < | x |
  \leqslant ( 2 \pi) 2^{- n}}
  \, &
  G_{ \varepsilon}^{2}  (x)
  -
  \frac{1}{16 \pi^{4}}
  \frac{1}{|x|^{4}}
  \ud x
  \bigg|\\
  &\lesssim
  \int_{ 0 < | x |
  \leqslant ( 2 \pi) 2^{- n}}
  \frac{1}{|x|^{3}}
  \ud x
  +
  \int_{ ( 2 \pi) 2^{- (n +1)} < | x |
  < \infty}
  \frac{ \ve}{(|x| + \ve)^{5}}
  \ud x\,.
  \end{aligned}
  \end{equation*}
  The first integral gives a factor proportional to $ 2^{- n}$, and the second
  integral yields
  \begin{equation*}
  \begin{aligned}
  \int_{ ( 2 \pi) 2^{- (n +1)} < | x |
  < \infty}
  \frac{ \ve}{(|x| + \ve)^{5}}
  \ud x
  &\lesssim
  \frac{ \ve}{2^{-(n+1)}+ \ve}
  \lesssim
  \ve 2^{n+1} = 2^{- (N_{\ve} -(n+1)) }\,.
  \end{aligned}
  \end{equation*}
  This concludes the proof.
  \end{proof}

\begin{lemma}\label{lem_logsum}
For every $k \in \mathbf{N}$ it holds that
\begin{equation*}
\begin{aligned}
\lim_{ \varepsilon \to 0}
\max_{ K \in \{0, \ldots, N_{\varepsilon} - 1 \}}
\lambda_{\varepsilon}^{2(k+1)}
\bigg|
\sum_{ n = K}^{N_{\varepsilon}-1} (N_{\varepsilon} -n )^{k}
-
\frac{ 1 }{k+1}
( N_{\varepsilon} - K )^{k +1}
\bigg|=0\,.
\end{aligned}
\end{equation*}
\end{lemma}

\begin{proof}
The sum can be bounded with  Faulhaber's
formula:
\begin{equation*}
\begin{aligned}
\lambda_{\varepsilon}^{2(k+1)}
\bigg|\sum_{ n = 1}^{N_{\varepsilon}-K}
n^{k}
-
\frac{1}{k +1}
( N_{\varepsilon}-K )^{k+1}\bigg|
&\leqslant
\frac{
\lambda_{\varepsilon}^{2(k+1)}
}{k+1}
\sum_{\ell =1}^{k}
\binom{ k+1}{ \ell}
|B_{\ell} |
( N_{\varepsilon}-K-1 )^{k+1 - \ell} \\
&\lesssim
\lambda_{\varepsilon}^{2(k+1)} N_{\varepsilon}^{k} \to 0 \,,
\end{aligned}
\end{equation*}
with $ (B_{k})_{k}$ the sequence of Bernoulli numbers.
\end{proof}

At this point we are ready to conclude the proof of the desired proposition.

\begin{proof}[Proof of Proposition~\ref{prop_bubbleDiag_exact}]
  Let $ \Gamma$ be a connected nested bubble diagram.
  We use Lemma~\ref{lem_exact_heppsplit} to decompose the valuation of the diagram into
  valuations across Hepp sectors, and 
  Lemma~\ref{lem_estimate_nonneg_Hepp} to reduce the sum to contributing Hepp trees in
  $\mT^{0}_{V_\star}$. 
  Next, extracting bubbles using Lemma~\ref{lem_bubbleTree_exact}, we have
  \begin{equation*}
  \begin{aligned}
  \lim_{ \varepsilon \to 0} \bigg\vert \lambda_{\varepsilon}^{| E_{\star}|}
  \Pi_{\varepsilon} \Gamma ( \varphi) -
  \lambda_{\varepsilon}^{| E_{\star}|}
  \sum_{T \in \mT^{0}_{V_{\star}}}
    \sum_{\substack{\mathbf{n} \in \mA^{\#, N_\ve}	(T)}}
    \int_{D_{\bf T}}
  \Psi ( x_{v_\star})
  \mW_{\ve} \mathfrak{K}_{\mf{F}_{T}} \Gamma (x_{V_{\star}})
  \ud x_{V_{\star}} \bigg\vert = 0 \,.
  \end{aligned}
  \end{equation*}
  Finally, applying Lemma~\ref{lem_nbd_singleTree} we obtain the desired result.
  \end{proof}

\subsection{Test function localisation}\label{sec:test-functions}


In the proof of Lemma~\ref{lem_exact_heppsplit}, we will use that it is sufficient to evaluate the test
function $\Psi$ at a single point  $\Psi(x_{v_\star})$. We prove this ingredient,
before the proof of the aforementioned lemma.

\begin{lemma}\label{lem_testfunc_loc}
Let $ \Gamma$ be a nested bubble diagram, then
\begin{equation}\label{e:lem-simp1}
\begin{aligned}
\lim_{ \varepsilon \to 0}
|\lambda_{\varepsilon}|^{|E_{\star}|} \bigg\vert
\Pi_{\varepsilon} \Gamma ( \varphi)
-
\int_{(\mathbf{T}^{d})^{V_{\star}}}
\Psi ( x_{v_\star})
\mW_{\ve} \Gamma (x_{V_{\star}})
 \ud
x_{V_{\star}}\bigg\vert = 0 \,.
\end{aligned}
\end{equation}
Here $ \Psi( x_{v_\star })$ denotes $ \lim_{ \varepsilon \to 0} \Psi_{\varepsilon} ( x_{V_{\star}})$ when every occurrence
of $ x_{v}$, $ v \in V_{\star} $, is replaced by $ x_{v_\star}  $.
\end{lemma}

\begin{proof}
By Taylor expansion we find
\begin{equation*}
\begin{aligned}
\big| \Psi_{\varepsilon} (x_{V_{\star}}) - \Psi_{\varepsilon}(x_{v_\star})
\big|
\leqslant \| \Psi_{\varepsilon}' \|_{\infty}
\sum_{v \in V_{\star}} \sum_{i = 1}^{d}
|x_{v}^{i} - x_{v_\star}^{i}|
\lesssim
\sum_{v \in V_{\star}} | x_{v}- x_{v_\star}|\,,
\end{aligned}
\end{equation*}
where
\begin{equation*}
\begin{aligned}
\Psi_{\varepsilon}' (x_{V_{\star}}) =
\sum_{v \in V_{\star}} \sum_{i = 1}^{d}
\partial_{x_{v}^{i}} \Psi_{\varepsilon}( x_{V_{\star}})\,.
\end{aligned}
\end{equation*}
Therefore
\begin{equation}\label{eq_supp1_locatestfunc}
\begin{aligned}
\lambda_{\varepsilon}^{|E_{\star}|}
\bigg|
\Pi_{\varepsilon} \Gamma ( \varphi)
- \int_{(\mathbf{T}^{d})^{V_{\star}}}
\Psi_{\varepsilon} ( x_{v_\star}) &
\mW_{\ve} \Gamma (x_{V_{\star}})
 \ud
x_{V_{\star}}
\bigg| \\
& \lesssim
|\lambda_{\varepsilon}|^{|E_{\star}|}
\sum_{v \in V_{\star}}
\int_{(\mathbf{T}^{d})^{V_{\star}}}
 | x_{v}- x_{v_\star}|\cdot
|\mW_{\ve} \Gamma (x_{V_{\star}})|
 \ud
x_{V_{\star}} \;.
\end{aligned}
\end{equation}
In other words, we have added a new kernel (of positive order), which is
represented by
 a new edge $e$ connecting $ v $ and $ v _\star $ in $ \Gamma$, with $
\mathfrak{t} ( e ) =1$. The resulting
diagram $ \overline{\Gamma}$ has positive degree, because
\begin{equation*}
\begin{aligned}
\deg\, \overline{\Gamma} = \deg\, \Gamma + \mathfrak{t} (e) \geqslant 1\,,
\end{aligned}
\end{equation*}
and $ \Gamma$ was assumed to be non--negative.
In particular, using  Proposition~\ref{prop_countlogs_nonneg}
\begin{equation*}
\begin{aligned}
 |\lambda_{\varepsilon}|^{| E_{\star}|}| \Pi_{\varepsilon}
\overline{\Gamma} ( \varphi) |
\lesssim
\sum_{T \in \mT_{V_{\star}}}
|\lambda_{\varepsilon}|^{2 |  \trim  |}
\big(\log{ \tfrac{1}{ \varepsilon}} \big)^{ \mathrm{Null}( T )} \to 0 \,,
\qquad \text{as }\ \varepsilon \to 0\,,
\end{aligned}
\end{equation*}
where we used that $\mathrm{Null}( T ) \leqslant |\trim|-1$ which holds
because the root $\bullet $ of $ T$ satisfies $ \overline{\deg}
(\bullet) = \sum_{u \succ \bullet} \eta(u) =
\deg\, \overline{\Gamma} >0$.

Finally, note that $ \Psi_{\varepsilon}$ can be replaced by $ \Psi $ in the
left side of~\eqref{eq_supp1_locatestfunc}, since
\begin{equation*}
\begin{aligned}
\int_{ \mathbf{T}^{d}} \rho_{\varepsilon}(z) \big( \varphi \star_{i} G (x-
\delta_{i} z) - \varphi \star_{i} G ( x)\big) \ud z \to 0 \,,
\end{aligned}
\end{equation*}
because the integrand is supported on $ |z| \leqslant \varepsilon$ and $ \varphi \star_{i} G$ is uniformly continuous.
Here $ \star_{i}$ denotes the convolution with the $i$-th component of $
\varphi$.
\end{proof}

We are now ready to complete the proof of Lemma~\ref{lem_exact_heppsplit}.

\begin{proof}[Proof of Lemma~\ref{lem_exact_heppsplit}]
First, since
\begin{equation*}
\begin{aligned}
(\mathbf{T}^{d})^{V_{\star}}_{\#}
:=
\{
x \in (\mathbf{T}^{d})^{V_{\star}}\,: \,
x_{v} \neq x_{w} \,, \ \forall v \neq w
\}\,,
\end{aligned}
\end{equation*}
has full Lebesgue measure, we have by Lemma~\ref{lem_testfunc_loc}
\begin{equation*}
\begin{aligned}
\lim_{ \varepsilon \to 0}
|\lambda_{\varepsilon}|^{|E_{\star}|} \bigg\vert
\Pi_{\varepsilon} \Gamma ( \varphi)
-
\int_{(\mathbf{T}^{d})^{V_{\star}}_{\#}
}
\Psi ( x_{v_\star})
\mW_{\ve} \Gamma (x_{V_{\star}})
 \ud
x_{V_{\star}}\bigg\vert = 0 \,.
\end{aligned}
\end{equation*}
To conclude the lemma, it suffices to show for every $T \in
\mT_{V_{\star}}$
\begin{equation}\label{eq_estdiagonalrest}
\begin{aligned}
\lim_{ \varepsilon \to 0}
|\lambda_{\varepsilon}|^{|E_{\star}|}
\sum_{\substack{\mathbf{n} \in \mA ( T) \setminus \mA^{\#, N_\ve}(T) }}
\int_{D_{(T, \mathbf{n})}}
\big|\Psi ( x_{v_\star})
\mW_{\ve} \Gamma (x_{V_{\star}})\big|
 \ud
x_{V_{\star}} =0\,,
\end{aligned}
\end{equation}
because $ \{ (T, \mathbf{n}) \,: \, T \in \mT_{V_{\star}} \ \text{ and }
\mathbf{n} \in \mA (T)\}$ covers all of $(\mathbf{T}^{d})^{V_{\star}}_{\#}
$, see Lemma~\ref{lem_hepp_almost_covers}, while restricting to $ \mathbf{n}
\in \mA^{\#, N_\ve}(T)$ provides a disjoint family of subsets by
Lemma~\ref{lem_disjoin_hepp}.

To prove~\eqref{eq_estdiagonalrest}, first, we note that if $ \mathrm{Null}( T)  <
|T|$, then the integral in~\eqref{eq_estdiagonalrest} vanishes as a consequence of
Lemma~\ref{lem_estimate_nonneg_Hepp}. Henceforward, we assume $ T \in
\mT_{ V_{\star}}^0(\Gamma)$.
Following the derivation of the upper bound~\eqref{eq_pos_Hepp_est_supp1}, we
write
\begin{equation}
\begin{aligned}
\sum_{\substack{\mathbf{n} \in  \mA ( T) \setminus \mA^{\#, N_\ve}(T)}}
\int_{D_{(T, \mathbf{n})}}
\big|\Psi ( x_{v_\star})
\mW_{\ve} \Gamma (x_{V_{\star}})\big|
 \ud
x_{V_{\star}}
\lesssim
\sum_{\substack{\mathbf{n} \in \mA ( T) \setminus \mA^{\#, N_\ve}(T)}}
\prod_{u \in \trim}
(2^{- \mathbf{n}(u)} + \varepsilon)^{ \eta(u) -d}
2^{-d \, \mathbf{n}(u)}\,,
\end{aligned}
\end{equation}
with $ \eta$ defined in~\eqref{eq_def_eta}.
We now follow the proof of Lemma~\ref{lem_estimate_nonneg_Hepp} verbatim, and
see that the restriction to at least one diagonal $ \mathbf{n} (u) =
\mathbf{n} (u')$ or $ \mathbf{n} (u) \geqslant N_{\varepsilon}$ (cf.~\eqref{eq_supp2_aboveN}) results in the bound
\begin{equation*}
\begin{aligned}
|\lambda_{\varepsilon}|^{|E_{\star}|}
\sum_{\substack{\mathbf{n} \in  \mA ( T) \setminus \mA^{\#, N_\ve}(T)}}
\int_{D_{(T, \mathbf{n})}}
\big|\Psi ( x_{V_{\star}})
\mW_{\ve} \Gamma (x_{V_{\star}})\big|
 \ud
x_{V_{\star}}
\lesssim
\lambda_{\varepsilon}^{2| \overline{T}|}
\big( \log{ \tfrac{1}{\varepsilon}} \big)^{| \overline{T}|-1}\,,
\end{aligned}
\end{equation*}
which vanishes as $ \varepsilon \to 0$. The proof is complete.
\end{proof}

\section{Computing the effective variance coefficient}\label{sec:effective-variance}
In this short section, we conclude our analysis of the limiting behaviour of
Feynman diagrams in weak coupling, by providing an exact formula for the limiting effective variance.
We prove Proposition~\ref{prop:effective-variance} using an inductive
argument, building on
  Proposition~\ref{prop:bubble}, which in turn relies on
  Proposition~\ref{prop_bubbleDiag_exact} and Corollary~\ref{cor_contr_diagram}.

\begin{proof}[Proof of Proposition~\ref{prop:effective-variance}]
  First, we observe that if $n_1+n_2$ is odd, then there is no complete
  pairing between $\mI_{n_1}$ and $\mI_{n_2}$, meaning that $\mK(\mI_{n_1},
  \mI_{n_2}) = \emptyset$ and therefore $\sigma_{\rm{eff}}^{2,(n_1 + n_2)} = 0$ by Definition~\ref{def:sigma-n}.

  Let us therefore concentrate on the case $n_1+n_2=2n$ even. Recall that from Definition~\ref{def:sigma-n} and from
  Proposition~\ref{prop:bubble}, we obtain the following identity for $n \in \NN_*$:
  \begin{equation*}
    \sigma_{\rm{eff}}^{2,(2n)} = \sum_{n_1 + n_2 = 2n} \sum_{\kappa \in \mK^{\mathrm{cont}}(\mI_{n_1}, \mI_{n_2})} \Big(\frac{ 1}{ \sqrt{8} \pi}\Big)^{2(n -1)}
    \frac{|\mathbb{S}( (\mI_{n_1}, \mI_{n_2})_\kappa )|}{ (n -1 )! } \;,
  \end{equation*}
  where we have used that the total number of inner vertices in $(\mI_{n_1},
  \mI_{n_2})_\kappa $ is given by $n$. Therefore, our claim is proven if we show
  that
  \begin{equation*}
    \sum_{n_1 + n_2 = 2n} \sum_{\kappa \in \mK^{\mathrm{cont}}(\mI_{n_1}, \mI_{n_2})}
    \frac{|\mathbb{S}( (\mI_{n_1}, \mI_{n_2})_\kappa )|}{ (n -1 )! } = 4^{n-1} \;.
  \end{equation*}
  Note that the claim is true for $n=1$, since $|\mathbb{S}( (\mI_{1},
  \mI_{1})_\kappa )|=1$ from Definition~\ref{def:extraction-sequences}. The
  claim is also true for $n=2$. We check this by hand, since in this case the only contributions come from
  $n_1=1, n_2 =3$, from $n_1=3,n_2=1$, and from $n_1=n_2=2$. In this
  case a quick calculation shows that the only contributing pairings are the
  following:
  \begin{equation*}
    \<var1> \;, \qquad \<var2> \;, \qquad  \<c22> \;, \qquad \<c23> \;,
  \end{equation*}
  implying the desired identity: In each pairing $|\mathbb{S}( (\mI_{n_1},
  \mI_{n_2})_\kappa )|=1$, since there is only one bubble to extract.

  We will prove the claim for arbitrary $n \geqslant 3$ by
  induction. To do so,
  we must keep track of all the information of a given pairing $(\mI_{n_1},
  \mI_{n_2})_\kappa$: not just the resulting graph, but also the original
  trees and the pairing $\kappa$. We do this by labelling all the edges in the left tree
  $\mI_{n_1}$, including the legs, from $1_{0}$ to $1_{n_1}$, from bottom to
  top (following the direction of the edges). And similarly all
  the edges in the right tree $\mI_{n_2}$ from $2_{0}$ to $2_{n_2}$. For example, if $n_1 =
  n_2 = 4$, we find the following diagram with edge labelling after a given pairing:
  \begin{equation*}
    \<varc42> = \<varcc42> \;.
  \end{equation*}
  Note that because of the labelling of the edges we can omit directing the
  edges, since the direction flows along the labelling from $\alpha_i$ to
  $\alpha_{i+1}$ (for $\alpha \in \{1,2\}$).
  Moreover, from the labelling of the edges in the graph $(\mI_{n_1},
  \mI_{n_2})_\kappa$ it is immediately possible to recover the original trees
  $\mI_{n_1}, \mI_{n_2}$ together with the pairing $\kappa$, simply by
  following the labels $(\alpha_i)_{i=0}^{n_\alpha}$ and  $(\beta_j)_{j=0}^{n_\beta}$ from smallest to the largest.

  In the edge-labelled diagrams, consider a subdiagram $\overline{\Gamma}$ that
  is isomorphic as a graph (omitting the edge labelling) to a bubble with four
  legs attached $\scalebox{0.5}{\<varB5>}$. Then, as a graph with edge labelling,
  $\overline{\Gamma}$ is isomorphic to one of the following $4$ subdiagrams
  (here an isomorphism is a map that sends edges and vertices to edges and
  vertices, and additionally maps edge labels to edge labels).
  Namely, there exist $\{\alpha, \beta\} \subseteq \{1,2\}$ and $i,j \in \NN$
  such that
  \begin{equation}\label{e:cond-ab}
    \text{either } \{\alpha, \beta\} = \{1,2\}  \text{ and }   i, j \in \NN \;, \quad \text{ or } \alpha = \beta \text{ and }  i < j \;,
  \end{equation}
  and such that $\overline{\Gamma}$ is isomorphic to
  \begin{equation}\label{e:cont-cases}
    \<varB1> \;, \qquad \<varB2> \;, \qquad \<varB3> \;, \qquad \<varB4>
     \;.
  \end{equation}
  Now, we can extend the contraction of a bubble from
  Definition~\ref{def:contract} to take into account the labelings, in such a
  way that after the contraction the resulting diagram is still the pairing of two
  smaller trees. Namely, we contract
  \begin{equation}\label{e:cont-case-1}
    \begin{aligned}
      & \<varB1> \mapsto  \<varC> \;.
    \end{aligned}
  \end{equation}
  Here the map $\pi$ is simply a rewarping of the labels: since we have
  contracted some edges, the indices associated to the remaining labels form an increasing sequence
  but with some gaps, and the map $\pi$ shifts the indices back so that
  they are given by a sequence of consecutive integers, for example:
  \begin{equation*}
    \pi (\alpha_0, \alpha_1, \alpha_4, \alpha_5)=  (\pi \alpha_0, \pi \alpha_1, \pi \alpha_4, \pi \alpha_5) = (\alpha_0, \alpha_1, \alpha_2, \alpha_3) \;.
  \end{equation*}
  Note that in~\eqref{e:cont-case-1} it is possible that $\alpha=\beta$ or
  $\alpha \neq \beta$. Depending on the case, the map $\pi$ has a different effect,
  and for clarity we provide a formula (we will omit this in the upcoming cases,
  since it is not important). If $\alpha = \beta$ we find that \lqm warped\rqm labelings $\pi \alpha_k, \pi
  \alpha^c_k$ (with $\{\alpha, \alpha^c\}=\{1,2\}$) are given by
  \begin{equation}\label{e:el11}
    \pi \alpha_k = \begin{cases}
      \alpha_k & \quad \forall k \leqslant j \;, \\
      \alpha_{k-2} & \quad \forall  j+3 \leq k \leq n_{\alpha} \;,
    \end{cases}  \quad \pi \alpha^c_k = \alpha^c_k \;, \quad \forall k \leq n_{\alpha^c} \;,
    \quad \text{ and } \begin{cases}
      \tilde{n}_\alpha & = n_\alpha- 2 \;, \\
      \tilde{n}_{\alpha^c} & = n_{\alpha^c}\;.
    \end{cases}
  \end{equation}
  Instead, if $\alpha \neq \beta$, then
  \begin{equation}\label{e:el12}
    \pi \alpha_k = \alpha_k \quad \forall k \leq n_\alpha, \quad \pi \beta_k = \begin{cases}
    \beta_k & \quad \forall k \leq j \;, \\
    \beta_{k-2} & \quad \forall  j+3 \leq k \leq n_\beta \;,
    \end{cases}
    \quad \text{ and } \begin{cases}
      \tilde{n}_\alpha & = n_\alpha \;, \\
      \tilde{n}_{\beta} & = n_{\beta}- 2\;.
    \end{cases}
  \end{equation}
  The only relevant aspect of this rewarping is that labels are again of the form $\pi \alpha = \{\alpha_k\}_{k
  \leq \tilde{n}_\alpha}$ for $\alpha \in \{1,2\}$ and $\tilde{n}_\alpha$ as in~\eqref{e:el11} and~\eqref{e:el12}. Also the exact form of $\tilde{n}_\alpha$
  is not important, except for the obvious observation that $\tilde{n}_1 +
  \tilde{n}_2 = 2n-2$.

  With the contraction~\eqref{e:cont-case-1} we mean that if $\Gamma^B$ is a subdiagram of $\Gamma = (\mI_{n_1}, \mI_{n_2})_{\kappa}$ isomorphic to
  a bubble (when considered without edge labelling) and connected to the remaining
  diagram as in the first case of~\eqref{e:cont-cases}, then we define the
  labelled contraction $\bigslant{\Gamma}{\Gamma^B}$ to be the diagram obtained by contracting the
  bubble following Definition~\ref{def:contract}, and with the edge labelling
  given by the rewarping map $\pi$ as in~\eqref{e:cont-case-1}.
  The result of such a contraction is still the pairing of two trees:
  $\bigslant{\Gamma}{\Gamma^B} = (\mI_{ \tilde{n}_1},
  \mI_{\tilde{n}_2})_{\tilde{\kappa}}$ for some pairing $\tilde{\kappa}$.

  The same applies to the other three cases of~\eqref{e:cont-cases}. In  the
  second case we find:
  \begin{equation}\label{e:cont-case-2}
    \<varB2> \mapsto  \<varC2> \;,
  \end{equation}
  And in the third case:
  \begin{equation}\label{e:cont-case-3}
    \<varB3> \mapsto  \<varC3> \;.
  \end{equation}
  And in the last case:
  \begin{equation}\label{e:cont-case-4}
    \<varB4> \mapsto  \<varC4> \;.
  \end{equation}

  This procedure can be inverted in the following sense. Given an edge-labelled Feynman diagram
 $\Gamma$ arising from a  pairing
  $\Gamma= (\mI_{\tilde{n}_1}, \mI_{\tilde{n}_2})_{\tilde{\kappa}}$, and given a vertex $v \in V_\star((\mI_{\tilde{n}_1},
  \mI_{\tilde{n}_2})_{\tilde{\kappa}})$, then there exist $\alpha, \beta, i, j$ satisfying~\eqref{e:cond-ab} such that the four edges incident in $v$ are of the
  following form:
  \begin{equation*}
    \<varC5> \;.
  \end{equation*}
  The existence of $\alpha, \beta, i, j$ follows from the fact that $\kappa$ is
  a complete pairing and either $v$ arises from an internal pairing (so
  $\alpha=\beta$, in which case we simply choose $i$ to be the smallest (with
  respect to the ordering of the edges) edge
  incident in $v$) or from an external one (so $\alpha \neq \beta$).

  Now, we can define $4$ larger diagrams (indexed by $\ell \in \{1, 2, 3,4\}$)
  $\Gamma^{v, \ell}$ in which we expand the vertex $v$ into a bubble,
  creating one of the four cases appearing in~\eqref{e:cont-cases}. Because of
  the labelling of the edges, these four diagrams are distinct (without edge
  labelling this would not necessarily be the case, because we would not be able
  to distinguish two edges incident in the same two vertices). To construct the
  diagram $\Gamma^{v, \ell}$ we leave the graph untouched, except for the vertex $v$ and the edges
  incident in $v$, and the overall labelling. In the first case $\ell=1$, we expand the
  vertex $v$ as follows:
  \begin{equation*}
    \<varC5> \mapsto \<varD1> \;.
  \end{equation*}
  Here we have colored the edges, in order to distinguish them, since in the process we are also modifying the labelling. Indeed,
  the labelling is chosen as the unique labelling through consecutive integer
  indices such that if $\Gamma^B\subseteq \Gamma^{v, 1}$ is the
  bubble with edges $\beta_{j+1}, \beta_{j+2}$, then, following the contraction~\eqref{e:cont-case-1}
  above we have $\Gamma = \bigslant{\Gamma^{v,1}}{\Gamma^B}$ as a labelled
  diagram (in this case it means that if an edge had label $\beta_k$ in the old
  diagram, then it has a label $\beta_{k+2}$ in the new one if $k \geq j+1$, and
  $\beta_k$ otherwise, and the labels $\beta^c_k$ are left untouched).

  We proceed similarly in the other cases. We list, respectively, the cases
  $\ell=2,3,4$ as follows:
  \begin{equation*}
    \<varC5> \mapsto \<varD2> \;, \quad
    \<varC5> \mapsto \<varD3> \;, \quad
    \<varC5> \mapsto \<varD4> \;.
  \end{equation*}
  Note that $j'$ can be different from $j$, depending on whether $\alpha =
  \beta$ of $\alpha \neq \beta$ (in the latter case $j' = j$).

  Combining everything we have proven so far, we find that for every $n\geq 3$,
  for every pairing
  $(\mI_{\tilde{n}_1}, \mI_{\tilde{n}_2})_{\tilde{\kappa}}$ with
$\tilde{n}_1+\tilde{n}_2=2n-2$, for every $v \in
  V_{\star}((\mI_{\tilde{n}_1}, \mI_{\tilde{n}_2})_\kappa)$, and for every $\ell
  \in \{1,2,3,4\}$ there exists a pairing
  $(\mI_{n_1}, \mI_{n_2})_{\kappa}$ with $n_1
  +n_2 = 2 n$ and a bubble subdiagram $\Gamma^{B, v, \ell} \subseteq (\mI_{n_1},
  \mI_{n_2})_{\kappa}$ such that $(\mI_{\tilde{n}_1},
  \mI_{\tilde{n}_2})_{\tilde{\kappa}} = \bigslant{(\mI_{n_1},
  \mI_{n_2})_{\kappa}}{\Gamma^{B, v, \ell}}$ as an edge-labelled diagram and such that
  $\Gamma^{B, v, \ell}$ contracts to $v$ following~\eqref{e:cont-case-1} if $\ell=1$,~\eqref{e:cont-case-2} if $\ell=2$,~\eqref{e:cont-case-3} if $\ell=3$, or~\eqref{e:cont-case-4} if $\ell=4$.

  This operation lifts to contributing diagrams with associated extraction
  sequences. For $\tilde{n}_1 + \tilde{n}_2 = 2 n- 2$, let $(\mI_{\tilde{n}_1}, \mI_{\tilde{n}_2})_{\tilde{\kappa}}$
  be the diagram associated to a contributing pairing $\tilde{\kappa}$ and
  $\tilde{\mE} = \{ \tilde{\Gamma}_i\}_{i=1}^{n-2}$ any bubble extraction sequence for $(\mI_{\tilde{n}_1}, \mI_{\tilde{n}_2})_{\tilde{\kappa}}$ in
  the sense of Definition~\ref{def:extraction-sequences}. Then for every $v \in
  V_\star ((\mI_{\tilde{n}_1}, \mI_{\tilde{n}_2})_{\tilde{\kappa}})$, and every
  $\ell \in \{1,2,3,4\}$ we find a new diagram (the same as described above) $(\mI_{n_1},
  \mI_{n_2})_{\kappa}$, with associated bubble extraction sequence $\mE^{v,
  \ell} = \{ \tilde{\Gamma}^{v, \ell}_i \}_{i=1}^{n-1}$ given by
  \begin{equation*}
    \tilde{\Gamma}^{v, \ell}_1 = \Gamma^{B, v, \ell} \;, \qquad  \tilde{\Gamma}^{v, \ell}_i =  \tilde{\Gamma}_{i-1} \cup  \Gamma^{B, v, \ell} \;,
  \end{equation*}
  with the natural interpretation that if $v \in V(\tilde{\Gamma}_i)$, then
  $\tilde{\Gamma}_{i} \cup  \Gamma^{B, v, \ell}$ is the image of
  $\tilde{\Gamma}_{i}$ in the graph $(\mI_{n_1},
  \mI_{n_2})_{\kappa}$ with the vertex $v$ replaced by the bubble diagram
  $\Gamma^{B, v, \ell}$, and the edges attached to $\Gamma^{B, v, \ell}$ (and
  the labels prescribed) according to the rule determined by $\ell$.

  Now, we observe that the mapping
  \begin{equation*}
    ((\mI_{\tilde{n}_1}, \mI_{\tilde{n}_2})_{\tilde{\kappa}}, v, \ell, \tilde{\mE}) \mapsto ((\mI_{n_1},
    \mI_{n_2})_{\kappa}, \mE^{v, \ell})
  \end{equation*}
  is a bijection. Its inverse is given by contracting the first bubble
  of the bubble extraction sequence of the larger diagram,
  {which also induces an extraction sequence on the smaller diagram by
  contracting the first bubble in every subdiagram in $ \mE^{v, \ell} $.} 
  Indeed, a contraction uniquely
  identifies the vertex $v$ (to which we have contracted the bubble), and the
  case $\ell$ (through the labelling of the larger diagram).

  Since there are $4$ choices of $\ell$ and $n-1$ choices for $v$, we have proven that for
  $n\geq 3$:
  \begin{equation*}
    \begin{aligned}
      4  (n-1) \!\!\! & \sum_{n_1 + n_2 = 2(n-1)} \sum_{\kappa \in \mK^{\mathrm{cont}}(\mI_{n_1}, \mI_{n_2})}
      |\mathbb{S}( (\mI_{n_1}, \mI_{n_2})_\kappa )| \\
      & =  \sum_{n_1 + n_2 = 2 n} \sum_{\kappa \in \mK^{\mathrm{cont}}(\mI_{n_1}, \mI_{n_2})}
      |\mathbb{S}( (\mI_{n_1}, \mI_{n_2})_\kappa )| \;.
    \end{aligned}
  \end{equation*}
  Therefore, we obtain by induction:
  \begin{equation*}
    \begin{aligned}
      \sum_{n_1 + n_2 = 2n} & \sum_{\kappa \in \mK^{\mathrm{cont}}(\mI_{n_1}, \mI_{n_2})}
      \frac{|\mathbb{S}( (\mI_{n_1}, \mI_{n_2})_\kappa )|}{ (n -1 )! } \\
      & =4\!\!\!  \sum_{n_1 + n_2 = 2(n-1)} \sum_{\kappa \in \mK^{\mathrm{cont}}(\mI_{n_1}, \mI_{n_2})}
      \frac{|\mathbb{S}( (\mI_{n_1}, \mI_{n_2})_\kappa )|}{ (n -2 )! } = 4^{n-1} \;,
    \end{aligned}
  \end{equation*}
  as desired. This concludes the proof.
\end{proof}

\section{Negative Anderson--Feynman diagrams}\label{sec_negative}

The goal of this
section is to prove Lemma~\ref{lem:negvanish}, which states that
\begin{equation*}
\begin{aligned}
\lim_{ \varepsilon \to 0}
\lambda_{\varepsilon}^{| E_{\star}|}
| \hat{\Pi}_{\varepsilon} \Gamma ( \varphi ) | =0\,,
\end{aligned}
\end{equation*}
for any negative
Anderson--Feynman diagram $ \Gamma$ with two or four legs. We recall that the renormalised valuation
$ \hat{\Pi}_{\varepsilon}$ is expressed in terms of Zimmermann's
forest formula~\eqref{eq_zimmermann}
\begin{equation}\label{eq_zimmermann_2}
\begin{aligned}
  \hat{\Pi}_\ve \Gamma = \sum_{\mf{F} \in \mathcal{F}_\Gamma^{-}}
(-1)^{|\mf{F}|} \Pi_\ve \mC_{\mf{F}} \Gamma \,,
\end{aligned}
\end{equation}
where $\mF_{\Gamma}^{-} $ denotes the set of all diverging forests that can be
found within the diagram $ \Gamma $.
In order to prove Lemma~\ref{lem:negvanish} we must carefully review the
approach from~\cite{BPHZ} for BPHZ renormalisation, and extend it to the present
critical setting.

  In doing so, we face two new challenges. These arise from the fact that the
  renormalisation from~\eqref{eq_zimmermann_2} corresponds to a zeroth order
  Taylor expansion. However, in order for diagrams not to diverge we must proceed to a second order Taylor expansion (because the blow-up of
  the renormalisation constant is of order $\ve^{-2}$, and each further order of
  Taylor expansion roughly leads to an improvement of one power of $\ve$, until
  there is no blowup). We find that
  \begin{itemize}
    \item The first order renormalisation does not contribute to the
    renormalisation by a symmetry argument, because of the isotropy and spatial
    homogeneity of the problem, see Section~\ref{sec:higher-order-ren}.
    \item The second order renormalisation does not contribute to the weak
    coupling limit. In particular, there is no renormalisation of the Laplacian
    as in KPZ/Burgers models~\cite{cannizzaro2023weak,cannizzaro2024gaussian}.
    This is due to a precise cancellation, see Section~\ref{sec:no-laplace}.
  \end{itemize}

To leverage on such higher order Taylor expansions, we will add and
subtract higher order renormalisation terms in different stages.
In this section we show how this can be achieved by following the setup of~\cite{BPHZ} as closely as possible.

\subsection{Preliminary structures in renormalisation}
Consider an Anderson--Feynman diagram $\Gamma = (V, E)$, and a divergent subdiagram $\gamma
\in \mG_\Gamma^{-}$.
As in Section~\ref{sec:zero-forests}, we will consider extraction maps $
\hat{\mC}_{\gamma}$ which rewire the edges of a given diagram attached to a
divergence $ \gamma $. These replace the maps $
\mC_{ \gamma } $ appearing in~\eqref{eq_zimmermann}.
While the extraction maps from~\eqref{eq_forest_extr} rewire edges to the highest
vertex of a subdiagram (according to some given ordering amongst the vertices),
in this section it will be convenient to follow a different convention.

First, as a consequence of Lemma~\ref{lem:neg}, every strict divergence $ \gamma \in
\mG^{-}_{\Gamma}$, as in~\eqref{e_def_Gminus}, has exactly one incoming edge $ e_{\gamma}^{\mathrm{in}}$ and
one outgoing edge $ e_{\gamma}^{\mathrm{out}}$ in $ \Gamma$.
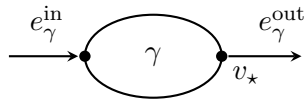
\begin{equation*}
  \begin{tikzpicture}[thick,>=stealth,scale=0.5,baseline=-0.2em]
    \tikzset{
      dot/.style={circle,fill=black,inner sep=1.5pt}
    }
    \node[above] at (-3.2,0.1) {$e_\gamma^{\mathrm{in}}$};
    \node[above] at ( 3.0,0.1) {$e_\gamma^{\mathrm{out}}$};
    \draw (-0.4,0) ellipse (1.8 and 1.1);
    \node at (-0.4,0) {$\gamma$};
    \node[dot] (bL) at (-2.2,0)   {};   
    \node[dot] (vL) at (1.4,0) {};   
    \draw[->] (-4.2,0) -- (bL);
    \draw[->] (vL) -- (3.5,0);
    \node[below right] at (vL) {$v_\star$};
  \end{tikzpicture}
\end{equation*}
Then for every
divergence $\gamma \in \mG_{\Gamma}^{-}$ we fix the distinguished vertex
\begin{equation*}
  v_\star(\gamma) = e_{\gamma, -}^{\mathrm{out}} \in V(\gamma) \,.
\end{equation*}
Recall from Section~\ref{sec:zero-forests} that for a given Feynman diagram $\Gamma =(V,E)$, the set $
\mD_{\Gamma}$ denotes the collection of all Feynman diagrams $ \tilde{\Gamma}=
(V, \tilde{E}) $ over the same set of vertices of $\Gamma$, and for
for which there exists an edge bijection $ \tau_{\tilde{\Gamma}, \Gamma} : \tilde{E} \to
E $, mapping legs onto legs.
We write $ \tau$ instead of $ \tau_{\tilde{\Gamma}, \Gamma}$ when the two
diagrams are clear from context. We also define
\begin{equation*}
\begin{aligned}
\sT_{\Gamma}:
= \langle \mD_{\Gamma} \rangle \,,
\end{aligned}
\end{equation*}
 to be the space of
all finite linear combinations of elements in $\mD_{\Gamma}$.
Then, for a fixed Feynman diagram $\Gamma$ and divergent subdiagram $\gamma \in
\mG^{-}_{\Gamma}$ we define the extraction map $\hat{\mC}_\gamma$ on
edges $e \in \tilde{E}$ of some $\tilde{\Gamma} \in \mD_{\Gamma}$ as follows:
\begin{equation}\label{e:Chat}
  \hat{\mC}_{\gamma} e = \begin{cases}
    (e_-, v_{\star}(\gamma)) & \quad \text{if } e = \tau^{-1}_{\tilde{\Gamma}, \Gamma} (e_{\gamma}^{\mathrm{ in}}) \text{ and } e_+ \in V(\gamma) \;, \\
    e & \quad \text{otherwise} \;.
  \end{cases}
\end{equation}
The condition $e_{+} = \tau^{-1}(e^{\rm{in}}_\gamma)_{+} \in V(\gamma)$ ensures
that if we apply iteratively extraction maps to nested diagrams that share the
same incoming edge, then we eventually rewire an incoming
edge to the outermost divergence, see the example discussed below. This guarantees the commutation
property in Lemma~\ref{lem:commute} below.
We then define the action on the full diagram $\tilde{\Gamma}$ as follows:
\begin{equation}\label{e:tauChat}
  \hat{\mC}_{\gamma} \tilde{\Gamma} := ( V, \{ \hat{\mC}_{\gamma} e : e \in \tilde{E} \} ) \,.
\end{equation}
In order to produce a new diagram in
$\mD_\Gamma$ we also provide the bijection $\tau_{\hat{\mC}_{\gamma}
\tilde{\Gamma}, \Gamma} = 
\tau_{\tilde{\Gamma}, \Gamma} \circ \hat{\mC}_{\gamma}^{-1}$.
For later convenience we introduce the following shorthand notation:
\begin{equation*}
  e_\star (\gamma, \tilde{\Gamma}) = \tau^{-1}_{\tilde{\Gamma}, \Gamma} (e^{\rm{in}}_{\gamma}) \;.
\end{equation*}
Whenever $\tilde{\Gamma}$ is clear from context, we abbreviate this by
$e_{\star}(\gamma)$. Moreover, if $\gamma$ is a subdiagram of $\Gamma$, we
set
\begin{equation*}
  \tilde{E}(\gamma) :=  \big\{ \tau_{\tilde{\Gamma}, \Gamma}^{-1} (e) : e \in
E(\gamma) \big\} \;, \qquad \tilde{V}(\gamma) := \bigcup_{e \in \tilde{E}(\gamma) }\{ e_-, e_+\} \;,
\end{equation*}
and define
\begin{equation*}
  \tau_{\tilde{\Gamma}, \Gamma}^{-1}(\gamma) := ( \tilde{V}(\gamma) \;, \tilde{E}(\gamma) ) \subseteq \tilde{\Gamma} \;.
\end{equation*}
In short, $\tau_{\tilde{\Gamma}, \Gamma}^{-1}(\gamma)$ is the image of $\gamma$
in $\tilde{\Gamma}$.

Let us explain graphically this construction. The picture to keep in mind is
that $\hat{\mC}_\gamma$ acts as follows:
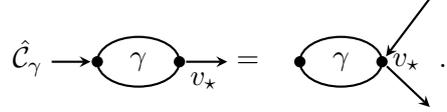
\begin{equation*}
  \hat{\mC}_{\gamma} \;
  \begin{tikzpicture}[thick,>=stealth,scale=0.3,baseline=-0.2em]
    \tikzset{
      dot/.style={circle,fill=black,inner sep=1.5pt}
    }
    \node[above] at (-3.2,0.1) {};
    \node[above] at ( 3.0,0.1) {};
    \draw (-0.4,0) ellipse (1.8 and 1.1);
    \node at (-0.4,0) {$\gamma$};
    \node[dot] (bL) at (-2.2,0)   {};   
    \node[dot] (vL) at (1.4,0) {};   
    \draw[->] (-4.2,0) -- (bL);
    \draw[->] (vL) -- (3.5,0);
    \node[below right] at (vL) {$v_\star$};
  \end{tikzpicture}
  =
  \begin{tikzpicture}[thick,>=stealth,scale=0.3,baseline=-0.2em]
    \tikzset{
      dot/.style={circle,fill=black,inner sep=1.5pt}
    }
    \node[above] at (-3.2,0.1) {};
    \node[above] at ( 3.0,0.1) {};
    \draw (-0.4,0) ellipse (1.8 and 1.1);
    \node at (-0.4,0) {$\gamma$};
    \node[dot] (bL) at (-2.2,0)   {};   
    \node[dot] (vL) at (1.4,0) {};   
    \draw[->] (3.5,2.8  ) -- (vL);
    \draw[->] (vL) -- (3.5,-2.0);
    \node[right] at (vL) {$v_\star$};
  \end{tikzpicture}  \;.
\end{equation*}
More concretely, we must renormalise diagrams that are for instance of the
form
\begin{equation}\label{e:example-gamma}
  \Gamma = \begin{tikzpicture}[thick, scale = 0.65, baseline=-0.2em]
    \tikzset{
      vtx/.style={circle,fill=black,inner sep=2pt}
    }
    \foreach \i/\x in {2/1.6,3/3.2,4/4.8,5/6.4}{
      \node[dot] (v\i) at (\x,0) {};
    }
    \foreach \i/\x in {1/0,6/8.0}{
      \node[bigroot] (v\i) at (\x,0) {};
    }

    \node[below left] at (v2) {\footnotesize $1$};
    \node[below right] at (v3) {\footnotesize $2$};
    \node[below left] at (v4) {\footnotesize $3$};
    \node[below right] at (v5) {\footnotesize $4$};
    \draw[midarrow] (v1) -- (v2) ;
    \draw[midarrow] (v2) -- (v3) ;
    \draw[midarrow] (v3) -- (v4) ;
    \draw[midarrow] (v4) -- (v5) ;
    \draw[midarrow] (v5) -- (v6) ;
    \draw[midarrow] (v2) .. controls +(0,-0.7) and +(0,-.7) .. (v3);
    \draw[midarrow] (v3) .. controls +(0,0.7) and +(0,0.7) .. (v2);
    \draw[midarrow] (v5) .. controls +(0,0.7) and +(0,0.7) .. (v4);
    \draw[midarrow] (v4) .. controls +(0,-0.7) and +(0,-0.7) .. (v5);
    \end{tikzpicture} \;.
\end{equation}
In this case there are three divergent subdiagrams. The full subdiagram $\gamma_1$ spanned
by the vertices $1,2$, the diagram $\gamma_2$ spanned by the vertices $3,4$, and
the diagram $\gamma_3$ spanned by all vertices $1,2,3,4$. For example, we then
find the following extractions
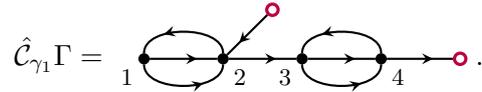
\begin{equation*}
  \hat{\mC}_{\gamma_1} \Gamma = \begin{tikzpicture}[thick, scale = 0.65, baseline=-0.2em]
    \tikzset{
      vtx/.style={circle,fill=black,inner sep=2pt}
    }
    \foreach \i/\x in {2/1.6,3/3.2,4/4.8,5/6.4}{
      \node[dot] (v\i) at (\x,0) {};
    }
    \foreach \i/\x in {6/8.0}{
      \node[bigroot] (v\i) at (\x,0) {};
    }

    \node[below left] at (v2) {\footnotesize $1$};
    \node[below right] at (v3) {\footnotesize $2$};
    \node[below left] at (v4) {\footnotesize $3$};
    \node[below right] at (v5) {\footnotesize $4$};
    \node[bigroot] (v1) at (4.2,1.0) {};
    \draw[midarrow] (v1) -- (v3) ;
    \draw[midarrow] (v2) -- (v3) ;
    \draw[midarrow] (v3) -- (v4) ;
    \draw[midarrow] (v4) -- (v5) ;
    \draw[midarrow] (v5) -- (v6) ;
    \draw[midarrow] (v2) .. controls +(0,-0.7) and +(0,-.7) .. (v3);
    \draw[midarrow] (v3) .. controls +(0,0.7) and +(0,0.7) .. (v2);
    \draw[midarrow] (v5) .. controls +(0,0.7) and +(0,0.7) .. (v4);
    \draw[midarrow] (v4) .. controls +(0,-0.7) and +(0,-0.7) .. (v5);
    \end{tikzpicture} \;.
\end{equation*}
Next, extracting the larger divergence we obtain
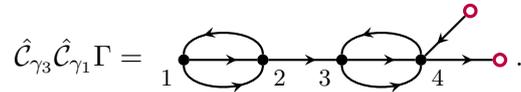
\begin{equation*}
  \hat{\mC}_{\gamma_3} \hat{\mC}_{\gamma_1} \Gamma = \begin{tikzpicture}[thick, scale =
	  0.65, baseline=-0.2em]
    \tikzset{
      vtx/.style={circle,fill=black,inner sep=2pt}
    }
    \foreach \i/\x in {2/1.6,3/3.2,4/4.8,5/6.4}{
      \node[dot] (v\i) at (\x,0) {};
    }
    \foreach \i/\x in {6/8.0}{
      \node[bigroot] (v\i) at (\x,0) {};
    }
    \node[below left] at (v2) {\footnotesize $1$};
    \node[below right] at (v3) {\footnotesize $2$};
    \node[below left] at (v4) {\footnotesize $3$};
    \node[below right] at (v5) {\footnotesize $4$};
    \node[bigroot] (v1) at (7.4,1.0) {};
    \draw[midarrow] (v1) -- (v5) ;
    \draw[midarrow] (v2) -- (v3) ;
    \draw[midarrow] (v3) -- (v4) ;
    \draw[midarrow] (v4) -- (v5) ;
    \draw[midarrow] (v5) -- (v6) ;
   \draw[midarrow] (v2) .. controls +(0,-0.7) and +(0,-.7) .. (v3);
    \draw[midarrow] (v3) .. controls +(0,0.7) and +(0,0.7) .. (v2);
    \draw[midarrow] (v5) .. controls +(0,0.7) and +(0,0.7) .. (v4);
    \draw[midarrow] (v4) .. controls +(0,-0.7) and +(0,-0.7) .. (v5);
    \end{tikzpicture} \;.
\end{equation*}
Lastly, if we extract the smaller divergence $\gamma_2$ we find:
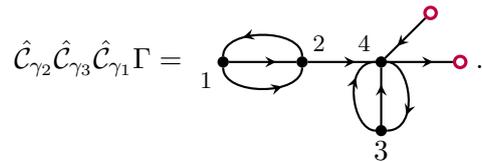
\begin{equation*}
  \hat{\mC}_{\gamma_2} \hat{\mC}_{\gamma_3} \hat{\mC}_{\gamma_1} \Gamma =
  \begin{tikzpicture}[thick, scale = 0.65, baseline=-0.2em]
    \tikzset{
      vtx/.style={circle,fill=black,inner sep=2pt}
    }
    \foreach \i/\x in {2/1.6,3/3.2,4/4.8}{
      \node[dot] (v\i) at (\x,0) {};
    }
    \foreach \i/\x in {5/6.4}{
      \node[bigroot] (v\i) at (\x,0) {};
    }
    \node[below left] at (v2) {\footnotesize $1$};
    \node[above right] at (v3) {\footnotesize $2$};
    \node[above left] at (v4) {\footnotesize $4$};
    \node[bigroot] (v1) at (5.8,1.0) {};
    \node[dot] (vn) at (4.8,-1.4) {};
    \node[below] at (vn) {$3$};
    \draw[midarrow] (v1) -- (v4) ;
    \draw[midarrow] (v2) -- (v3) ;
    \draw[midarrow] (v3) -- (v4) ;
    \draw[midarrow] (vn) -- (v4) ;
    \draw[midarrow] (v4) -- (v5) ;
    \draw[midarrow] (v2) .. controls +(0,-0.7) and +(0,-.7) .. (v3);
    \draw[midarrow] (v3) .. controls +(0,0.7) and +(0,0.7) .. (v2);
    \draw[midarrow] (v4) .. controls +(.6, 0) and +(.8, 0) ..  (vn);
    \draw[midarrow] (vn) .. controls +(-.6,0) and +(-.8,0) ..  (v4);
    \end{tikzpicture} \;.
\end{equation*}
One can check by hand that the order in which we performed these operations did
not matter. This is in particular the consequence of the choice~\eqref{e:Chat}
of how to rewire the incoming edges. Indeed, this property holds in general.


\begin{lemma}\label{lem:commute}
Let $ \Gamma $ be a negative Anderson--Feynman diagram.
For $ \gamma_{1}, \gamma_{2} \in \mG_{\Gamma}^{-}$ such that either $
\gamma_{1}$ and $ \gamma_{2}$ are vertex disjoint or one is included in the
other,
we have $ \hat{\mC}_{\gamma_{1}}
\hat{\mC}_{\gamma_{2}} \tilde{\Gamma} =  \hat{\mC}_{\gamma_{2}}
\hat{\mC}_{\gamma_{1}} \tilde{\Gamma}$, for every $ \tilde{\Gamma} \in
\mD_{\Gamma}$.
Moreover, if $ \gamma_{1} \subset \gamma_{2}$ such that $
e_{\star} ( \gamma_{1}, \Gamma ) = e_{\star} ( \gamma_{2}, \Gamma) $, then $
\hat{\mC}_{\gamma_{1}} \hat{\mC}_{\gamma_{2}} = \hat{\mC}_{\gamma_{2}}$.
\end{lemma}

\begin{proof}
  If the subdiagrams $\gamma_i$ are vertex disjoint, then the result follows
  because the extraction maps act on different sets of edges. Therefore,
  consider the case $\gamma_1 \subseteq \gamma_2$ and note that the only edges
  that are potentially moved are $e_{\star}(\gamma_i,
  \tilde{\Gamma})$, for $i \in \{1, 2\}$. If $e_{\star}(\gamma_1,
  \tilde{\Gamma}) \neq e_{\star}(\gamma_2,
  \tilde{\Gamma})$, or if they are equal but $e_{\star, +}(\gamma_1,
  \tilde{\Gamma}) \not\in V(\gamma_2)$ then the result follows again from the fact that the
  extraction maps do not act on the same edges (in the second case both
  extractions act trivially). We are left with the case $e_{\star}(\gamma_1,
  \tilde{\Gamma}) = e_{\star}(\gamma_2,
  \tilde{\Gamma}) = (e_-, e_+)$ such that $ e_{+} \in V (
\gamma_{2})$. We find
  \begin{equation*}
    \hat{\mC}_{\gamma_1} \hat{\mC}_{\gamma_2} e_{\star}(\gamma_1,
    \tilde{\Gamma}) = \hat{\mC}_{\gamma_1} (e_-, v_\star(\gamma_2)) = \begin{cases}
      (e_-, v_\star(\gamma_2)) & \text{ if } v_\star(\gamma_2) \in V(\gamma_2 )\setminus V(\gamma_1)\;, \\
      (e_-, v_\star(\gamma_1)) & \text{ otherwise }\;.
    \end{cases}
  \end{equation*}
  However,  if $ v_\star(\gamma_2) \in V(\gamma_2 ) \cap V(\gamma_1)$,
  then $v_\star(\gamma_2) = v_\star(\gamma_1)$ by Lemma~\ref{lem:neg}, which proves the
  commutation and the second statement of the lemma.
\end{proof}

Now, if we replace the contraction maps $ \mC$ appearing in~\eqref{eq_zimmermann_2}
with the extraction maps $ \hat{\mC}$, we can rewrite Zimmermann's forest formula as
\begin{equation}\label{eq_zimmer_chat}
\begin{aligned}
\hat{\mR} \Gamma = \sum_{\mf{F} \in \mathcal{F}_\Gamma^{-}}
(-1)^{| \mf{F}|}
\mf{K}_{\mf{F}} \Gamma \,, \qquad \text{where} \quad
\mf{K}_{\mf{F}} := \prod_{ \gamma \in \mf{F}} \hat{\mC}_{\gamma} \,.
\end{aligned}
\end{equation}
As a consequence of the spatial homogeneity of
the kernels, we have the following identification of diagrams. See also \cite[Pp. 27--28]{BPHZ}.
\begin{lemma}
	For every negative Anderson--Feynman diagram, we have
	$ \Pi_{\ve} \hat{\mR} \Gamma = \hat{\Pi}_{\ve} \Gamma$.
\end{lemma}

Before we continue, let us immediately treat the case of so called \emph{tadpole} divergences.
We say $ \gamma \in \mG_{\Gamma}^{-}$ is a
tadpole divergence, if it is of the form
\begin{equation*}
\begin{aligned}
\begin{tikzpicture}
  \node[dot] (m) at (0,0) {};
  \draw[midarrow, thick] (-1,0) -- (m);
  \draw[midarrow, thick] (m) -- (1,0);
  \draw[midarrow, thick] (m) to[out=135,in=45,looseness=20] (m);
    \end{tikzpicture}
\end{aligned}
\end{equation*}
i.e. $ \gamma$ is a self--loop or equivalently  $ | V ( \gamma) |= 1$. See for
example~\eqref{e:loop} for an example where such a tadpole divergence arises.

The BPHZ renormalisation always cancels
tadpole divergences exactly, as proven in the following lemma.

\begin{lemma}\label{lem_tadpole}
If $ \Gamma$ is a connected negative Anderson--Feynman diagram containing a tadpole
divergence, then $ \hat{\Pi}_{\ve} \Gamma =0$.
\end{lemma}

\begin{proof}
Let $ \gamma \in \mG_{\Gamma}^{-}$ be a tadpole divergence. We partition the
set of forests $ \mF_{\Gamma}^{-}$ into two sets $ \mF_{1}$ and $
\mF_{2}$: those forests that contain $
\gamma$ (say $ \mF_{1}$) and those forests that do not contain $ \gamma$ (say $
\mF_{2}$). Moreover, we note that $\{ \gamma\} \cup \mf{F}$ is a forest in $
\mF_{1}$, whenever $ \mf{F} \in \mF_{2}$, because for any other divergence $
\gamma' \in \mG_{\Gamma}^{-}$ the tadpole $ \gamma$ is either included
in $ \gamma' $ or vertex disjoint.
Hence, $ \mF_{1}$ and $ \mF_{2}$ are in bijection.

Thus, the valuation of Zimmermann's forest formula~\eqref{eq_zimmer_chat} can be expressed in
terms of
\begin{equation}\label{eq_supp1_tadpole}
\begin{aligned}
 \hat{\Pi}_\ve \Gamma = \sum_{\mf{F} \in \mF_{2}}
(-1)^{| \mf{F}|}
\Pi_\ve \mf{K}_{\mf{F}} (\mathrm{id}- \hat{\mC}_{\gamma}) \Gamma \,.
\end{aligned}
\end{equation}
Moreover, we note that $ \hat{\mC}_{\gamma}$ has no effect on $ \Gamma$,
since the corresponding rewiring leaves the diagram invariant. Thus,
$(\mathrm{id}- \hat{\mC}_{\gamma}) \Gamma =0$ and~\eqref{eq_supp1_tadpole} vanishes.
This concludes the proof.
\end{proof}

\begin{nota}
As a consequence of Lemma~\ref{lem_tadpole}, for the entire remainder of this
section we consider only diagrams that don't contain any
tadpoles.
\end{nota}

\subsection{Higher order renormalisation}\label{sec:higher-order-ren}

So far we have only introduced a zeroth order expansion term $
\hat{\mC}_{\gamma}$
 in the
renormalisation procedure.
To estimate $\hat{\Pi}_{\varepsilon} \Gamma
( \varphi) $, we will now introduce an additional  term $
\hat{\mC}^{(1)}_{\gamma}$ in the renormalisation procedure, which corresponds to a first order Taylor
expansion about our divergences.
Eventually, we will prove that the first order term $\hat{\mC}^{(1)}_{\gamma}$
is in fact not present in
renormalisation map \eqref{eq_zimmermann_2}, because it vanishes upon integration as a consequence
of the isotropy and spatial homogeneity of the problem.
Therefore, the additional renormalisation we introduce in this subsection
should be thought of as a zero-sum term. It is only useful in the technical
aspects of our bounds.

In order to deal with higher order expansions in~\eqref{eq_zimmermann_2}, we will augment the space of Feynman diagrams with additional labels:
\begin{itemize}
\item We add labels $ \mathfrak{h}  : E \to \NN^{d} $ to every edge, including legs.
The label $ \mathfrak{h} $ will track the number of directional derivatives applied to the
corresponding edge.
\item We add labels $\mathfrak{n} : V_{\star} \to \NN^{d}$ to all inner
vertices. The labels $\mathfrak{n}$ track the powers of monomials gained through Taylor expansion.
\item We add labels $\mathfrak{a} : V_{\star} \to V_\star$ to all inner
vertices. The labels $\mathfrak{a}$ track the point $\mf{a}$ in which we centre
monomials. Overall, to a given vertex $v \in V_\star$ we associate the monomial
$(x_v - x_{\mf{a}(v)})^{\mf{n}(v)}$, with the usual convention
\begin{equation*}
  (x_v - x_{\mf{a}(v)})^{\mf{n}(v)} = \prod_{i=1}^d (x_v^{i} - x_{\mf{a}(v)}^{i})^{\mf{n}(v)_i} \;.
\end{equation*}
\end{itemize}
Every Feynman diagram $\Gamma$ is considered equipped with the
labels $ \mathfrak{n} = \mathfrak{h} \equiv 0$, and $\mf{a}= \rm{id}$, when
viewed as an element of $\mD_{\Gamma}$. \\

Collecting all labels, we are now working with diagrams
$\tilde{\Gamma}$ that are decorated with the following information:
\begin{equation*}
  (\tilde{\Gamma}, \mathfrak{h}, \mathfrak{n}, \mf{a},  \tau_{\tilde{\Gamma}, \Gamma})\,.
\end{equation*}
We denote with
\begin{equation*}
	\begin{aligned}
		\mD_{\Gamma}^{\rm{lab}}  \qquad \text{and} \qquad
\sT^{\rm{lab}}_{\Gamma} = \langle \mD_{\Gamma}^{\rm{lab}} \rangle \,,
	\end{aligned}
\end{equation*}
respectively, the space of such labelled diagrams, and the vector space of all their finite linear
combinations.

On the space $\mD_{\Gamma}^{\rm{lab}}$, we define extraction maps similarly to \cite[Eq.
(3.1)]{BPHZ}.
Namely, let us fix $\gamma \in \mG_\Gamma^{-}$.
First, we extend $\hat{\mC}$ to labelled diagrams by defining $\hat{\mC}_\gamma \colon \mD_{\Gamma}^{\rm{lab}} \to
\sT_{\Gamma}^{\rm{lab}} $ in terms of
\begin{equation}\label{e:Chatnew}
    \hat{{\mC}}_{\gamma} ( \tilde{\Gamma}, \mathfrak{h}, \mathfrak{n}, \mf{a})
    = \sum_{ \overline{\mf{n}} \leq \mf{n} \vert_{V_{\mf{a}}(\gamma)}} { \mf{n} \choose \overline{\mf{n}} } \big( \hat{\mC}_{\gamma} \tilde{\Gamma},
        \mathfrak{h} \;,
        (\mathfrak{n}, \overline{\mf{n}})_\gamma \;, \mf{a}_{\gamma}
        \big)\,,
  \end{equation}
  where $\hat{\mC}_{\gamma} \tilde{\Gamma}$ on the right is the unlabelled
  extraction map from~\eqref{e:Chat}. Moreover, we have defined
  \begin{equation}\label{e:Va}
    V_{\mf{a}}(\gamma) := \{ v \in V(\gamma) \setminus \{ v_{ \star}(
    \gamma) \} \colon \mf{a}(v)  \not\in V(\gamma)\} \;,
  \end{equation}
  as well as
    $\mf{a}_\gamma (v)$ by\footnote{The map $ \mf{a}$ replaces the role of the
    ``depth'' $\mf{d}$ in \cite{BPHZ}.}:
  \begin{equation}\label{e:aChat}
    \mf{a}_\gamma (v ) := \begin{cases}
      v_{\star}(\gamma) & \quad  \text{if } v \in V ( \gamma) \setminus \{ v_{ \star}(
  \gamma) \} \text{ s.t. } \mf{a} (v) = v \;, \\
      v_{\star}(\gamma) & \quad  \text{if } v \in V_{\mf{a}} ( \gamma) \;, \\
      \mf{a} (v) & \quad \text{otherwise} \;.
    \end{cases}
  \end{equation}
  The sum in \eqref{e:Chatnew} runs over  $\overline{\mf{n}} \colon V_\star \to \NN^d$ such that
  $\overline{\mf{n}}(v)_i \leq \mf{n}(v)_i$ for all $v \in V_\star$ and $i \in
  \{1, \ldots, d\}$, and such that $\overline{\mf{n}}(v) = 0$ for all $v \not\in
  V_{\mf{a}}(\gamma)$, with the
  definition:
  \begin{equation*}
    { \mf{n} \choose \overline{\mf{n}}} = \frac{\mf{n}!}{\overline{\mf{n}}! (\mathfrak{n} -\overline{\mf{n}})!} \;, \qquad \mf{n} ! = \prod_{ v \in V_\star} \prod_{i=1}^d \mf{n}(v)_i !  \;.
  \end{equation*}

  In our setting, we will always have $\mf{n}(v)_i \in \{0,1\}$, so that
  ${\mf{n} \choose \overline{\mf{n}}} =1$. However, we keep the combinatorial
  representation as it makes the explanation of the terms clearer and establishes the
  link to \cite[Eq. (3.1)]{BPHZ}. Finally, \eqref{e:Chatnew} contains the \lqm
  contracted\rqm label $(\mf{n},
  \overline{\mf{n}})_\gamma$, which is defined as follows for all $v \in V_\star$:
  \begin{equation*}
    (\mf{n},  \overline{\mf{n}})_\gamma (v) = \begin{cases}
      \overline{\mf{n}}(v) & \quad \text{if } v \in V_{\mf{a}}(\gamma)\;, \\
      \mf{n}(v)+ \sum_{u \in V_{\mf{a}}(\gamma) } (\mf{n}(u) - \overline{\mf{n}}(u)) & \quad \text{if } v = v_{\star}(\gamma) \;, \\
      \mf{n}(v) & \quad \text{otherwise} \;.
    \end{cases}
  \end{equation*}
In the subsequent sections, we will always have
$\mf{a}(v) = \mf{a}(v_\star(\gamma))$ for all
  $v \in V_{\mf{a}}(\gamma)$, because
we only extract divergences that belong to a fixed forest.
Therefore, the choice of $\mf{n}$ and $\mf{a}$
  above correspond to the recentering of monomials via
the multinomial theorem:
  \begin{equation*}
    \begin{aligned}
    \prod_{v \in V_\star} & (x_{v} - x_{\mf{a}(v)})^{\mf{n}(v)}  \\
    & =  \prod_{v \in V_\star \setminus V_{\mf{a}}(\gamma)} (x_{v} - x_{\mf{a}(v)})^{\mf{n}(v)} \prod_{v \in V_{\mf{a}} (\gamma)} (x_{v} - x_{v_\star(\gamma)} + x_{v_\star(\gamma)}- x_{\mf{a}(v)})^{\mf{n}(v)} \\
    & = \sum_{ \overline{\mf{n}} \leq \mf{n} \vert_{V_{\mf{a}}(\gamma)}} { \mf{n} \choose \overline{\mf{n}} } \prod_{v \in V_\star \setminus V_{\mf{a}}(\gamma)} (x_{v} - x_{\mf{a}(v)})^{\mf{n}(v)} \prod_{v \in V_{\mf{a}} (\gamma)} (x_{v} - x_{v_\star(\gamma)})^{\overline{\mf{n}}(v)} ( x_{v_\star(\gamma)}- x_{\mf{a}(v_\star(\gamma))})^{\mf{n}(v) - \overline{\mf{n}}(v)} \\
    & = \sum_{ \overline{\mf{n}} \leq \mf{n} \vert_{V_{\mf{a}}(\gamma)}} { \mf{n} \choose \overline{\mf{n}} } \prod_{v \in V_\star } (x_{v} - x_{\mf{a}(v)})^{( \mf{n}, \overline{\mf{n}})_\gamma (v)} \;,
    \end{aligned}
  \end{equation*}
  so that eventually all the monomials
  are expressed in terms of the difference of two variables,
  that are either both associated to vertices of $\gamma$, or both associated to
  $V(\gamma)^c \cup \{ v_\star(\gamma)\}$.

Similarly, we define the map $\hat{\mC}^{(1)}_\gamma
\colon \mD_{\Gamma}^{\rm{lab}} \to \sT_{\Gamma}^{\rm{lab}} $ as follows: With
$ e_{\star}(\gamma)= e_{\star}(\gamma, \tilde{\Gamma})$,
\begin{equation}\label{eq_def_Cbar1}
\begin{aligned}
\hat{{\mC}}_{\gamma}^{(1)}( \tilde{\Gamma}, \mathfrak{h}, \mathfrak{n}, \mf{a})
= \begin{cases}
  0 \quad \text{ if } \mathfrak{h} ( e_{\star}(\gamma )) \neq 0 \;, &  \\
  0 \quad \text{ if } \hat{\mC}_\gamma \tilde{\Gamma} = \tilde{\Gamma} \;, & \\
\sum_{i =1}^{d}
\Big(
\hat{\mC}_{\gamma} \tilde{\Gamma} \;,
\mathfrak{h} +  \mathds{1}_{
(
  e_{\star, -}(\gamma), v_{\star}  ( \gamma) ) } \delta_{i} \;,
\mathfrak{n} +
\mathds{1}_{ e_{\star, +}(\gamma)}\delta_{i} \;, \mf{a}_{\gamma}
\Big) &  \text{else} \;,
\end{cases}
\end{aligned}
\end{equation}
where $\delta_i$ is the $i$-th canonical basis vector in $\RR^d$,
as well as $\tau$ and $\hat{\mC}_\gamma$ considered in the unlabelled
setting~\eqref{e:Chat}.
Graphically we represent $\hat{\mC}_\gamma^{(1)}$ as follows:
\begin{equation*}
  \hat{\mC}_{\gamma}^{(1)} \;
  \begin{tikzpicture}[thick,>=stealth,scale=0.3,baseline=-0.2em]
    \tikzset{
      dot/.style={circle,fill=black,inner sep=1.5pt}
    }
    \node[above] at (-3.2,0.1) {};
    \node[above] at ( 3.0,0.1) {};
    \draw (-0.4,0) ellipse (1.8 and 1.1);
    \node at (-0.4,0) {$\gamma$};
    \node[dot] (bL) at (-2.2,0)   {};   
    \node[dot] (vL) at (1.4,0) {};   
    \draw[->] (-4.2,0) -- (bL);
    \draw[->] (vL) -- (3.5,0);
    \node[below right] at (vL) {$v_\star$};
  \end{tikzpicture}
  = \sum_{i=1}^d
  \begin{tikzpicture}[thick,>=stealth,scale=0.3,baseline=-0.2em]
    \tikzset{
      dot/.style={circle,fill=black,inner sep=1.5pt}
    }
    \node[above] at (-3.2,0.1) {};
    \node[above] at ( 3.0,0.1) {};
    \draw (-0.4,0) ellipse (1.8 and 1.1);
    \node at (-0.4,0) {$\gamma$};
    \node[dot] (bL) at (-2.2,0)   {};   
    \node at (-2.9,0.3) {\scalebox{0.8}{$\delta_i$}};
    \node[dot] (vL) at (1.4,0) {};   
    \draw[->] (3.5,2.8  ) -- (vL);
    \node at (1.9,2.0) {\scalebox{0.8}{$\delta_i$}};
    \draw[->] (vL) -- (3.5,-2.0);
    \node[right] at (vL) {$v_\star$};
  \end{tikzpicture}  \;,
\end{equation*}
where we have added the nontrivial labelings of $\mf{n}$ and $\mf{h}$ on the
right. We omitted $\mf{a}$ and assumed that we were not in one of the cases in which the map $
\hat{\mC}^{(1)}_{\gamma} $ acts trivially.

  We will denote the action on labelled diagrams with
  \begin{equation*}
    \hC_\gamma \tilde{\Gamma} = \hat{\mC}_{\gamma} ( \tilde{\Gamma}, \mathfrak{h}, \mathfrak{n}, \mf{a}) \;, \qquad \hC_\gamma^{(1)} \tilde{\Gamma} = \hat{\mC}_{\gamma}^{(1)} ( \tilde{\Gamma}, \mathfrak{h}, \mathfrak{n}, \mf{a})
  \end{equation*}
  for short, when no confusion can occur. We also extend the maps $\hC_\gamma,
  \hC^{(1)}_\gamma$ linearly to all of $\sT^{\rm{lab}}_{\Gamma}$.

\begin{remark}
  The conditions concerning the vanishing of the maps in~\eqref{eq_def_Cbar1}
  correspond to the conditions in \cite[Eq. (3.1)]{BPHZ}.
  \begin{itemize}
    \item The first condition in~\eqref{eq_def_Cbar1}
is only relevant if we successively extract divergences
    that share the same incoming edge, namely if $\gamma_1 \subseteq \gamma_2 \in
    \mG^{-}_\Gamma$ satisfy $e_\star(\gamma_1, \Gamma) = e_\star(\gamma_2,
    \Gamma)$. Our choice corresponds to enforcing renormalisation only in
    presence of a \emph{strictly} negative divergence (in contrast to the non-positive degree condition in \cite[Eq.
    (3.1)]{BPHZ}).

    \item   The second condition in~\eqref{eq_def_Cbar1} corresponds to the fact
    that the first order approximation would vanish if $\hC_\gamma \tilde{\Gamma} =
    \tilde{\Gamma}$, since it would lead to a monomial $(x_{v_\star(\gamma)}^i
    -x_{v_\star(\gamma)}^i)=0$.

   \end{itemize}
    Overall, the map $\hC_\gamma + \hC^{(1)}_\gamma$ is equivalent to the
    extraction map $\hC_\gamma$ described in \cite[Pp.~{26-27}]{BPHZ}, except for the explicit choice of
    $v_\star(\gamma)$, and for the requirement that the extracted divergences are strictly negative.
    Notice that in our setting there is no need to
    introduce weights on half-edges, because we are always differentiating the
    only incoming edge.
\end{remark}
Given a forest of divergences $ \mf{F} \in \mF_{\Gamma}^{-}$, we define the
iterative application of extraction maps:
\begin{equation*}
\begin{aligned}
\mf{K}_{\mf{F}}^{(1)} :=  \mf{K}_{\mf{F} \setminus \vr(\mf{F})}^{(1)} \prod_{ \gamma \in \vr(\mf{F})} \big(\hat{\mC}_{\gamma} +
\hat{\mC}_{\gamma}^{(1)} \big)\,,
\end{aligned}
\end{equation*}
where $\vr (\mf{F})$ is the set of roots of $\mf{F}$.
It turns out that the order in which extraction maps are applied does not play
a role. This is the content of the next lemma, which implies
that we can simply write
\begin{equation}\label{eq_mK1}
  \begin{aligned}
  \mf{K}_{\mf{F}}^{(1)} =  \prod_{ \gamma \in \mf{F}} \big(\hat{\mC}_{\gamma} +
  \hat{\mC}_{\gamma}^{(1)} \big)\,.
  \end{aligned}
  \end{equation}
However, for the sake of clarity let us first evaluate the map
$\mf{K}^{(1)}$ on a representative example. Let us consider $\Gamma$ as in
\eqref{e:example-gamma}, namely
\begin{equation*}
\begin{aligned}
\Gamma = \begin{tikzpicture}[thick, scale = 0.6, baseline=-0.2em]
    \tikzset{
      vtx/.style={circle,fill=black,inner sep=2pt}
    }
    \foreach \i/\x in {2/1.6,3/3.2,4/4.8,5/6.4}{
      \node[dot] (v\i) at (\x,0) {};
    }
    \foreach \i/\x in {1/0,6/8.0}{
      \node[bigroot] (v\i) at (\x,0) {};
    }

    \node[below left] at (v2) {\footnotesize $1$};
    \node[below right] at (v3) {\footnotesize $2$};
    \node[below left] at (v4) {\footnotesize $3$};
    \node[below right] at (v5) {\footnotesize $4$};
    \draw[midarrow] (v1) -- (v2) ;
    \draw[midarrow] (v2) -- (v3) ;
    \draw[midarrow] (v3) -- (v4) ;
    \draw[midarrow] (v4) -- (v5) ;
    \draw[midarrow] (v5) -- (v6) ;
    \draw[midarrow] (v2) .. controls +(0,-0.7) and +(0,-.7) .. (v3);
    \draw[midarrow] (v3) .. controls +(0,0.7) and +(0,0.7) .. (v2);
    \draw[midarrow] (v5) .. controls +(0,0.7) and +(0,0.7) .. (v4);
    \draw[midarrow] (v4) .. controls +(0,-0.7) and +(0,-0.7) .. (v5);
    \end{tikzpicture} \;.
\end{aligned}
\end{equation*}
In this example, let now $\gamma_1$ be the subdiagram
spanned by the vertices $1,2$, and $\gamma_2$ be the subdiagram spanned by the
vertices $1,2,3,4$ (notice the slight change of notation with respect to the
discussion below \eqref{e:example-gamma}). We will compute $ (\hC_{\gamma_1}
+ \hC^{(1)}_{\gamma_1})  (\hC_{\gamma_2}
+ \hC^{(1)}_{\gamma_2}) \Gamma$. To graphically encode the map $\mf{a}$, we color in
blue the vertices with $\mf{a}(v) = 4$, and in green the vertices satisfying
$\mf{a}(v)=2$. First, we find that
\begin{equation}\label{e:comm-example}
  \begin{aligned}
  (\hC_{\gamma_2}+ \hC^{(1)}_{\gamma_2}) \Gamma = \begin{tikzpicture}[thick, scale = 0.5, baseline=-0.2em]
    \tikzset{
      vtx/.style={circle,fill=black,inner sep=2pt}
    }
    \foreach \i/\x in {2/1.6,3/3.2,4/4.8,5/6.4}{
      \node[dot, blue] (v\i) at (\x,0) {};
    }
    \node[bigroot] (v6) at (8.0,-1.0) {};
    \node[bigroot] (v7) at (8.0,1.0) {};
    \node[below left] at (v2) {};
    \node[below right] at (v3) {};
    \node[below left] at (v4) {};
    \node[below right] at (v5) {};
    \draw[midarrow] (v2) -- (v3) ;
    \draw[midarrow] (v3) -- (v4) ;
    \draw[midarrow] (v4) -- (v5) ;
    \draw[midarrow] (v5) -- (v6) ;
    \draw[midarrow] (v7) -- (v5) ;
    \draw[midarrow] (v2) .. controls +(0,-0.7) and +(0,-.7) .. (v3);
    \draw[midarrow] (v3) .. controls +(0,0.7) and +(0,0.7) .. (v2);
    \draw[midarrow] (v5) .. controls +(0,0.7) and +(0,0.7) .. (v4);
    \draw[midarrow] (v4) .. controls +(0,-0.7) and +(0,-0.7) .. (v5);
    \end{tikzpicture}
    +
    \begin{tikzpicture}[thick, scale = 0.5, baseline=-0.2em]
      \tikzset{
        vtx/.style={circle,fill=black,inner sep=2pt}
      }
      \foreach \i/\x in {2/1.6,3/3.2,4/4.8,5/6.4}{
        \node[dot, blue] (v\i) at (\x,0) {};
      }
      \node[bigroot] (v6) at (8.0,-1.0) {};
      \node[bigroot] (v7) at (8.0,1.0) {};
      \node[below left] at (v2) {$\delta_i$};
      \node[below right] at (v3) {};
      \node[below left] at (v4) {};
      \node[below right] at (v5) {};
      \draw[midarrow] (v2) -- (v3) ;
      \draw[midarrow] (v3) -- (v4) ;
      \draw[midarrow] (v4) -- (v5) ;
      \draw[midarrow] (v5) -- (v6) ;
      \draw[midarrow] (v7) -- (v5) ;
      \node at (7.0,1.0) {\scalebox{1.0}{$\delta_i$}};
      \draw[midarrow] (v2) .. controls +(0,-0.7) and +(0,-.7) .. (v3);
    \draw[midarrow] (v3) .. controls +(0,0.7) and +(0,0.7) .. (v2);
    \draw[midarrow] (v5) .. controls +(0,0.7) and +(0,0.7) .. (v4);
    \draw[midarrow] (v4) .. controls +(0,-0.7) and +(0,-0.7) .. (v5);
      \end{tikzpicture}
    \;.
  \end{aligned}
\end{equation}
Then, applying $\hC_{\gamma_1} + \hC^{(1)}_{\gamma_1}$ to each term we
obtain
\begin{equation*}
  \begin{aligned}
  (\hC_{\gamma_1} + \hC^{(1)}_{\gamma_1})(\hC_{\gamma_2} + \hC^{(1)}_{\gamma_2}) \Gamma = &
  \begin{tikzpicture}[thick, scale = 0.5, baseline=-0.2em]
    \tikzset{
      vtx/.style={circle,fill=black,inner sep=2pt}
    }
    \foreach \i/\x in {3/3.2,4/4.8,5/6.4}{
      \node[dot, blue] (v\i) at (\x,0) {};
    }
    \node[dot, green!50!black!70] (v2) at (1.6,0) {};
    \node[bigroot] (v6) at (8.0,-1.0) {};
    \node[bigroot] (v7) at (8.0,1.0) {};
    \node[below left] at (v2) {};
    \node[below right] at (v3) {};
    \node[below left] at (v4) {};
    \node[below right] at (v5) {};
    \draw[midarrow] (v2) -- (v3) ;
    \draw[midarrow] (v3) -- (v4) ;
    \draw[midarrow] (v4) -- (v5) ;
    \draw[midarrow] (v5) -- (v6) ;
    \draw[midarrow] (v7) -- (v5) ;
    \draw[midarrow] (v2) .. controls +(0,-0.7) and +(0,-.7) .. (v3);
    \draw[midarrow] (v3) .. controls +(0,0.7) and +(0,0.7) .. (v2);
    \draw[midarrow] (v5) .. controls +(0,0.7) and +(0,0.7) .. (v4);
    \draw[midarrow] (v4) .. controls +(0,-0.7) and +(0,-0.7) .. (v5);
    \end{tikzpicture} \\
    &  +
    \begin{tikzpicture}[thick, scale = 0.5, baseline=-0.2em]
      \tikzset{
        vtx/.style={circle,fill=black,inner sep=2pt}
      }
      \foreach \i/\x in {3/3.2,4/4.8,5/6.4}{
        \node[dot, blue] (v\i) at (\x,0) {};
      }
      \node[dot, green!50!black!70] (v2) at (1.6,0) {};
      \node[bigroot] (v6) at (8.0,-1.0) {};
      \node[bigroot] (v7) at (8.0,1.0) {};
      \node[below left] at (v2) {$\delta_i$};
      \node[below right] at (v3) {};
      \node[below left] at (v4) {};
      \node[below right] at (v5) {};
      \draw[midarrow] (v2) -- (v3) ;
      \draw[midarrow] (v3) -- (v4) ;
      \draw[midarrow] (v4) -- (v5) ;
      \draw[midarrow] (v5) -- (v6) ;
      \draw[midarrow] (v7) -- (v5) ;
      \draw[midarrow] (v2) .. controls +(0,-0.7) and +(0,-.7) .. (v3);
    \draw[midarrow] (v3) .. controls +(0,0.7) and +(0,0.7) .. (v2);
    \draw[midarrow] (v5) .. controls +(0,0.7) and +(0,0.7) .. (v4);
    \draw[midarrow] (v4) .. controls +(0,-0.7) and +(0,-0.7) .. (v5);
      \node at (7.0,1.0) {\scalebox{1.0}{$\delta_i$}};
      \end{tikzpicture}
      + \begin{tikzpicture}[thick, scale = 0.5, baseline=-0.2em]
        \tikzset{
          vtx/.style={circle,fill=black,inner sep=2pt}
        }
        \foreach \i/\x in {3/3.2,4/4.8,5/6.4}{
          \node[dot, blue] (v\i) at (\x,0) {};
        }
        \node[dot, green!50!black!70] (v2) at (1.6,0) {};
        \node[bigroot] (v6) at (8.0,-1.0) {};
        \node[bigroot] (v7) at (8.0,1.0) {};
        \node[below left] at (v2) {};
        \node at (3.55,-0.5) {$\delta_i$};
        \node[below left] at (v4) {};
        \node[below right] at (v5) {};
        \draw[midarrow] (v2) -- (v3) ;
        \draw[midarrow] (v3) -- (v4) ;
        \draw[midarrow] (v4) -- (v5) ;
        \draw[midarrow] (v5) -- (v6) ;
        \draw[midarrow] (v7) -- (v5) ;
        \draw[midarrow] (v2) .. controls +(0,-0.7) and +(0,-.7) .. (v3);
    \draw[midarrow] (v3) .. controls +(0,0.7) and +(0,0.7) .. (v2);
    \draw[midarrow] (v5) .. controls +(0,0.7) and +(0,0.7) .. (v4);
    \draw[midarrow] (v4) .. controls +(0,-0.7) and +(0,-0.7) .. (v5);
        \node at (7.0,1.0) {\scalebox{1.0}{$\delta_i$}};
        \end{tikzpicture} \;.
      \end{aligned}
\end{equation*}
The same arguments, applied in reverse order, show that indeed  $$(\hC_{\gamma_2} + \hC^{(1)}_{\gamma_2})
(\hC_{\gamma_1} + \hC^{(1)}_{\gamma_1}) \Gamma = (\hC_{\gamma_1} + \hC^{(1)}_{\gamma_1})
(\hC_{\gamma_2} + \hC^{(1)}_{\gamma_2}) \Gamma\,.$$
This identity holds in general.

\begin{lemma}\label{lem_C1_comm}
Let $\Gamma$ be an Anderson--Feynman diagram and
 $ \mf{F} \in
 \mF_{\Gamma}^{-}$ be a forest of divergences. Write $\mf{K}^{(1)}_{\mf{F}}
\Gamma = \sum \tilde{\Gamma}$, where each $\tilde{\Gamma} $ is a nontrivial
element of $\mD^{\rm{lab}}_{\Gamma}$.
 Then for any $ \gamma_{1}, \gamma_{2} \in \mG^{-}_{\Gamma}$ such that
$\mf{F} \cup \{ \gamma_{1}, \gamma_{2}\} \in \mF_{\Gamma}^{-}$ and for any term
$\tilde{\Gamma}$, we have
 \begin{equation}\label{e:nd-1}
     ( \hat{\mC}_{\gamma_{2}} +
 \hat{\mC}_{\gamma_{2}}^{(1)})(\hat{\mC}_{\gamma_{1}} +
 \hat{\mC}_{\gamma_{1}}^{(1)}) \tilde{\Gamma} =
 (\hat{\mC}_{\gamma_{1}} +
 \hat{\mC}_{\gamma_{1}}^{(1)})
 ( \hat{\mC}_{\gamma_{2}} +
 \hat{\mC}_{\gamma_{2}}^{(1)})  \tilde{\Gamma}
 \;.
 \end{equation}
 \end{lemma}
The proof of this result follows as in \cite[Proposition 2.19 and Eq.\ (3.7)]{BPHZ}.
By virtue of Lemma~\ref{lem_C1_comm} the order of operations in
$\mf{K}^{(1)}_{\mf{F}}$ is irrelevant, and we are finally ready to enhance our definition of
renormalisation in
\eqref{eq_zimmer_chat} to include first order terms.
\begin{definition}\label{def_R1}
  Consider the maps $\hat{\mC}_\gamma, \hat{\mC}_\gamma^{(1)}$ from~\eqref{e:Chatnew}
 and~\eqref{eq_def_Cbar1}  extended by linearity to all of
  $\sT_\Gamma^{\rm{lab}}$.
We define
  \begin{equation}\label{e_Rhat_one}
  \begin{aligned}
\hat{\mR}^{(1)} \Gamma := \sum_{\mf{F} \in \mathcal{F}_\Gamma^{-}}
(-1)^{| \mf{F}|}
\mf{K}_{\mf{F}}^{(1)} \Gamma \,,
  \end{aligned}
\end{equation}
with $ \mf{K}_{\mf{F}}^{(1)}$ as in \eqref{eq_mK1}.
In particular, the evaluation with respect to $ \hat{\mR}^{(1)}$ is given by
\begin{equation*}
\begin{aligned}
\Pi_{\ve} \hat{\mR}^{(1)} \Gamma ( \varphi)
=
\sum_{\mf{F} \in \mathcal{F}_\Gamma^{-}}
(-1)^{| \mf{F}|}
\int_{( \TT^{d})^{V_{\star}}}
  \overline{\mW}_{\ve} \mf{K}_{\mf{F}}^{(1)} \Gamma (x_{V_\star})
 \ud x_{V_{\star}} \,,
\end{aligned}
\end{equation*}
where we have extended the map $ \mW_{\ve} $ from
\eqref{e_def_mW} to labelled diagrams as
follows:
\begin{equation}\label{e_updW_labelled}
  \overline{\mW}_{\ve}  (  \tilde{\Gamma}, \mathfrak{h}, \mathfrak{n},  \mf{a}) (x_{V_\star})
  =
  \int\limits_{ ( \TT^{d})^{L}}
  \varphi ( x_{L})
   \prod_{e \in \tilde{E} }
  D^{\mathfrak{h} ( e)} G_{\ve}(x_{e_{+}} - x_{e_{-}})
  \prod_{ v \in V_{\star}}
  \big( x_{ v } - x_{\mf{a}(v) }\big)^{\mathfrak{n}
  (v)} \ud x_{L}\,.
  \end{equation}
\end{definition}
We stress that, in contrast to $ \mW_{\ve}$, $ \overline{\mW}_{\ve}$ now also
includes the Green's functions associated to legs and integrated against the
test function $ \varphi$.\\

The next result proves that the additional renormalisation does not affect the
evaluation of a given Feynman diagram, because we have only added antisymmetric
terms that vanish under integration.
\begin{lemma}\label{lem_include_C1}
For any negative Anderson--Feynman diagram $ \Gamma $,
we have
\begin{equation}\label{e_bphzintervals_1}
\begin{aligned}
\hat{\Pi}_{\varepsilon} \Gamma
( \varphi)
=
\Pi_{\ve} \hat{\mR}^{(1)} \Gamma ( \varphi) \,.
\end{aligned}
\end{equation}
\end{lemma}

\begin{proof}
Let $ \mf{F} \in \mF_{\Gamma}^{-}$ be a forest of divergences,
and $( \gamma_{i} )_{i =1}^{M}$ be an enumeration of the divergences
inside, such that $i<j$ whenever $ \gamma_{i} \subset \gamma_{j} $.
By explicitly expanding the product in \eqref{eq_mK1}, we can write
\begin{equation*}
\begin{aligned}
\big(\mf{K}^{(1)}_{\mf{F}} -
\mf{K}_{\mf{F}} \big) \Gamma
=
\sum_{m = 1}^{M}
\bigg( \prod_{i=m+1}^M
 (\hat{\mC}_{\gamma_{i}}+
\hat{\mC}_{\gamma_{i}}^{(1)}) \bigg)
 \hat{\mC}^{(1)}_{ \gamma_{m}}
\bigg( \prod_{i =1}^{m-1}
\hat{\mC}_{\gamma_{i}} \bigg)
\Gamma\,.
\end{aligned}
\end{equation*}
Our goal is to argue that the valuation of
the right--hand side vanishes after integration, which together
with \eqref{eq_zimmer_chat} and \eqref{e_Rhat_one} implies
\eqref{e_bphzintervals_1}.

We fix $m$, and write $ \tilde{\Gamma} :=  \prod_{i =1}^{m-1}
\hat{\mC}_{\gamma_{i}} \Gamma \in  \mD_\Gamma^{\rm{lab}}$ (indeed, since no
differentiation has been applied $\tilde{\Gamma}$ consists of a single labelled
diagram, and not of a linear combination of them).
 Our objective
is then to prove that the following integral vanishes:
\begin{equation}\label{e:no-idea}
  \int_{ ( \mathbf{T}^{d})^{ V_{\star}}}
 \overline{\mW}_{\varepsilon}
\bigg( \prod_{i=m+1}^M
 (\hat{\mC}_{\gamma_{i}}+
\hat{\mC}_{\gamma_{i}}^{(1)}) \bigg)
 \hat{\mC}^{(1)}_{ \gamma_{m}}
\tilde\Gamma (x_{V_\star})
 \ud x_{ V_{\star} } = 0\;.
\end{equation}
From~\eqref{e:no-idea} we see that in $\tilde{\Gamma}$ the decoration $\mf{h}$ vanishes,
i.e. $\mf{h}
\equiv 0$, since no map $\hat{\mC}^{(1)}$ was ever applied. Moreover, because of
the ordering of the $\gamma_i$, the map $\hC_{\gamma_m} \tilde{\Gamma} \neq
\tilde{\Gamma}$ as an unlabelled diagram.
Therefore, by~\eqref{eq_def_Cbar1} the map
$\hat{\mC}_{\gamma_m}^{(1)}$ acts nontrivially.
Moreover, we note that
 edges of
the subdiagram $ \tilde\gamma_{m}$
of $\tilde{\Gamma}$ spanned by the preimage of the edges of $\gamma_m$,
i.e.
$\tilde{\gamma}_m = \tau_{\tilde{\Gamma}, \Gamma}^{-1} (\gamma)$,
are not affected by any of the further extraction maps
 $ \hat{\mC}_{\gamma_{k}}$ and $
\hat{\mC}^{(1)}_{\gamma_{k}}$ for $k > m$.
This is because those extractions only rewire
edges incoming into $\gamma_k$, but $ \tilde\gamma_{k}$ is either disjoint from $ \tilde\gamma_{m}$ or
fully contains $\tilde{\gamma}_m$, and therefore the incoming edges of any
$ \gamma_{k}$ does not belong
to $ \tilde{E} (\gamma_m)$.

Consequently,~\eqref{e:no-idea} contains the following subintegral, with
$v_{m} :=  e_{\star, +}(\gamma_{m}, \tilde{\Gamma})$ and $ \mf{a} (v_{m}) = v_{\star} (
  \gamma_{m})$:
\begin{equation}\label{eq_cancelfirstorder_1}
\begin{aligned}
&
\sum_{i =1}^{d}
\int_{ ( \mathbf{T}^{d})^{ V ( \gamma ) \setminus \mf{a}(v_{m})}}
\big( x^{i}_{ v_{m}} - x_{ \mf{a} ( v_{m} )}^{i}\big)
\prod_{e \in \tilde{E} ( \gamma_{m} )}
G_{\ve}(x_{e_{+}} - x_{e_{-}})
 \ud x_{ V ( \gamma_{m} ) \setminus \mf{a}(v_{m}) }\,.
\end{aligned}
\end{equation}
By Lemma~\ref{lem_sym} below, we conclude that the integral
\eqref{eq_cancelfirstorder_1} vanishes.
Hence, also the overall integral~\eqref{e:no-idea} vanishes as desired, which completes the
proof.

\end{proof}

\begin{lemma}\label{lem_sym}
Let $ f : ( \TT^{d})^{n+1} \to \RR$ be of the form
\begin{equation*}
\begin{aligned}
f(x_{0, \ldots, n})
=
\prod_{\{k,\ell\} \in I  } G_{\ve}( x_{\ell} - x_{k})
\end{aligned}
\end{equation*}
where $ I$ is a multiset of edges in the complete graph with vertices $\{0,
\ldots, n\}$. Then for any $i =1, \ldots, d $ and $ x_{0} \in \TT^{d}$
\begin{equation*}
\begin{aligned}
\int_{(\TT^{d})^{n}} ( x_{1}^{i}- x_{0}^{i}) f ( x_{0,\ldots, n}) \ud
x_{1 ,\ldots, n}
=0\,.
\end{aligned}
\end{equation*}
\end{lemma}

\begin{proof}
First, we note that
\begin{equation}\label{e:anti}
\begin{aligned}
\int_{(\TT^{d})^{n}} ( x_{1}^{i}- x_{0}^{i}) f ( x_{0,\ldots, n}) \ud
x_{1 ,\ldots, n}
=
\int_{(\TT^{d})^{n}}  x_{1}^{i} f ( 0, x_{1,\ldots, n}) \ud
x_{1 ,\ldots, n}\,,
\end{aligned}
\end{equation}
by recentering all integration variables around $ x_{0}$.
Next, let $ \uppi^{(i)} : \TT^{d} \to \TT^{d}$ be the map that flips the sign of the $i$--th coordinate,
and define the function
\begin{equation*}
\begin{aligned}
g ( x_{1})
:=
\int_{(\TT^{d})^{n-1}} f (0, x_{1,\ldots, n}) \ud x_{2 ,\ldots, n}
=
\int_{(\TT^{d})^{n-1}} \prod_{\{k, \ell\} \in I\setminus\{1\}  } G_{\ve}(
x_{\ell} - x_{k})
\prod_{ \substack{\{1,\ell\} \in I }} G_{\ve}( x_{\ell} - x_{1}) \ud x_{1 ,\ldots, n}
\,,
\end{aligned}
\end{equation*}
where we integrated out all variables except for  $ x_{0}= 0$ and $ x_{1}$.

If we can show that  $ g ( x_{1}) =g (
\uppi^{(i)}x_{1})$, then the statement of the lemma follows from the
antisymmetry of the integrand in~\eqref{e:anti}.
To prove that $ g ( x_{1}) =g (
  \uppi^{(i)}x_{1})$,
we rewrite the integral as follows:
\begin{equation*}
\begin{aligned}
g ( \uppi^{(i)} x_{1} )
&=
\int_{(\TT^{d})^{n-1}} \prod_{\{k, \ell\} \in I\setminus\{1\}  } G_{\ve}(
x_{\ell} - x_{k})
\prod_{ \substack{\{1,\ell\} \in I }} G_{\ve}( x_{\ell} - \uppi^{(i)} x_{1}) \ud x_{2
,\ldots, n}\\
& =
\int_{(\TT^{d})^{n-1}}  \prod_{\{k, \ell\}\in I\setminus\{1\} } G_{\ve}(
\uppi^{(i)}( \uppi^{(i)} x_{\ell} - \uppi^{(i)} x_{k}))
\prod_{ \substack{\{1,\ell\} \in I }} G_{\ve}(\uppi^{(i)} ( \uppi^{(i)} x_{\ell} -
x_{1})) \ud x_{2
,\ldots, n}\\
& =
\int_{(\TT^{d})^{n-1}} \prod_{\{k, \ell\}\in I\setminus\{1\} }  G_{\ve}(
 \uppi^{(i)} x_{\ell} - \uppi^{(i)} x_{k})
\prod_{ \substack{\{1,\ell\} \in I }} G_{\ve}(\uppi^{(i)} x_{\ell} - x_{1}) \ud x_{2
,\ldots, n}
\end{aligned}
\end{equation*}
where
 we used that $\uppi^{(i)} x_{0} = x_{0} \equiv 0$. Moreover, in the last equality we used that $ G_{\ve}( \cdot)$ is radially
symmetric.
Finally, we perform the change of variables $ x_{k} \mapsto
\uppi^{(i)} x_{k}$ in the above integral, which yields
$g ( \uppi^{(i)} x_{1} )
=
g (x_{1} ) $.
\end{proof}

\subsection{Forest intervals}

It is not immediate how to leverage on cancellations in Zimmermann's forest
formula, when expressed in terms of \eqref{e_Rhat_one}.
Hence, it will be convenient to rewrite the
expression.
To this end, we introduce \emph{forest
intervals}
$ \mathbb{M} =[ \underline{\mathbb{M}}, \overline{\mathbb{M}} ] \subset
\mathcal{F}_\Gamma^{-}$, for subsets $\underline{\mathbb{M}},
\overline{\mathbb{M}} \subset \mF_{\Gamma}^{-}$. More precisely, a forest interval
$[ \underline{\mathbb{M}}, \overline{\mathbb{M}} ]$ consists of all those
forests $ \mf{F} $ satisfying $\underline{\mathbb{M}}
\subset  \mf{F} \subset \overline{\mathbb{M}}$.
We write $ \delta (\mathbb{M}) :=
\overline{\mathbb{M}}\setminus \underline{\mathbb{M}}$.
The following result is a consequence of~\cite[Lemma~3.4]{BPHZ}.

\begin{lemma}\label{lem_pointwise_interval}
	For any negative Anderson--Feynman diagram $ \Gamma $ and any partition $ \mP$ of $
\mF_{\Gamma}^{-}$ into forest intervals, we can write
\begin{equation}\label{e_bphzintervals}
\begin{aligned}
\hat{\mR}^{(1)} \Gamma
=
\sum_{\mathbb{M} \in \mP}
\hat{\mR}_{\mathbb{M}}^{(1)} \Gamma\,,
\end{aligned}
\end{equation}
with
\begin{equation}\label{e_renorm}
\begin{aligned}
\hat{\mR}_{\mathbb{M}}^{(1)} \Gamma
:= \prod_{ \gamma \in \delta( \mathbb{M})} \big( \mathrm{id} - (
\hat{\mC}_{\gamma} + \hat{\mC}^{(1)}_{\gamma})\big)
\prod_{ {\gamma} \in \underline{\mathbb{M}}} (-1)( \hat{\mC}_{ {\gamma
}}
+\hat{\mC}^{(1)}_{ {\gamma}}) \Gamma\,.
\end{aligned}
\end{equation}
The order in both products is irrelevant in view of
Lemma~\ref{lem_C1_comm}.
\end{lemma}

Note that in the special case of the finest partition of singleton forests,
we recover Zimmermann's forest formula~\eqref{eq_zimmer_chat}.
A consequence of the above lemma is the following pointwise identity

\begin{corollary}\label{cor_ptw_forestinterval}
	For any negative Anderson--Feynman diagram $ \Gamma $, test function $ \varphi$, any
$x_{V_\star} \in (\TT^d)^{V_\star}$, and any partition $ \mP$ of $
\mF_{\Gamma}^{-}$ into forest intervals
\begin{equation}\label{eq_ptw_forestinterval}
\begin{aligned}
\overline{\mW}_{\varepsilon}
\hat{\mR}^{(1)} \Gamma (x_{V_{\star}})
=
\sum_{\mathbb{M} \in \mP}
\overline{\mW}_{\varepsilon}
\hat{\mR}_{\mathbb{M}}^{(1)} \Gamma (x_{V_{\star}})\,,
\end{aligned}
\end{equation}
where the dependence on $ \varphi$ is implicit in $
\overline{\mW}_{\ve}$ by~\eqref{e_updW_labelled}.
\end{corollary}


\subsection{Some properties of higher order renormalisation}

Before we proceed, let us collect some properties of the extraction maps
$\hC_\gamma, \hC^{(1)}_\gamma$. First we note that the maps $\hC_\gamma$ and
$\hC^{(1)}_\gamma$ act substantially different on the labels $\mf{n}$.
The map $\hC_\gamma$ from \eqref{e:Chatnew} only shifts the labels from their
original location to some new location (the total sum $\sum_v \mf{n}(v)$ remains
unchanged). On the contrary $\hC^{(1)}_\gamma$ does non shift any labels, but
only \lqm creates\rqm new ones at $v_\star(\gamma)$. Therefore, if we are given
a forest $\mf{F} \in \mF_\Gamma^-$ and a vertex $v \in
V_\star $ such that $|\mf{n}(v)|=1$ in some term $\tilde{\Gamma}$ of
$\mf{K}^{(1)}_{\mf{F}} \Gamma = \sum \tilde{\Gamma}$, there exists a unique
$\gamma \in \mf{F}$ such that $\mf{n}(v)$ was created by the application of
$\hC_{\gamma}^{(1)}$, in the sense just described.
We also observe that each term $\tilde{\Gamma}$ is of the form
\begin{equation*}
    \tilde{\Gamma} = \Gamma_n \;, \qquad \Gamma_l     =
    \hC_{\gamma_l}^{(\alpha_l), k_l} \cdots \hC_{\gamma_1}^{(\alpha_1), k_1} \Gamma
    \in \mD_{\Gamma}^{\rm{lab}} \;,
  \end{equation*}
  where $\mf{F} = \{\gamma_m\}_{m=1}^n$, $\alpha_m \in \{0, 1\}$ and with $\hC^{(
	  \alpha_{m}), k_m}_{\gamma_m} $ being one of the
  summands that defines $\hC_{\gamma_m}$ in \eqref{e:Chatnew} if $\alpha_m = 0$,
  and similarly one of the summands of $\hC^{(1)}_{\gamma_m}$ if $\alpha_m=1$ (i.e. $ k
  _{m}$ indexes either a multi--index $ \overline{\mf{n}}$ in $ \hC$ or a direction $i=1, \ldots ,d
  $ in $ \hC^{(1)}$ ).
  We have also assumed that the product
  is defined by applying the operators in order, starting from $\hC_{\gamma_1}^{(\alpha_1), k_1}$ (the order matters, as this product is not
  commutative). These observations and notations are used in the next lemma.
\begin{lemma}\label{lem:n-prop}
  Let $\mf{F} \in \mF^{-}_{\Gamma}$ be a forest of
   divergences with $|\mf{F}|=n$. Decompose $\mf{K}_{\mf{F}}^{(1)} \Gamma = \sum
   \tilde{\Gamma}$, for a finite number of nontrivial terms $\tilde{\Gamma} \in
   \mD^{\rm{lab}}_\Gamma$. Then for any such term $\tilde{\Gamma}$:
   \begin{enumerate}
    \item It holds that $|\mf{n} (v)| \leq 1$ for all $v \in V_{\star}$.
    \item If $|\mf{n}(v)|=1$, then there exists a unique $l \in \{1, \ldots, n\}$ such that one of the two holds:
    \begin{enumerate}
      \item Either $v = e_{\star, +}(\gamma_l, \Gamma_{l-1})$ and $\alpha_l=1$.
      \item Or $\mf{n}(v)$ was created by the application of $\hC_{\gamma_l}^{(1)}$
      in the sense described above, and there exists a (unique) $\gamma \in
      \mf{F}$ such that $v = v_\star(\gamma)$, $\gamma \subseteq \gamma_l$, and
      $e_{\star, +}(\gamma_l, \Gamma_{l-1}) \in V(\gamma)$. In this case, there
      exists no external edge (with respect to $\gamma$) incoming\footnote{There will be still an
      outgoing external edge.}
into $v$ in $\tilde{\Gamma}$, and $e_\star(\gamma, \Gamma) = e_\star(\gamma_l,
      \Gamma)$.
    \end{enumerate}
    In addition, $\mf{a}(v) = v_\star (\gamma')$ for some $\gamma'
    \subseteq \gamma_l$.
   \end{enumerate}
\end{lemma}
\begin{proof}
  The proof is by induction over $n = |\mf{F}|$. The statement is
  true for $n =0$, since in that case $\tilde{\Gamma} = \Gamma$ and $\mf{n}=0$.
  Assume the statement is true for $n-1 \in \NN$ and write $\tilde{\Gamma} =
  \hC_{\gamma_n}^{(\alpha_n), k_n} \Gamma_{n-1}$, so that the induction
  hypothesis applies to $\Gamma_{n-1}$. We must distinguish two cases, depending
  on the value of $\alpha_n$.

  If $\alpha_n=0$, then consider the set $A =V_{\mf{a}}(\gamma_n) \cap \{\mf{n}
  \neq 0\}$ (where the maps $\mf{n}, \mf{a}$ are considered with respect to
  $\Gamma_{n-1}$ and $V_{\mf{a}}$ is defined in \eqref{e:Va}). Our first objective is to prove that $|A| \leq 1$.  Note that if we set $\mf{F}_{n-1} = \mf{F} \setminus
  \{\gamma_n\}$, then
  \begin{equation}\label{e:a-representation}
    \mf{a}(\Gamma_{n-1}) (v) = \begin{cases}
      v_\star(\gamma)\,, & \quad \text{ if }  v \in \bigcup_{\eta \in \vr(\mf{F}_{n-1})} V(\eta)\setminus \{v_\star (\eta)\}\;, \\
      v\,, & \quad \text{ otherwise } \;.
    \end{cases}
  \end{equation}
  where $\gamma \in \mf{F}_{n-1}$ is the smallest (by inclusion) divergence of $\mf{F}_{n-1}$ such that $v
  \in V(\gamma) \setminus v_\star (\gamma)$. This follows by a much
  simpler induction argument over the size of the forest, by~\eqref{e:aChat}.
  Now suppose by contradiction that $|A| \geq 2$ and consider $v_1, v_2 \in A$.
  From the definition of $V_{\mf{a}}$ and by \eqref{e:a-representation}, we have
  that $\mf{a}(v_1) = \mf{a}(v_2) = v_\star(\overline{\gamma})$, where
  $\overline{\gamma}$ is the parent of $\gamma_n$ in $\mf{F}_{n}$ (the smallest
  divergence such that $\gamma_n \subseteq \overline{\gamma}$ and $\gamma_n \neq
  \overline{\gamma}$). Note that in principle the parent might not exist, if
  $\gamma_n$ is one of the roots of $\mf{F}_n$. However in this case $|A|=0$, so
  we can exclude this possibility.

  Now by the induction hypothesis, let $l_1, l_2$ be such that $\mf{n}(v_i)$ was
  created by the application of $\hC^{(1)}_{\gamma_{l_i}}$. 
  { We then have $\overline{\gamma} \subseteq \gamma_{l_i}$, because
$ v_{i} \in V ( \overline{\gamma}) \cap V ( \gamma_{ \ell_{i}})$ so $
\overline{\gamma}\subset \gamma_{ \ell_{i}}$ because $ \overline{\gamma}$ was the smallest
parent of $ \gamma_{n}$.
	  Hence, by the forest property
$\gamma_{l_1} \subseteq \gamma_{l_2}$ (up to swapping the labels $1$ and $2$). } 

  Moreover, note that we also necessarily have $e_\star(\gamma_{l_1}, \Gamma)
  \neq e_\star(\gamma_{l_2}, \Gamma)$, since otherwise we would have
  $\tilde{\Gamma} = 0$ by the first property of~\eqref{eq_def_Cbar1}. 
	  However,
  this contradicts the induction hypothesis, which implies that $e_{\star,
  +}(\gamma_{l_2}, \Gamma_{l_2-1}) \in V(\gamma_n)$.
    Indeed note that by induction either $v_2 = e_{\star,
    +}(\gamma_{l_2}, \Gamma_{l_2-1}) $ or there exists a $\gamma \in
    \mf{F}_{n-1}$ such that $v_\star(\gamma) = v_2$ and $e_{\star,
    +}(\gamma_{l_2}, \Gamma_{l_2-1}) \in V(\gamma)$. In this case we have
    $\gamma \subseteq \gamma_n$ since $v_2 \neq v_\star(\gamma_n)$ by definition
    of $V_{\mf{a}}$, and so $V(\gamma) \subseteq V(\gamma_n)$.
    Since $\gamma_n
  \subseteq \gamma_{l_2}$, the fact $e_{\star,
  +}(\gamma_{l_2}, \Gamma_{l_2-1}) \in V(\gamma_n)$ implies in turn that
  $e_{\star}(\gamma_{l_2}, \Gamma_{l_2-1}) =e_{\star}(\gamma_n,
  \Gamma_{l_2-1})$.   This is  because the number of incoming edges into any
  divergence of the forest $\mf{F} \cup \{\gamma_n\}$ (and in particular into
  $\gamma_n$) is decreasing with the application of any $\hC^{(\alpha_m),
  k_m}_{\gamma_m}$ and bounded by one by Lemma~\ref{lem:neg},
we would have
  deduced that $e_\star(\gamma_n, \Gamma) =
  e_\star(\gamma_{l_2}, \Gamma)$.
  Then, because $\gamma_n \subseteq \gamma_{l_1} \subseteq \gamma_{l_2}$,
  we deduce $e_{\star}(\gamma_{l_1}, \Gamma) = e_{\star}(\gamma_{l_2}, \Gamma) $, which is a contradiction and therefore $|A|
  \leq 1$ as desired.
  
  Moreover, we note that if $|A|=1$, then the preceding argument proves
  that $e_\star(\gamma_n, \Gamma) = e_\star(\gamma_l, \Gamma)$, $\gamma_n
  \neq\gamma_l, \gamma_n \subseteq \gamma_l$ and in particular
  $\hC_{\gamma_n}$ acts as the identity on $\Gamma_{n-1}$ as an unlabelled
  diagram (because there is no incoming edge into $\gamma_n$ as a subdiagram of $\Gamma_{n-1}$).

  We deduce that the sum appearing in \eqref{e:Chatnew} contains one term in
  which $\mf{n}$ is left invariant, and if $|A|=1$ a second term in which the
  weight of $\mf{n}(v)$ is shifted from $v$ to $v_\star(\gamma_n)$ (with $A =
  \{v\}$). In the first case, by \eqref{e:a-representation} the induction claim
  is immediately verified. In the second case, let us start by verifying that
  $|\mf{n}(\cdot, \Gamma_n)| \leq 1$. The only issue appears for $ v_{\star}(
  \gamma_{n})$.
  From the induction hypothesis we find $|\mf{n}(v_\star(\gamma_n), \Gamma_n)|
  \geq 2$ and the equality can only happen if  $|\mf{n}(v_\star(\gamma_n), \Gamma_{n-1})| = 1$.
  As before, this implies the existence of $\gamma_{l_1}, \gamma_{l_2}$ (one
  associated to $\mf{n}(v)$, the other to $\mf{n}(v_\star(\gamma_n))$) such that
  $\gamma_n \subseteq \gamma_{l_i}$, and it must be
  $e_\star(\gamma_{l_1}, \Gamma) \neq e_\star(\gamma_{l_2}, \Gamma)$, otherwise
  $\tilde{\Gamma}$ would be trivial. However, by induction we find
  $e_{\star}(\gamma_{l_2}, \Gamma) = e_{\star}(\gamma, \Gamma)$ and in the
  previous argument we have proven that $e_{\star}(\gamma_{l_1}, \Gamma) =
  e_{\star}(\gamma, \Gamma)$, which leads to a contradiction, since we would
  have $\tilde{\Gamma}=0$.
  To conclude checking the induction hypothesis for $\Gamma_n$ we observe that
  by the previous arguments $e_{\star}(\gamma_l, \Gamma) = e_\star(\gamma_n,
  \Gamma)$, so that indeed there is no incoming vertex into $v_\star(\gamma_n)$
  and all the desired properties hold with $\gamma=\gamma_n$ and $\gamma' =
  \overline{\gamma}$ the parent of $\gamma$ in $\mf{F}$.

  If $\alpha_n=1$ we must only prove that $|\mf{n}(e_{\star,+}(\gamma_n,
  \Gamma_{n-1}))| =0$. Indeed, we will prove that $|\mf{n}(e_{\star,+}(\gamma_n,
  \Gamma_{n-1}))| =1$ implies $\tilde{\Gamma}=0$. If  $|\mf{n}(e_{\star,+}(\gamma_n,
  \Gamma_{n-1}))| =1$, then $2.(a)$ or $2.(b)$ apply. In the case $2.(a)$ we
  would have $e_{\star,+}(\gamma_n,
  \Gamma_{n-1}) = e_{\star,+}(\gamma_l,
  \Gamma_{l-1})$ for some $l < n$, and by the forest property either $\gamma_n
  \subseteq \gamma_l$ or $\gamma_l \subseteq \gamma_n$. In both cases we deduce
  $e_\star(\gamma_l, \Gamma) = e_\star(\gamma_n, \Gamma)$ since the number of
  external incoming edges into a divergence is decreasing under successive extractions and bounded
  by one by Lemma~\ref{lem:neg} (see also the argument above). Therefore
  $\tilde{\Gamma}=0 $ by the first property in \eqref{eq_def_Cbar1}. A similar
  argument applies in the case $2.(b)$: let us consider $\gamma$ as in
  that statement.
  First we observe that $\gamma \subseteq \gamma_n$. Indeed, we have by
  assumption that $v_\star(\gamma) =  e_{\star,+}(\gamma_n, \Gamma_{n-1})\in
  V(\gamma_n)$ and therefore $\gamma \subseteq \gamma_n$ if $v_\star(\gamma)
  \neq v_\star(\gamma_n)$. If instead $v_\star(\gamma) = v_\star(\gamma_n)$,
  then we would have $\tilde{\Gamma}=0$ by the second property of \eqref{eq_def_Cbar1}.
  Therefore, we again find $\gamma
  \subseteq \gamma_n \subseteq \gamma_l$ or $\gamma
  \subseteq \gamma_l \subseteq \gamma_n$. But this leads to a contradiction of
  the induction hypothesis, since $v = v_\star(\gamma)$ would have an incoming
  external edge in the diagram $\Gamma_{n-1}$. Therefore,
  the claims of the induction hold for $\tilde{\Gamma}$. More precisely,
  $2.(b)$ holds with $l= n$ for $v = e_{\star, +}(\gamma_n, \Gamma_{n-1})$ and $\gamma'
  = \gamma_n$.

  Finally, the fact that there is no external (to $\gamma$) incoming edge into
  any vertex $v$ as in $2.(b)$ follows from the the
  fact that $e_\star(\gamma, \Gamma_{l-1}) = e_\star(\gamma_l, \Gamma_{l-1})$.
  This completes the proof.
\end{proof}
The next result shows that if we renormalise two nested divergences that share
the same incoming edge, then this is equivalent (after valuation) to
renormalising only the smaller one.

\begin{lemma}\label{lem:no_doubles}
   Let $\Gamma$ be an Anderson--Feynman diagram and $\gamma_1 \subseteq \gamma_2,
   \gamma_i \in \mG^{-}_\Gamma$,
 such that $e_{\star}(\gamma_1, \Gamma) =
   e_\star(\gamma_2, \Gamma)$. Let $\mf{F} \in \mF^{-}_{\Gamma}$ be a forest of
   divergences such that $\mf{F} \cup \{\gamma_1, \gamma_2\} \in
   \mF_\Gamma^{-}$. Decompose $\mf{K}_{\mf{F}}^{(1)} \Gamma = \sum
   \tilde{\Gamma}$, for a finite number of nontrivial terms $\tilde{\Gamma} \in
   \mD^{\rm{lab}}_\Gamma$. Then for any such term $\tilde{\Gamma}$:
 \begin{enumerate}
 \item It holds that
 \begin{equation}\label{e:nd-2}
    \overline{\mW}_\ve (\mathrm{id} - \hat{\mC}_{\gamma_{2}} -
 \hat{\mC}_{\gamma_{2}}^{(1)})(\mathrm{id} - \hat{\mC}_{\gamma_{1}} -
 \hat{\mC}_{\gamma_{1}}^{(1)}) \tilde{\Gamma} = \overline{\mW}_\ve (\mathrm{id} -
 \hat{\mC}_{\gamma_{1}} - \hat{\mC}_{\gamma_{1}}^{(1)}) \tilde{\Gamma} \;.
   \end{equation}
 \item Moreover,
 \begin{equation}\label{e:nd-4}
  \overline{\mW}_\ve( \hat{\mC}_{\gamma_{2}} +
 \hat{\mC}_{\gamma_{2}}^{(1)})(\mathrm{id} - \hat{\mC}_{\gamma_{1}} -
 \hat{\mC}_{\gamma_{1}}^{(1)}) \tilde{\Gamma} = 0 \;.
   \end{equation}
 \end{enumerate}
 Notice that the order of operations does not matter by Lemma~\ref{lem_C1_comm}.
 \end{lemma}
 \begin{proof}
  Since \eqref{e:nd-4} implies \eqref{e:nd-2}, we only prove the former. In the
  following, the maps $\mf{a}, \mf{h}, \mf{n}$ are all considered with respect
  to $\tilde{\Gamma}$, unless stated otherwise.
 Note that $e_{\star}(\gamma_1, \Gamma) =
   e_\star(\gamma_2, \Gamma)$ is equivalent to $e_{\star}(\gamma_1, \tilde{\Gamma}) =
   e_\star(\gamma_2, \tilde{\Gamma})$, and set $v = e_{\star, +}(\gamma_1, \tilde{\Gamma})$.

   We start by observing that if  $\mf{h}(e_{\star}(\gamma_1,
   \tilde{\Gamma})) \neq 0$, then by the first identity in
   \eqref{eq_def_Cbar1} we have $\hC^{(1)}_{\gamma_1}
   \tilde{\Gamma} = \hC^{(1)}_{\gamma_2} \tilde{\Gamma} = 0$, and therefore in this
   case it suffices to prove that:
   \begin{equation}\label{e:int-aim-1}
    \overline{\mW}_\ve  \hat{\mC}_{\gamma_{2}} \tilde{\Gamma} = \overline{\mW}_\ve \hC_{\gamma_2}
 \hat{\mC}_{\gamma_{1}}  \tilde{\Gamma} \;.
   \end{equation}
   Now let us define $A_i = V_{\mf{a}}(\gamma_i) \cap \{\mf{n} \neq 0\}$. Then
   following the same arguments of the proof of Lemma~\ref{lem:n-prop}, we have
   that $|A_i| \leq 1$. In particular, if $A_1 = \emptyset$, then
   $\hC_{\gamma_1}$ from \eqref{e:Chatnew} consists of only one term and we have
   $\hC_{\gamma_2} \hC_{\gamma_1} \tilde{\Gamma}= \hC_{\gamma_2}
   \tilde{\Gamma}$: Indeed,
   the identity holds at the level of unlabelled diagrams because the two divergences share an
   incoming edge. It extends to an identity between labelled diagrams because
   $V_{\mf{a}}(\gamma_2, \hC_{\gamma_1} \tilde{\Gamma}) \cap \{\mf{n}(\cdot
   ,\hC_{\gamma_1} \tilde{\Gamma}) \neq 0\}=V_{\mf{a}}(\gamma_2,  \tilde{\Gamma}) \cap \{\mf{n}(\cdot
   , \tilde{\Gamma}) \neq 0\}$ since $A_1 = \emptyset$ (we made explicit the diagrams with respect to
   which we are considering the labelling). Hence, in this case the
   desired identity holds even before applying $\overline{\mW}_{\ve}$. Next
   consider the case $A_1= \{v_1\}$, then we must distinguish the three further cases
   $A_2 = \emptyset, A_2 = \{v_2\}$ with $v_2 \neq v_1$, and $A_2 = \{v_1\}$.
   Define $\mf{n}_{12}(v) = \mf{n}(v)1_{V_\star \setminus \{v_1, v_2\}}(v)$ and
   set $\tilde{\Gamma}_{12}$ for the labelled diagram
   \begin{equation*}
    \tilde{\Gamma}_{12} = ( \hC_{\gamma_2}\tilde{\Gamma}, \mathfrak{h}, \mathfrak{n}_{12}, \mf{a}_{\gamma_2}) \;,
   \end{equation*}
   where in the first term we intend the action of $\hC_{\gamma_2}$ on unlabelled
   diagrams. Then we can rewrite:
   \begin{equation*}
    \overline{\mW}_\ve  \hat{\mC}_{\gamma_{2}} \tilde{\Gamma} (x_{V_\star}) = \begin{cases}
      (x_{v_1} - x_{\mf{a}(v_1)})^{ \mf{n}(v_{1})} (x_{v_2} -
x_{\mf{a}(v_2)})^{\mf{n}(v_{2})}  \overline{\mW}_\ve  \tilde{\Gamma}_{12} (x_{V_\star})  & \quad \text{ if } v_1 \neq v_2\;,\\
      (x_{v_1} - x_{\mf{a}(v_1)})^{\mf{n}(v_{1})}  \overline{\mW}_\ve  \tilde{\Gamma}_{12} (x_{V_\star}) & \quad \text{ else } \;,
    \end{cases}
   \end{equation*}
   where we used \eqref{e:Chatnew} and $(x_{v_2} - x_{\mf{a}(v_2)})^{
\mf{n}(v_{2})}
   = (x_{v_2} - x_{v_\star(\gamma_2)})^{\mf{n}(v_{2})}  + (x_{v_\star(\gamma_2)} -
x_{\mf{a}(v_2)})^{ \mf{n}(v_{2})}  $, (and the same for $v_1$). Here we also used that
$\mf{a}_{\gamma_2}(v_2) = v_\star(\gamma_2)$ and that if $v_1\neq v_2$ then $\mf{a}(v_1) \in
V(\gamma_2)$ since otherwise we would have $v_1 \in V_{\mf{a}}(\gamma_2)$. One obtains the
   desired identity by performing the same calculation on the right side of
   \eqref{e:int-aim-1}.

   Next, consider the case $\mf{h}(e_\star(\gamma_1, \tilde{\Gamma}))=0$. Then
   by Lemma~\ref{lem:n-prop}, and following the same arguments as in the proof of Lemma~\ref{lem:n-prop}, we
   have that $|A_i| = 0 $ for $i=1,2$, since in that proof we show that if
   $|A_i| >0$ then $e_\star(\gamma_i, \Gamma) = e_\star(\gamma_{m_{i}}, \Gamma)$ for
some $\gamma_{m_{i}} \in \mf{F}$ such that $\alpha_{m_{i}} =1$ and therefore
   $\mf{h}(e_\star(\gamma_i, \tilde{\Gamma}))\neq 0$. In this case, the result
   follows as in the case $ \mf{h}(e_\star(\gamma_1, \tilde{\Gamma}))\neq 0$
   and through the same calculations performed in the example
   \eqref{e:comm-example}. This completes the proof.
 \end{proof}

\subsection{Renormalisation expressed through a single diagram}

The aim of this section is to rewrite the expression~\eqref{eq_ptw_forestinterval}
in a convenient form, namely as a sum of valuations of single Feynman diagrams
in $ \mD_{\Gamma}^{\rm lab}$, where $ \Gamma= (V,E)$ is a negative Anderson--Feynman
diagram. In this representation, each edge will be associated with a
linear combination of kernels. This reformulation
simplifies and systematises the subsequent use of Taylor expansions.

First, we extend the kernel assignment $ \overline{\mW}_{\ve} $ of a labelled
Feynman diagram $ \tilde{\Gamma} =  (V, \tilde{E}) \in \mD_{\Gamma}^{\rm{lab}} $, which
previously only involved
Green's
functions $G_\ve$, into an arbitrary kernel assignment $\mK$. Namely, fix any
$\tilde{\Gamma} \in \mD_{\Gamma}^{\rm{lab}}$ as well as  an associated collection of
smooth kernels $\{\mK_{e, \ve}\}_{e \in {E}}$ such that $\mK_{e, \ve} \colon
(\TT^d)^{V} \rightarrow \RR$ (for later convenience we allow
the kernels to depend on all inner variables). We then define
\begin{equation}\label{e:def-WK}
  \overline{\mW}_{\ve}^{( \mK)} (  \tilde{\Gamma}, \mathfrak{h}, \mathfrak{n},  \mf{a}) (x_{V_\star})
  :=
  \int\limits_{ ( \TT^{d})^{L}}
  \varphi ( x_{L})
   \prod_{ e \in \tilde{E} }
  D^{\mathfrak{h} ( e)} \mK_{  \tau_{ \tilde\Gamma, \Gamma} (e) , \ve} (  x_{V})  \prod_{ v \in V_{\star}}
  \big( x_{ v } - x_{\mf{a}(v) }\big)^{\mathfrak{n}
  (v)} \ud x_{L}
\,,
\end{equation}
where derivatives of $ \mK_{\tau_{\tilde{\Gamma}, \Gamma}(e), \ve}$ are always taken with respect to $
x_{e_{+}}$, the variable associated to $ e_{+}$ in $ \tilde\Gamma $.
Moreover, we notice that the kernel assignment $\mK$ is actually a function of the unlabelled diagram,
rather than the labelled one.

We also introduce the operator $ \mH $ acting on smooth kernels $
K :\mathbf{T}^{d} \to \mathbf{R}$ by
\begin{equation} \label{e:hk-def}
  \begin{aligned}
  \mH K ( x, y | x_{\star}) &:=
  K ( x - y)
- K (x_{\star} -y) -  \sum_{i = 1}^{d}
(x^{(i)}- x_{\star}^{(i)})
\partial_{i}K ( x_{\star} - y)\\
  &=
  \sum_{i , j =1}^{d}
  (x^{i} - x_{\star}^{i})
  (x^{j} - x_{\star}^{j})
  \int_{0}^{1} (1-t)
  \partial^2_{x^{i} x^{j}} K
  ( x_{\star}
  + t( x- x_{\star}  )-  y)
  \ud t\,,
  \end{aligned}
  \end{equation}
meaning that $\mH K$ is the remainder of the first order Taylor expansion of $K(x,
y)$, as a function of $x$ around $K( x_{\star}- y)$.

The following lemma provides an alternative description of the valuation
\eqref{eq_ptw_forestinterval}.


\begin{lemma}[Peeling off divergences]\label{lem_pushonKernel}
Let
$\Gamma$ be an Anderson--Feynman diagram and $ \mathbb{M} = [ \underline{ \mathbb{M}}, \overline{\mathbb{M}}]$ be
an interval of forests in $ \mF_{ \Gamma}^{-}$.
Whenever $ \hat{\mR}^{(1)}_{\mathbb{M}}
\Gamma\neq 0 $, we have
\begin{equation}\label{eq_crux}
\begin{aligned}
\overline{\mW}_{\ve}
\hat{\mR}^{(1)}_{\mathbb{M}}
\Gamma ( x_{ V_{\star}})
=
\overline{\mW}_{\ve}^{( \mK )}  \mf{K}^{(1)}_{\underline{\mathbb{M}}}
\Gamma (x_{ V_{\star}})\,,
\end{aligned}
\end{equation}
where for any non trivial $ \tilde\Gamma \in \mD_{\Gamma}^{\rm{lab}}$ and any $ e \in E $,
with $ \tilde{e} = \tau^{-1}_{\tilde\Gamma , \Gamma}(e) \in \tilde{E}$,
\begin{equation} \label{e:def-K}
  \begin{aligned}
  \mK_{ e, \ve} ( x_{V} )
  :=
  \begin{cases}
  (\mH G_{\ve})( x_{ \tilde{e}_{+}},x_{ \tilde{e}_{-}}| x_{ \mf{b} ( e)}) \,, & \quad \text{if }
  \exists \gamma \in \delta( \mathbb{M}) \text{ with }
  e =e_{\star}( \gamma, \Gamma ) \,,\\
  G_{\ve} (
 x_{ \tilde{e}_{+}} - x_{\tilde{e}_{-}}
)  \,, & \quad \text{otherwise.}
  \end{cases}
  \end{aligned}
\end{equation}
Here,  $ \mf{b} ( e) :=
  v_{\star}( \theta)$ with $ \theta \in \delta( \mathbb{M})$
  denoting the smallest divergence (in the sense of inclusion) such that $
  e_{\star}(\theta, \Gamma) = e$.
  Whenever $ \tilde\Gamma$ is trivial, we set $ \mK \equiv 0$.
\end{lemma}

\begin{proof}
We begin by reducing the number of cases to be considered. We observe the
following regarding the valuation of $ \hat{\mR}^{(1)}_{\mathbb{M}}
\Gamma $ from \eqref{e_renorm}:
\begin{enumerate}
\item Using Lemma~\ref{lem:no_doubles} 1. and Lemma~\ref{lem_C1_comm}, we can assume that for all
  $\gamma_1, \gamma_2 \in \delta(\mathbb{M})$, such that $\gamma_1 \neq
  \gamma_2$, we have
\begin{equation*}
	e_{\star}(\gamma_1, {\Gamma}) \neq e_{\star}(\gamma_2, {\Gamma})  \;.
\end{equation*}
In other words, if there were multiple such divergences, we only keep the
smallest one (in the sense of inclusion).

\item Using  Lemma~\ref{lem:no_doubles} 2. and Lemma~\ref{lem_C1_comm},
we can assume that for all
  $\gamma_1, \gamma_2 \in \underline{\mathbb{M}}$, such that $\gamma_1 \neq \gamma_2$, we have
\begin{equation*}
	e_{\star}(\gamma_1, {\Gamma}) \neq e_{\star}(\gamma_2, {\Gamma})  \;.
\end{equation*}
In other words, if there were multiple such divergences, we only keep the
largest one (in the sense of inclusion).

\item If two divergences $ \gamma_1 \in \underline{\mathbb{M}}$ and $ \gamma_2 \in \delta ( \mathbb{M})$ share the same incoming edge, i.e. $ e_{\star}(\gamma_1,
{\Gamma}) = e_{\star}(\gamma_2, {\Gamma})$, then $ \gamma_1
\subset \gamma_2$. Otherwise
	\begin{equation*}
	\begin{aligned}
	 \overline{\mW}_\ve \hat{\mR}^{(1)}_{\mathbb{M}}
\Gamma=	\overline{\mW}_\ve ( \mathrm{id} - \hat{\mC}_{ \gamma_{2}}- \hat{\mC}_{
		\gamma_{2}}^{(1)}) ( \hat{\mC}_{\gamma_{1}}+
		\hat{\mC}^{(1)}_{\gamma_{1}})
		\hat{\mR}^{(1)}_{[\underline{\mathbb{M}} \setminus \{\gamma_{1}\} ,
		\overline{\mathbb{M}} \setminus \{ \gamma_{1}, \gamma_{2}\}]} \Gamma= 0
		\,,
	\end{aligned}
\end{equation*}
by Lemma~\ref{lem:no_doubles} 2., which we excluded by assumption.
\end{enumerate}
 Note that this simplification alone explains the choice of $\mf{b}$ appearing in the statement of the result.

In the following, we will treat each non--trivial diagram $ \tilde\Gamma \in  \mD_{\Gamma}^{\rm{lab}}$ separately in the decomposition
\begin{equation*}
\begin{aligned}
\mf{K}_{ \underline{\mathbb{M}}}^{(1)} \Gamma
=
\sum_{\tilde\Gamma} \tilde\Gamma \,.
\end{aligned}
\end{equation*}
The idea of the proof is to ``peel off'' renormalisation operators
$ (\mathrm{id} - \hat{\mC}_{\gamma} - \hat{\mC}^{(1)}_{\gamma})$ in $
\hat{\mR}^{(1)}_{[\emptyset, \delta({\mathbb{M}})]}$, starting from the outer
most operator, until we are left with $\tilde\Gamma$.
Every time we ``peel off'' an operator, we replace it with a suitable kernel assignment.
It will be convenient to first peel operators associated to larger divergences (in the
sense of inclusion), and proceed towards smaller ones.
We therefore fix an order $ \{ \gamma_{i} \}_{i =1}^{M}$ of divergences in $ \delta
( \mathbb{M})$, such that $i \leq j $ whenever $ \gamma_{i}\subseteq
\gamma_{j} $.

Thus, starting with one of the largest divergences, let us consider
 a component $ \hat{\Gamma} = (
\hat{\Gamma}, \hat{\mf{h}}, \hat{\mf{n}}, \hat{\mf{a}}) \in
\mD_{\Gamma}^{\rm{lab}}$ of  the linear combination of diagrams
\begin{equation*}
\begin{aligned}
	\hat{\mR}^{(1)}_{[\emptyset, \delta({\mathbb{M}})\setminus \gamma_{M}]} \tilde\Gamma	=
\prod_{i =1}^{M-1}
(\mathrm{id} - \hat{\mC}_{\gamma_{i}} -
\hat{\mC}^{(1)}_{\gamma_{i}})\tilde\Gamma\,,
\end{aligned}
\end{equation*}
and ``peel off'' the operator $  \big(\mathrm{id} -  \hat{\mC}_{\gamma_{M}} -
\hat{\mC}^{(1)}_{\gamma_{M}} \big) $.
First,
when applying  that $  \big(\mathrm{id} -  \hat{\mC}_{\gamma_{M}} -
\hat{\mC}^{(1)}_{\gamma_{M}} \big) $ to $ \hat{\Gamma}$, we observe
the sum over
vertex labels in \eqref{e:Chatnew} will only contain the term $
\overline{\mf{n}}
\equiv 0 $, because
\begin{equation}\label{eq_supp1_Va_empty}
\begin{aligned}
V_{ \hat{\mf{a}}}( \gamma_{M}) \cap \{ v \, : \, \hat{\mf{n}} (v)\neq
0 \} = \emptyset\,.
\end{aligned}
\end{equation}
To verify \eqref{eq_supp1_Va_empty}, we notice that for a vertex $ v \in V ( \gamma_{M}) \setminus
\{ v_{\star}( \gamma_{M}) \} $ to
satisfy $ \hat{\mf{a}}(v) \notin V ( \gamma_{M})$, we must have
that
\begin{equation}\label{eq_supp_nexist}
\begin{aligned}
&\nexists \ \gamma' \in \overline{\mathbb{M}}\setminus \gamma_{M}\,:\, \gamma'
\subsetneq \gamma_{M}\,, \  v \in V(\gamma') \setminus v_\star(\gamma')
\ \text{ and }\   v_{\star} ( \gamma') \in V ( \gamma_{M})\,.
\end{aligned}
\end{equation}
Now, if $ \hat{\mf{n}}(v) \neq 0$, then by Lemma~\ref{lem:n-prop} there must
exist a $ \gamma' \in
\overline{\mathbb{M}} \setminus \gamma_{M}$ such that

\begin{enumerate}
\item
	Either $ v = e_{\star , +} ( \gamma', \Gamma) \in V ( \gamma_{M}) \setminus \{
	v_{\star} ( \gamma_{M })\}$. Therefore, by the forest property:
	\begin{enumerate}
		\item[(a)] Either $ \gamma' \subsetneq \gamma_{M}$. However, this leads
			to a
			contradiction to \eqref{eq_supp_nexist}, because $ v  = e_{\star , +} ( \gamma', \Gamma)  \in V (
			\gamma' ) \setminus v_{\star} ( \gamma' ) $ and $ v_{\star}(
			\gamma' ) \in V ( \gamma' ) \subset V ( \gamma_{M})$.
		\item[(b)] Or $ \gamma_{M}\subsetneq \gamma' $. Then necessarily $
			v= e_{\star, +} ( \gamma' , \Gamma) = e_{\star , +}(
			\gamma_{M}, \Gamma) $. However, this case is  not possible
			because
			\begin{itemize}
\item Either $ \gamma' \in \delta( \mathbb{M})$, then $ \gamma_{M} \subsetneq
\gamma' $ is not possible due to our ordering.
\item Or $ \gamma' \in \underline{\mathbb{M}}$, then 3. above excludes this
case.
\end{itemize}
	\end{enumerate}
\item Or $ v = v_{\star}( \gamma' ) \in V ( \gamma_{M}) \setminus \{
	v_{\star}( \gamma_{M}) \}$.
Therefore, by the forest property:
\begin{enumerate}
	\item[(a)] Either $ \gamma' \subsetneq
\gamma_{M}$. Then by Lemma~\ref{lem:n-prop} 2. there must exist a divergence
	$ \gamma'' \in \overline{\mathbb{M}} \setminus \gamma_{M}$ (where the label $
	\hat{\mf{n}}(v)$ was created)
	such that $ \gamma'
	\subsetneq \gamma'' $, $ e_{\star ,+}( \gamma' , \Gamma) = e_{\star ,+}(
	\gamma'', \Gamma) $, and thus $ v_{\star}( \gamma' ) \neq v_{\star}( \gamma'')$. By the forest property, we must have
	\begin{itemize}
		\item Either $ \gamma'' \subsetneq \gamma_{M}$. However, this
			contradicts \eqref{eq_supp_nexist}, since $ v= v_{\star}( \gamma'
			)\in V ( \gamma''
			)\setminus v_{\star}( \gamma'') $ and $ v_{\star } ( \gamma'')
			\in V ( \gamma_{M})$.
		\item Or $ \gamma_{M} \subsetneq \gamma''$. Then necessarily $
			e_{\star , +}( \gamma_{M}, \Gamma) \neq e_{\star,+} ( \gamma'',
			\Gamma)= e_{\star, +}( \gamma' , \Gamma) \notin V (
			\gamma_{M}) $ by our induced ordering and 3. above.
			However, this implies that $ \gamma' \nsubseteq \gamma_{M}$
			contradicting our assumption (and the forest property).
	\end{itemize}
\item[(b)] Or $ \gamma_{M}\subsetneq \gamma' $.
	Then the outgoing edge of $\gamma'$ is also outgoing for $\gamma_M$ (in $
	\Gamma$), because
	necessarily $ v_{\star}( \gamma') = v_{\star}( \gamma_{M})$. However, we assumed
	that $v_\star(\gamma')\neq v_\star(\gamma_M)$.
\end{enumerate}



\end{enumerate}
We conclude that no such $ \gamma' \in \mathbb{M} $ exists and therefore
\eqref{eq_supp1_Va_empty} holds true.

Next, we argue that the action of $  \big(\mathrm{id} - \hat{\mC}_{\gamma_{M}} -
\hat{\mC}^{(1)}_{\gamma_{M}} \big)$ leaves monomials represented by $
\hat{\mf{n}}, \hat{\mf{a}}$ invariant.
Clearly this is true when applying the identity, but in view of the discussion
above we also have by \eqref{e:Chatnew} that
\begin{equation*}
\begin{aligned}
\hat{\mC}_{\gamma_{M}}
( \hat{\Gamma}, \hat{\mf{h}}, \hat{\mf{n}}, \hat{\mf{a}})
= ( \hat{\mC}_{\gamma_{M}} \hat{\Gamma}, \hat{\mf{h}}, \hat{\mf{n}},
\hat{\mf{a}}_{\gamma_{M}})
\end{aligned}
\end{equation*}
and provided that $ \hat{\mC}^{(1)}_{\gamma_{M}}$ in \eqref{eq_def_Cbar1} acts
non-trivially (in particular $ \hat{\mf{h}}( e_{\star}(
\gamma_{M}, \hat{\Gamma}))=0$), that
\begin{equation*}
\begin{aligned}
\hat{\mC}^{(1)}_{\gamma_{M}}
( \hat{\Gamma}, \hat{\mf{h}}, \hat{\mf{n}}, \hat{\mf{a}})
= \sum_{i=1}^d
\big( \hat{\mC}_{\gamma_{M}} \hat{\Gamma},
\hat{\mathfrak{h}} +  \mathds{1}_{
(
  e_{\star, -}(\gamma_{M}, \hat{\Gamma}), v_{\star}  ( \gamma_{M}) ) } \delta_{i} \;,
\hat{\mathfrak{n}}+
\mathds{1}_{ e_{\star, +}(\gamma_{M}, \hat{\Gamma})}\delta_{i} \;,
\hat{\mf{a}}_{\gamma_{M}}\big)\,.
\end{aligned}
\end{equation*}
Valuating any of the diagrams in the two previous displays, we see
the monomials (omitting the new monomial created by $
\hat{\mC}^{(1)}_{\gamma_{M}}$)
\begin{equation}\label{eq_supp_mon_oinv}
\begin{aligned}
\prod_{ v \in V_{\star}}
\big( x_{ v } - x_{ \hat{\mf{a}}_{\gamma_{M}}(v) }\big)^{ \hat{\mathfrak{n}}
(v)}
=
\prod_{ v \in V_{\star}}
\big( x_{ v } - x_{ \hat{\mf{a}}(v) }\big)^{ \hat{\mathfrak{n}}
(v)}\,,
\end{aligned}
\end{equation}
where we used \eqref{eq_supp1_Va_empty} in combination with the definition of the map $
\hat{\mf{a}}_{\gamma_{M}}$ \eqref{e:aChat}.\\

The following graphical representation of the renormalisation may be helpful in what follows
(we display only the divergent subgraph $\gamma$, together with its incoming and outgoing edges):
\begin{equation*}
    (\rm{id} - \hC_\gamma - \hC_{\gamma}^{(1)}) {\Gamma} =
    \begin{tikzpicture}[thick,>=stealth,scale=0.3,baseline=-0.2em]
      \tikzset{
        dot/.style={circle,fill=black,inner sep=1.5pt}
      }
      \node[above] at (-3.2,0.1) {};
      \node[above] at ( 3.0,0.1) {};
      \draw (-0.4,0) ellipse (1.8 and 1.1);
      \node at (-0.4,0) {${\gamma}$};
      \node[dot] (bL) at (-2.2,0)   {};   
      \node[dot] (vL) at (1.4,0) {};   
      \draw[->] (-4.2,0) -- (bL);
      \draw[->] (vL) -- (3.5,0);
      \node[below right] at (vL) {$v_\star$};
    \end{tikzpicture}
    \; - \;
    \begin{tikzpicture}[thick,>=stealth,scale=0.3,baseline=-0.2em]
      \tikzset{
        dot/.style={circle,fill=black,inner sep=1.5pt}
      }
      \node[above] at (-3.2,0.1) {};
      \node[above] at ( 3.0,0.1) {};
      \draw (-0.4,0) ellipse (1.8 and 1.1);
      \node at (-0.4,0) {${\gamma}$};
      \node[dot] (bL) at (-2.2,0)   {};   
      \node[dot] (vL) at (1.4,0) {};   
      \draw[->] (3.5,2.8  ) -- (vL);
      \draw[->] (vL) -- (3.5,-2.0);
      \node[right] at (vL) {$v_\star$};
    \end{tikzpicture}
    \; - \; \sum_{i=1}^d
    \begin{tikzpicture}[thick,>=stealth,scale=0.3,baseline=-0.2em]
      \tikzset{
        dot/.style={circle,fill=black,inner sep=1.5pt}
      }
      \node[above] at (-3.2,0.1) {};
      \node[above] at ( 3.0,0.1) {};
      \draw (-0.4,0) ellipse (1.8 and 1.1);
      \node at (-0.4,0) {${\gamma}$};
      \node[dot] (bL) at (-2.2,0)   {};   
      \node at (-2.9,0.3) {\scalebox{0.8}{$\delta_i$}};
      \node[dot] (vL) at (1.4,0) {};   
      \draw[->] (3.5,2.8  ) -- (vL);
      \node at (1.9,2.0) {\scalebox{0.8}{$\delta_i$}};
      \draw[->] (vL) -- (3.5,-2.0);
      \node[right] at (vL) {$v_\star$};
    \end{tikzpicture}  \;.
  \end{equation*}
Now, the fact that previous monomials are not affected by the renormalisation of $
\gamma_{M}$, see \eqref{eq_supp_mon_oinv},
allows us to write
\begin{equation}\label{eq_firsttaylor_supp1}
\begin{aligned}
\overline{\mW}_{\ve} \big( & \mathrm{id} - ( \hat{\mC}_{\gamma_{M}} +
\hat{\mC}^{(1)}_{\gamma_{M}}) \big)
\hat{\Gamma} ( x_{ V_{\star}}) \\
& =
\Big(
D^{ \hat{\mathfrak{h}} ( e_{\star})} G_{\ve}(x_{e_{\star,+}} - x_{
e_{\star,-}})
\\
& \qquad \qquad  -
D^{ \hat{\mathfrak{h}} ( e_{\star})} G_{\ve}(x_{ v_{\star}( \gamma_{M})}- x_{
e_{\star,-}})
-
\mathds{1}_{ \hat{\mathfrak{h}} (e_{\star}) =0}
\sum_{i =1}^{d}
\partial_{ i}
G_{\ve}(x_{ v_{\star}( \gamma_{M})}- x_{
e_{\star,-}})
( x_{ e_{\star,+}}^{(i)} -
x_{ v_{\star}( \gamma_{M})}^{(i)} )
\Big)\\
&\quad \times
\int_{ ( \TT^{d})^{L}}
\varphi ( x_{L})
 \prod_{e \in \hat{E} \setminus e_{\star}}
D^{ \hat{\mathfrak{h}} ( e)} G_{\ve}(x_{e_{+}} - x_{e_{-}})
\prod_{ v \in V_{\star}}
\big( x_{ v } - x_{ \hat{\mf{a}}(v) }\big)^{ \hat{\mathfrak{n}}
(v)}
\ud x_{L} \\
& =
D^{ \hat{\mathfrak{h}} ( e_{\star})}
(\mH  G_{\ve})
\big( x_{ e_{\star,+}}, x_{ e_{\star,-}} | x_{ v_{\star}(
\gamma_{M})}\big) \\
& \quad \times
\int_{ ( \TT^{d})^{L}}
\varphi ( x_{L})
 \prod_{e \in \hat{E} \setminus e_{\star}}
D^{ \hat{\mathfrak{h}} ( e)} G_{\ve}(x_{e_{+}} - x_{e_{-}})
\prod_{ v \in V_{\star}}
\big( x_{ v } - x_{ \hat{\mf{a}}(v) }\big)^{ \hat{\mathfrak{n}}
(v)}
\ud x_{L}\,,
\end{aligned}
\end{equation}
where we used the shorthand
$ e_{\star} = e_{\star} (
\gamma_{M},\tilde\Gamma)$
(note that $ e_{\star}( \gamma_{M},  \hat{\Gamma}) =  e_{\star} ( \gamma_{M},
\tilde\Gamma)$, because of 1. in the beginning of the proof).
In the last step of \eqref{eq_firsttaylor_supp1}, we used that
\begin{equation*}
\begin{aligned}
D
(\mH  G_{\ve})
\big( x_{ e_{\star,+}}, x_{ e_{\star,-}} | x_{ v_{\star}(
\gamma_{M})}\big)
=
D  G_{\ve}(x_{e_{\star,+}} - x_{
e_{\star,-}})
-
D G_{\ve}(x_{ v_{\star}( \gamma_{M})}- x_{
e_{\star,-}})\,,
\end{aligned}
\end{equation*}
which is an immediate consequence of
\begin{equation}\label{eq_derive_mH}
\begin{aligned}
\partial_{i} \mH K ( x, y |  x_{\star} )
&= \frac{\ud}{ \ud x_{i}}  \Big(K (x-y) - K ( x_{\star}- y ) - \sum_{j =1}^{d}
(x^{(j)} - x_{\star}^{(j)})\partial_{j} K (x_{\star}, y)\Big)\\
&=
\partial_{i} K ( x - y )
-
\partial_{i} K (x_{\star} - y)\,.
\end{aligned}
\end{equation}
Hence, the right--hand side of \eqref{eq_firsttaylor_supp1} can be recast into
\begin{equation*}
\begin{aligned}
\overline{\mW}_{\ve}^{(\mK)}
\hat{\Gamma} ( x_{ V_{\star}})
=
\int_{ ( \TT^{d})^{L}}
\varphi ( x_{L})
 \prod_{e \in \hat{E} }
D^{ \hat{\mathfrak{h}} ( e)} \mK_{ \tau_{ \hat{\Gamma}, \Gamma}(e),\ve}(x_{ V_{\star} } )
\prod_{ v \in V_{\star}}
\big( x_{ v } - x_{ \hat{\mf{a}}(v) }\big)^{ \hat{\mathfrak{n}}
(v)}
\ud x_{L} \,,
\end{aligned}
\end{equation*}
with the kernel assignment
\begin{equation}\label{eq_kernelassign_first}
\begin{aligned}
\mK_{ \tau_{ \hat{\Gamma}, \Gamma}(e), \ve} ( x_{ V_{\star}}) =
\begin{cases}
\mH G_{\ve} ( x_{e_{+}}, x_{e_{-}} | x_{ v_{\star}( \gamma_{M}) }) \,, \quad &
\text{if }\ \tau_{ \hat{\Gamma}, \Gamma }(e)
= e_{\star} ( \gamma_{M}, \Gamma)\,, \\
G_{\ve} ( x_{e_{+}} - x_{e_{ -}}) \,, \quad & \text{otherwise.}
\end{cases}
\end{aligned}
\end{equation}
Notice that the kernel assignment in the display above is independent of
$\hat{\Gamma}$, but only depends on $\Gamma$.
In particular, it is independent of $ \hat{\mf{n}}$ and $ \hat{\mf{a}}$.
Hence, we can write
\begin{equation*}
\begin{aligned}
\overline{\mW}_{\varepsilon}
\hat{\mR}_{[ \emptyset, \delta(\mathbb{M})]}^{(1)}
\tilde\Gamma
(x_{V_{\star}})
=\overline{\mW}_{\varepsilon}^{( \mK)}
\hat{\mR}_{[ \emptyset, \delta(\mathbb{M})\setminus \gamma_{M}]}^{(1)}
\tilde\Gamma (x_{V_{\star}})\,.
\end{aligned}
\end{equation*}
Informally, we pushed the renormalising action from $ ( \mathrm{id} -
\hat{\mC}_{\gamma_{M}}- \hat{\mC}_{\gamma_{M}}^{(1)}) $
onto the kernel assignment $ \overline{\mW}^{ (\mK )}$.
We now proceed by ``peeling off'' the subsequent renormalisations associated to $ \gamma_{M-1}$, $ \gamma_{M-2}$, and so on, until
we reach $ \tilde\Gamma $.
It is therefore necessary to repeat the steps above with
the difference that now Green's functions $ G$ in \eqref{eq_firsttaylor_supp1} are replaced
by a general kernel assignment $\mK$.
However, one can check that all the previous steps still apply when considering the action of
 $ ( \mathrm{id} -
\hat{\mC}_{\gamma_{M-1}} - \hat{\mC}^{(1)}_{\gamma_{M-1}}) $ on any of the
components $ \hat{\Gamma} \in \mD_{\Gamma}^{\rm{lab}}$ in the linear combination
\begin{equation}\label{eq_gammahat_nextstep}
\begin{aligned}
	\hat{\mR}^{(1)}_{[\emptyset, \delta ( \mathbb{M})\setminus \{ \gamma_{M}, \gamma
	_{M-1}\}]}\tilde\Gamma
	=
\prod_{i =1}^{M-2}
(\mathrm{id} - \hat{\mC}_{\gamma_{i}} -
\hat{\mC}^{(1)}_{\gamma_{i}})\tilde\Gamma\,.
\end{aligned}
\end{equation}
Indeed,
the statement \eqref{eq_supp1_Va_empty} remains true, when replacing $
\gamma_{M}$ by $ \gamma_{M-1}$ and $ \hat{\Gamma}$ with a component of
\eqref{eq_gammahat_nextstep}, by the same argument.
Moreover, we notice that the edge $ e_{\star}( \gamma_{M}, \tilde\Gamma)$ is not affected by the action of $
\hat{\mC}_{\gamma_{M-1}}$ because either
$ \gamma_{M-1} \subset \gamma_{M}$ (such that they
don't share the same incoming edge), or $ \gamma_{M-1}$ is disjoint from $
\gamma_{M}$. Therefore, an analogue of \eqref{eq_supp_mon_oinv} still holds, and
 the identity \eqref{eq_firsttaylor_supp1} remains true when replacing $
\gamma_{M}$ with $ \gamma_{M-1}$, and Green's functions with the kernel assignment
$ \mK$ in \eqref{eq_kernelassign_first}, because it only relies on
\eqref{eq_derive_mH}.
Note that the centring of the
monomials in \eqref{eq_derive_mH} is determined by the map $\hat{\mf{a}}$ even
after any extraction.

Overall,  after ``peeling off'' the renormalistion with respect to $ \gamma_{M-1}$, this
yields
\begin{equation*}
\begin{aligned}
\overline{\mW}_{\varepsilon}
\hat{\mR}_{[ \emptyset,\delta(\mathbb{M})]}^{(1)}
\tilde\Gamma
(x_{V_{\star}})
=\overline{\mW}_{\varepsilon}^{( \mK)}
\hat{\mR}_{[ \emptyset, \delta(\mathbb{M}) \setminus \{ \gamma_{M},
\gamma_{M-1}\}]}^{(1)}
\tilde\Gamma (x_{V_{\star}})\,,
\end{aligned}
\end{equation*}
with
\begin{equation*}
\begin{aligned}
\mK_{ \tau_{ \hat{\Gamma}, \Gamma}(e), \ve} ( x_{ V_{\star}}) =
\begin{cases}
\mH G_{\ve} ( x_{e_{+}}, x_{e_{-}} | x_{ v_{\star}( \gamma_{M-1}) }) \,, \quad
& \text{if }\ \tau_{ \hat{\Gamma}, \Gamma}(e)
= e_{\star} ( \gamma_{M-1}, {\Gamma})\,, \\
\mH G_{\ve} ( x_{e_{+}}, x_{e_{-}} | x_{ v_{\star}( \gamma_{M}) }) \,, \quad & \text{if }\
\tau_{ \hat{\Gamma}, \Gamma} (e)
= e_{\star} ( \gamma_{M}, {\Gamma})\,, \\
G_{\ve} ( x_{e_{+}} - x_{e_{ -}}) \,, \quad & \text{otherwise.}
\end{cases}
\end{aligned}
\end{equation*}
 Now, proceeding in the same fashion for all remaining divergences until
exhausting  $ \delta
(\mathbb{M})$, we ultimately arrive at
\begin{equation*}
\begin{aligned}
\overline{\mW}_{\ve}
\hat{\mR}^{(1)}_{[ \emptyset, \delta(\mathbb{M})]}
\tilde\Gamma ( x_{ V_{\star}})
= \overline{\mW}_{\ve}^{( \mK )}  \tilde\Gamma (x_{ V_{\star}})\,,
\end{aligned}
\end{equation*}
where for $ e \in \tilde{E}$
\begin{equation}
  \begin{aligned}
  \mK_{ \tau_{ \tilde\Gamma, \Gamma}(e) , \ve} ( x_{V} )
  :=
  \begin{cases}
  (\mH G_{\ve})( x_{ e_{+}},x_{e_{-}}| x_{ \mf{b} ( e)}) \,, & \quad \text{if }
  \exists \gamma \in \delta( \mathbb{M}) \text{ with }
  \tau_{\tilde\Gamma, \Gamma}(e)=e_{\star}( \gamma, \Gamma ) \,,\\
  G_{\ve} ( x_{ e_{+}}- x_{e_{-}})  \,, & \quad \text{otherwise.}
  \end{cases}
  \end{aligned}
\end{equation}

Finally, to complete the proof, we observe that the kernel assignment
$\mK$ depends only on the unlabelled version of the diagram
$\Gamma$ and on $ \mathbb{M} $. It depends on $\tilde{\Gamma}$
only in so far as it determines which variables in $ x_{V_{\star}}$
are taken as arguments. In any case, it depends solely on the unlabelled
version of $ \tilde{\Gamma} $, which is the same for all components of
$\mf{K}^{(1)}_{ \underline{\mathbb{M}}} \Gamma$.
Hence,
\begin{equation*}
\begin{aligned}
\overline{\mW}_{\ve}
\hat{\mR}^{(1)}_{\mathbb{M}}
\Gamma ( x_{ V_{\star}})
=
\sum_{\tilde\Gamma}
\overline{\mW}_{\ve}
\hat{\mR}^{(1)}_{[ \emptyset, \delta(\mathbb{M})]}
\tilde\Gamma ( x_{ V_{\star}})
=
\sum_{\tilde\Gamma}
\overline{\mW}_{\ve}^{( \mK )}  \tilde\Gamma (x_{ V_{\star}})
=
\overline{\mW}_{\ve}^{( \mK )} \mf{K}^{(1)}_{\underline{\mathbb{M}}} \Gamma (x_{ V_{\star}})
\,,
\end{aligned}
\end{equation*}
where in the last step we used that $\mK$ is independent of the specific choice
of $ \tilde\Gamma$, as explained above.
This concludes the proof.
\end{proof}

Combining Lemma~\ref{lem_include_C1}, Corollary~\ref{cor_ptw_forestinterval}, and Lemma~\ref{lem_pushonKernel} finally yields
the following representation, which conveniently translates linear combinations
of Feynman diagrams into linear combinations of kernels.

\begin{corollary}\label{cor_rep_singlediagram}
For any negative Anderson--Feynman diagram $ \Gamma $ and any family of
partitions $ (\mP_{x_{V_{\star}}} )_{x_{V_{\star}} \in (
\TT^{d})^{V_{\star}}}$
of $ \mF_{\Gamma}^{-}$ into forest intervals
\begin{equation*}
\begin{aligned}
\hat{\Pi}_{\varepsilon} \Gamma
( \varphi)
=
\int_{ ( \mathbf{T}^{d})^{ V_{\star}}}
\sum_{\mathbb{M} \in \mP_{x_{V_{\star}}}}
\overline{\mW}_{\ve}^{( \mK )}  \mf{K}^{(1)}_{\underline{\mathbb{M}}}
\Gamma (x_{ V_{\star}})
 \ud x_{ V_{\star} }\,,
\end{aligned}
\end{equation*}
where $ \mK = \mK ( \mathbb{M})$ was defined in Lemma~\ref{lem_pushonKernel}.
\end{corollary}

\subsection{Safe and unsafe divergences}

In the previous section, we augmented the renormalisation map $ \hat{\mR}$ with
first order terms, that left the valuation of renormalised
Anderson--Feynman diagrams unchanged. Therefore, as a consequence of
Corollary~\ref{cor_rep_singlediagram} and
Lemma~\ref{lem_hepp_almost_covers}, we deduce the following upper bound for a
renormalised Anderson--Feynman diagram:
\begin{equation}\label{eq_est_Hepp}
\begin{aligned}
| \hat{\Pi}_{\varepsilon} \Gamma  ( \varphi)|
& \leqslant
\sum_{T \in \mT_{V_{\star}}}
\sum_{{\bf n} \in \mA ( T)}
\sum_{\mathbb{M} \in \mP_{T}}
\int_{ D_{( T, \mathbf{n})}}
\big|
 \overline{\mW}_{\varepsilon}^{( \mK)}
\mathfrak{K}_{ \underline{\mathbb{M}}}^{(1)}
\Gamma
(x_{V_{\star}})
\big|
 \ud x_{ V_{\star} }\\
& \leqslant
\sum_{T \in \mT_{V_{\star}}}
\sum_{{\bf n} \in \mA ( T)}
\sum_{\mathbb{M} \in \mP_{T}}
\sup_{x_{V_{\star}} \in D_{( T, \mathbf{n})}}
\big|
 \overline{\mW}_{\varepsilon}^{( \mK)}
\mathfrak{K}_{ \underline{\mathbb{M}}}^{(1)}
 \Gamma
(x_{V_{\star}})
\big|
\prod_{u \in \overline{T}} 2^{-d \mathbf{n} (u)}
\,,
\end{aligned}
\end{equation}
with $ \mK = \mK (\mathbb{M})$ as in~\eqref{e:def-K}, $\overline{\mW}^{(\mK)}$ as in~\eqref{e:def-WK},
and $ \mP_{T}$ any partition into forest intervals (which we allow to depend on the choice of Hepp tree $T$).

In this section, we follow the classical approach of considering safe and unsafe
forest intervals to obtain a partition $ \mP_{T}$ that will
allow us
to efficiently bound~\eqref{eq_est_Hepp}.
We will see that the bound \eqref{eq_est_Hepp} is too crude for our purposes.
At the end of this section, we will refine it in order to exploit certain symmetries.
However, for the sake of the exposition, we first focus on estimating
the right--hand side of \eqref{eq_est_Hepp}.

Before introducing safe and unsafe divergences, we recall that the rewiring of edges via $ \hat{\mC}_{\gamma}$ leads to
the valuation of diagrams $\tilde\Gamma \in  \mD_{\Gamma}^{\rm{lab}}$, whose scale assignment
is provided by a Hepp sector $ (T, \mathbf{n})$ associated to the original diagram $
\Gamma$. The fact that we consider edges in the new $\tilde{\Gamma}$, but
Hepp sectors associated to the original $\Gamma$, can be a source of confusion. Given an edge $e \in \tilde{E}$, the associated lengthscale will be
given by
\begin{equation*}
  | x_{e_{+}}- x_{ e_{-}}| \simeq
2^{-\mathbf{n} ( e^{\uparrow})} \;,
\end{equation*}
where we have set
\begin{equation}\label{e:e-up-def}
  e^{\uparrow}:= e_{+}\wedge e_{-} \in \trim \;, \qquad \forall e \in \tilde{E} \;.
\end{equation}
However, the point we would like to stress, is that generally $ e^{\uparrow}
\neq \tau_{\tilde{\Gamma}, \Gamma}( e)^{\uparrow} \in \trim$.\\


Safe and unsafe divergences are
fundamental for the analytic estimates that will follow later in this section.
Safe divergences are those in which blow ups within the divergence are compensated by
``good'' edges at the boundary of the divergence (which we define precisely below).
Thus, preventing the overall Feynman
diagram to notice the blow up of this divergence.
Whenever a divergence is not safe, we call it unsafe.

Let us consider a
forest of divergences $ \mf{F} \in \mF_{ \Gamma}^{-}$. For every $ \gamma =
(E(\gamma), V(\gamma)) \in \mG_\Gamma^-$ such that $\mf{F}\cup\{\gamma\}$
is  a forest,
we write $ \mA ( \gamma, \mf{F}) \subseteq \Gamma $ for
the parent of $\gamma$ in $\mf{F} \cup \{\Gamma\}$:
\begin{equation}\label{e:def-a-gamma}
  \mA ( \gamma, \mf{F}) = \min \{ \gamma' \in \mf{F} \cup \{\Gamma\} \, \colon \,
  \gamma \subsetneq \gamma' \} \;,
\end{equation}
where the minimum is taken in the sense of inclusions.
We say that $ \gamma'$ is a child of $\gamma $ in $\mf{F}$, if  $\gamma = \mA (
\gamma ', \mf{F}) $, and we write $\mC(\gamma, \mf{F})$ for the set of all
children of $\gamma$ in $\mf{F}$.
Whenever the
forest is clear from context we shorten the notation to $\mA(\gamma)$.

Next, we define the set of proper edges of $\gamma$ as:
\begin{equation}\label{e:proper-edges}
  \mE_{\mf{F}}(\gamma) = E(\gamma) \setminus
  \bigcup_{ \gamma' \in \mC ( \gamma, \mf{F})}
   E(\gamma')  \subseteq E(\Gamma) \;,
\end{equation}
as well as the set of boundary edges:

\begin{equation}\label{e:def-bdry-edges}
  \partial \mE_{\mf{F}} (\gamma) = \{ e \in E(\mA(\gamma, \mf{F})) \setminus (E(\gamma) \cup E_L) \, \colon \, e_{+} \in V(\gamma) \text{ or } e_{-} \in V(\gamma) \} \;.
\end{equation}
We note that in our definition, legs can not be boundary edges.
We observe that both definitions depend on the choice of the forest $\mf{F}$.

Moreover, define $\mathfrak{K} ( \gamma, \mf{F} ) $ to be the subdiagram of
$\tilde{\Gamma} = \mathfrak{K}_{\mf{F}} \Gamma $ spanned by the
edges in
\begin{equation}\label{e:kgamma-def}
\begin{aligned}
\tilde{E} (\mathfrak{K} (\gamma, \mf{F}) ) :=
\big\{ \tilde{e} \in \tilde{E} ( \gamma)\,: \, \tau_{\tilde\Gamma,
\Gamma}(\tilde{e}) \in\mE_{\mf{F}} (\gamma)\big\}\,.
\end{aligned}
\end{equation}
The following depiction
(see also \cite[Figure~3]{BPHZ})
will be helpful to keep in mind:
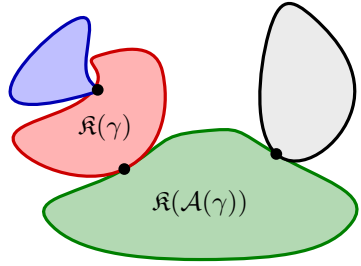
\begin{equation*}
	\begin{aligned}
		\begin{tikzpicture}[scale=1, every node/.style={font=\small}]
\usetikzlibrary{decorations.pathmorphing}

\draw[draw=green!50!black, line width=1.2pt, fill=green!50!black!30]
  plot[smooth cycle, tension=1] coordinates {
    (-2,-.5) (-1,0.2) (0,.7) (1,0.4)
    (2,-0.5) (0,-1)
};
\node at (0.0,-0.25) {$\mf{K}(\mA(\gamma))$};

\draw[draw=red!80!black, line width=1.2pt, fill=red!30]
  plot[smooth cycle, tension=1] coordinates {
    (-1,0.2) (-2.45,0.55) (-1.45,1.2) (-1.35,1.7)
     (-.5,1.5)
  };
  \node at (-1.25,.75) {$\mf{K}(\gamma)$};

\draw[draw=blue!70!black, line width=1.2pt, fill=blue!25]
  plot[smooth cycle, tension=1] coordinates {
    (-1.45,1.2) (-2.5,1.2)
    (-1.60,2.20) (-1.5,1.5)
};

\draw[draw=black, line width=1.2pt, fill=black!8]
  plot[smooth cycle, tension=1] coordinates {
    (1,0.4) (0.9,2) (1.5,2.2) (2,1)
};

\fill[black] (1,0.4) circle (2.2pt);
\fill[black] (-1,0.2) circle (2.2pt);
\fill[black] (-1.35,1.25) circle (2.2pt);

\end{tikzpicture}
	\end{aligned}
\end{equation*}
Here, we see four divergences: The blue one and the black one, as well as the divergence
$ \gamma $ (spanned by edges in the red and blue area), and $ \mA (
\gamma) $ (spanned by edges in all shaded areas).\\

Finally, given a Hepp sector $(T,
\mathbf{n})$ and a forest $\mf{F} \in \mF_{\Gamma}^{-}$, we associate lengthscales to edges in $\Gamma$
that are however computed in the new diagram $\mf{K}_{\mf{F}} \Gamma$,
the unlabelled version of $\mf{K}_{\mf{F}}^{(1)} \Gamma$ (see also
\eqref{eq_forest_extr}). To highlight this dependence, we define the following scale
map\footnote{Note that the notation ``scale'' is slightly misleading: If
$n =  \mathrm{scale}^{\mf{F}}_{T, \mathbf{n}} (e)$ is very large, then
the length--scale $
2^{-n} $ is actually very small. However, we didn't alter the
notation to remain close to the literature.}:
\begin{equation*}
  \mathrm{scale}^{\mf{F}}_{T, \mathbf{n}} (e) = \mathbf{n} ( \tau_{\mathfrak{K}_\mathfrak{F} \Gamma, \Gamma}^{-1} (e)_+ \wedge \tau_{\mathfrak{K}_\mf{F} \Gamma, \Gamma}^{-1} (e)_- )
  = \n \big( \tau^{-1}_{\mf{K}_{\mf{F}} \Gamma, \Gamma}(e)^{\uparrow} \big)\;,
  \qquad \forall e \in E(\Gamma) \;,
\end{equation*}
where $\wedge$ indicates the least common ancestor in the Hepp tree $T$. In
words, $\mathrm{scale}^{\mf{F}}_{T, \mathbf{n}} (e)$ is the (inverse, exponential) scale
associated to $e$ in the graph $\mathfrak{K}_\mathfrak{F} \Gamma$. Once more,
this choice is highly dependent on the forest $\mf{F}$.

\begin{definition}[Safe and unsafe forests]\label{def_safe_unsafe}
Let $ \Gamma$ be a Feynman diagram and
fix a Hepp sector $ \mathbf{T}=(T, \mathbf{n})$.
Let $\mf{F} \in \mF_{\Gamma}^{-} $ be a forest of divergences in $ \Gamma$. We say that
\begin{enumerate}
\item The divergence $ \gamma \in \mf{F} $
is \emph{safe} in $\mf{F}$, if
\begin{equation}\label{eq_def_safe}
\begin{aligned}
\sup_{e \in  \partial \mE_{\mf{F}} ( \gamma) }
\mathrm{scale}_{T, \mathbf{n}}^{\mf{F}} (e)
\geqslant \min_{e \in  \mE_{\mf{F}} ( \gamma) } \mathrm{scale}_{T, \mathbf{n}}^{\mf{F}} (e)  \,,
\end{aligned}
\end{equation}
with the condition that the supremum over an empty set is $-\infty$.

\item The divergence $ \gamma \in \mf{F} $ is \emph{unsafe} in $\mf{F}$, if~\eqref{eq_def_safe} fails.
\end{enumerate}
Moreover, we say that $ \gamma$ is \emph{safe to be added} to a forest $\mf{F} $,
if $ \mf{F} \cup \{ \gamma\}$ is a forest and $ \gamma$ is safe in it. Likewise, $
\gamma$ is \emph{unsafe to be added} if $ \mf{F} \cup \{
  \gamma\}$ is a forest and $ \gamma $ is unsafe in $ \mf{F} \cup \{
\gamma\}$.
\end{definition}

For example, if
  $ \gamma$ is a leaf in $ \mf{F}$, then $\gamma$ is
  safe in $\mf{F}$ if there exists an edge $ e^{*} \in \mE_{\mf{F}} ( \gamma)$ and a
  boundary edge $ e' \in \partial \mE_{\mf{F}} (\gamma)$ such that
  \begin{equation*}
    \begin{aligned}
    | x_{\tau^{-1}(e')_{+}} - x_{\tau^{-1}(e')_{-}}|
    \leqslant
    | x_{\tau^{-1}(e^{*})_{+}} - x_{\tau^{-1}(e^{*})_{-}}|
    \,,
    \end{aligned}
    \end{equation*}
with $\tau = \tau_{\mf{K}_{\mf{F}} \Gamma, \Gamma}$.

\begin{remark}\label{rem:safe}
  The condition~\eqref{eq_def_safe} does not depend on the choice of $\mathbf{n}
  \in \mA(T)$, but only on $T$. Therefore, we can define the set
  \begin{equation*}
    \mF^{s}_\Gamma(T) \subseteq \mF_\Gamma^{-} \;,
  \end{equation*}
  of safe divergent forests, depending only on $T \in \mT_{V_\star}$ and not on
  the particular scale assignment.
\end{remark}

Having introduced safe and unsafe divergences, given a Hepp sector, we can construct
a suitable partition of all forests into forest interval, following
\cite[Lemma~3.6]{BPHZ} (where the proofs of the following statements can be found).

\begin{lemma}\label{lem_partition}
Let $ \Gamma$ be a Feynman diagram and
fix a Hepp tree $ T \in \mT_{V_\star}$.
	For every safe forest $\mf{F}_{s} \in \mF_{\Gamma}^{s}(T)$, we define $\mf{F}_{u}$
	to be the set of divergences that are unsafe to be added to $ \mf{F}_{s}$
	 \begin{equation*}
  \mf{F}_u = \{ \gamma \in \mG^{-}_{\Gamma} \, \colon \, \mf{F}_s \cup \{\gamma\} \text{ is not a safe forest } \}\;.
\end{equation*}
Note that $\mf{F}_u$ depends both on $\mf{F}_s$ and on
$T$. However, we omit from writing this explicitly.
Then the following holds:
\begin{enumerate}
\item
The disjoint union $\mf{F}_{s} \sqcup
\mf{F}_{u}$ is a forest.

\item Every $ \gamma \in \mf{F}_{u}$ is unsafe in $ \mf{F}_{s} \cup
\mf{F}_{u}$. Likewise, every $ \gamma \in \mf{F}_{s}$ is safe in
$\mf{F}_s {\cup \mf{F}_{u}}$.

\item The collection of forest intervals
\begin{equation*}
\begin{aligned}
\mP_{T}= \left\{\big[ \mf{F}_{s}, \mf{F}_{s} \cup \mf{F}_{u} \big]
\,:\,
\mf{F}_{s}  \in \mF_{\Gamma}^{s}(T) 
\right\}
\end{aligned}
\end{equation*}
partitions the set of all forests $\mF^{-}_{\Gamma}$.
\end{enumerate}
\end{lemma}
This partition of $\mF^{-}_\Gamma$ into forest intervals in combination with
\eqref{eq_est_Hepp} yields the upper bound that
we will be working with for the remainder of the section:
\begin{corollary}\label{cor_bphzinterval}
Let $ \Gamma$ be a negative Feynman diagram, then
\begin{equation*}
\begin{aligned}
| \hat{\Pi}_{\varepsilon} \Gamma   ( \varphi) |
\lesssim
\sum_{T \in \mT_{V_{\star}}}
\sum_{ \mf{F}_{s}\in \mF_{\Gamma}^{s}(T)}
\sum_{{\bf n} \in \mA ( T)}
\sup_{x \in D_{( T, \mathbf{n})}}
\big| \overline{\mW}_{\varepsilon}^{(\mK)}
\mathfrak{K}_{ \mf{F}_{s}}^{(1)}	 \Gamma
(x_{V_{\star}})
\big|
\prod_{u \in \overline{T}} 2^{-d \mathbf{n} (u)}
\,,
\end{aligned}
\end{equation*}
with $\mK {= \mK ( \mathbb{M})}$ defined in~\eqref{e:def-K}, with $\mathbb{M} = [\mf{F}_s,\mf{F}_s
\cup \mf{F}_u]$ or equivalently $
\delta( \mathbb{M}) = \mf{F}_{u}$.
\end{corollary}
With this result at hand, we are ready to perform our analytic estimates.

\subsection{Integration of safe forests}

First, we restrict ourselves to the case of safe forest intervals $ [
\mf{F}_{s}, \mf{F}_{s} \cup \mf{F}_{u}]$, in the sense  that $ \mf{F}_{u}=
\emptyset$.
Similarly to~\eqref{eq_def_eta}, we define the analogue of $ \eta$ for labelled
	diagrams $ (\tilde\Gamma, \mf{h}, \mf{n} , \mf{a}) \in  \mD^{\rm{lab}}_\Gamma$
\begin{equation}\label{eq_eta_safe}
  \begin{aligned}
  \eta(u) :=
   d -
\sum_{\substack{e \in \tilde{E}_{\star}\\ e_{+}\wedge e_{-}=u}} (2 + |\mf{h}(e) |)
+
|\{ v \in V_{\star}\,: \, \mf{n}(v) =1\ \text{and} \ v \wedge \mf{a}(v) = u\}|
\,,
  \end{aligned}
  \end{equation}
and set $\overline{\deg}(u) = \sum_{v \succ u} \eta(v) $.

We find the following result
in analogy to \cite[pages 32 -- 34]{BPHZ}.

\begin{lemma}[Safe integration]\label{lem_induction_neg}
Let $ \Gamma$ be a negative Anderson--Feynman diagram with $2k$ legs,
$ T \in \mT_{V_{\star}}$, and
$\mf{F}_{s} \in \mF_{ \Gamma}^{s}(T)$ such that $ \mf{F}_{u} = \emptyset $.
Then
\begin{equation*}
  \overline{\mW}_{\varepsilon}^{(\mK)}
 \mathfrak{K}_{ \mf{F}_{s}}^{(1)}  \Gamma
(x_{V_{\star}}) = \Psi_{\ve} (x_{V_\star})\mW_\ve  \mathfrak{K}_{
\mf{F}_{s}}^{(1)} \Gamma
(x_{V_{\star}})
\end{equation*}
with $ \Psi_{\ve}$ defined in terms of $ \varphi$ in~\eqref{eq_def_Psi} and
$\mW_\ve$ from~\eqref{e_def_mW}, and where  $ \mK_{\varepsilon}\equiv
G_{\varepsilon}$ since $\delta(\mathbb{M}) = \emptyset$ in~\eqref{e:def-K}.
In addition,
it holds that
\begin{equation*}
\begin{aligned}
\sum_{{\bf n} \in \mA ( T)}
\sup_{x \in D_{( T, \mathbf{n})}}
\big| \overline{\mW}_{\varepsilon}^{( \mK)}
\mathfrak{K}_{ \mf{F}_{s} }^{(1)} \Gamma (x_{V_{\star}})\big|
\prod_{u \in \overline{T}} 2^{-d \mathbf{n} (u)}
\lesssim
	\sum_{\tilde\Gamma}
( \log{ \tfrac{1}{\ve}} )^{{\mathrm{Null}} (
T , \tilde\Gamma) }
\,,
\end{aligned}
\end{equation*}
where
$\mathrm{Null }  (T, \tilde\Gamma) = \left| \left\{ u \in \overline{T} \, :\,
\overline{\deg}(u) =0 \right\}\right|$. Here,
we wrote $ \mf{K}^{(1)}_{\mf{F}_{s}} \Gamma=
\sum \tilde\Gamma $ with each $\tilde{\Gamma} \in \mD^{\rm{lab}}_\Gamma$ being nontrivial.
\end{lemma}

The proof follows the same lines of Lemma~\ref{lem_estimate_nonneg_Hepp} in the case of
non--negative Feynman diagrams. However, here more notation is required to
track the scales $\mathbf{n}$ through iterative application of extraction maps $
\hat{\mC}$ and $ \hat{\mC}^{(1)}$.

\begin{proof}
The first statement is a consequence of Lemma~\ref{lem_pushonKernel} because $
\mK \equiv G $,
together with the
definition of $\Psi_\ve$ and $\mW_\ve$.
As a consequence, we bound
\begin{equation*}
\begin{aligned}
\sup_{x_{V_\star} \in D_{( T, \mathbf{n})}}
\big| \overline{\mW}_{\varepsilon}^{( \mK)}
\mathfrak{K}_{ \mf{F}_{s} }^{(1)}\Gamma (x_{V_{\star}})\big|
\leqslant
\| \Psi_{\ve}\|_{\infty}
\sup_{x_{V_\star} \in D_{( T, \mathbf{n})}}
\big| \mW_{\varepsilon}
\mathfrak{K}_{ \mf{F}_{s} }^{(1)} \Gamma (x_{V_{\star}})\big| \lesssim_\varphi \sup_{x_{V_\star} \in D_{( T, \mathbf{n})}}
\big| \mW_{\varepsilon}
\mathfrak{K}_{ \mf{F}_{s} }^{(1)} \Gamma (x_{V_{\star}})\big| \,,
\end{aligned}
\end{equation*}
with $ \sup_{ \ve \in (0,1)} \| \Psi_{\ve}\|_{\infty} \lesssim_\varphi 1$ since
the test function $\varphi$ is smooth (and in particular it is bounded).

Next, we decompose $ \mathfrak{K}_{ \mf{F}_{s}}^{(1)} \Gamma = \sum_{\tilde\Gamma}
\tilde\Gamma$, where $ \tilde\Gamma \in \mD_{\Gamma}^{\rm{lab}}$.
We proceed by estimating each such $\tilde\Gamma = ( \tilde\Gamma, \mf{h},
\mf{n}, \mf{a})$ individually (note that there are only finitely many).
Writing $\tilde{E}_\star = \tau^{-1}_{\tilde{\Gamma}, \Gamma} E_\star$,
and using that $ | \partial_{i}^{h} G_{ \varepsilon} (x) | \lesssim ( |x| +
\varepsilon )^{-2-h}$, $h \in \{0,1\}$, we see that
 \begin{equation}\label{eq_safe_supp1}
  \begin{aligned}
  \sum_{{\bf n} \in \mA ( T)}
  \sup_{x \in D_{( T, \mathbf{n})}} &
  \big| \mW_{\varepsilon}
  \tilde\Gamma (x_{V_{\star}})\big|
  \prod_{u \in \overline{T}} 2^{-d \mathbf{n} (u)}
  \\
  & \lesssim
  \sum_{{\bf n} \in \mA (T)}
  \prod_{ e \in \tilde{E}_{\star}}
  (2^{-\mathbf{n} ( e_{+} \wedge e_{-})}
  + \varepsilon)^{-2 - |\mf{h} ( e)| }
\prod_{v \in V_{\star}} 2^{- \mf{n}(v)\cdot  \mathbf{n} ( v \wedge \mf{a}(v))}
  \prod_{u \in  \overline{T}} 2^{-d \mathbf{n}(u)}\\
  & {=} 
  \sum_{{\bf n} \in \mA ( T) }
  \prod_{u \in \overline{T}}
  (2^{-\mathbf{n} ( u)}
  + \varepsilon)^{ \eta(u)-d}
   2^{-d \mathbf{n}(u)}
  \,.
\end{aligned}
\end{equation}
By Lemma~\ref{lem_safeisnonneg} below, we see that $ \overline{\deg}(u) =
\sum_{v \succ u} \eta(v) \geqslant 0 $.
Therefore, we proceed as in the proof of Lemma~\ref{lem_estimate_nonneg_Hepp},
to conclude
that~\eqref{eq_safe_supp1} is bounded by
\begin{equation}\label{eq_safe_supp10}
\begin{aligned}
\sum_{{\bf n} \in \mA ( T) }
\prod_{u \in \overline{T}}
(2^{-\mathbf{n} ( u)}
+ \varepsilon)^{ \eta(u) -d}
 2^{-d \mathbf{n}(u)}
\lesssim
\big(
\log{ \tfrac{1}{ \varepsilon}}
\big)^{\mathrm{Null} (T, \tilde\Gamma)}\,,
\end{aligned}
\end{equation}
with $ \mathrm{Null } (T, \tilde\Gamma) = \left| \left\{ u \in \overline{T} \, :\,
\overline{\deg}(u) =0 \right\}\right|$. This finishes the proof.
\end{proof}

Before we proceed, we state the following supporting statement for the proof of
Lemma~\ref{lem_induction_neg}. Its proof follows verbatim to the proof of
\cite[Eq. (3.10)]{BPHZ} with the definition of $\eta$ in \cite[Pp. 32+]{BPHZ}

\begin{lemma}\label{lem_safeisnonneg}
Let $ \Gamma$ be a negative Anderson--Feynman diagram, $ T \in
\mT_{V_{\star}}$, and $ \mf{F}_{s} \in \mF_{ \Gamma}^{s}(T)$ {such that $
\mf{F}_{u} = \emptyset$}.
Then
\begin{equation*}
\begin{aligned}
\overline{\deg}(u) = \sum_{v \succ u} \eta(v) \geqslant 0 \,,
\quad \text{ for every } u \in \overline{T}\,,
\end{aligned}
\end{equation*}
with $ \eta$ defined as in~\eqref{eq_eta_safe} on each of the diagrams $
	\tilde\Gamma = (
V, \tilde{E})$ in $
\mathfrak{K}_{\mf{F}_{s}}^{(1)} \Gamma = \sum \tilde\Gamma $.
\end{lemma}

\subsection{Renormalisation of unsafe forests}

After having dealt with safe forests $ \mf{F}_{s}$ for which $
\mf{F}_{u} = \emptyset$, we now move on to the case $\mf{F}_{u} \neq \emptyset$. In this case, we recall that the renormalisation map
\eqref{e_Rhat_one} is of the form
\begin{equation}\label{eq_supp_unsafe_renorm}
\begin{aligned}
\hat{\mR}_{[\mf{F}_{s}, \mf{F}_{s}\cup \mf{F}_{u}]}^{(1)} \Gamma
= (-1)^{| \mf{F}_{s}|} \prod_{\gamma \in \mf{F}_{u} }( \mathrm{id} -
\hat{\mC}_{ \gamma} - \hat{\mC}_{\gamma}^{(1)} ) \mathfrak{K}_{\mf{F}_{s}}^{(1)}
\Gamma\,.
\end{aligned}
\end{equation}
The idea behind BPHZ renormalisation is to leverage cancellations that take the form of Taylor expansion in the valuation $
\Pi_{\varepsilon}( \mathrm{id} -
\hat{\mC}_{ \gamma} - \hat{\mC}_{\gamma}^{(1)} ) \Gamma$. On the level
of a given Hepp tree $ T$,
this will result in a shift in the weights $ \eta$, which will cure
divergent sub-integrals.

  Before we proceed, let us introduce the setting we will be working in
  throughout this sub-section. First, we fix a negative Anderson--Feynman diagram $\Gamma$, and a Hepp tree $ T \in
  \mT_{V_{\star}}$. Next we fix a safe forest $\mf{F}_s$ and consider any divergence
  $\gamma \in \mf{F}_u$, where the latter is the forest of unsafe divergences
  for $\mf{F}_s$. We then define the largest scale of $\mf{K} ( \gamma,
\mf{F}_{s} )$, see \eqref{e:kgamma-def}:
      \begin{equation}\label{e:gup}
        \begin{aligned}
        \gamma^{\uparrow}
        :=
        \min \{ v \wedge v'\,: \, v, v' \in \mf{K} ( \gamma, \mf{F}_{s} ) \} \in \overline{T}\,,
        \end{aligned}
      \end{equation}
      where the minimum is taken with respect to the partial order on $T$. Equivalently, we could write
      \begin{equation} \label{e:gup-edge}
        \gamma^{\uparrow}
        =
        \min \{ e^\uparrow \, \colon \, e \in \tilde{\mE}_{\mf{F}_{s}}( \gamma) \} \;,
      \end{equation}
      where the identity holds because $\mf{K} ( \gamma, \mf{F}_{s} )$
      is connected in $\mf{K}_{\mf{F}_{s}}^{(1)} \Gamma$ (since the original
      $\gamma$ is connected), and for the same reason the infimum is a minimum.

      Similarly, we define another node, which corresponds to the
      scale of the smallest boundary edge connected to $\mf{K} ( \gamma, \mf{F}_{s} )$.
        \begin{equation}\label{e_def_gupup}
        \gamma^{\uparrow \uparrow} := \begin{cases}
          \max \{
            e^{\uparrow}\,: \, e \in \partial \tilde{\mE}_{\mf{F}_{s}} (\gamma)
            \}\,, \qquad & \text{ if }  \partial \tilde{\mE}_{\mf{F}_{s}} (\gamma) \neq \emptyset \;,\\
            \dagger & \text{ otherwise}\;.
        \end{cases}
        \end{equation}
        Here $\partial \tilde{\mE}_{\mf{F}_{s}} (\gamma) = \tau^{-1}
        \partial\mE_{\mf{F}_{s}} (\gamma)$ with the notation of~\eqref{e:def-bdry-edges}, and $\dagger$ is a \lqm cemetary
        state\rqm that appears in the case in
        which all edges $e \in E(\mA(\gamma))\setminus E(\gamma)$ such that $e$ is
        incident in $\gamma$ are legs: see~\eqref{e:def-bdry-edges}. We view $T
        \cup \{\dagger\}$ as a partially ordered set with the assumption that
        $\dagger \prec u$, for all $u \in T$. This reflects the fact that legs are
        smooth, so they correspond to the largest scale possible.

        It follows from the definition of $\gamma^{\uparrow}, \gamma^{\uparrow
        \uparrow}$ that
        \begin{equation*}
          \gamma \in \mf{F}_u \quad \Longrightarrow \quad \gamma^{\uparrow} \succ \gamma^{\uparrow \uparrow} \;, \gamma^{\uparrow} \neq \gamma^{\uparrow \uparrow} \,.
        \end{equation*}
	In the remainder of this section, we remove the
blow-up appearing at $\gamma^{\uparrow}$ (whenever $\overline{\deg} ( \gamma^{\uparrow}) <0$), and shifting and absorbing it into  $
\gamma^{\uparrow \uparrow} $, through a Taylor expansion.
	As in the case of safe divergences, our goal is to find
for every $ \tilde\Gamma$ in $ \mf{K}^{(1)}_{\mf{F}_{s}} \Gamma= \sum \tilde\Gamma$
a map $
\hat{\eta}$ on
the inner vertices $ \overline{T}$ such that
\begin{equation}\label{eq_est_unsafe}
\begin{aligned}
\sum_{{\bf n} \in \mA ( T)}
\sup_{x \in D_{( T, \mathbf{n})}}
\big| \overline{\mW}_{\varepsilon}^{(\mK)}
\tilde\Gamma
(x_{V_{\star}})\big|
\prod_{u \in \overline{T}} 2^{-d \mathbf{n} (u)}
\lesssim
\sum_{ {\bf n} \in \mA ( T)}
\prod_{u \in \overline{T}}
(2^{-\mathbf{n} ( u)}
+ \varepsilon)^{ \hat{\eta}  (u) -d}
 2^{-d \mathbf{n}(u)}\,,
\end{aligned}
\end{equation}
and such that the right side of~\eqref{eq_est_unsafe} is summable, which is again captured by the condition $\sum_{v \succ u} \hat{\eta}(v) \geq 0$.
In Lemma~\ref{lem_hepp_unsafe} below, we will verify
\eqref{eq_est_unsafe} with $\hat{\eta}$ given by:
\begin{equation}\label{e_eta_unsafe}
  \begin{aligned}
  \hat{\eta} (u)
 : = \eta(u)
  +
  2
  \sum_{\gamma \in \mf{F}_{u}} \big(
  \mathds{1}_{ \gamma^{\uparrow}}( u ) - \mathds{1}_{ \gamma^{\uparrow \uparrow}}( u )
  \big)
  \,,
  \end{aligned}
  \end{equation}
where $ \eta$ is as in~\eqref{eq_eta_safe}.
{The factor~$2$ multiplying $\mathds{1}_{ \gamma^{\uparrow}}( u ) - \mathds{1}_{
\gamma^{\uparrow \uparrow}}( u )$ captures the fact that a first--order Taylor expansion
yields a quadratic remainder term.} 

Before we proceed to
the proof of~\eqref{eq_est_unsafe}, we show that the choice of $\hat{\eta}$ from
\eqref{e_eta_unsafe} does indeed lead to the convergence of the right side of
\eqref{eq_est_unsafe}. This is the content of the next lemma.
Its proof follows verbatim to the proof of
\cite[Eq.~(3.10)]{BPHZ} with the definition of $\eta$ in \cite[Eq.~(3.17)]{BPHZ}.

\begin{lemma}\label{lem_unsafe_nonneg}
For $\hat{\eta}$ as in~\eqref{e_eta_unsafe} we have $ \widehat{\deg} (u) :=  \sum_{v \succ u} \hat{\eta}(v) \geqslant 0
$ for all $u \in \trim$.
\end{lemma}

In the next result, we verify that indeed choosing $\hat{\eta}$ as in
\eqref{e_eta_unsafe} leads to~\eqref{eq_est_unsafe}, and we derive an estimate
on negative Feynman diagrams.

\begin{lemma}\label{lem_hepp_unsafe}
Let $ \Gamma$ be a negative Anderson--Feynman diagram, $T \in
\mT_{V_{\star}}$ an
associated Hepp tree, $ \mf{F}_{s} \in \mF_{\Gamma}^{s}(T)$ be a safe forest, and $\mf{F}_{u}$
the respective unsafe forest. Write $ \mf{K}^{(1)}_{\mf{F}_{s}} \Gamma=
\sum \tilde\Gamma $ where each $\tilde{\Gamma}$ is nontrivial in $\mD^{\rm{lab}}_\Gamma$. Then for $\hat{\eta} $ (which depends on $
\tilde\Gamma$) as in~\eqref{e_eta_unsafe}
the bound~\eqref{eq_est_unsafe} holds. In particular, we obtain
\begin{equation*}
\begin{aligned}
\sum_{{\bf n} \in \mA ( T)}
\sup_{x \in D_{( T, \mathbf{n})}}
\big|
\overline{\mW}_{\varepsilon}^{(\mK)}
\mathfrak{K}_{\mf{F}_{s}}^{(1)} \Gamma (x_{V_{\star}})\big|
\prod_{u \in \overline{T}} 2^{-d \mathbf{n} (u)}
\lesssim
	\sum_{\tilde\Gamma}
( \log{ \tfrac{1}{\ve}} )^{\widehat{\mathrm{Null}} (
T , \tilde\Gamma) }
\,,
\end{aligned}
\end{equation*}
where $\widehat{\mathrm{Null}}(T ,\tilde\Gamma ) = \left| \left\{ u  \in \trim \, :\,  \sum_{v
\succ  u} \hat{\eta} (v) =0 \right\}\right| $.
\end{lemma}

The proof of the lemma follows the approach of~\cite[Section 3]{BPHZ}, with the
difference that we do not renormalise logarithmic
divergences. 

\begin{proof}
Consider $\mf{K}_{\mf{F}_s}^{(1)} \Gamma = \sum \tilde{\Gamma}$, with
$\tilde{\Gamma} = (V, \tilde{E}) \in \mD^{\rm{lab}}_{\Gamma}$.
{Then from \eqref{e:def-WK}}, we have that
\begin{equation}\label{eq_valuation_mK}
\begin{aligned}
\overline{\mW}_{\varepsilon}^{(\mK)}
\tilde{\Gamma}
{ ( x_{V_{\star}}) }
=
\int_{ ( \TT^{d})^{L}}
\varphi ( x_{L})
 \prod_{e \in  \tilde{E}}
\mK_{\tau_{\tilde{\Gamma}, \Gamma}(e), \ve}(x_V)
\ud x_{L} \;,
\end{aligned}
\end{equation}
where {$\mK = \mK ( \mathbb{M})$ is the kernel assignment from~\eqref{e:def-K} with $\mathbb{M} =
[\mf{F}_s, \mf{F}_s \cup \mf{F}_u]$}.
Now, define an edge labelling $\ell \colon \tilde{E} \mapsto \{0,1\}$ as follows:
\begin{equation}\label{eq_def_ell}
\begin{aligned}
\ell(e)
:=
\begin{cases}
1  & \quad \text{ if } \exists \gamma \in \mf{F}_{u} \text{ such that } e_\star(\gamma, \Gamma) = \tau_{\tilde{\Gamma}, \Gamma}(e)\,,\\
0 & \quad \text{ otherwise.}
\end{cases}
\end{aligned}
\end{equation}
Note that from~\eqref{e:def-K}, we have (with $\mf{b}$ as in that equation):
\begin{equation*}
  \begin{aligned}
    \mK_{\tau_{\tilde{\Gamma}, \Gamma} (e),\ve} (x_V) = \big(\mH^{ \ell ( e)} G_{\ve}\big)( x_{e_{+}}, x_{e_{-}} |  x_{ \mf{b} ( e)} ) = \begin{cases}
     (\mH G_{\ve})( x_{e_{+}}, x_{e_{-}} |  x_{\mf{b} ( e)} ) \;, & \text{ if } \ell(e) = 1 \;, \\
     G_{\ve}( x_{e_{+}} - x_{e_{-}} ) \;, & \text{ if } \ell(e) = 0\;.
    \end{cases}
  \end{aligned}
\end{equation*}
To estimate the valuation, we start with the part of the
integral that includes legs. Note that every leg vertex in $ L$ is associated to exactly
one edge (the corresponding leg) in $
\tilde{E}_{L} = \tau^{-1}_{\tilde{\Gamma}, \Gamma}(E_L)$. We will prove the following bound:
\begin{equation}\label{e_leg_contr}
\begin{aligned}
\bigg|\int_{ ( \TT^{d})^{L}}
\varphi ( x_{L})
 \prod_{e \in \tilde{E}_{L}}
\big(\mH^{ \ell ( e)} G_{\ve}\big)( x_{e_{+}}, x_{e_{-}} |  x_{ \mf{b} ( e)} )
\ud x_{L} \bigg|
\lesssim
\prod_{ \substack{e \in \tilde{E}_{L}\\ \ell( e) = 1}}
|x_{e_{+}} - x_{\mf{b}(e)}|^{ 2} \;.
\end{aligned}
\end{equation}
To see this, we notice that for every leg $ e \in \tilde{E}_{L}$ with $\ell ( e) =1$ (and thus $
e_{-} \in L$), the integral on the left of~\eqref{e_leg_contr} contains the 
subintegral
\begin{equation*}
  \int_{\TT^{d}}
  \varphi ( x_{L})
  \mH G_{\ve} ( x , y| z)
   \ud y \;,
\end{equation*}
where we have written $x, y, z$ for the variables
$x_{e_+},x_{e_-}$, and $x_{\mf{b}(e)}$ respectively.
Now, by the definition of $\mH$ from~\eqref{e:hk-def} and integration by
parts, we compute (with $M = L \setminus \{e_-\}$):
\begin{equation*}
\begin{aligned}
\int_{\TT^{d}}
\varphi & ( x_{M}, y)
\mH G_{\ve} ( x , y| z )
 \ud y\\
& =
\int_{\TT^{d}}
\varphi ( x_{M}, y)
\sum_{i , j =1}^{d}
(x^{i} - z^{i})
(x^{j} - z^{j})
\int_{0}^{1} (1-t)
\partial_{i}\partial_{j} G_{\ve}
( z - y
+ t( x - z  ))
\ud t
 \ud y\\
& =
\sum_{i , j =1}^{d}
(x^{i} - z^{i})
(x^{j} - z^{j})
\int_{\TT^{d}}
\partial^2_{ y^{i} y^{j} }
\varphi ( x_{M}, y)
\int_{0}^{1} (1-t)
G_{\ve}
( z - y
+ t( x - z  ))
\ud t
 \ud y \;.
\end{aligned}
\end{equation*}
Since $\varphi$ is smooth and $G_\ve$ is integrable uniformly in $\ve$, this
immediately implies~\eqref{e_leg_contr}.

For all edges that are not legs, we use the following argument. For a given edge $e \in \tilde{E}_\star$, consider the collection
$ \gamma_{1} \subsetneq
\gamma_{2} \subsetneq \cdots \subsetneq \gamma_{m}$ of all unsafe
divergences in $\mf{F}_{u} $ such that $e = e_\star(\gamma_i, \tilde{\Gamma})$.
In particular, it holds that $v: = \mf{b}(e) = v_\star(\gamma_1)$.
Then Lemma~\ref{lem_taylor_new} implies that for $ x_{V_\star} \in D_{(T,\n)}$
\begin{equation*}
\begin{aligned}
|\mH G_{\ve}( x_{e_{+}}, x_{e_{-}}| x_{ v  })|
& \lesssim
\frac{ | x_{e_{+}} - x_{ v}|^{2} }{\big( \min\{ | x_{v} - x_{e_{-}}|, | x_{e_{+}}
- x_{e_{-}}  | \} + \ve\big)^{2} }
( | x_{e_{+}}- x_{e_{-}}|+ \ve)^{-2}\\
& \lesssim
\frac{ 2^{-2 \mathbf{n}( e_{+} \wedge v )} }{\big( \min\{ 2^{- \mathbf{n}( e_{-}
\wedge v )} , 2^{- \mathbf{n}( e^{\uparrow})}  \} + \ve \big)^{2} }
( | x_{e_{+}}- x_{e_{-}}|+ \ve)^{-2}\\
& \lesssim
\frac{ 2^{- 2 \mathbf{n}( \gamma_{1}^{\uparrow})} }{( 2^{- \mathbf{n}(
{\gamma }_{m}^{\uparrow \uparrow})}  + \ve )^{2} }
( | x_{e_{+}}- x_{e_{-}}|+ \ve)^{-2}\,,
\end{aligned}
\end{equation*}
where in the last inequality, we used $e^\uparrow \prec \gamma_m^{\uparrow
  \uparrow}$,  since $e \in \partial
  \tilde{\mE}_{\mf{F}_s}(\gamma_m)$, and $e^\uparrow \prec \gamma_m^{\uparrow
  \uparrow}$, which holds by a similar argument. 
  Hence, for every $ x_{V_\star} \in D_{( T, \n)}$ and $e \in \tilde{E}_\star$
  with $\ell(e)=1$
\begin{equation}\label{eq_unsafe_est_supp2}
\begin{aligned}
|\mH G_{\ve}( x_{e_{+}}, x_{e_{-}}| x_{ v })|
\lesssim
\prod_{ i=1}^{m}
\frac{ 2^{- 2\mathbf{n}( \gamma_{i}^{\uparrow})} }{( 2^{- \mathbf{n}(
\gamma_{i}^{\uparrow \uparrow})}  + \ve )^{2} }
( | x_{e_{+}}- x_{e_{-}}|+ \ve)^{-2}\,,
\end{aligned}
\end{equation}
using that $ \mathbf{n} ( \gamma_{i+1}^{\uparrow})
\geqslant
 \mathbf{n} ( \gamma_{i }^{\uparrow\uparrow})  $
since  $ \gamma_{i}$ is an
unsafe divergence,
and necessarily
 $
e^{\mathrm{out}}_{ \gamma_{i}} \in E ( \gamma_{i +1})$ because all $
\gamma_{i} $ share the same incoming edge.
If $ e$ is a leg we set $ e^{\uparrow} = \dagger =
\gamma_{m}^{\uparrow\uparrow}$, and use~\eqref{e_leg_contr} together with the same
argument as for inner edges, to see that
\begin{equation}\label{eq_legestH}
\begin{aligned}
| x_{e_{+}} - x_{v}|^{2}
\lesssim
2^{- 2\mathbf{n}( \gamma_{1}^{\uparrow})}
\leqslant
2^{- 2\mathbf{n}( \gamma_{m}^{\uparrow})}
\prod_{ i=1}^{{m-1}}
\frac{ 2^{- 2\mathbf{n}( \gamma_{i}^{\uparrow})} }{( 2^{- \mathbf{n}(
\gamma_{i}^{\uparrow \uparrow})}  + \ve )^{2} } \,.
\end{aligned}
\end{equation}
For convenience, we set $ \mathbf{n} ( \dagger ) =0$, so that
$2^{- 2\mathbf{n}( \gamma_{m}^{\uparrow})}$ in~\eqref{eq_legestH} can be absorbed into the product.

Overall, we have derived the bound
\begin{equation}\label{eq_supp111_unsafe}
\begin{aligned}
\big| \overline{\mW}_{\varepsilon}^{(\mK)}
\tilde{\Gamma} (x_{V_{\star}})
\big|
& \lesssim_{ \varphi}
\prod_{ \substack{e \in \tilde{E}_{L}\\ \ell( e) =1 }}
|x_{e_{+}}- x_{\mf{b}(e)} |^{2}
 \prod_{e \in \tilde{E}_{\star} }
\big| \mK_{\tau_{\tilde{\Gamma}, \Gamma}(e), \ve} (x_{V_\star})\big|\\
& \lesssim
\left\{\prod_{ \gamma \in \mf{F}_{u}}
\frac{ 2^{- 2\mathbf{n}( \gamma^{\uparrow})} }{( 2^{- \mathbf{n}(
\gamma^{\uparrow \uparrow})}  + \ve )^{2} } \right\}
\left\{
\prod_{e \in \tilde{E}_{\star}}
( |  x_{e_{+}} - x_{e_{-}} | + \ve)^{-2}
\right\}\,,
\end{aligned}
\end{equation}
where we used~\eqref{e_leg_contr},~\eqref{eq_unsafe_est_supp2} and~\eqref{eq_legestH} in the last
inequality.
Because incoming edges to divergences are unique, we do not count unsafe
divergences multiple times in~\eqref{eq_unsafe_est_supp2}.
Recalling the definition of $ \hat{\eta}$ in~\eqref{e_eta_unsafe}, this finally
yields
\begin{equation}\label{eq_supp100_unsafe}
\begin{aligned}
\sum_{ {\bf n} \in \mA ( T)}
\sup_{x \in D_{\bf T}}
\big| \overline{\mW}_{\varepsilon}^{(\mK)}
\tilde{\Gamma} (x_{V_{\star}})
\big|
\prod_{u \in \overline{T}} 2^{-d \mathbf{n} (u)}
& \lesssim
\sum_{ {\bf n} \in \mA ( T)}
\prod_{u \in \overline{T}}
(2^{-\mathbf{n} ( u)}
+ \varepsilon)^{ \hat{\eta}  (u) -d}2^{-d \mathbf{n} (u)}
\,.
\end{aligned}
\end{equation}
The sum on the right side can the be evaluated as in~\eqref{eq_safe_supp10} (see
also Lemma~\ref{lem_estimate_nonneg_Hepp}), because $
\hat{\eta} $ is nonnegative, see Lemma~\ref{lem_unsafe_nonneg}.
\end{proof}

\subsection{Uniform bound in the weak coupling limit for negative diagrams}

In the previous two sections, we established control of safe and unsafe forests, which
allows us to prove the following
upper bound on negative Anderson--Feynman diagrams.

\begin{corollary}\label{cor_unsafe}
	Let $ \Gamma$ be a negative $2k$-Anderson--Feynman diagram with $k \in \{1,2\}$. Then
\begin{equation*}
\begin{aligned}
\lambda_{\ve}^{| E_{\star}|}
| \hat{\Pi}_{\varepsilon} \Gamma ( \varphi) |
\lesssim
\big(\log{ \tfrac{1}{\ve}} \big)^{- \frac{2-k}{2} }\,.
\end{aligned}
\end{equation*}
In particular, the right--hand side vanishes as $ \ve \to 0 $ if $ k =1$.
\end{corollary}

\begin{proof}
By Lemma~\ref{lem_induction_neg} and Lemma~\ref{lem_hepp_unsafe}, we obtain that for every $ T \in
\mT_{V_{\star}}$, and every safe forest $
\mf{F}_{s} \in \mF_{\Gamma}^{s} (T)$
\begin{equation}\label{eq_upperbound_negdiagrams}
\begin{aligned}
\lambda_{\ve}^{| E_{\star}|}
\sum_{{\bf n} \in \mA ( T)}
\sup_{x \in D_{( T, \mathbf{n})}}
\big|
\overline{\mW}_{\varepsilon}^{(\mK)}
\mathfrak{K}_{\mf{F}_{s}}^{(1)} \Gamma ( x_{V_{\star}})
\big|
\prod_{u \in \overline{T}} 2^{-d \mathbf{n} (u)}
\lesssim
\sum_{\tilde\Gamma}
( \log{ \tfrac{1}{\ve}} )^{ \widehat{\mathrm{Null}}( \overline{T}, \tilde\Gamma)
- \frac{| E_{\star}|}{2} }
\,,
\end{aligned}
\end{equation}
where $\widehat{\mathrm{Null}}( T, \tilde\Gamma) = \left| \left\{ u \in \trim \, :\,  \sum_{v
\succ u} \hat{\eta} (v) =0 \right\}\right| $.
Therefore, 
by Corollary~\ref{cor_bphzinterval},
our claim follows once we prove that
\begin{equation}\label{e:some-aim}
  \widehat{\rm{Null}}( T, \tilde\Gamma) - \frac{| E_{\star}|}{2} \leq - \frac{2-k}{2} \;.
\end{equation}
For $k=2$, \eqref{e:some-aim} is trivially true, since $ |E_{\star}| = 2 |
\overline{T}|$ and $\widehat{\mathrm{Null} }(T, \tilde\Gamma) \leqslant | \overline{T}| $.
On the other hand, for $ k=1$ 
we use $ |
E_{\star}| = 2 | \overline{T}| +(2-k)$, cf.~\eqref{eq_Estar_card}, which indeed yields
\begin{equation*}
	\begin{aligned}
		\widehat{\rm{Null}}( T, \tilde\Gamma) - \frac{| E_{\star}|}{2}
		\leqslant | \trim | - ( | \trim| + \tfrac{1}{2}) = -
		\frac{1}{2} \,.
	\end{aligned}
\end{equation*}
This concludes the proof.
\end{proof}

The above corollary states that all negative Feynman diagrams with two legs
vanish in the weak coupling limit.
As we will see in the next section, the four--legged case is more subtle (we
only need to cover these two cases since they are enough to prove that second
moments of trees vanish in weak coupling).
However, in the case in which $ \mf{F}_{u}= \emptyset$, our previous estimates are in fact sharp
enough to yield a vanishing contribution also with $k=2$. This is the content of
the next result.

\begin{lemma}\label{lem_safe_vanish_2k}
Let $ \Gamma$ be a $ 4$--legged, negative Anderson--Feynman diagram,
$ T \in \mT_{V_{\star}}$, and
$\mf{F}_{s} \in \mF_{ \Gamma}^{s}(T)$ such that $ \mf{F}_{u} = \emptyset $.
Then
\begin{equation*}
\begin{aligned}
\sum_{{\bf n} \in \mA ( T)}
\sup_{x \in D_{( T, \mathbf{n})}}
\big| \overline{\mW}_{\varepsilon}
\mathfrak{K}_{ \mf{F}_{s} }^{(1)} \Gamma (x_{V_{\star}})\big|
\prod_{u \in \overline{T}} 2^{-d \mathbf{n} (u)}
\lesssim \big( \log{ \tfrac{1}{\varepsilon}} \big)^{| \overline{T}| -1}\,.
\end{aligned}
\end{equation*}

\end{lemma}

\begin{proof}
	We use the bound derived in Lemma~\ref{lem_induction_neg}.
It is only left to show that the weak coupling limit of the above quantity
vanishes, which follows if we prove that
\begin{equation*}
	\mathrm{Null} (T, \tilde\Gamma) < \frac{|E_\star|}{2} = | \overline{T}|  \;,
\end{equation*}
for every $ \tilde\Gamma$ in $ \mf{K}_{\mf{F}_{s}}^{(1)} \Gamma= \sum \tilde\Gamma$.
In fact,
it suffices to show there exists an inner node
$u \in \overline{T}$ such that $ \eta(u) >0$, since then $ \mathrm{Null} (T, \tilde\Gamma) \leqslant | \trim |-1 =  \tfrac{|
E_{\star}|}{2} - 1$, where we used~\eqref{eq_Estar_card} in
the last identity.

Indeed, we will show by induction that $ \tilde\Gamma= \Gamma$, and that it must be a nested bubble diagram,
whenever $ \eta (u) = 0$ for every $ u \in \overline{T} $.
However,  this
contradicts our assumption that $ \Gamma$ contains a negative sub--diagram.

First, notice that for $ \Gamma'$ with $ | V_{\star}| =| \overline{T}| +1 =2$,
we only have the diagram $ \scalebox{0.7}{\<cc22>}$ (modulo orientation of edges) which clearly is a nested bubble diagram.
Next, we assume that the following statement is true for some $n$.
For every Anderson--Feynman diagram $ \Gamma$, with $ | V_{\star}| =n$:
If some  $ \tilde\Gamma \in \mD_{\Gamma}^{\rm{lab}}$ in $
\mf{K}^{(1)}_{\mf{F}_{s}} \Gamma= \sum_{\tilde\Gamma} \tilde\Gamma $
satisfies $ \eta \equiv
0 $, then $
\tilde\Gamma = \Gamma$ and it is a nested bubble diagram.

Now, let $ \Gamma$ be an Anderson--Feynman diagram of the size $ |
V_{\star}| = n+1$, {and let $ \tilde\Gamma $ be  a diagram in $ \mf{K}^{(1)}_{\mf{F}_{s}}
\Gamma = \sum \tilde\Gamma$ such that $ \eta \equiv 0 $.} 
Choose an arbitrary node $ u \in \overline{T}$ such that it is the parent of
two leaves $ v_{1}, v_{2} \in V_{\star}$, and denote by $ \overline{\Gamma}_{1}$ the
subdiagram in $ \tilde\Gamma$ spanned by $
v_{1}$ and $ v_{2}$.
By Lemma~\ref{lem:no-weird-bubble}, $ \overline{\Gamma}_{1}$ must be a bubble.
We then define $ \tilde{\Gamma}_{2}:= \mK_{ \overline{\Gamma}_{1}} \tilde\Gamma$ (as a
labelled diagram), cf. Definition~\ref{def:bubbl}, and the Hepp tree $ T_{2}$ where the subtree rooted at $u$ in $
\overline{T} $ was contracted to a single leaf.
Note that $ \tilde{\Gamma}_{2}$ must also arise from an Anderson--Feynman diagram $
\Gamma_{2}$:
\begin{itemize}
\item Either $ v_{1}$ and $ v_{2}$ also spanned a bubble in the original
diagram $ \Gamma$. Then setting $ \Gamma_{2}:= \mK_{
\overline{\Gamma}_{1}} \Gamma$ will be again an Anderson--Feynman diagram
(since we preserve that every inner node is connected to four edges). Likewise,
we can perform the same extractions on $ \Gamma_{2}$, that led to $
\tilde\Gamma $ from $ \Gamma$, to construct $ \tilde\Gamma_{2}$.

\item Or the bubble $ \overline{\Gamma}_{1}$ was created by a rewiring of edges
through the operators $ \hat{\mC}$ (not $ \hat{\mC}^{(1)}$, since the edges are
not labelled).
Because only edges that are incoming into divergences are rewired, we know that
say $ e_{1} = (v_{1}, v_{\star} ( \gamma) ) = e_{\star} ( \gamma, \tilde\Gamma)$ (where we
now included a direction of the edge), for some $ \gamma \in \mf{F}_{s}$. In
particular, this implies $ v_{2}= v_{\star} ( \gamma) $.

On the other hand, $ e_{2}= ( v_{2}, v_{1}) $ which must equal $
\tau^{-1}_{ \tilde\Gamma, \Gamma} (e^{\rm out}_{\gamma}) $,
the unique outgoing edge from $ \gamma$ in $ \tilde\Gamma $.
In fact, $ e_{2}= e^{\rm out}_{\gamma}$, i.e. $e^{\rm out}_{\gamma}$ has not
been subject to any rewiring in $ \tilde\Gamma$. If it were,
 there must exist $ \gamma ' \in \mf{F}_{s}$ such that $ e^{\rm out}_{\gamma}
= e_{\star}( \gamma' , \Gamma) $ and $ v_{\star}( \gamma'')= v_{1}$, which
eventually yields $\tau^{-1}_{ \tilde\Gamma, \Gamma} (e^{\rm out}_{\gamma}) =
(v_{2}, v_{ 1})$.
However, this would imply that $ \Gamma$ contains a connected component with no
legs, but only the two divergences $ \gamma $ and $ \gamma' $ connected by the two
edges $ e_{\gamma}^{\rm out}= e_{\star}( \gamma', \Gamma) $ and $
e_{\gamma'}^{\rm out} = e_{\star}( \gamma, \Gamma) $, which is impossible.

Hence, the original diagram $ \Gamma$ must contain a sub--diagram of the form
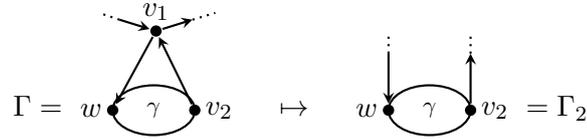
\begin{equation*}
\begin{aligned}
\Gamma=
\begin{tikzpicture}[thick,>=stealth,scale=0.3,baseline=-0.2em]
      \tikzset{
        dot/.style={circle,fill=black,inner sep=1.5pt}
      }
      \node[above] at (-3.2,0.1) {};
      \node[above] at ( 3.0,0.1) {};
      \draw (-0.4,0) ellipse (1.8 and 1.1);
      \node[dot] (bL) at (-2.2,0)   {};
      \node at (-.4,0) {\scalebox{0.8}{$\gamma$}};
      \node[dot] (vL) at (1.4,0) {};   
      \node[dot] (w) at (-.3,3.5)   {};
      \node[right] at (vL) {$v_2$};
      \node[left] at (bL) {$w$};
	\node[above] at (w) {$v_1$};
\draw[->] (w) to (bL);
\draw[->] (vL) to (w);
\draw[->] (w) to (1.3,4);
\draw[->] (-1.9,4) to (w);
\draw[dotted] (-1.9,4) to (-2.8,4.3);
\draw[dotted] (1.3,4) to (2.3,4.3);
    \end{tikzpicture}
\quad\mapsto \quad
\begin{tikzpicture}[thick,>=stealth,scale=0.3,baseline=-0.2em]
      \tikzset{
        dot/.style={circle,fill=black,inner sep=1.5pt}
      }
      \node[above] at (-3.2,0.1) {};
      \node[above] at ( 3.0,0.1) {};
      \draw (-0.4,0) ellipse (1.8 and 1.1);
      \node[dot] (bL) at (-2.2,0)   {};
      \node[dot] (vL) at (1.4,0) {};   
      \node at (-.4,0) {\scalebox{0.8}{$\gamma$}};
      \node[right] at (vL) {$v_2$};
      \node[left] at (bL) {$w$};
\draw[->] (vL) to (1.4,2.5);
\draw[->] (-2.2,2.5) to (bL);
\draw[dotted] (-2.2,2.5) to (-2.2,3.3);
\draw[dotted] (1.4,2.5) to (1.4,3.3);
    \end{tikzpicture}
= \Gamma_{2}
\end{aligned}
\end{equation*}
where $ w := e_{\star, +} ( \gamma, \Gamma) $.
Thus, when removing the edges $ e_{\star } ( \gamma, \Gamma) =
e_{\gamma}^{\rm in}$ and $ e_{\gamma}^{\rm out}$, and inserting the divergence
$ \gamma$ in place of $ v_{1}$ (see the second diagram in the display above), we
create a new Anderson--Feynman diagram $ \Gamma_{2}$ (all inner nodes still
have degree four).
Performing the same extractions that led from $ \Gamma$ to $ \tilde\Gamma$ will
now lead from $ \Gamma_{2}$ to $\tilde\Gamma_{2}$.
\end{itemize}

Therefore, $ \tilde{\Gamma}_{2}$ and $ \overline{T}_{2}$ satisfy our induction
hypothesis, and we conclude that $ \tilde{\Gamma}_{2}$ is a nested bubble
diagram. Consequently, also $ \tilde{\Gamma}$ must have been a nested bubble diagram.
\end{proof}

\subsection{No renormalised Laplacian}\label{sec:no-laplace}

In this subsection, we complete the proof that there
is no contribution to the weak coupling limit appearing from second order Taylor
corrections, which could
lead potentially to a renormalised Laplacian in \eqref{e:resolvent-renormalised}.
Corollary~\ref{cor_unsafe} already yields that any negative Anderson--Feynman diagram $
\Gamma$ with $2$ legs vanishes in the weak coupling limit.
Unfortunately, the derived upper bound falls short of treating the case of $4$ legs.
In this case,
Lemma~\ref{lem_hepp_unsafe} yields
\begin{equation}\label{e_upperbnd_4legs}
\begin{aligned}
\lambda_{\ve}^{|E_{\star}|} \big| \hat{\Pi}_{\ve} \Gamma ( \varphi) \big|
\lesssim
\sum_{ T \in \mT_{V_{\star}}}
\sum_{\tilde\Gamma}
\big( \log{ \tfrac{1}{\ve}}  \big)^{ \widehat{\mathrm{Null}} (T, \tilde\Gamma) - | \overline{T}| }\,,
\end{aligned}
\end{equation}
cf.~\eqref{eq_upperbound_negdiagrams}.
One may hope that $ \widehat{\mathrm{Null}}( T, \tilde\Gamma) < |
\overline{T}|  $ for all $\tilde\Gamma$ and $ T \in \mT_{V_{\star}}$, so that
\eqref{e_upperbnd_4legs} vanishes.
However, this is not the case, and one can construct examples of diagrams and trees in which $
\widehat{\mathrm{Null}}( T, \tilde\Gamma) = |
\overline{T}|  $.

While we expect the estimates performed so
far in this section to hold for all models (mutatis mutandis), whether or not we
see a contribution from {higher order} Taylor remainders (which typically takes
to form of a Laplacian renormalisation) depends on the model. For our Anderson
model, we can prove that such renormalisation does not appear.
Indeed, the remainder of this section is aimed at improving the bound~\eqref{e_upperbnd_4legs}:

\begin{proposition}\label{prop_improve}
Let $ \Gamma $ be a negative 4-Anderson--Feynman diagram. Then
\begin{equation*}
\begin{aligned}
 |\hat{\Pi}_{\ve} \Gamma ( \varphi) |
\lesssim
\big( \log{ \tfrac{1}{\ve}} \big)^{| E_{\star}|/2 -1}\,.
\end{aligned}
\end{equation*}
\end{proposition}

In the proof of Lemma~\ref{lem_hepp_unsafe}, we performed a Taylor expansion up to first order,
to estimate a renormalised diagram $ \Pi_{\varepsilon} ( 1- \hat{\mC}_{\gamma})
\Gamma= \Pi_{\varepsilon} ( 1-
\hat{\mC}_{\gamma} -\hat{\mC}_{ \gamma}^{(1)}) \Gamma $.
There we didn't make use of the fact that prior to any renormalisation, every
edge of $ \Gamma$ corresponds to a Green's kernel $ G_{\varepsilon}$, instead
all our estimates are based on power counting.
For the improved upper bound, we will take this additional information into account.
In Lemma~\ref{lem_emptysafe} below, we will postprocess our bound by adding and subtracting
a second order term to the renormalisation map $ \hat{\mR}^{(1)}$.
This additional renormalisation produces a kernel $ \Delta
G_{\ve}$. Formally, $\Delta G_\ve$ is a kernel of degree $-4$, so that
power counting would lead to a logarithmic blow-up after integrating.
However, we will use the identity $ ( 1- \Delta) G_{\varepsilon} = \delta_{
\varepsilon}$ to avoid this blow-up.
Before we delve into the proof, we present the main argument that allows
us to conclude the improved bound in Proposition~\ref{prop_improve}.

\begin{proof}[Proof of Proposition~\ref{prop_improve}]
	We start from the representation in Definition~\ref{def_R1}:
\begin{equation*}
  \hat{\Pi}_{\ve} \Gamma (\varphi) = \int_{(\TT^d)^{V_\star}} \sum_{\mf{F} \in
\mF^{-}_\Gamma} (-1)^{|\mf{F}|} \overline{\mW}_{\ve} \mf{K}^{(1)}_{\mf{F}} \Gamma (x_{V_\star}) \ud x_{V_\star} \;,
\end{equation*}
where $ \overline{\mW}_{\ve} $ was defined in \eqref{e_updW_labelled}.
{Notice that by Lemma~\ref{lem_disjoin_hepp} the sets $D_{(T, \n)}$
for $T \in \mT_{V_\star}$ and $\n \in \mA^{\#, N_\ve}(T)$ are disjoint.
Here
where $\mA^{\#, N_\ve}$ is as in \eqref{e:ANE} and $N_\ve$ as in \eqref{e:Ne}.
Thus, for a suitable partition of the domain into Hepp sectors, we also have to consider
the remainder volume
\begin{equation*}
  \mathring{D}^\ve = (\TT^d)^{V_\star} \setminus \bigcup_{T \in \mT_{V_\star}}
  \bigcup_{\n \in  \mA^{\#, N_\ve}(T)} D_{(T, \n)} \;.
\end{equation*}}
Therefore, we have the upper bound:
\begin{equation*}
  \begin{aligned}
    | \hat{\Pi}_\ve \Gamma (\varphi)| \leq &  \bigg\vert \sum_{T \in
\mT_{V_{\star}}}
\sum_{\n \in \mA^{\#, N_\ve}} \int_{D_{(T, \n)}} \sum_{\mf{F} \in
\mF^{-}_\Gamma} (-1)^{|\mf{F}|} \overline{\mW}_{\ve}
\mf{K}_{\mf{F}}^{(1)} \Gamma (x_{V_\star}) \ud x_{V_\star} \bigg\vert \\
    & + \bigg\vert \int_{\mathring{D}^\ve} \sum_{\mf{F} \in \mF^{-}_\Gamma}
(-1)^{|\mf{F}|} \overline{\mW}_{\ve} \mf{K}_{\mf{F}}^{(1)} \Gamma (x_{V_\star}) \ud x_{V_\star}\bigg\vert \;.
  \end{aligned}
\end{equation*}
Again, to any $T \in \mT_{V_\star}$ we associate a partition of $\mF^{-}_{\Gamma}$ into
forest intervals of the form $[\mf{F}_s, \mf{F}_s \cup \mf{F}_u ]$, where
$\mf{F}_s$ belongs to the set $\mF_{\Gamma}^{s}(T) \subseteq
\mF_{\Gamma}^-$ of safe forests (with
respect to the tree $T$). We then rewrite the estimate as follows:
\begin{align}
    | \hat{\Pi}_\ve \Gamma (\varphi)| \leq &
\sum_{T \in \mT_{V_{\star}}}
\sum_{\substack{ \mf{F}_s \in \mF_\Gamma^{s}(T) \\ \mf{F}_u = \emptyset }}\sum_{\n \in \mA^{\#, N_\ve} (T)}
\int_{D_{(T, \n)}}  |\overline{\mW}_{\ve}
\mf{K}^{(1)}_{\mf{F}_{s}}
\Gamma (x_{V_\star}) |
\ud x_{V_\star} \label{e:A12} \\
& + \sum_{T \in \mT_{V_{\star}}} \sum_{\substack{ \mf{F}_s \in \mF_\Gamma^{s}(T) \\ \mf{F}_u \neq \emptyset }}
\sum_{\n \in \mA^{\#, N_\ve} (T)}
\bigg\vert
 \int_{D_{(T, \n)}}  \overline{\mW}_{\ve}
 \hat{\mR}_{[\mf{F}_s, \mf{F}_{s} \cup \mf{F}_u ]}^{(1)} \Gamma (x_{V_\star}) \ud x_{V_\star}
\bigg\vert \label{e:A11} \\
    & + \sum_{T \in \mT_{V_{\star}}} \sum_{\mf{F}_s \in \mF_\Gamma^{s}(T)}\sum_{\n \in
\mA^{c, N_\ve}(T)}  \int_{{D}_{(T, \n)}}|\overline{\mW}_{\ve}
\hat{\mR}_{[\mf{F}_s, \mf{F}_s \cup \mf{F}_u ]}^{(1)} \Gamma (x_{V_\star}) |  \ud x_{V_\star} \;. \label{e:A2}
\end{align}
Here we have separated the case $\mf{F}_{u}= \emptyset$ from its complement, and moreover we have defined
\begin{equation}\label{eq_Ac}
  \mA^{c, N_\ve} (T) := \mA(T) \setminus \mA^{\#, N_\ve}(T) \;.
\end{equation}
Notice that it is important that the absolute value in the second term is still
outside the integral, so that we can make use of cancellations that are due to
symmetries.

Both \eqref{e:A12} and \eqref{e:A2} can be controlled using previous arguments.
More precisely, \eqref{e:A12} is bounded by Lemma~\ref{lem_safe_vanish_2k}.
On the other hand, \eqref{e:A2} can be controlled since the sum over $ \mA^{c ,
N_{\ve}}(T)$ is almost restricted to a diagonal, see Lemma~\ref{lem_Ddot} below.

Only, the treatment of \eqref{e:A11} requires some more care and the use of
symmetries.
Here, we can restrict ourselves to the case when $\widehat{\rm{Null}}(T,
\tilde\Gamma) =
|\trim|$, for some $\tilde\Gamma$ in $ \mf{K}^{(1)}_{\mf{F}_{s}} \Gamma = \sum
\tilde\Gamma$ (note
that this condition depends only on $\tilde\Gamma$, $T$, and $
\mathfrak{F}_{u}$) since otherwise \eqref{eq_upperbound_negdiagrams} yields that the
contribution of the term vanishes.
In Lemma~\ref{lem_emptysafe} we show that \eqref{e:A11} is bounded by {$ ( \log
\tfrac{1}{ \ve })^{| E_{\star}|/2-1}$} , using a second--order Taylor estimate.
This concludes the proof.
\end{proof}

We conclude this subsection with the control of the remainder terms
in the above proof.

\begin{lemma}\label{lem_Ddot}
Let $ \Gamma$ be a negative 4-Anderson--Feynman diagram and $ T \in
\mT_{V_{\star}}$ a Hepp tree. Then
\begin{equation*}
\begin{aligned}
\sum_{T \in \mT_{V_{\star}}} \sum_{\mf{F}_s \in \mF_\Gamma^{s}(T)}\sum_{\n \in
\mA^{c, N_\ve}(T)}  \int_{{D}_{(T, \n)}}|\overline{\mW}_{\ve}
\hat{\mR}_{[\mf{F}_s, \mf{F}_s \cup \mf{F}_u ]}^{(1)} \Gamma (x_{V_\star}) |
\ud x_{V_\star}
\lesssim \big( \log \tfrac{1}{\ve } \big)^{ |\trim| -1}\,,
\end{aligned}
\end{equation*}
where $ \mA^{c , N_{\ve}}(T)$ is defined in \eqref{eq_Ac}.
\end{lemma}

\begin{proof}
Let $ T \in \mT_{V_{\star}}$ and $ \mf{F}_{s} \in \mF_{\Gamma}^{s}(T)$, with
associated unsafe forest $ \mf{F}_{u}$.
Following the same steps as in Lemma~\ref{lem_hepp_unsafe} that led to
\eqref{eq_supp100_unsafe}, yields
\begin{equation*}
\begin{aligned}
\sum_{\n \in
\mA^{c, N_\ve}(T)}  \int_{{D}_{(T, \n)}}|\overline{\mW}_{\ve}
\hat{\mR}_{[\mf{F}_s, \mf{F}_s \cup \mf{F}_u ]}^{(1)} \Gamma (x_{V_\star}) |
\ud x_{V_\star}
\lesssim
\sum_{\n \in
\mA^{c, N_\ve}(T)}
\prod_{u \in \overline{T}}
(2^{-\mathbf{n} ( u)}
+ \varepsilon)^{ \hat{\eta}  (u) -d}
 2^{-d \mathbf{n}(u)}\,,
\end{aligned}
\end{equation*}
where now we only sum over scales in $ \mA^{c, N_\ve}(T)$.
In the proof of  Lemma~\ref{lem_exact_heppsplit}, we saw that this is
equivalent to summing over a diagonal, which leaves us with the upper bound $
\big( \log \tfrac{1}{\ve } \big)^{ |\trim| -1}$.
Indeed, the same proof applies because
 $ \hat{\eta}$ is non--negative, in the sense of
Lemma~\ref{lem_unsafe_nonneg}.
\end{proof}

\subsubsection{Second--order Taylor expansion}

The aim of this subsection is to prove the remaining ingredient that we used in the proof of
Proposition~\ref{prop_improve} above:

\begin{lemma}\label{lem_emptysafe}
Let $ \Gamma$ be a negative $4$-Anderson--Feynman diagram, $ T \in
\mT_{V_{\star}}$ a Hepp tree, and $ \mf{F}_{s}$ a safe forest such that the associated
unsafe forest $ \mf{F}_{u} \neq \emptyset$ is non--empty.
Then
\begin{equation*}
\begin{aligned}
\sum_{\n \in \mA^{\#, N_\ve} (T)}
\bigg\vert \int_{D_{(T, \n)}}  \overline{\mW}_{\ve}
\hat{\mR}_{[\mf{F}_s, \mf{F}_{s} \cup \mf{F}_u ]}^{(1)} \Gamma (x_{V_\star}) \ud x_{V_\star}
\bigg\vert
\lesssim
( \log{ \tfrac{1}{\ve}} )^{ | \trim| -1 }
\,.
\end{aligned}
\end{equation*}
\end{lemma}

Before proceeding to the proof of this lemma, we require some preliminary
results. We start by singling out a leaf unsafe divergence.
\begin{lemma}\label{lem:gammastar}
  Fix a negative $4$-Anderson--Feynman diagram $\Gamma$ and $T \in \mT_{V_\star}$ a Hepp
  tree. Assume that the unsafe forest
  $\mf{F}_u$ associated to some safe forest $\mf{F}_s \in \mF_{\Gamma}^{s} (T)$ satisfies
  $\mf{F}_u \neq \emptyset$,
  and that $\widehat{\rm{Null}}(T, \tilde{\Gamma})
  = |\trim|$ for some $\tilde{\Gamma}$ in the decomposition
  $\mf{K}^{(1)}_{\mf{F}_s} \Gamma = \sum \tilde{\Gamma}$. Then there exists a $\gamma_* \in \mf{F}_u$ such
  that $\gamma_*$ is a leaf in $\mf{F}_u$ and there exists a $u \in \trim$ with $\Gamma_u =
  \gamma_*$ and $u = \gamma_*^\uparrow$.
  Moreover, $ \gamma_{\ast}^{\uparrow\uparrow} \neq \dagger$.
\end{lemma}

\begin{proof}
Choose a maximal (with respect to the tree order) $u$ such that $u =
\gamma_*^\uparrow$ for some $\gamma_* \in \mf{F}_u$. By maximal we mean
that for every $v \succ u, v \neq u$ it holds that $\eta(v) = \hat{\eta}(v)$, or
in other words $\mathds{1}_{\gamma^\uparrow}(v) = 0$ for all $\gamma \in
\mf{F}_u$ and $v \succ u, v \neq u$. Then since $\sum_{v \succ u}\hat{\eta}(v)
=0$ by assumption it must hold that $\overline{\deg}(\Gamma_u) = - 2$, cf.
\eqref{e_eta_unsafe}, so that
by Lemma~\ref{lem:neg} we have $\Gamma_u \in \mG_{\Gamma}^{-}$. Moreover, we also
have $\gamma_* \subseteq \Gamma_u$ and we can deduce that $\gamma_* =
\Gamma_u$. Indeed from our assumptions we have that $\eta(v) = \hat{\eta} (v) = 0$  for all
$v \succ u, v \neq u$, and therefore we find that the collection $v \succ u, v \neq u$ corresponds to a
bubble extraction sequence which upon successive contraction reduce $\Gamma_u$ to a sunset
diagram $\<cc4div>$. Therefore, $\Gamma_u$ can not strictly contain a divergence.
Finally, assume $ \gamma_{\ast}^{\uparrow\uparrow} = \dagger$. Then
$ \widehat{\mathrm{deg}}(\bullet) >0 $, cf. \eqref{e_eta_unsafe}, since the original diagram $ \Gamma $ satisfied $
\mathrm{deg} \,  \Gamma = 0 $. However, this would imply that
$\widehat{\rm{Null}}(T, \tilde{\Gamma})
  < |\trim|$, which violates our assumption.
\end{proof}

Let $ \gamma_{\ast} \in \mf{F}_{u} $ be an unsafe divergence as in
Lemma~\ref{lem:gammastar}, and denote with $u_* \in \trim$ the node in $ \trim$
such that $\Gamma_{u_*} = \gamma_*$.
We define slightly modified Hepp sectors for $ \mathbf{n} \in \mA^{\#}(T)$. Namely, we set, with $v_\star
= v_\star(\gamma_*)$:
\begin{equation}
  \begin{aligned}
    D^*_{(T, \n)} := \big\{ x_{V_\star} & \in (\TT^d)^{V_\star}  \, \colon \, \\
    & (\sqrt{d} \pi)  2^{- \mathbf{n} (v \wedge w)-1} < | x_{v} - x_{w} |
    \leqslant (\sqrt{d} \pi)   2^{- \mathbf{n} (v \wedge w)} \ \  \forall v, w \in
    (V_{\star} \setminus V(\gamma_*) ) \cup \{v_\star\} \;,  \\
    & (\sqrt{d} \pi)  2^{- \mathbf{n} (v \wedge w)-1} < | x_{v} - x_{w} |
    \leqslant (\sqrt{d} \pi)   2^{- \mathbf{n} (v \wedge w)} \ \  \forall v, w \in
    V(\gamma_*)
    \big\} \;.
  \end{aligned}
\end{equation}
In particular, this choice of set does not impose any condition on the
lengthscales $|x_v - x_w|$ when $v \in V_\star \setminus V(\gamma_*)$ and $w \in
V(\gamma_*) \setminus \{v_\star\}$.

Let us write
$ T_{\ast} = T_{u_{\ast}}$ for the sub--Hepp tree rooted at $u_{\ast} $,
$\n_{\ast} = \n \vert_{T_{\ast}}$, and
$ T' $ for the Hepp tree where we replaced the entire subtree $ T_{\ast}$ with
a single leaf $ v_{\star}$, and $\n'$ the restriction of the scales accordingly.
This allows for the following convenient decomposition:
\begin{equation}\label{eq_def_modHeppSec}
\begin{aligned}
D^*_{(T, \n)}
=
\big\{
x_{V_{\star}} \in (\TT^d)^{V_\star}
\,:\,
x_{(V_{\star}\setminus V ( \gamma_{\ast})) \cup v_{\star}}
\in
D_{ T' , \n '}
\text{ and }
x_{V ( \gamma_{\ast}) \setminus v_{\star}} \in
D_{(T_{\ast}, \n_{\ast})} (x_{ v_{\star}})
\big\}\,,
\end{aligned}
\end{equation}
where
\begin{equation*}
\begin{aligned}
D_{(T_{\ast}, \n_{\ast})} (x_{ v_{\star}})
:=
\big\{
x_{V ( \gamma_{\ast}) \setminus v_{\star}} \in ( \TT^{d})^{V ( \gamma_{\ast})
\setminus v_{\star}}
\,:\,
x_{V ( \gamma_{\ast})} \in D_{( T_{\ast}, \n_{\ast})}
\big\}\,.
\end{aligned}
\end{equation*}
Since we are only dropping some assumptions on
the internal variables, we find that $D_{(T, \n)} \subseteq D^*_{(T, \n)}$.

While having lost information of the exact scale assignment
between variables in $ V ( \gamma_{\ast})$ and outside of $ V ( \gamma_{\ast})$
in $ D^{\ast}_{( T, \n )}$, we can almost fully recover this information.
The following result is a consequence of $ \gamma_{\ast}$ being an unsafe divergence
and the triangle inequality.

\begin{lemma}\label{lem_recover}
	Let $ \gamma_{\ast}$ and $ u_{\ast}$ be as in Lemma~\ref{lem:gammastar} and $\mathbf{n} \in \mA^{\#}(T)$. Then
for all $ v \in V ( \gamma_{\ast}) \setminus v_{\star}( \gamma_{\ast})$  and $ w \in V_{\star}
\setminus V ( \gamma_{\ast} )$, we have
\begin{equation*}
\begin{aligned}
( \sqrt{d} \pi )
2^{- \n ( v \wedge w)-1} ( 1-2^{- ( \n (u_{\ast}) - \n ( v \wedge w))} )
<
| x_{v}- x_{w}|
\leqslant
 ( \sqrt{d} \pi ) 2^{ - \n ( v \wedge w)}
( 1+ 2^{- ( \n (u_{\ast}) - \n ( v \wedge w))} )\,,
\end{aligned}
\end{equation*}
for all $  x_{V_{\star}} \in D^{\ast}_{(T,\n)}$.
In particular
\begin{equation*}
\begin{aligned}
\frac{1}{2}
 (\sqrt{d} \pi )
2^{- \n ( v \wedge w)-1} <
| x_{v}- x_{w}|
\leqslant
2 ( \sqrt{d} \pi ) 2^{ - \n ( v\wedge w)}
\,.
\end{aligned}
\end{equation*}
As a consequence, if
	$p(u_*)$ is the parent of $u_*$ in $T$, then we have
\begin{equation*}
  \begin{aligned}
  \mathrm{Vol} \big( D_{(T, \n)}^{\ast}\setminus D_{(T, \n)} \big)
  \lesssim
  2^{- d ( \n ( u_{\ast}) - \n ( p ( u_{\ast})) )}
  \prod_{u \in \trim} 2^{- d \n (u)}\,.
  \end{aligned}
  \end{equation*}

\end{lemma}

\begin{proof}
We write $ v_{\star} = v_{\star}( \gamma_{\ast})$.
Then by application of the triangle inequality
\begin{equation}\label{e:ubd}
  |x_v - x_w| \leq |x_v - x_{v_\star}| + |x_{v_\star} - x_w| \leq (\sqrt{d}
\pi) 2^{- \n(v \wedge w)} (1 + 2^{- (\n( u_{\ast}) - \n(v \wedge w))})  \;,
\end{equation}
and
\begin{equation}\label{e:lbd}
  |x_v - x_w| \geq |x_{v_\star} - x_w| - |x_v - x_{v_\star}| \geq (\sqrt{d}
\pi) 2^{- \n(v \wedge w)-1} (1 - 2^{- (\n ( u_{\ast}) - \n(v \wedge w))})  \;.
\end{equation}
where we used that $\n( u_{\ast}) \leqslant \n ( v \wedge v_{\star})$.
Moreover, we used that $ v_{\star} \wedge w = v \wedge w$, because $ w$ is not
a descendant of $ u_{\ast}$. The second statement in the lemma is simply a consequence of the fact that $
2^{- (\n ( u_{\ast}) - \n(v \wedge w))} \leqslant 1/2$ because $ \n \in
\mA^{\#}(T)$.

As for the last claim, we observe that if $x_{V_\star} \in
D^*_{(T,\bf{n})} \setminus D_{(T,\bf{n})}$, then by the first two points and the
definition of $D_{(T, \bf{n})}$ there must exist (at least) one pair $v,w$
such that $ v \in V ( \gamma_*) \setminus \{v_\star(\gamma_*)\} $ and $ w \in
V_{\star} \setminus V ( \gamma_*)$. Consequently,
\begin{equation*}
  \begin{aligned}
    \frac{1}{\sqrt{d} \pi}|x_v - x_w| \in & [2^{- \n(v \wedge w)},  2^{- \n(v \wedge w)} (1 + 2^{- (\n( u_{\ast}) - \n(v \wedge w))})] \\
    & \qquad \cup [ 2^{- \n(v \wedge w)-1} (1 - 2^{- (\n ( u_{\ast}) - \n(v \wedge w))}),  2^{- \n(v \wedge w)-1} ] \;.
  \end{aligned}
\end{equation*}
Since the measure of this union of intervals is bounded by $$2^{- \n(v \wedge
w)}2^{- (\n( u_{\ast}) - \n(v \wedge w))} \leq 2^{- \n(v \wedge
w)}2^{- (\n( u_{\ast}) - \n(p (u_*)))}\;,$$ we obtain the desired result by summing
over all possible choices of pairs $u, w$.
\end{proof}

The following lemma shows that we may replace the integration domain $D_{(T,
\n)}$ by $D_{(T, \n)}^{\ast}$.

\begin{lemma}\label{lem_Dstarenough}
Let $ \Gamma$ be a negative 4-Anderson--Feynman diagram, $ T \in
\mT_{V_{\star}}$ a Hepp tree, and $ \mathfrak{F}_{s}$ a safe forest with associated
unsafe forest $\mf{F}_{u} \neq \emptyset$.
Furthermore, let $ \tilde\Gamma \in \mD_{\Gamma}^{\rm{lab}}$ be a term in the decomposition $\mf{K}_{\mf{F}_s}^{(1)} \Gamma = \sum \tilde{\Gamma}$,
such that $\widehat{\rm{Null}}(T, \tilde\Gamma) = |\trim|$.
Then
\begin{equation*}
\begin{aligned}
\sum_{\n \in \mA^{\#, N_\ve} (T)}
\int_{D_{(T, \n)}^{\ast}\setminus D_{(T, \n)}}  |\overline{\mW}_{\ve}^{(\mK)}
 \tilde\Gamma (x_{V_\star})| \ud x_{V_\star}
\lesssim
\big( \log{ \tfrac{1}{ \ve}} \big)^{| \trim| -1}
\,.
\end{aligned}
\end{equation*}
\end{lemma}

\begin{proof}
First, we fix
$ \gamma_{\ast}$ as in Lemma~\ref{lem:gammastar}.
We then follow the proof of Lemma~\ref{lem_hepp_unsafe} up to \eqref{eq_supp111_unsafe}, which yields
\begin{equation*}
\begin{aligned}
\int_{D_{(T, \n)}^{\ast}\setminus D_{(T, \n)}}  |\overline{\mW}_{\ve}^{(\mK)} \tilde{\Gamma}
(x_{V_\star})| \ud x_{V_\star}
\lesssim
\mathrm{Vol} \big( D_{(T, \n)}^{\ast}\setminus D_{(T, \n)} \big)
\prod_{u \in \overline{T}}
(2^{-\mathbf{n} ( u)}
+ \varepsilon)^{ \hat{\eta}  (u) -d}\,.
\end{aligned}
\end{equation*}
To estimate the volume term, we use  Lemma~\ref{lem_recover} with $  p (
u_{\ast})$ being parent of $u_{\ast}$ in $ \trim $, and recall that $
\gamma_{\ast}^{\uparrow \uparrow} \neq \dagger$ thus $u_{\ast}$ is not the root of $
\trim$ and indeed has a parent. Hence, because $\mathbf{n} \in
\mA^{\#}(T)$, we have
\begin{equation*}
\begin{aligned}
\mathrm{Vol} \big( D_{(T, \n)}^{\ast}\setminus D_{(T, \n)} \big)
\lesssim
2^{- d ( \n ( u_{\ast}) - \n ( p ( u_{\ast})) )}
\prod_{u \in \trim} 2^{- d \n (u)}\,.
\end{aligned}
\end{equation*}
Now, because $ \widehat{\mathrm{Null}} ( T, \tilde\Gamma ) = | \trim|$, and therefore $ \hat{\eta} \equiv
0 $, we have
\begin{equation}\label{e_supp1_DsameDstar}
\begin{aligned}
\sum_{\n \in \mA^{\#, N_\ve} (T)}
2^{- d ( \n ( u_{\ast}) - \n ( p ( u_{\ast})) )}
\prod_{u \in \trim}
(2^{-\mathbf{n} ( u)}
+ \varepsilon)^{ \hat{\eta}  (u) -d}
2^{- d \n (u)}
\lesssim N_{\ve}^{|\trim| -1}
\,,
\end{aligned}
\end{equation}
where we used that
\begin{equation*}
	\begin{aligned}
	\sum_{m =1}^{N_{\varepsilon}}
	2^{dm}
	\sum_{n =m}^{N_{\varepsilon}} 		2^{- dn}
	\lesssim
	\sum_{m =1}^{N_{\varepsilon}} 2^{-d (m-m)} = N_{\varepsilon}\,.
	\end{aligned}
\end{equation*}
This concludes the proof.
\end{proof}

Finally, we can proceed to the key step of the proof of Lemma~\ref{lem_emptysafe},
namely performing
one additional second--order Taylor expansion.
In the following, it suffices to restrict to the case  $\widehat{\rm{Null}}(T, \tilde{\Gamma})
  = |\trim|$ for some $\tilde{\Gamma}$ in the decomposition
  $\mf{K}^{(1)}_{\mf{F}_s} \Gamma = \sum \tilde{\Gamma}$,
  since otherwise the bound from Lemma~\ref{lem_hepp_unsafe} is sufficient for
  Lemma~\ref{lem_emptysafe}.

First, for $K \colon \TT^4 \to
\RR$, we introduce the operators
\begin{equation*}
  \begin{aligned}
  \mH_{2} K ( x, y | x_{\star}) &:=
  K ( x- y)
  - K (x_{\star} -y) -  \sum_{i = 1}^{d}
  (x^{i}- x_{\star}^{i})
  \partial_{i}K ( x_{\star} - y)\\
  &\qquad - \frac{1}{2}
   \sum_{ \substack{i,j = 1 }}^{d}
  (x^{i}- x_{\star}^{i})
  (x^{j}- x_{\star}^{j})
  \partial_{i}\partial_{j}K ( x_{\star} - y)\\
  &=
  \sum_{ \substack{\beta \in \NN^{d}\\ | \beta| = 3}}
  \frac{3}{ \beta!}
  (x - x_{\star})^{ \beta}
  \int_{0}^{1} (1-t)^{2}
  \partial^{\beta} K
  ( x_{\star} - y
  + t( x- x_{\star}  ))
  \ud t \;.
  \end{aligned}
  \end{equation*}
In addition we extract the symmetric bit of this expansion:
\begin{equation*}
  \begin{aligned}
  \mH^{\rm{sym}}_2 K
  ( x, y| x_{\star})
  & := \frac{1}{2}   \sum_{ \substack{i = 1}}^{d}
  \big(x^i- x_{\star}^{i}\big)^{2}
  \partial_{i}^{2} K ( x_{\star}- y )\,.
  \end{aligned}
\end{equation*}
Now, we alter the kernel assignment $\mK$ from Lemma~\ref{lem_pushonKernel} only with respect
to the renormalisation of the divergence $\gamma_*$, which we choose as in
Lemma~\ref{lem:gammastar}. Namely, we define new
kernel assignments $\mL$ and $\mM$ through
\begin{equation}\label{eq_def_L}
  \mL_{e, \ve} (x_V) := \begin{cases}
    (\mH_{2} G_{\ve})( x_{\tilde{e}_{+}},x_{ \tilde{e}_{-}}| x_{\mf{b}(e)}) \;, & \text{ if } e = e_{\star}(\gamma_*, {\Gamma}) \;,\\
    \mK_{e, \ve} (x_V)\;, & \text{ otherwise}\;,
  \end{cases}
\end{equation}
for every $ e \in E $ and $ \tilde{e} = \tau^{-1}_{ \tilde\Gamma, \Gamma}(
e)$. Similarly
\begin{equation}\label{eq_def_M}
  \mM_{e, \ve} (x_V) := \begin{cases}
    (\mH_{2}^{\rm{sym}} G_{\ve})( x_{\tilde{e}_{+}},x_{\tilde{e}_{-}}| x_{\mf{b}(e)}) \;, & \text{ if } e = e_{\star}(\gamma_*, {\Gamma}) \;,\\
    \mK_{e, \ve}\;,  (x_V) & \text{ otherwise}\;.
  \end{cases}
\end{equation}
We note that since we chose $\gamma_*$ to be a leaf of $\mf{F}_u$, we have
$\mf{b}(e) = v_\star(\gamma_*)$. The next result is a consequence of
antisymmetry.

\begin{lemma}\label{lem_offdiag_nochange}
Let $ \Gamma$ be a negative $4$-Anderson--Feynman diagram, $ T \in
\mT_{V_{\star}}$ a Hepp tree, and $ \mathfrak{F}_{s}$ a safe forest with associated
unsafe forest $\mf{F}_{u} \neq \emptyset$.
Furthermore, let $ \tilde\Gamma \in \mD_{\Gamma}^{\rm{lab}}$ be a term in the decomposition $\mf{K}_{\mf{F}_s}^{(1)} \Gamma = \sum \tilde{\Gamma}$,
such that $\widehat{\rm{Null}}(T, \tilde\Gamma) = |\trim|$.
Then
\begin{equation}\label{e:identity-to-prv}
\int_{ D_{( T, \mathbf{n})}^{\ast}}
\overline{\mW}_{\varepsilon}^{(\mK)}
\tilde\Gamma
(x_{V_{\star}}) \ud x_{V_{\star}}
=
\int_{ D_{( T, \mathbf{n})}^{\ast}}
\big( \overline{\mW}_{\varepsilon}^{( \mM)}
 \tilde\Gamma
- \overline{\mW}_{\varepsilon}^{( \mL)}
\tilde\Gamma\big)
(x_{V_{\star}}) \ud x_{V_{\star}}\,.
\end{equation}
\end{lemma}

\begin{proof}
We follow the same lines as in the proof of Lemma~\ref{lem_sym} to show that
the following sub integral (of the right--hand side of \eqref{e:identity-to-prv}) vanishes:
\begin{equation*}
\begin{aligned}
\sum_{ \substack{i,j =1\\ i \neq j }}^{d}
\int_{
D_{( T_{\ast}, \n_{\ast})} ( x_{v_{\star} })
}
\big( x^{i}_{ e_{\star , +}} - x_{ v_{\star}}^{i}\big)
\big( x^{j}_{ e_{\star, +}} - x_{ v_{\star} }^{j}\big)
\prod_{e \in \tilde{E} ( \gamma_{\ast} )}
G_{\ve}(x_{e_{+}} - x_{e_{-}})
 \ud x_{ V ( \gamma_{\ast} ) \setminus v_{\star}}\,,
\end{aligned}
\end{equation*}
where $ v_{\star} = v_{\star} ( \gamma_{\ast}) $ and $ e_{\star}= e_{\star}(
\gamma_{\ast}, \tilde\Gamma)$.
Here we used Lemma~\ref{lem_onlyNeed_emptysafe}, which yields that under the
assumptions of this lemma, the labels $ \mf{h}$ and $ \mf{n}$ (from $
\tilde\Gamma$) vanish on $\tilde{\gamma_*}= \tau_{\tilde{\Gamma}, \Gamma}^{-1}
\gamma_*$ (except that $\mf{n}$ might not vanish on $v_\star(\gamma_*)$, but
this does not affect the expression above).
However, compared to Lemma~\ref{lem_sym} there are
two key differences: there is an additional monomial, and we
integrate over the domain $ D^{\ast}_{(T, \n)}$ instead of the full space $
( \TT^{d})^{V_{\star} }$. Thus, some further arguments are necessary to justify
that the integral still vanishes due to antisymmetry.

Let $f $ be of the same form as in Lemma~\ref{lem_sym},
and set $ v_{0}= v_{\star}$, $ v_{1} = x_{e_{\star, +}}$ (and an arbitrary
order $ 2, \ldots, n $ for the remaining nodes in $ V(\gamma_{\ast})$).
Again, by shifting
all integration variables by $ x_{0}$, we see that the integral is independent
of $ x_{0}$
\begin{equation*}
\begin{aligned}
\int_{D_{( T_{\ast}, \n_{\ast})} ( x_{0} )} ( x_{1}^{i}- x_{0}^{i})( x_{1}^{j}- x_{0}^{j})  f ( x_{0,\ldots, n}) \ud
x_{1 ,\ldots, n}
=
\int_{D_{( T_{\ast}, \n_{\ast})} ( 0)}  x_{1}^{i} x_{1}^{j} f ( 0, x_{1,\ldots, n}) \ud
x_{1 ,\ldots, n}\,.
\end{aligned}
\end{equation*}
Now, define the function
  \begin{equation*}
    \begin{aligned}
    g ( x_{1}) :=
    x_{1}^{j}
    \int_{ ( \TT^{d})^{n-1}}
    \mathds{1}_{D_{( T_{\ast}, \n_{\ast})} ( 0)} ( x_{1 ,\ldots, n})  f ( 0, x_{1,\ldots, n}) \ud
    x_{2,\ldots, n}\,,
    \end{aligned}
    \end{equation*}
and recall that $ \uppi^{(i)} : \TT^{d} \to \TT^{d}$ is the map that flips the
sign of the $i$--th coordinate.
Again we have that $ g ( x_{1}) =g (
\uppi^{(i)}x_{1})$, because, with $D_0 = D_{( T_{\ast}, \n_{\ast})} ( 0)$
\begin{equation*}
  \begin{aligned}
  &g ( \uppi^{(i)} x_{1} ) \\
  &=
  (\uppi^{(i)}x_{1})^{j}
  \int\limits_{ ( \TT^{d})^{n-1}}
  \mathds{1}_{D_0}(\uppi^{(i)} x_{1}, x_{2,\ldots, n}) \prod_{\{k, \ell\} \in I\setminus\{1\}  } G_{\ve}(
  x_{\ell} - x_{k})
  \prod_{ \substack{\{1,\ell\} \in I }} G_{\ve}( x_{\ell} - \uppi^{(i)} x_{1}) \ud x_{2
  ,\ldots, n}\\
  & =
  x_{1}^{j}
  \int\limits_{ ( \TT^{d})^{n-1}}
  \mathds{1}_{D_0}(x_{1}, \uppi^{(i)}x_{2,\ldots, n}) \prod_{\{k, \ell\}\in I\setminus\{1\} }  G_{\ve}(
   \uppi^{(i)} x_{\ell} - \uppi^{(i)} x_{k})
  \prod_{ \substack{\{1,\ell\} \in I }} G_{\ve}(\uppi^{(i)} x_{\ell} - x_{1}) \ud x_{2
  ,\ldots, n}\,,
  \end{aligned}
  \end{equation*}
where we first used $ i \neq j $, and in the last equality that $ G_{\ve}( \cdot)$ is radially
symmetric, as well as the fact that
  \begin{equation*}
    \begin{aligned}
    \uppi^{(i)} \big( x_{1}, (\uppi^{(i)}) x_{2,\ldots, n}\big) \in
    D_0 \quad \text{if and only if }\quad
    (x_{1}, \uppi^{(i)} x_{2,\ldots, n}) \in D_0\,,
    \end{aligned}
    \end{equation*}
since the definition of $D^{\ast}_{ (T_{\ast}, \n)}$ only depends on differences
between coordinates.
Finally, we perform the change of variables $ x_{k} \mapsto
\uppi^{(i)} x_{k}$ in the above integral, which yields as desired that
$g ( \uppi^{(i)} x_{1} )
=
g (x_{1} ) $.
Overall, we  have
\begin{equation*}
\begin{aligned}
\int_{D_{( T_{\ast}, \n_{\ast})} ( 0)}  x_{1}^{i} x_{1}^{j} f ( 0, x_{1,\ldots, n}) \ud
x_{1 ,\ldots, n}
&=
\int_{ \TT^{d}}
\mathds{1}\{ \tfrac{ |x_{1}|}{ \sqrt{d} \pi} \in J\}\
  x_{1}^{i} g( x_{1}) \ud x_{1}\\
& = \int_{ \TT^{d}}
\mathds{1}\{x_{1}^{i}>0, \tfrac{ |x_{1}|}{ \sqrt{d} \pi} \in J \}\
  x_{1}^{i} g( x_{1}) \ud x_{1}\\
&\qquad+
\int_{ \TT^{d}}
\mathds{1}\{ ( \uppi^{(i)}x_{1})^{i}>0, \tfrac{ |x_{1}|}{ \sqrt{d} \pi}
\in J \}\
  x_{1}^{i} g( x_{1}) \ud x_{1}\,,
\end{aligned}
\end{equation*}
where we wrote $ J :=  (2^{- \n( v_{1} \wedge
v_{\star})-1}, 2^{- \n( v_{1} \wedge v_{\star})}] $.
Performing the change of variable $ x_{1} \mapsto \uppi^{(i)} x_{1}$ in the
second integral, and using that $ g ( x_{1}) = g ( \uppi^{(i)}x_{1}) $, shows
that the expression vanishes.
This concludes the proof.
\end{proof}

We proceed by estimating the integrals with respect to the kernel assignment $
\mM$ and $ \mL $ from Lemma~\ref{lem_offdiag_nochange} separately. We start with
the kernel assignment $\mL$, which corresponds to a second order Taylor
expansion about the divergence $\gamma_*$.

\begin{lemma}[Second--order Taylor]\label{lem_hepp_unsafe_improve}
Let $ \Gamma$ be a negative 4-Anderson--Feynman diagram, $ T \in
\mT_{V_{\star}}$ a Hepp tree, and $ \mathfrak{F}_{s}$ a safe forest with associated
unsafe forest $\mf{F}_{u} \neq \emptyset$.
Furthermore, let $ \tilde\Gamma \in \mD_{\Gamma}^{\rm{lab}}$ be a term in the decomposition $\mf{K}_{\mf{F}_s}^{(1)} \Gamma = \sum \tilde{\Gamma}$,
such that $\widehat{\rm{Null}}(T, \tilde\Gamma) = |\trim|$.
Then, with $ \mL$ as in \eqref{eq_def_L}, we have
\begin{equation}\label{eq_unsafe_est}
\begin{aligned}
\sum_{{\bf n} \in \mA^{\#, N_{\ve}} ( T)}
\sup_{x \in D_{( T, \mathbf{n})}^{\ast}}
\big| \overline{\mW}_{\varepsilon}^{(\mL)}
\tilde\Gamma (x_{V_{\star}})\big|
\prod_{u \in \overline{T}} 2^{-d \mathbf{n} (u)}
\lesssim
( \log{ \tfrac{1}{\ve}} )^{  |\trim|-1  }
\,.
\end{aligned}
\end{equation}
\end{lemma}

\begin{proof}
First, we recall that the valuation with respect to the integration kernel $ \mL$
reads as follows, similarly to \eqref{eq_valuation_mK},
\begin{equation}\label{eq_defI12}
\begin{aligned}
& \overline{\mW}_{\ve}^{(\mL)}\tilde\Gamma (x_{V_{\star}})
=
\int_{ ( \TT^{d})^{L}}
\varphi ( x_{L})
\big(\mH_{2} G_{\ve}\big) ( x_{ e_{\star,+}}, x_{ e_{\star,-}} |  x_{
v_{\star} ( \gamma_{\ast})} )
\prod_{e \in \tilde{E} \setminus e_{\star}  }
 \mK_{ \tau_{\tilde\Gamma, \Gamma}(e), \varepsilon}( x_{V_{\star}})
\ud x_{L} 	\,,
\end{aligned}
\end{equation}
where we wrote $ e_{\star} = e_{\star}( \gamma_{\ast}, \tilde\Gamma)$,
and used that $ \mf{b} (
e_{\star})= v_{\star}( \gamma_{\ast}) $ since $ \gamma_{\ast}$ is a leaf in $
\mf{F}_{u}$.

We proceed as in the proof of Lemma~\ref{lem_hepp_unsafe}, with
the difference that we now have the improved bound, with $ u_{\ast} =
\gamma_{\ast}^{\uparrow}$,
\begin{equation*}
\begin{aligned}
|\mH_{2} G_{\ve}( x_{ e_{\star,+}}, x_{ e_{\star,-}}|  x_{b (
e_{\star})} )|
& \lesssim
\frac{ 2^{- 3 \mathbf{n}( u_{\ast})} }{( 2^{- \mathbf{n}(
{\gamma '}_{m}^{\uparrow \uparrow})}  + \ve )^{3} }
( | x_{ e_{\star,+}}- x_{ e_{\star,-}}|+ \ve)^{-2}\,,
\end{aligned}
\end{equation*}
where $ \gamma '_{m}$ is the largest divergence of $ \gamma_{\ast}=  \gamma_{1}, \ldots,
\gamma_{M}$ such that $ e_{\star} ( \gamma'_{m}, \Gamma) = e_{\star}$.
Overall, this yields the improved upper bound
\begin{equation*}
\begin{aligned}
&\big|\overline{\mW}_{\ve}^{(\mL)}
\tilde\Gamma (x_{V_{\star}}) \big|
 \lesssim
\frac{ 2^{- 3\mathbf{n}( \gamma_{\ast}^{\uparrow})} }{( 2^{- \mathbf{n}(
\gamma_{\ast}^{\uparrow \uparrow})}  + \ve )^{3} }
\left\{\prod_{ \gamma \in \mf{F}_{u}\setminus \gamma_{\ast}}
\frac{ 2^{- 2\mathbf{n}( \gamma^{\uparrow})} }{( 2^{- \mathbf{n}(
\gamma^{\uparrow \uparrow})}  + \ve )^{2} } \right\}
\prod_{e \in \tilde{E}_{\star}}
( |  x_{e_{+}} - x_{e_{-}} | + \ve)^{-2}
\,,
\end{aligned}
\end{equation*}
in analogy to~\eqref{eq_supp111_unsafe}.
Let $\hat{\eta}'(u) := \hat{\eta} (u) + \mathds{1}_{
\gamma_{\ast}^{\uparrow}}(u) - \mathds{1}_{
\gamma_{\ast}^{\uparrow\uparrow}}(u) $, then evaluating the sum as in
\eqref{eq_supp100_unsafe}, we have
\begin{equation*}
\begin{aligned}
\sum_{ {\bf n} \in \mA^{\#, N_{\ve}} ( T)}
\sup_{x \in D_{(T, \n)}^{\ast}}
\big| \overline{\mW}_{\ve}^{(\mL)}\tilde\Gamma (x_{V_{\star}})\big|
\prod_{u \in \overline{T}} 2^{-d \mathbf{n} (u)}
&\lesssim
\sum_{ {\bf n} \in \mA^{\#, N_{\ve}}  ( T)}
\prod_{u \in \overline{T}}
(2^{-\mathbf{n} ( u)}
+ \varepsilon)^{ \hat{\eta}'  (u) -d}
 2^{-d \mathbf{n}(u)}\\
&\lesssim
\big( \log{ \tfrac{1} { \ve}})^{ \max\{\widehat{\mathrm{Null}}  ( T, \tilde\Gamma)
-1,0 \}}\,,
\end{aligned}
\end{equation*}
where we used in the last step that
\begin{equation*}
\begin{aligned}
\Big| \Big\{ u \, :\,  \sum_{v
\succ u} \hat{\eta}' (v) =0 \Big\}\Big|
=\max\Big\{ \Big| \Big\{ u \, :\,  \sum_{v
\succ u} \hat{\eta} (v) =0 \Big\}\Big|
-1, 0 \Big\}\,.
\end{aligned}
\end{equation*}
This concludes the proof.
\end{proof}

\begin{lemma}[Control of the diagonal second--order term]\label{lem_diagonal}
Let $ \Gamma$ be a negative 4-Anderson--Feynman diagram, $ T \in
\mT_{V_{\star}}$ a Hepp tree, and $ \mathfrak{F}_{s}$ a safe forest with associated
unsafe forest $\mf{F}_{u} \neq \emptyset$.
Furthermore, let $ \tilde\Gamma \in \mD_{\Gamma}^{\rm{lab}}$ be a non-trivial term in the decomposition $\mf{K}_{\mf{F}_s}^{(1)} \Gamma = \sum \tilde{\Gamma}$,
such that $\widehat{\rm{Null}}(T, \tilde\Gamma) = |\trim|$.
Then, with $\mM$ as in \eqref{eq_def_M}, we have
\begin{equation}\label{eq_statement_tildeK}
\begin{aligned}
\sum_{{\bf n} \in \mA^{\#, N_{\ve}} ( T)}
\bigg|
\int_{ D_{( T, \mathbf{n})}^{\ast}}
 \overline{\mW}_{\varepsilon}^{( \mM)}
 \tilde\Gamma
(x_{V_{\star}}) \ud x_{V_{\star}}
\bigg|
\lesssim
( \log{ \tfrac{1}{\ve}} )^{| \overline{T}| -1 }\,.
\end{aligned}
\end{equation}
\end{lemma}

\begin{proof}
First, we note that the integrated valuation of the kernel assignment $\mM$
with respect to $ \Gamma$ reads, with $ e_{\star}= e_{\star}(
\gamma_{\ast}, \tilde\Gamma)$,
\begin{equation}\label{eq_diagpert_supp1}
\begin{aligned}
\int_{ D_{( T, \mathbf{n})}^{\ast}}
 \overline{\mW}_{\varepsilon}^{( \mM)}
  \tilde\Gamma
(x_{V_{\star}}) \ud x_{V_{\star}}
&=
\frac{1}{2}
\sum_{i =1}^{d}
\int_{ D_{( T, \mathbf{n})}^{\ast}}
\int_{ ( \TT^{d})^{L}}
\varphi ( x_{L})
\big( x_{ e_{\star,+}}^{i} - x_{ v_{\star} ( \gamma_{\ast}
)}^{i}\big)^{2}
\partial_{i}^{2} G_{\ve} ( x_{ v_{\star} ( \gamma_{\ast}
)}- x_{ e_{\star,-}})\\
& \qquad \qquad \times
\prod_{e \in \tilde{E} \setminus e_{\star} }
\mK_{ \tau_{\tilde\Gamma, \Gamma} (e) , \varepsilon} ( x_{V_{\star}} )
\ud x_{L} \ud x_{V_{\star}} \;,
\end{aligned}
\end{equation}
where
we used that $ \mf{b}
(e_{\star}) = v_{\star}( \gamma_{\ast}) $ since $ \gamma_{\star}$ is a leaf in
$ \mf{F}_{u}$.

Recalling the decomposition of the modified Hepp sector $ D_{(T ,
\mathbf{n})}^{\ast}$ in \eqref{eq_def_modHeppSec}, we see that
each summand on the right contains the subintegral
\begin{equation}\label{eq_gamma1_int}
\begin{aligned}
\int_{D_{(T_{\ast}, \mathbf{n}_{\ast})}( x_{v_{\star}( \gamma_{\ast}) })}
\big( x_{ e_{\star, +}}^{i} - x_{ v_{\star} ( \gamma_{\ast}
)}^{i}\big)^{2}
\prod_{e \in \tilde{E} ( \gamma_{\ast}) }
G_{\ve} ( x_{ e_{+}}- x_{e_{-}})
 \ud x_{V( \gamma_{\ast} ) \setminus v_{\star}(
\gamma_{\ast} )}\,,
\end{aligned}
\end{equation}
since there is no contribution coming from the labels $\mf{h}$ and $\mf{n}$ (of
$\tilde\Gamma$) to this subintegral, by Lemma~\ref{lem_onlyNeed_emptysafe}.
Note that the integral \eqref{eq_gamma1_int} is also independent of  $ x_{v_{\star} (
\gamma_{\ast} )} $, as can be seen by performing a shift of all integration
variables.
The integral is also independent of $i=1, \ldots, d$, because $ G_{ \ve}$ is radially
symmetric, and so is the domain $ D_{(T_{\ast},
\mathbf{n}_{\ast})}(0)$.

Hence, we may rewrite~\eqref{eq_diagpert_supp1} as
\begin{equation*}
\begin{aligned}
\int_{ D_{( T, \mathbf{n})}^{\ast}}
 \overline{\mW}_{\varepsilon}^{( \mM)}
 \tilde\Gamma
(x_{V_{\star}}) \ud x_{V_{\star}}
&=
\frac{1}{2}
\int_{ D_{( T, \mathbf{n})}^{\ast}}
\int_{ ( \TT^{d})^{L}}
\varphi ( x_{L})
(\Delta G_{\ve}) ( x_{ v_{\star} ( \gamma_{\ast}
)}- x_{ e_{\star,-}})
\big( x_{ e_{\star,+}}^{1} - x_{ v_{\star} ( \gamma_{\ast}
)}^{1}\big)^{2}\\
&\qquad \qquad \times
\prod_{e \in \tilde{E} \setminus e_{\star} }
\mK_{ \tau_{\tilde\Gamma, \Gamma} (e) , \varepsilon} ( x_{V_{\star}} )
\ud x_{L} \ud x_{V_{\star}}\,.
\end{aligned}
\end{equation*}
Next, we use that $ (1-\Delta) G_{\ve} =  \delta_{\ve} $, with $\delta_\ve =
\rho_\ve \star \delta$ which allows us to
bound the integrand uniformly over $ D_{( T, \mathbf{n})}^{\ast}$ as follows:
\begin{equation}\label{eq_supp1_vanishLapalce}
\begin{aligned}
\Big|
 \overline{\mW}_{\varepsilon}^{( \mM)}
 \tilde\Gamma
(x_{V_{\star}}) \Big|
&\leqslant
\frac{1}{2}
\bigg|
\int_{ ( \TT^{d})^{L}}
\varphi ( x_{L})
\prod_{e \in E_{L} }
\mH^{ \ell ( e)} G_{\ve}( x_{e_{+}}, x_{e_{-}} |  x_{ \mf{b} ( e)} )
\ud x_{L}
\bigg|\\
&\qquad \times
( \delta_{\ve} +  G_{\ve}) ( x_{ v_{\star} ( \gamma_{\ast}
)}- x_{ e_{\star,-}})
\big| x_{ e_{\star, +}} - x_{ v_{\star} ( \gamma_{\ast}
)}\big|^{2}
\prod_{e \in \tilde{E}_{\star} \setminus e_{\star} }
\mK_{ \tau_{\tilde\Gamma, \Gamma} (e) , \varepsilon} ( x_{V_{\star}} )
\,.
\end{aligned}
\end{equation}
Here we used that $ e_{\star} \in \tilde{E}_{\star}$, because if $ e_{\star}$ were a leg
then then $ \gamma_{\ast}^{\uparrow \uparrow} = \dagger$ which is excluded by
Lemma~\ref{lem:gammastar}.
Note also that we can estimate the integral over legs in~\eqref{eq_supp1_vanishLapalce} as in~\eqref{e_leg_contr}. Hence it
suffices to find an upper bound on the second line
of~\eqref{eq_supp1_vanishLapalce}.

\begin{enumerate}
\item First, we treat the contribution
of the integral containing $
\delta_{\ve}$, given by the kernel
\begin{equation*}
\begin{aligned}
\delta_{\ve}( x_{ v_{\star} ( \gamma_{\ast}
)}- x_{ e_{\star,-}})
\big( x_{ e_{\star,+}}^{1} - x_{ v_{\star} ( \gamma_{\ast}
)}^{1}\big)^{2}
 \prod_{e \in \tilde{E}_{\star} \setminus e_{\star} }
\mK_{ \tau_{\tilde\Gamma, \Gamma} (e) , \varepsilon} ( x_{V_{\star}} )
\,,
\end{aligned}
\end{equation*}
which vanishes, except when $ | x_{ v_{\star} ( \gamma_{\ast}
)}- x_{ e_{\star,-}}| \leqslant \ve$, since $\delta_{\ve} ( x)
=
\rho_{\ve} (x)
\lesssim  \ve^{- d} \mathds{1}_{|x| \leqslant \ve}$, in particular, when
$\mathbf{n} ( v_{\star} ( \gamma_{\ast}) \wedge  e_{\star,-}) \geqslant
N_{\ve} $.
However, such small scales are excluded by definition of $\mA^{\#, N_{\ve}}( T
) $, see \eqref{e:ANE}.
Therefore
\begin{equation}\label{e_supp1_deltacontr_new}
\begin{aligned}
\delta_{\ve}( x_{ v_{\star} ( \gamma_{\ast}
)}- x_{ e_{\star,-}})
\big( x_{ e_{\star,+}}^{1} - x_{ v_{\star} ( \gamma_{\ast}
)}^{1}\big)^{2}
\prod_{e \in \tilde{E}_{\star} \setminus e_{\star} }
\mK_{ \tau_{\tilde\Gamma, \Gamma} (e) , \varepsilon} ( x_{V_{\star}})=0 \,.
\end{aligned}
\end{equation}

\item  On the other hand, the term with $ G_{\ve}$ is better behaved
since it only blows up to order $ -2 $ (instead of $-4$ for $ \Delta
G_{\ve}$ or $ \delta_{\ve}$). We can therefore follow the same steps as in the
proof of Lemma~\ref{lem_hepp_unsafe}, to obtain
\begin{equation}\label{eq_supp1_Gcont}
\begin{aligned}
\sum_{ {\bf n} \in \mA^{\#, N_{\ve}} ( T)}
& \sup_{x \in D_{(T , \mathbf{n})}} \bigg\{
 G_{\ve} ( x_{ v_{\star} ( \gamma_{\ast}
)}- x_{ e_{\star,-}})
\big| x_{ e_{\star,+}} - x_{ v_{\star} ( \gamma_{\ast}
)}\big|^{2}
\prod_{e \in \tilde{E}_{\star} \setminus e_{\star} }
\big|\mK_{ \tau_{\tilde\Gamma, \Gamma} (e) , \varepsilon} ( x_{V_{\star}} )\big|
\bigg\}
\prod_{u \in \overline{T}} 2^{-d \mathbf{n} (u)}
\\
& \lesssim
\sum_{ {\bf n} \in \mA^{\#, N_{\ve}}  ( T)}
( 2^{- \mathbf{n} (u_{\ast})} + \ve )^{2}
\prod_{u \in \overline{T}}
(2^{-\mathbf{n} ( u)}
+ \varepsilon)^{ \hat{\eta}  (u)  -d}
 2^{-d \mathbf{n}(u)}
\lesssim \big( \log{ \tfrac{1}{ \ve}} \big)^{| \trim| -1}\,.
\end{aligned}
\end{equation}
The last inequality holds, because
\begin{itemize}
\item Either $ \sum_{ u_{\ast} \prec u } \hat{\eta} (u) =0
$, in which case we count one zero less, since with the above calculation we increased $ \hat{\eta} (
u_{\ast})$ by $2$.

\item Or  $ \sum_{ u_{\ast} \prec u } \hat{\eta} (u) > 0$, in which
case the number of logarithms is unchanged, but we must have started with $
\widehat{\mathrm{Null}} ( T, \tilde\Gamma) < | \overline{T}| $ to begin with.
\end{itemize}

\end{enumerate}
Overall, combining \eqref{e_supp1_deltacontr_new} and
\eqref{eq_supp1_Gcont} concludes the proof.
\end{proof}

\begin{proof}[Proof of Lemma~\ref{lem_emptysafe}]
First, by Lemma~\ref{lem_pushonKernel}, we have that
\begin{equation*}
	\begin{aligned}
		\sum_{\n \in \mA^{\#, N_\ve} (T)}
\bigg\vert \int_{D_{(T, \n)}}  \overline{\mW}_{\ve}
\hat{\mR}_{[\mf{F}_s, \mf{F}_{s} \cup  \mf{F}_u ]}^{(1)} \Gamma (x_{V_\star}) \ud x_{V_\star}
\bigg\vert
\leqslant
\sum_{\tilde{\Gamma}}
\sum_{\n \in \mA^{\#, N_\ve} (T)}
\bigg\vert \int_{D_{(T, \n)}}
\overline{\mW}_{\ve}^{(\mK)} \tilde{\Gamma}
(x_{V_\star}) \ud x_{V_\star}
\bigg\vert
	\end{aligned}
\end{equation*}
where we wrote $\mf{K}_{\mf{F}_s}^{(1)} \Gamma = \sum \tilde{\Gamma}$, with each
$\tilde{\Gamma} \in \mD_{\Gamma}^{\rm{lab}}$ being non-trivial.
Hence, the desired bound is true for all $ \tilde\Gamma $ satisfying
$\widehat{\rm{Null}}(T, \tilde\Gamma) < |\trim|$, by Lemma~\ref{lem_hepp_unsafe}.

Therefore, it is only left to consider those $ \tilde\Gamma $ with
$\widehat{\rm{Null}}(T, \tilde\Gamma) = |\trim|$.
In this case, let $ \gamma_{*}$ be such as in
Lemma~\ref{lem:gammastar}, and define $ \mM$ and $\mL$ as in
\eqref{eq_def_L} and \eqref{eq_def_M}.
Then applying Lemma~\ref{lem_Dstarenough} yields
\begin{equation*}
\begin{aligned}
\sum_{\n \in \mA^{\#, N_\ve} (T)}
\bigg\vert \int_{D_{(T, \n)}} &
\overline{\mW}_{\ve}^{(\mK)} \tilde{\Gamma}
(x_{V_\star}) \ud x_{V_\star}
\bigg\vert \\
& \lesssim
\sum_{\n \in \mA^{\#, N_\ve} (T)}
\bigg\vert \int_{D_{(T, \n)}^{*}}
\overline{\mW}_{\ve}^{(\mK)} \tilde{\Gamma}
 (x_{V_\star}) \ud x_{V_\star}
\bigg\vert
+
\big( \log{ \tfrac{1}{ \ve}} \big)^{| \trim| -1}\\
& \leqslant
\sum_{{\bf n} \in \mA^{\#, N_{\ve}} ( T)}
\sup_{x \in D_{( T, \mathbf{n})}^{\ast}}
\big| \overline{\mW}_{\varepsilon}^{(\mL)}
\tilde\Gamma (x_{V_{\star}})\big|
\prod_{u \in \overline{T}} 2^{-d \mathbf{n} (u)}\\
&\quad +
\sum_{{\bf n} \in \mA^{\#, N_{\ve}} ( T)}
\bigg|
\int_{ D_{( T, \mathbf{n})}^{\ast}}
 \overline{\mW}_{\varepsilon}^{( \mM)}
 \tilde\Gamma
(x_{V_{\star}}) \ud x_{V_{\star}}
\bigg|
+
\big( \log{ \tfrac{1}{ \ve}} \big)^{| \trim| -1}\,,
\end{aligned}
\end{equation*}
where we used Lemma~\ref{lem_offdiag_nochange} in the second inequality.
The desired statement is now a consequence of Lemmas~\ref{lem_hepp_unsafe_improve} and~\ref{lem_diagonal}.
\end{proof}

\subsubsection{Some graph properties}
In the following, we use the natural notion of degree for labelled diagrams. For a
connected $\gamma \subseteq \tilde\Gamma = (\tilde{\Gamma}, \mf{h}, \mf{n},
\mf{a}) \in \mD_{\Gamma}^{\rm{lab}}$, we set
\begin{equation*}
 \deg (\gamma, \tilde\Gamma) := d(|V(\gamma)| -1) -2 |E(\gamma)| -\sum_{e \in E(\gamma)} \sum_{i=1}^d\mf{h}(e)_i + \sum_{v \in V(\gamma) \colon\mf{a}(v) \in V(\gamma)} \sum_{i=1}^d \mf{n}(v)_i \;.
\end{equation*}
Notice that $ \deg ( \gamma, \Gamma ) $ agrees with $ \deg \gamma$, defined in
	\eqref{e:deg-extended}. Moreover, we stress that the above definition of $ \deg (
	\cdot, \tilde\Gamma)$ does not take into account the
kernel assignment induced by the unsafe divergences from Lemma~\ref{lem_pushonKernel},
i.e. the gain from Taylor expansions.

We then find the following characterisation of degree zero labelled diagrams over two vertices.

\begin{lemma}\label{lem:no-weird-bubble}
  Let $\mf{F} \in \mF^{-}_{\Gamma}$ be a forest of
   divergences with $|\mf{F}|=n$. Decompose $\mf{K}_{\mf{F}}^{(1)} \Gamma = \sum
   \tilde{\Gamma}$, for a finite number of nontrivial terms $(\tilde{\Gamma}, \mf{h},
   \mf{n}, \mf{a}) \in
   \mD^{\rm{lab}}_\Gamma$. Let $\gamma {\subseteq \tilde\Gamma}$
   be a connected subdiagram spanned by two vertices $v_1, v_2 \in V_\star$ such
  that $\deg(\gamma, {\tilde\Gamma})=0$. Then it is isomorphic to a bubble, namely to
  $\<bubbleA>$ or $\<bubbleB>$, in the sense that:
\begin{itemize}
	\item For all $e \in E(\gamma)$ it holds that $\mf{h}(e) = 0$.
  \item If $\mf{n}(v) \neq 0$ for one $v \in V(\gamma)$, then $\mf{a}(v)
  \not\in V(\gamma)$.
\end{itemize}
\end{lemma}

\begin{proof}
  First, let us show that it can not be that both $\mf{n}(v_1), \mf{n}(v_2)
  \neq 0$. Assume the converse is true. Then following Lemma~\ref{lem:n-prop}
  and the preceeding discussion, there exist two divergences $\gamma_{l_1},
  \gamma_{l_2} \in \mf{F}$ such that $\mf{n}(v_i)$ is associated to the extraction
  $\mC^{(1)}_{\gamma_i}$. Now, from the forest property of $\mf{F}$ it must be
  that either $\gamma_{l_1} \subseteq \gamma_{l_2}$ or $\gamma_{l_2} \subseteq
  \gamma_{l_1}$. Indeed for $\deg(\gamma)=0$ to hold, $\gamma$ must consist of
  at least two edges. If $\gamma_{l_1}$ were disjoint of $\gamma_{l_2}$, and
  since every divergence has at most one incoming edge (which in the case of
  $\gamma_{l_i}$ must satisfy $\mf{h}(e)\neq 0$) and exactly one outgoing
  one, the only possible configuration is that $\gamma$ consists exactly of two
  edges, with the following disposition:
  
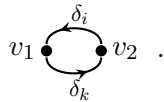
\begin{equation*}
    \begin{tikzpicture}[thick,>=stealth,scale=0.2,baseline=-0.2em]
      \tikzset{
        dot/.style={circle,fill=black,inner sep=1.5pt}
      }
      \node[dot] (bL) at (-2.2,0)   {};
      \node[dot] (vL) at (1.4,0) {};   
      \node at (0,-2.6) {\scalebox{0.8}{$\delta_k$}};
      \node at (0,2.4) {\scalebox{0.8}{$\delta_i$}};
      \node[right] at (vL) {$v_2$};
      \node[left] at (bL) {$v_1$};
      \draw[midarrow] (vL) to[out=90,in=90] (bL);
      \draw[midarrow] (bL) to[out=-90,in=-90] (vL);
    \end{tikzpicture} \;.
  \end{equation*}
  However, this disposition is impossible, for example because the resulting
  diagram has no legs as the original diagram would consist only of $\gamma_{l_1}$
  and $\gamma_{l_2}$ connected by two edges (see the same argument in the
  proof of Lemma~\ref{lem_safe_vanish_2k}).

  We have therefore concluded (up to relabelling the vertices) that $\gamma_{l_1}
  \subseteq \gamma_{l_2}$. Now for $\tilde{\Gamma} \neq 0$ to hold we must have
  $e_\star(\gamma_{l_2}, \Gamma) \neq e_{\star}(\gamma_{l_1}, \Gamma) $, by the
  first identity of \eqref{eq_def_Cbar1}. We will now prove that this too can
  not be the case. We start by observing by Lemma~\ref{lem:n-prop} 2.(a) or
  2.(b) that $v_2 \in V(\gamma_{l_1})$, since otherwise there would be at least one
  external edge (with respect to $\gamma_{l_1}$) incoming into $v_1$, or at
  least two external outgoing edges from $v_1$ (which would contradict
  Lemma~\ref{lem:neg}). Furthermore, if Lemma~\ref{lem:n-prop} 2.(a) holds at
  $v_2$, then we have that $e_{\star, +}(\gamma_{l_2}, \Gamma_{l_2-1}) \in
  V(\gamma_{l_1})$ which implies $e_\star(\gamma_{l_2}, \Gamma) =
  e_{\star}(\gamma_{l_1}, \Gamma)$. If instead Lemma~\ref{lem:n-prop} 2.(b)
  holds, then $v_2=v_\star(\gamma_2)$ for some $\gamma_2 \subseteq \gamma_{l_2}$
  and $v_2 \in V(\gamma_{l_1})$ implies that $\gamma_2 \subseteq \gamma_{l_1}$
  as well. Then Lemma~\ref{lem:n-prop} 2.(b) implies that $e_{\star, +}(\gamma_{l_2}, \Gamma_{l_2-1}) \in
  V(\gamma_{2}) \subseteq V(\gamma_{l_1})$ leading again to the conclusion $e_\star(\gamma_{l_2}, \Gamma) =
  e_{\star}(\gamma_{l_1}, \Gamma)$. We have deduced that $\mf{n}(v_1),
  \mf{n}(v_2) \neq 0$ is not possible.

  Next we prove that $\mf{h}(e)=0$ for all $e \in E(\gamma)$. Given our
  discussion so far, the only other possibility is that the diagram has the following
  structure (up to relabelling vertices):
  
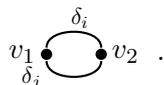
\begin{equation*}
    \begin{tikzpicture}[thick,>=stealth,scale=0.2,baseline=-0.2em]
      \tikzset{
        dot/.style={circle,fill=black,inner sep=1.5pt}
      }
      \node[dot] (bL) at (-2.2,0)   {};
      \node[dot] (vL) at (1.4,0) {};   
      \node at (0,2.4) {\scalebox{0.8}{$\delta_i$}};
      \node[right] at (vL) {$v_2$};
      \node[left] at (bL) {$v_1$};
      \draw (vL) to[out=90,in=90] (bL);
      \draw (bL) to[out=-90,in=-90] (vL);
      \node at (-3.2,-1.6) {\scalebox{0.8}{$\delta_j$}};
    \end{tikzpicture} \;.
  \end{equation*}
  namely that $\mf{n}(v_1) \neq 0$ and $\mf{a}(v_1) = v_2$ and $\gamma$ consists
  of two edges $E(\gamma) = \{e_1, e_2\}$ with $\mf{h}(e_1) \neq 0,
  \mf{h}(e_2)=0$ (note that we have not prescribed a direction on the edges). In this case $v_1$ must satisfy either 2.(a) or 2.(b) of
  Lemma~\ref{lem:n-prop}. Let us start by assuming the case 2.(a). In this case,
  by \eqref{e:a-representation} we have that $v_2 = v_\star(\gamma_2)$ where
  $\gamma_2$ is the smallest divergence in $\mf{F}$ such that $v_1 \in
  V(\gamma_2) \setminus \{v_\star(\gamma_2)\}$.

  If $e_1=(v_1, v_2)$, that is if $e_1$ is incoming into $v_2$, then it must be
  that $v_2=v_\star(\gamma')$ for some $\gamma' \in\mf{F}$ and that $e_1 =
  e_\star (\gamma', \tilde\Gamma)$. We deduce that $\gamma'
  \subseteq \gamma_2$ and that $e_1, e_2 \not\in E(\gamma')$ (otherwise we could
  not have $e_1 = e_\star (\gamma', \tilde\Gamma)$). However this implies that $e_1,
  e_2$ must be the only incoming and outgoing edges of $\gamma'$, but in order
  for $v_2 =v_\star(\gamma_2)$ there must be at least another outgoing edge,
  leading to a contradiction.

  If $e_1=(v_2, v_1)$, that is if $e_1$ is incoming into $v_1$, then it
  must be that $v_1=v_\star(\gamma')$ for some $\gamma' \in \mf{F}$ such that
  $v_2 \not\in \gamma'$. In
  addition, $v_1$ must satisfy 2.(a) from Lemma~\ref{lem:n-prop} and is linked
  to the extraction of some diagram $\gamma_{l_1} \in \mf{F}$, since 2.(b) does
  not allow for an external incoming edge. Here we use the same notation as in
  Lemma~\ref{lem:n-prop} and the preceeding discussion. By the forest property either
  $\gamma' \subseteq \gamma_{l_1}$ or $\gamma_{l_1} \subseteq \gamma'$ as they
  share at least a vertex. Now in the diagram $\Gamma_{l_1-1}$, if $\gamma_{l_1} \subseteq \gamma'$ we would have $\hC^{(1)}_{\gamma_{l_1}} \Gamma_{l_1 -1}
  =0$ by the second identity in \eqref{eq_def_Cbar1}. Therefore, we must have
  $\gamma' \subsetneq \gamma_{l_1}$. However, the number of incoming vertices
  into a divergence is decreasing (see the proof of Lemma~\ref{lem:n-prop} as we
  extract divergences), and therefore it is impossible that in $\tilde{\Gamma}$
  there is an incoming edge into $v_1$ external to $\gamma'$.

  If instead $v_1$ satisfies 2.(b), then it must be that $v_i =
  v_\star(\gamma_i)$ for some $\gamma_i \in \mf{F}$ and that $\gamma_2$ is the
  parent of $\gamma_1$ in $\mf{F}$ by \eqref{e:a-representation} since
  $\mf{a}(v_1) =v_2$. Moreover, by
  2.(b) the vertex $v_1$ can not have any incoming edges that are external to
  $\gamma_1$. However, since there are at least two edges in $\gamma$, this
  implies that there are at least two edges outgoing of $\gamma_1$, which
  contradicts Lemma~\ref{lem:neg} (note that the total number of outgoing edges
  out of divergence is preserved by extractions and equal to one). Hence such a
  diagram can not exist. This concludes the proof.
\end{proof}

Next, consider $\widehat{\rm{Null}}$ as in Lemma~\ref{lem_hepp_unsafe}. We
notice that in order for $ \widehat{\mathrm{Null}}(T, \tilde{\Gamma}) = | \overline{T}| $ to
hold, we must necessarily have a nested bubble structure in the leaves of $\mf{F}_u$.
\begin{lemma}\label{lem_onlyNeed_emptysafe}
Let $ \Gamma$ be a negative Anderson--Feynman diagram, $ T
\in \mT_{V_{\star}}$ a Hepp tree, and $\mf{F}_s$ a forest of safe divergences.
Write $\mf{K}_{\mf{F}_s}^{(1)} \Gamma = \sum \tilde{\Gamma}$, where each
$\tilde{\Gamma} \in \mD_{\Gamma}^{\rm{lab}}$ is non-trivial. Assume that
$\widehat{\rm{Null}}(T, \tilde{\Gamma}) = |\trim|$ for some $\tilde{\Gamma}$.
Let $\gamma_* \in \mf{F}_u$ be as in Lemma~\ref{lem:gammastar}. Then the edge set $\tilde{E}(\gamma_*)$ of
$\tilde{\gamma}_* =
\tau_{\tilde{\Gamma}, \Gamma}^{-1} \gamma_*$ satisfies $\mf{h}(e) = 0 $ for all $e
\in \tilde{E}(\gamma_*)$.
Moreover, if $(\rm{id} - \hC_{\gamma_*} - \hC_{\gamma_*}^{(1)})
\tilde{\Gamma} \neq 0$, then there exists no $v \in V(\gamma_*) \setminus \{v_\star(\gamma_*)\}$ such that $\mf{n}(v)
\neq 0$.
\end{lemma}

\begin{proof}
Consider the vertex $\gamma_*^\uparrow \in \trim$, and note that by Lemma~\ref{lem:gammastar}
$\tilde{\gamma}_* = \tau_{\tilde{\Gamma}, \Gamma}^{-1} \gamma_*$ is the diagram spanned by all the vertices associated to
leaves $v \in T$ such that $v \succ \gamma_*^{\uparrow}$.
Note that $\deg(\tilde{\gamma}_*, \tilde\Gamma) = -2$, or else $\widehat{\rm{Null}}(T, \tilde{\Gamma})
< |\trim|$ since $ \hat{\eta} ( \gamma_{\ast}^{\uparrow}) = \eta (
\gamma_{\ast}^{\uparrow})+2$.
For any $u \in \trim$
such that $u \succ \gamma_*^\uparrow, u \neq \gamma_*^\uparrow$, we have that
the subdiagram $\tilde{\Gamma}_u \subseteq \tilde{\Gamma}$ spanned by all the
leaves $v$ of $T$ such that $v \succ u$, satisfies $\deg(\tilde{\Gamma}_u,
\tilde\Gamma) =
0$.
Therefore, by Lemma~\ref{lem:no-weird-bubble} it coincides with a nested bubble
diagram and the
associated edges satisfy $\mf{h}(e) = 0$. The result then follows by iteratively
contracting all the bubbles, as in the proof of Lemma~\ref{lem_safe_vanish_2k},
until $\tilde{\gamma}_*$ is reduced to a diagram of degree $-2$, spanned by two
vertices. Also in this case we have $\mf{h}(e)=0$ for all the edges in this last
diagram. If not, one of the two vertices must be of the form $v_\star(\gamma)$
for some $\gamma \in \mf{F}_s, \gamma \subseteq \gamma_*$. However, in this case
the entire diagram $\gamma$ would have contracted to a point, which can not be
because degrees are preserved when contracting bubbles, and the original degree
of $\gamma$ was $-2$, while the degree of a singleton is zero.

As for the last claim, suppose that $\mf{n}(v) \neq 0$ for some $v  \in V (
\gamma_{\ast}) \setminus \{ v_{\star}( \gamma_{\ast}) \}$. Then
either $\mf{a}(v) \in V(\gamma_*)$ or $\mf{a}(v) \in V_\star \setminus
V(\gamma_*)$.
In the first case, since contracting bubbles does not
change the degree (even considering the labels) by
Lemma~\ref{lem:no-weird-bubble}, we would conclude that after contracting all
the bubbles the degree of $\tilde{\gamma}_*$ is greater than $-1$ (since there
are no edges with $\mf{h} \neq 0$ but at least one monomial), which contradicts
$\deg(\tilde{\gamma}_*, \tilde\Gamma)=-2$.
On the other hand, if instead $\mf{a}(v) \in V_\star \setminus
V(\gamma_*)$, then by Lemma~\ref{lem:n-prop} there is no incoming edge into
$\gamma_*$ (the incoming edge that was removed
in order to produce the $\mf{n}(v)$ has been rewired outside of $V(\gamma_*)$
since $\mf{a}(v)$ does not belong to $V(\gamma_*)$)  and in this case $(\rm{id} - \hC_{\gamma_*} -
\hC_{\gamma_*}^{(1)})
\tilde{\Gamma} = 0$. This concludes the proof.
\end{proof}

\subsection{Proof of Lemma~\ref{lem:negvanish}}\label{sec:the-end-proof}

Finally, we collect all the results obtained so far to prove Lemma~\ref{lem:negvanish}, which states
that any negative Anderson--Feynman diagram (with two or four legs) vanishes in the weak coupling limit, after renormalisation.

\begin{proof}[Proof of Lemma~\ref{lem:negvanish}]
	For $k =1$, the statement is a consequence of Corollary~\ref{cor_unsafe}.
	On the other hand, if $k=2$ and
   $ \Gamma $ a negative $ 4$-Anderson--Feynman diagram, then
  \begin{equation*}
  \begin{aligned}
  \lim_{\ve \to 0}
  \lambda_{\ve}^{| E_{\star}|} | \hat{\Pi}_{\ve} \Gamma ( \varphi) | =0\,,
  \end{aligned}
  \end{equation*}
  by Proposition~\ref{prop_improve}. We remark the the tadpole case is treated
  separately in Lemma~\ref{lem_tadpole}.
\end{proof}
\appendix

\section{The Green's function of the Laplacian}

We consider the Green's function $G $ of the massive Laplacian, which is the
solution to
\begin{equation*}
\begin{aligned}
( 1- \Delta) G = \delta_{0}\,, \quad \text{on } \ \mathbf{T}^{4} \;.
\end{aligned}
\end{equation*}
Then $G$ satisfies the following estimates.
\begin{lemma} \label{lem:G}
  We have that $G \in C(\TT^4 \setminus \{0\} ; \RR)$ and $G(x)> 0$ for all $x
  \in \TT^4 \setminus \{0\}$. Moreover
  \begin{align*}
    \sup_{x \in \TT^4 \setminus\{0\} } \Big\vert G(x) - \frac{1}{4 \pi^2 |x|^2} - \frac{1}{8 \pi^2 } \log(|x|) \Big\vert < \infty \;.
  \end{align*}
\end{lemma}

\begin{proof}
  The continuity of $G$ outside zero is standard (in fact $G$ is smooth outside
  the origin), and similarly the strict positivity. Let us concentrate on the
  bound on the blow-up of the solution. If we consider the function $G_0(x) = 1/(4
  \pi^2 |x|^2)$ defined on $\RR^4$ (with $|\cdot|$ the Euclidean distance rather
  than the periodic one), then $G$ solves the Poisson equation $- \Delta G_0 =
  \delta_0$. In fact, $G_0$ is explicit from the Poisson formula
  \begin{equation*}
    \begin{aligned}
      G_0(x) & = \int_{0}^\infty \frac{1}{(4 \pi t)^{\frac{d}{2}}} e^{- \frac{|x|^2}{4t}} \ud t  = \frac{|x|^2}{4} \int_{0}^\infty \frac{1}{( \pi r |x|^2)^{\frac{d}{2}}} e^{- \frac{1}{r}} \ud r = \frac{1}{4 \pi^2 |x|^2 } \;.
    \end{aligned}
  \end{equation*}
  However, $G_0$ does not contain a mass term. To account for the mass term we
  will need a logarithmic correction. To make this construction clearer, let us assume that
  we want to construct the Green's function of the massive Laplacian $(\mf{m} -
  \Delta) G = \delta$, for arbitrary $\mf{m} > 0$. We will then set $\mf{m}
  =1$, but since the argument is a perturbative expansion, it is easier to
  follow in the general case.

  Let us  define the corrected Green's function
  \begin{equation*}
    G_1 (x) := \frac{1}{4 \pi^2 |x|^2} + \frac{\mf{m} }{8 \pi^2}\log( |x|) +  \frac{\mf{m}^2}{64 \pi^2}|x|^2 \log (| x|)\;.
  \end{equation*}
  For this corrected function, we find that
  \begin{equation*}
    (\mf{m} - \Delta) G_1 = \delta + \frac{2 \mf{m}}{8 \pi^2 |x|^2}  -  \frac{2
\mf{m} }{8 \pi^2} \frac{1}{|x|^{2}} + \frac{\mf{m}^2}{8 \pi^2} \log (|x|) -
\frac{ 8 \mf{m}^2}{64 \pi^2}\log(|x|) + R_{ \mf{m}} (x) \;,
  \end{equation*}
  where
  \begin{equation*}
    R_{ \mf{m}} (x) = - \frac{6 \mf{m}^2}{64 \pi^2} +  \frac{\mf{m}^3}{64 \pi^2}|x|^2 \log (| x|)
  \end{equation*}
  Here we have repeatedly written the Laplacian on $\RR^d$ in
  polar coordinates, meaning that for radial functions $\varphi$ it holds $\Delta \varphi
  = \partial_{r}^2 \varphi + (d-1) r^{-1} \partial_r \varphi$. For example, we find
  \begin{equation*}
    \Delta \big(|x|^2 \log(|x|)\big) = 2 \log(|x|) + 3  + (d-1) (2\log(|x|) + 1) = 2 d \log(|x|) + d +2 \;.
  \end{equation*}
  This leaves us with
  \begin{equation*}
    (\mf{m} - \Delta) G_1 = \delta + R_{\mf{m}} \;, \qquad \forall x \in \RR^4 \;.
  \end{equation*}
  Now we introduce a smooth, radial, function $\varphi$ with compact support
  in $(-\pi, \pi)^4$ and such that $\varphi \equiv 1$ in a neighbourhood of
  zero. We use $\varphi$ as a cutoff function, and define $G_{1, \rm{per}}$ the
  periodic extension of $x \mapsto G_1(x) \varphi(x)$. Hence, there exists an
  $\tilde{R}_{\mf{m}} \in L^{\infty}(\TT^4)$ such that
  \begin{equation*}
    (\mf{m} - \Delta) G_{1, \rm{per}} = \delta + \tilde{R}_{ \mf{m}}  \;, \qquad \forall x \in \TT^4 \;.
  \end{equation*}
  Finally, we can compare $G_{1, \rm{per}}$ to the actual Green's function of
  the massive Laplacian on the torus. If we define $\mE = G - G_{1, \rm{per}}$,
  then we obtain
  \begin{equation*}
    (\mf{m} - \Delta) \mE = - \tilde{R}_{ \mf{m}} \;,
  \end{equation*}
  and since $\tilde{R}_{\mf{m}} \in L^\infty$, we conclude that
  \begin{equation*}
    \|G - G_{1, \rm{per}}\|_{\infty} \lesssim_{ \mf{m}}  1 \;,
  \end{equation*}
  which implies the desired result.
\end{proof}
Similarly, we obtain an estimate on mollifications of the Green's function.

\begin{lemma}\label{lem_greens_function_order_bound}
Let $ \rho \in C^{\infty}_c(\RR^4)$ be a non-negative radial function with
compact support in $(-1, 1)^4$, and $
\int \rho = 1$. Then set $ \rho_{\varepsilon}
( \cdot ) := \varepsilon^{-d} \rho ( \cdot / \varepsilon) $ and
define $ G_{\varepsilon}:= \rho_{ \varepsilon} \star  G$ the periodic,
mollified Green's function of the massive Laplacian. Then for some constant $C >
0$
\begin{equation}\label{e:G-bd1}
  |G_{\ve} (x)| \leq \frac{C}{|x|^2 + \ve^2} \qquad \forall x \in \TT^4 \;.
\end{equation}
In addition uniformly over $\ve, x$ such that $|x| \geq 2 \ve$ it holds that for
some $C > 0$
\begin{equation}\label{e:G-bd2}
  \frac{1}{4 \pi^2 ( |x| + \ve)^2} +\frac{1}{8 \pi^2} \log(|x|) - C \leq G_{\ve} (x) \leq \frac{1}{4 \pi^2 ( |x| - \ve)^2} +\frac{1}{8 \pi^2} \log(|x|) +C \;.
\end{equation}
\end{lemma}

\begin{proof}
We rely on the estimate in Lemma~\ref{lem:G} without further reference to it. To
obtain the first estimate we control
\begin{equation*}
  |G_{\ve}(x)| \lesssim \int_{\TT^4} \frac{1}{|x-z|^2} \vr_{\ve}(z) \ud z \lesssim \begin{cases}
    \frac{1}{|x|^2} & \quad \text{ if } |x| > 2 \ve \;, \\
    \int_{\TT^{4}} |z|^{-2} \vr_{\ve} (z) \ud z \lesssim \ve^{-2} \;, & \quad \text{ if } |x| \leq 2 \ve \;.
  \end{cases}
\end{equation*}
This is an estimate of the required order for~\eqref{e:G-bd1}.

As for~\eqref{e:G-bd2}, we start singling out the logarithmic term. Here we use that
\begin{equation}\label{eq_supp1_Green}
  \int \log(|x - z|) \varrho_\ve (z) \ud z = \log(|x| + \ve) + \mO(1) = \log(|x| ) + \mO(1)  \;,
\end{equation}
for $|x| \geq 2 \ve$, see for example~\cite[Lemma 3.5]{MR3652040}.

Instead, for the polynomial term we estimate from below
\begin{equation*}
  \int \frac{1}{|x - z|^2} \vr_{\ve}(z) \ud z \geq \frac{1}{(|x|+ \ve)^2} \;.
\end{equation*}
And similarly from above, using that $|x| \geq 2 \ve$:
\begin{equation*}
  \int \frac{1}{|x - z|^2} \vr_{\ve}(z) \ud z \leq \frac{1}{(|x|- \ve)^2} \;.
\end{equation*}
Combining these estimates leads to~\eqref{e:G-bd2}, and thus the proof of the
result is complete.
\end{proof}

\begin{lemma}\label{lem_Gmol2ratiomol}
In the same setup as in the previous lemma, there exists a $C>0$ such that
\begin{equation*}
\begin{aligned}
\bigg| G_{\ve}(x) - \frac{1}{4 \pi^{2}} \frac{1}{|x|^{2}} \bigg|
\leqslant C\bigg(
\frac{ \ve}{ (|x| + \ve )^{3}} + \log \frac{1}{|x|} + 1 \bigg)
\end{aligned}
\end{equation*}
for all $\ve \in (0, 1)$ and $ |x| \geq 2 \ve$.
\end{lemma}
\begin{proof}	
From~\eqref{e:G-bd2} in Lemma~\ref{lem_greens_function_order_bound}, we find
that
\begin{equation*}
  G_\ve(x) -  \frac{1}{4 \pi^2 |x|^{2}}  \lesssim \frac{\ve}{( |x| - \ve)^3} + \log\frac{1}{|x|} +1\;.
\end{equation*}
Similarly from below:
\begin{equation*}
  G_\ve(x) -  \frac{1}{4 \pi^2 |x|^{2}}  \gtrsim - \frac{\ve}{( |x| + \ve)^3} -  \log\frac{1}{|x|}- 1\;.
\end{equation*}
This concludes the proof, by using $|x| \geq 2 \ve$.
\end{proof}

\section{Structure of negative Anderson--Feynman diagrams}\label{app:bphz}

We start by proving the following result, which guarantees that any divergent
diagram with strictly negative degree must be obtained from an internal pairing.
Given a diagram $\Gamma$ and subdiagram $\overline{\Gamma}$ we say that $e =
(e_-, e_+) \in
E$ is incoming into $\overline{\Gamma}$ (resp. outgoing) if $e_+ \in
\overline{V}, e_- \in V \setminus\overline{V} $ (resp. $e_- \in
\overline{V}, e_+ \in V \setminus\overline{V} $). Furthermore, recall that for a
subdiagram $\overline{\Gamma} \subseteq \Gamma$ of a Feynman diagram $\Gamma$ we
write
\begin{equation*}
  \overline{\deg} (\overline{\Gamma})  = 4 ( | \overline{V} | -1) - 2 | \overline{E} | \;.
\end{equation*}
The only difference to the degree defined in~\eqref{e_def_deg} is that we do not
require $\overline{\Gamma}$ to be connected, so that
$\overline{\deg}(\overline{\Gamma}) = \deg(\overline{\Gamma}) + {4}( c -1)$, where $c$ is
the number of connected components of $\overline{\Gamma}$.

\begin{lemma}\label{lem:neg}
Let $ \Gamma $ be a connected Anderson--Feynman diagram, and $ \overline{\Gamma}
= (\overline{V}, \overline{E}) \subseteq \Gamma
$ a subdiagram. Then $ \overline{\deg} ( \overline{\Gamma}) \in -2 +\NN$, and $ \overline{\deg} ( \overline{\Gamma}) = -2$
if and only if $\overline{\Gamma}$ is connected and there exist exactly one outgoing and one incoming edge.

Moreover, fix $n_1, n_2 \in \NN_*$ and a paring $\kappa$ such that $\Gamma =
(\mI_{n_1}, \mI_{n_2})_{\kappa}$. If $\deg (\overline{\Gamma}) = -2$, then the edges of $\overline{\Gamma}$ belong either
entirely to $\mI_{n_1}$ or entirely $\mI_{n_2}$.
\end{lemma}

\begin{proof}
The edge set $ \overline{E} $ of $ \overline{\Gamma} $
satisfies $ | \overline{E} | = 2 | \overline{V} | - L$, where
$ L \geqslant 0 $ is the number of outgoing (or equivalently,
incoming)
edges that connect $ \overline{\Gamma} $ to
the rest of the diagram $ \Gamma $.
In particular, it follows that
\begin{equation*}
\begin{aligned}
\overline{\deg} (\overline{\Gamma}) = 4  (| \overline{V} | -1) -2  |
\overline{E} | = - 4 + 2 L \;.
\end{aligned}
\end{equation*}
Now the first result is proven if we show that $ L \geqslant 1 $. But note that by
construction, in the original directed diagram $ \Gamma $ it is always possible
to reach a leaf (or leg), from any internal vertex, which implies $ L \geqslant
1 $. Reversing the argument, we see that in order for $ L =1 $ to hold, $
\overline{\Gamma} $ must have exactly an incoming and an outgoing edge that
connect it to $ \Gamma $. For the same reason $\overline{\Gamma}$ must be connected.

Finally, a similar argument shows that if $ \overline{\Gamma} $ contains edges
from $ k $ distinct trees, then $ L \geqslant k$ (for each
tree present there must be a path out of $ \overline{\Gamma} $). The result is
proven.
\end{proof}

Now we are ready to conclude.

\begin{proof}[Proof of Proposition~\ref{prop:bphz}]
  We restrict to proving the result in the case $k=2$, since the general case
  follows along the same arguments, with only an additional burden of notation.
  Note that for every pairing $\kappa$ between $\tau_1$ and $\tau_2$ there are
  two unique internal pairings $\kappa_i \in \mK^{\rm{int}} (\tau_i)$ and a
  unique external pairings $\kappa_{\rm{ext}}$ such that $(\tau_1,
  \tau_2)_\kappa = ( (\tau_1)_{\kappa_1},
  (\tau_2)_{\kappa_2})_{\kappa_{\rm{ext}}}$. Here by external pairing we
  mean that it links to vertices that do not belong to the same tree: if
  $\kappa_{\rm{ext}}$ pairs a vertex $v \in \tau_1$, then the vertex
  $\kappa_{\rm{ext}}(v)$ with which it is paired belongs to $\tau_2$ (and
  viceversa). Therefore, it suffices to prove that for any fixed pairing
  $\kappa \in \mK(\tau_1, \tau_2)$ we have the identity
  \begin{equation*}
  \Pi_{\varepsilon}  \left( \mf{R}_{\ve} (\tau_1)_{\kappa_1} ,  \mf{R}_{\ve} (\tau_2)_{\kappa_2} \right)_{\kappa_{\rm{ext}}}= \hp (\tau_1, \tau_2)_\kappa = \hp ( (\tau_1)_{\kappa_1}, (\tau_2)_{\kappa_2})_{\kappa_{\rm{ext}}} \;.
  \end{equation*}
  Here note that the last identity follows by definition, while the first one needs to be
  checked. To see that the term on the left is even meaningful, and corresponds to a
  slice of the expected value that we are trying to estimate, we note that $\mf{R}_\ve$
  as defined in~\eqref{e:ren-1}, that~\eqref{e:ren-2} leaves single homogeneous chaoses invariant, and that
  \begin{equation}\label{eq_supp_propbphz1}
    \begin{aligned}
      \EE [ \mf{R}_{\ve} (\tau_1)_{\kappa_1}  \mf{R}_{\ve} (\tau_2)_{\kappa_2}] &= \sum_{\kappa_{\rm{ext}}} \left( \mf{R}_{\ve} (\tau_1)_{\kappa_1} ,  \mf{R}_{\ve} (\tau_2)_{\kappa_2} \right)_{\kappa_{\rm{ext}}} \\
      & = \sum_{\kappa_{\mathrm{ext}}}
      \int_{( \mathbf{T}^{4})^{ V_{o} }} \hat{\Pi}_\ve \tau_{1, \kappa_1}
    (x_L, x_{V_{o}}) \hat{\Pi}_\ve \tau_{2, \kappa_2}
    (x_L, x_{\kappa_{\rm{ext}}(V_{o})})  \ud x_{V_{o}} \;,
    \end{aligned}
  \end{equation}
  where the sum runs over all complete external pairings among the noise
vertices $V_o$ in $\tau_{1, \kappa_1}$ and the corresponding vertices
$\kappa_{\rm{ext}}(V_o)$ (the overall set of paired vertices does not
depend on the choice of the external pairing) in $\tau_{2,\kappa_2}$. Now
let us denote with $\mG^-_{i}$ the set of divergent diagrams in the Feynman
diagram $\tau_{i, \kappa_i}$ (we view the latter as a Feynman diagram as in the
discussion following~\eqref{e:ren-1}), and $\mF_i^-$ the set of all forests of
elements in $\mG^-_i$. Then by definition, for fixed $\kappa_{\rm{ext}}$, the
last term in~\eqref{eq_supp_propbphz1} can be rewritten as
  \begin{equation*}
    \sum_{ (\mf{F}_1, \mf{F}_2) \in \mF^-_{1} \times \mF^-_{2}} (-1)^{|\mf{F}_1^-|+|\mf{F}_2^-|} \left( \Pi_\ve \mC_{\mf{F}_1} \tau_{1,
    \kappa_1} ,  \Pi_\ve \mC_{\mf{F}_2} \tau_{2,\kappa_2} \right)_{\kappa_\mathrm{ext}} \;.
  \end{equation*}
  At this point, let $\mG^-$ indicate the set of negative divergences in $(\tau_1, \tau_2)_{\kappa}$. Then Lemma~\ref{lem:neg} implies that $\mG^- = \mG^-_{1} \sqcup \mG^-_{2}$. As a consequence, every forest $\mf{F}$ of subsets of $\mG^-$ can be decomposed uniquely into the disjoint union of two forests $\mf{F}_i \in \mF^-_{i}$, such that $\mf{F} = \mf{F}_1\sqcup \mf{F}_2$.

 It follows that
 \begin{equation*}
  \begin{aligned}
  \sum_{ (\mf{F}_1, \mf{F}_2) \in \mF^-_{1} \times \mF^-_{2}} & (-1)^{|\mf{F}_1^-|+|\mf{F}_2^-|} \left( \Pi_\ve \mC_{\mf{F}_1} \tau_{1,
    \kappa_1} ,  \Pi_\ve \mC_{\mf{F}_2} \tau_{2,\kappa_2} \right)_{\kappa_\mathrm{ext}}  \\
    & = \sum_{\mf{F} \in \mF^-} (-1)^{|\mf{F}|} \Pi_\ve \mC_{\mf{F}}  \left(\tau_{1,\kappa_2} ,  \tau_{2,\kappa_2} \right)_{\kappa_\mathrm{ext}} = \hp ( \tau_{1,\kappa_1}, \tau_{2,\kappa_2})_{\kappa_{\rm{ext}}} \;,
  \end{aligned}
 \end{equation*}
 as desired. This concludes the proof of the result.
\end{proof}

\section{Properties of Hepp sectors} \label{sec:hepp}

Throughout this section, let $ \Gamma = ( V, E ) $ be a Feynman diagram, with
inner nodes $ V_{\star}$.

\begin{lemma}\label{lem_hepp_almost_covers}
  For every $ x \in ( \mathbf{T}^{d})^{ V_{\star}}$ such that
  $ x_{v}\neq x_{w}$ for all $ v \neq w \in V_{\star}$, there exists at least one
  Hepp sector $\mathbf{T}=  (T, \mathbf{n})$ such that $ x \in D_{\mathbf{T}}$,
  as defined in~\eqref{eq_def_Hepp_dom}.
  \end{lemma}

  \begin{proof}
  Let $ x \in  ( \mathbf{T}^{d})^{V_{\star}}$ 
	  such that  $ x_{v}\neq x_{w}$ for all $ v \neq w \in V_{\star}$, and let $ \{
	  \overline{e}_{i}\}_{i}$ be an order on the edges of the (undirected) complete graph $ \mG $ with
	  vertices $
	  V_{\star} $ such that 
	  \begin{equation}\label{eq_supp1_disjoint}
		  \begin{aligned}
			  0 < | x_{ \overline{e}_{i,+}} - x_{ \overline{e}_{i,-}}|  \leqslant |
			  x_{ \overline{e}_{j , +}}- x_{ \overline{e}_{j , -}}|\,, \qquad \forall i <j \,.
		  \end{aligned}
	  \end{equation}
  Now, let $ \Sigma =( V_{\star}, \{e_{i}\}_{i=1}^{|V_{\star}| -1}) $,
  with $\{e_{i}\}_{i=1}^{|V_{\star}| -1} \subset \{ \overline{e}_{i}\}_{i} $,  be the minimal spanning tree of $ \mG $ with respect to the
  total order induced by~\eqref{eq_supp1_disjoint}, where we labelled the edges in
  the order they have been added to the spanning tree, from $ 1$ to  $|V_{\star}|-1  $.

  Then, starting from the root, we construct a tree $ T= T(\Sigma) \in
  \mT_{V_{\star}}$ as follows.
  If $|V_\star|=2$, then $T$ is simply the tree with leaves the vertices
  $V_\star$ grafted under a root labelled by the unique edge in $\Sigma$.

  For $|V_\star|>2$ we proceed by induction.
  We remove the edge $ e_{|V_{\star}|-1}$ from $ \Sigma$, (the last edge
  to be added to the minimal spanning tree). This splits
  the tree $ \Sigma $ into two subtrees $ \Sigma_{1}, \Sigma_{2}$ with disjoint vertex sets $ V_{1} \sqcup V_{2}=
  V_{\star}$. Then we define inductively $T(\Sigma) = [T(\Sigma_1),
  T(\Sigma_2)]_{e_{|V_\star|-1}}$, where the latter denotes the binary tree
  obtained by grafting the two trees $ T(\Sigma_i)$ under a new root labelled
  by $ e_{|V_{\star}|-1}$, see the following example:
  
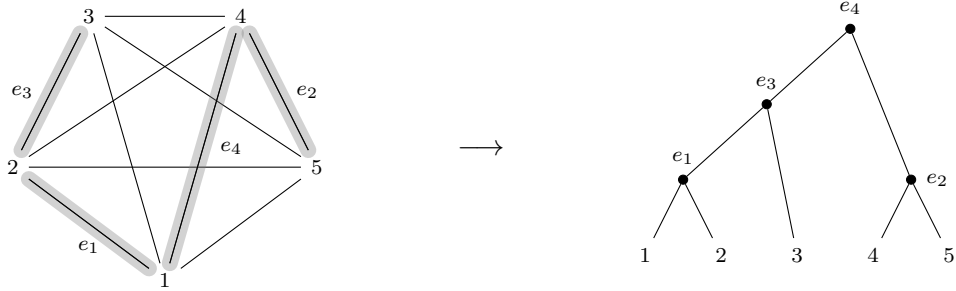
\begin{equation*}
    \begin{tikzpicture}[scale=1,
      every node/.style={font=\scriptsize},
      dot/.style={circle,fill=white!50!black,inner sep=1.4pt},
      baseline = 7ex
    ]
    \node (v1) at (-0.0,-0.5)   {1};
    \node (v2) at (-2,1)  {2};
    \node (v3) at (-1,3)  {3};
    \node (v4) at (1,3)   {4};
    \node (v5) at (2,1)   {5};
    \foreach \a/\b in {1/2,1/3,1/4,1/5,2/3,2/4,2/5,3/4,3/5,4/5}
      \draw (v\a) -- (v\b);
    \begin{scope}[line width=6pt, draw=white!50!black, line cap=round, opacity=.35]
      \draw (v1) -- node[midway, below, opacity =1.0] {\scriptsize $e_1$} (v2);
      \draw (v2) -- node[midway, left, opacity =1.0] {\scriptsize $e_3$}(v3);
      \draw (v1) -- node[midway, right, opacity =1.0] {\scriptsize $e_4$}(v4);
      \draw (v4) -- node[midway, right, opacity =1.0] {\scriptsize $e_2$}(v5);
    \end{scope}
    \foreach \a/\b in {1/2,2/3,1/4,4/5}
      \draw (v\a) -- (v\b);
  \end{tikzpicture} \qquad \qquad \longrightarrow \qquad \qquad
  \begin{tikzpicture}[scale=1,
      every node/.style={font=\scriptsize},
      baseline = 8ex
    ]
    \node (x1) at (6.5,0) {1};
    \node (x2) at (7.5,0) {2};
    \node (x3) at (8.5,0) {3};
    \node (x4) at (9.5,0) {4};
    \node (x5) at (10.5,0) {5};
    \node[dot, label=above:{$e_1$}] (e1) at (7,1)   {};
    \node[dot, label=above:{$e_3$}] (e3) at (8.1,2) {};
    \node[dot, label=above:{$e_4$}] (e4) at (9.2,3) {};
    \node[dot, label=right:{$e_2$}] (e2) at (10,1)  {};
    \draw (x1) -- (e1);
    \draw (x2) -- (e1);
    \draw (e1) -- (e3);
    \draw (x3) -- (e3);
    \draw (e3) -- (e4);
    \draw (e4) -- (e2);
    \draw (x4) -- (e2);
    \draw (x5) -- (e2);
    \end{tikzpicture}
  \end{equation*}
  Finally, having constructed the tree, we fix the compatible map  $ \mathbf{n} \in \mA(T)$ such that
  \begin{equation}\label{e:compatibility}
  \begin{aligned}
  (\sqrt{d}\pi)   2^{- \mathbf{n} (v \wedge w)-1} < | x_{v} - x_{w} |
  \leqslant (\sqrt{d}\pi)   2^{- \mathbf{n} (v \wedge w)}\,,
  \end{aligned}
  \end{equation}
  for every $ v, w \in V_{\star} $. Note that there exists an $\mathbf{n} (v \wedge w) \in \NN$ such
  that~\eqref{e:compatibility} is satisfied because the maximal distance between
 to points on the $d$-dimensional torus is $\sqrt{d} \pi$. Hence, by construction and definition of Hepp sectors
\eqref{eq_def_Hepp_dom}, we have $ x \in
  D_{( T, \mathbf{n})}$.
  This finishes the proof.
  \end{proof}
  The next result proves that most Hepp sectors are disjoint.

  \begin{lemma}\label{lem_disjoin_hepp}
  Let $ \Gamma = ( V, E ) $ be a Feynman diagram, with inner vertices $
V_{\star} $ and define $$\mA^{\#}(T):= \{\mathbf{n} \in  \mA (T)\,:\,
\mathbf{n}(u) \neq \mathbf{n}(u') \;, \forall u \neq u' \in \overline{T} \}$$ for any tree $T \in \mT_{V_\star}$. Then:
  \begin{enumerate}
  \item For every $ T \in \mT_{V_{\star}}$ and any pair $\mathbf{n}, \mathbf{n}'
  \in \mA(T)$ with $\mathbf{n} \neq \mathbf{n}'$ we have $D_{(T, \mathbf{n}) }
  \cap D_{(T, \mathbf{n}') } = \emptyset$.

  \item Moreover for $T, T' \in \mT_{V_\star}$ and
  $\mathbf{n} \in  \mA^{\#}(T), \mathbf{n}' \in  \mA^{\#}(T')$,
  we have $D_{(T, \mathbf{n}) }
  \cap D_{(T', \mathbf{n}') } = \emptyset$.

  \end{enumerate}
  \end{lemma}

  We note that the exclusion of diagonals in the second part of the lemma is
  necessary, because the constant assignment $ \mathbf{n} \equiv n \in
  \mathbf{N}$ is compatible with any
  tree $ T \in \mT_{V_{\star}}$. For example, in the case $|V_{\star}|=3$
  and $ d \geqslant 2$,
  we can find $ x \in ( \mathbf{T}^{d})^{ V_{\star}}$ such that
  \begin{equation*}
  \begin{aligned}
  | x_{v}- x_{w}| = \alpha \,, \qquad \forall v,w \in V_{\star}, v \neq w \,,
  \end{aligned}
  \end{equation*}
  for any $\alpha$ sufficiently small, by embedding an equilateral triangle in
  the torus.

  \begin{proof}[Proof of Lemma~\ref{lem_disjoin_hepp}]
  We prove the two points separately.

  \vspace{0.5em}

  \textit{Step 1.} Let $ T \in \mT_{V_{\star}}$ and assume that $ x \in D_{(T, \mathbf{n})} \cap D_{(T,
  \mathbf{n}')}$, for distinct and compatible $\mathbf{n}, \mathbf{n}'$.
  Since $ \mathbf{n} \neq \mathbf{n} '$, there exists $u \in \trim $ such that $
  \mathbf{n} (u) >
  \mathbf{n} '(u)$ (or vise versa).
  Now, let $ v, w \in V_{\star}$ such that $u = v \wedge w$. Consequently, from
  the definition~\eqref{eq_def_Hepp_dom},
  \begin{equation*}
  \begin{aligned}
  | x_{v} - x_{w} |
  \leqslant ( \sqrt{d} \pi)
  2^{- \mathbf{n}(u)}
  \leqslant
( \sqrt{d} \pi)
  2^{- \mathbf{n}'(u)-1}
  <
  | x_{v} - x_{w} |\,,
  \end{aligned}
  \end{equation*}
which is a contradiction.
  Hence, $ D_{(T, \mathbf{n})} \cap D_{(T,
  \mathbf{n}')}= \emptyset$.

  \vspace{0.5em}

  \textit{Step 2.} The case for fixed $ T$ and varying $ \mathbf{n}$ has been covered in the
  previous point.
  Consider therefore $ T \neq T' $ and
   assume there exists a triplet of points $
  v_{1}, v_{2} , v_{3} \in V_{\star} $ such that
  \begin{equation}\label{eq_triplet}
  \begin{aligned}
  ( v_{1} \wedge v_{2}) \succ_{T} ( v_{2} \wedge v_{3})
  \quad \text{but} \quad
  ( v_{2} \wedge v_{3})  \succ_{T'} ( v_{1} \wedge v_{2})\,.
  \end{aligned}
  \end{equation}
  Then,~\eqref{eq_triplet} and the assumption $\mathbf{n} \in \mA^\#(T),
  \mathbf{n}' \in \mA^\#(T')$ imply
  that every $x \in D_{(T, \mathbf{n})} \cap
  D_{(T', \mathbf{n}')}$ satisfies
  \begin{equation*}
  \begin{aligned}
  | x_{v_{1}} - x_{v_{2}}| \leqslant
  ( \sqrt{d} \pi) 2^{- \mathbf{n}( v_{1} \wedge v_{2})}
  \leqslant
  ( \sqrt{d} \pi ) 2^{-  \mathbf{n} ( v_{2} \wedge v_{3})-1}
   < | x_{v_{2}} - x_{v_{3}}| \,,
  \end{aligned}
  \end{equation*}
  as well as
  \begin{equation*}
  \begin{aligned}
  | x_{v_{1}} - x_{v_{2}}| >
  ( \sqrt{d} \pi ) 2^{- \mathbf{n}'( v_{1} \wedge v_{2})-1}
  \geqslant
  ( \sqrt{d} \pi ) 2^{-  \mathbf{n}' ( v_{2} \wedge v_{3})}
  \geqslant
   | x_{v_{2}} - x_{v_{3}}| \,.
  \end{aligned}
  \end{equation*}
  We deduce that the intersection $  D_{(T, \mathbf{n})} \cap
  D_{(T', \mathbf{n}')}$ is empty.

  Therefore, the proof is complete if we show
  that a triplet of points $
  v_{1}, v_{2} , v_{3} \in V_{\star} $ satisfying~\eqref{eq_triplet} always exists.
  To this end, we start by partitioning $ V_{\star}$ into two sets $ V_{1} ,
  V_{2}$, which correspond to the vertex sets spanned by the two subtrees $T_1,
  T_2$ such that $T = [T_1, T_2]$ is obtained by grafting these trees to the root.
  Similarly, let $ V_{1}' ,V_{2}' $ be the specular partitioning induced by $ T ' $.
  If $ V_{\alpha}= V_{\beta} '$ for some $\{\alpha, \beta\} \subseteq \{1,2\}$, we proceed in the same
  fashion with the subtrees $T_\alpha, T'_\beta$ and $T_{\alpha^c},
  T'_{\beta^c}$ (where $\{\alpha , \alpha^c\} = \{\beta, \beta^c\} = \{1,2\}$)
  until we find two subtrees $\tilde{T} = [\tilde{T}_1, \tilde{T}_2] \subseteq
  T, \tilde{T}' = [\tilde{T}'_1, \tilde{T}'_2] \subseteq T'$
  such that $\tilde{V}_\alpha = V(\tilde{T}_\alpha) \neq V(\tilde{T}'_\beta = \tilde{V}'_\beta$ for any $\{\alpha,
  \beta\} \subseteq \{1,2\}$ (here $V(\tilde{T}_\alpha)$ is the set of vertices
  spanned by the subtree $\tilde{T}_\alpha$).
  Note that $\tilde{T}, \tilde{T}'$ exist and are non-empty, since otherwise $ T
  =  T ' $.
  Without loss of generality, we assume that $ \tilde{V}_{1} \neq
  \tilde{V}_{1}'$.
  Now, let  $v_{1} \in \tilde{V}_{1}' \cap \tilde{V}_{1}$, $ v_{2} \in \tilde{V}_{1} \setminus \tilde{V}_{1}' $
  (or $ v_{2} \in \tilde{V}_{1}' \setminus
  \tilde{V}_{1}$ if $ \tilde{V}_{1} \subset \tilde{V}_{1}' $), and $ v_{3} \in
  V(\tilde{T})\setminus (
    \tilde{V}_{1} \cup \tilde{V}_{1}' ) = V(\tilde{T}') \setminus (
    \tilde{V}_{1} \cup \tilde{V}_{1}' )$. See the following diagram:
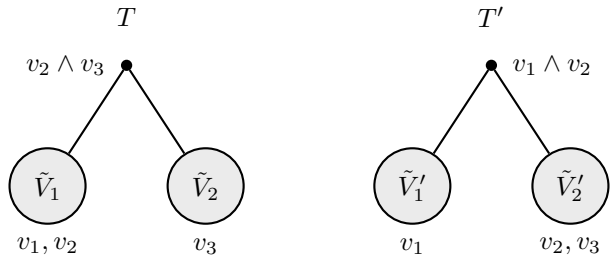
\begin{equation*}
     \begin{tikzpicture}[>=stealth, thick, scale=.8,
      every node/.style={font=\small}]
    \tikzstyle{bigset}=[circle,draw,fill=black!8,minimum size=1cm]
    \node at (-0.0,2.8) {$T$};
    \node at (-1,2) {$v_2 \wedge v_3$};
    \node[dot] (r) at (0,2) {};
    \node[bigset]  (V1) at (-1.3,0) {$\tilde{V}_1$};
    \node[bigset]  (V2) at ( 1.3,0) {$\tilde{V}_2$};
    \draw (r) -- (V1);
    \draw (r) -- (V2);
    \node at (-1.3,-1) {$v_1,v_2$};
    \node at ( 1.3,-1) {$v_3$};
    \begin{scope}[xshift=6cm]
    \node at (0.0,2.8) {$T'$};
    \node at (1,2) {$v_1 \wedge v_2$};
    \node[dot] (r2) at (0,2) {};
    \node[bigset]   (V1p) at (-1.3,0) {$\tilde{V}_1'$};
    \node[bigset]   (V2p) at ( 1.3,0) {$\tilde{V}_2'$};
    \draw (r2) -- (V1p);
    \draw (r2) -- (V2p);
    \node at (-1.3,-1) {$v_1$};
    \node at ( 1.3,-1) {$v_2,v_3$};
    \end{scope}
    \end{tikzpicture}
  \end{equation*}
  By construction, the points $ v_{1}, v_{2}, v_{3} \in V_{\star} $ satisfy
 ~\eqref{eq_triplet}. This completes the proof of the result.
  \end{proof}

\section{Taylor expansions of kernels}

A key ingredient in estimates connected to the BPHZ theorem is
a suitably stated Taylor expansion for the kernels involved in Feynman diagrams.

\begin{lemma}\label{lem_taylor_new}
For any $\ve \in (0, 1)$, let $ K_{ \varepsilon}: \mathbf{T}^{d} \to \mathbf{R}$ be a
function satisfying, for every $\ell \in \mathbf{N}^{d}$ with $| \ell
|=k$ and uniformly over $\ve\in (0, 1)$:
\begin{equation*}
\begin{aligned}
|D^{\ell}K_{ \varepsilon} (x) |\lesssim ( |x | + \varepsilon)^{
 \mathfrak{t}- k}\,, \quad \text{ for some }\ \mathfrak{t}<0\,.
\end{aligned}
\end{equation*}
Let $ x, x_{\star},  y \in \mathbf{T}^{d}$ such that $| x- x_{\star}| <  | x- y
|, | x_{\star} - y |$.
Then for every $N \in \mathbf{N}$
\begin{equation}\label{eq_taylor}
\begin{aligned}
\Big|
K_{\varepsilon}( x- y)
- \sum_{| \ell | < N} \frac{1}{\ell !}
( x - x_{\star})^{\ell}
( D^{\ell}K_{ \varepsilon}) ( x_{\star}- y)
\Big|
\lesssim_{N} \frac{ | x- x_{\star}|^{N} \big(|
x- y| +
\varepsilon\big)^{ \mathfrak{t}}}{( \min\{ | x- y
|, | x_{\star} - y | \}+ \varepsilon)^{N}} \,.
\end{aligned}
\end{equation}
\end{lemma}

The proof of this estimate follows the same steps of the proof of
\cite[Lemma~3.8]{BPHZ}, while explicitly keeping track of $ \ve$ in the upper bound.


\end{document}